\documentclass[10pt,reqno,oneside]{amsproc}

\allowdisplaybreaks

\usepackage{marginnote}

  \chardef\forshowkeys=0
  \chardef\refcheck=0
  \chardef\showllabel=0
  \chardef\sketches=0

\usepackage{color}

\ifnum\forshowkeys=1
  \definecolor{mygray}{rgb}{.6, .6, .6}
  
  \usepackage[notref,notcite,color]{showkeys}
\fi

\usepackage{times}
\usepackage{graphicx,bbm}
\usepackage[usenames,dvipsnames,svgnames,table]{xcolor}
\usepackage[colorlinks=true, pdfstartview=FitV, linkcolor=blue, citecolor=blue, urlcolor=blue]{hyperref}

\usepackage{tikz}

\usepackage{xcolor}

\usepackage{datetime}
\usepackage{fancyhdr,lastpage}

\ifnum \refcheck=1 
  \usepackage{refcheck}    
\fi

\chardef\coloryes=1 
\chardef\isitdraft=0 

\ifnum\isitdraft=1 
   \def\eqref#1{({\ref{#1}})}                

\definecolor{mygray}{rgb}{.6, .6, .6}

  \def\nnewpage{} 
\else

  \def\startnewsection#1#2{\section{#1}\label{#2}\setcounter{equation}{0}}   
\def\nnewpage{} 

\fi

\begin{document}

\def\nto#1{{\colC \footnote{\em \colC #1}}}
\def\fractext#1#2{{#1}/{#2}}
\def\fracsm#1#2{{\textstyle{\frac{#1}{#2}}}}   
\def\nnonumber{}
\def\les{\lesssim}

\def\colr{{}}
\def\colg{{}}
\def\colb{{}}
\def\colu{{}}
\def\cole{{}}
\def\colA{{}}
\def\colB{{}}
\def\colC{{}}
\def\colD{{}}
\def\colE{{}}
\def\colF{{}}

\ifnum\coloryes=1

  \definecolor{coloraaaa}{rgb}{0.1,0.2,0.8}
  \definecolor{colorbbbb}{rgb}{0.1,0.7,0.1}
  \definecolor{colorcccc}{rgb}{0.8,0.3,0.9}
  \definecolor{colordddd}{rgb}{0.0,.5,0.0}
  \definecolor{coloreeee}{rgb}{0.8,0.3,0.9}
  \definecolor{colorffff}{rgb}{0.8,0.9,0.9}
  \definecolor{colorgggg}{rgb}{0.5,0.0,0.4}
  \definecolor{coloroooo}{rgb}{0.45,0.0,0}

 \def\colb{\color{black}}

 \def\colr{\color{red}}
\def\cole{\color{black}}
\def\coly{\color{gray}}

 \def\colu{\color{blue}}
 \def\colg{\color{colordddd}}
 \def\colgray{\color{colorffff}}

 \def\colA{\color{coloraaaa}}
 \def\colB{\color{colorbbbb}}
 \def\colC{\color{colorcccc}}
 \def\colD{\color{colordddd}}
 \def\colE{\color{coloreeee}}
 \def\colF{\color{colorffff}}
 \def\colG{\color{colorgggg}}

\fi
\ifnum\isitdraft=1
   \chardef\coloryes=1 
   \baselineskip=17.6pt
\pagestyle{myheadings}
\def\const{\mathop{\rm const}\nolimits}  
\def\diam{\mathop{\rm diam}\nolimits}    
\def\rref#1{{\ref{#1}{\rm \tiny \fbox{\tiny #1}}}}
\def\theequation{\fbox{\bf \thesection.\arabic{equation}}}
\setcounter{equation}{0}
\pagestyle{fancy}
\lhead{\colb Section~\ref{#2}, #1 }
\cfoot{}
\rfoot{\thepage\ of \pageref{LastPage}}
\lfoot{\colb{\today,~\currenttime}~(jkl2)}}

\chead{}
\rhead{\thepage}
\def\nnewpage{\newpage}
\newcounter{startcurrpage}
\newcounter{currpage}
\def\llll#1{{\rm\tiny\fbox{#1}}}
   \def\blackdot{{\color{red}{\hskip-.0truecm\rule[-1mm]{4mm}{4mm}\hskip.2truecm}}\hskip-.3truecm}
   \def\bluedot{{\colC {\hskip-.0truecm\rule[-1mm]{4mm}{4mm}\hskip.2truecm}}\hskip-.3truecm}
   \def\purpledot{{\colA{\rule[0mm]{4mm}{4mm}}\colb}}
   \def\pdot{\purpledot}
\else
   \baselineskip=12.8pt
   \def\blackdot{{\color{red}{\hskip-.0truecm\rule[-1mm]{4mm}{4mm}\hskip.2truecm}}\hskip-.3truecm}
   \def\purpledot{{\rule[-3mm]{8mm}{8mm}}}
   \def\pdot{}
\fi

\def\out#1{}
\def\inc{\mathrm{inc}}
\def\pp{p}
\def\qq{{\tilde p}}
\def\MM{M}
\def\ema#1{{#1}}
\def\emb#1{#1}

\ifnum\isitdraft=1
  \def\llabel#1{\nonumber}
\else
  \def\llabel#1{\nonumber}
  \def\llabel#1{\label{#1}}
\fi

\def\tepsilon{\tilde\epsilon}
\def\epsilonz{\epsilon_0}
\def\restr{\bigm|}
\def\into{\int_{\Omega}}
\def\intu{\int_{\Gamma_1}}
\def\intl{\int_{\Gamma_0}}
\def\tpar{\tilde\partial}
\def\bpar{\,|\nabla_2|}
\def\barpar{\bar\partial}
\def\FF{F}
\def\gdot{{\color{green}{\hskip-.0truecm\rule[-1mm]{4mm}{4mm}\hskip.2truecm}}\hskip-.3truecm}
\def\bdot{{\color{blue}{\hskip-.0truecm\rule[-1mm]{4mm}{4mm}\hskip.2truecm}}\hskip-.3truecm}
\def\cydot{{\color{cyan} {\hskip-.0truecm\rule[-1mm]{4mm}{4mm}\hskip.2truecm}}\hskip-.3truecm}
\def\rdot{{\color{red} {\hskip-.0truecm\rule[-1mm]{4mm}{4mm}\hskip.2truecm}}\hskip-.3truecm}

\def\tdot{\fbox{\fbox{\bf\color{blue}\tiny I'm here; \today \ \currenttime}}}
\def\nts#1{{\color{red}\hbox{\bf ~#1~}}} 

\def\ntsr#1{\vskip.0truecm{\color{red}\hbox{\bf ~#1~}}\vskip0truecm} 

\def\ntsf#1{\footnote{\hbox{\bf ~#1~}}} 
\def\ques#1{{\hbox{\bf\colr ~#1~}}} 
\def\ntsfred#1{\footnote{\color{red}\hbox{\bf ~#1~}}} 
\def\bigline#1{~\\\hskip2truecm~~~~{#1}{#1}{#1}{#1}{#1}{#1}{#1}{#1}{#1}{#1}{#1}{#1}{#1}{#1}{#1}{#1}{#1}{#1}{#1}{#1}{#1}\\}
\def\biglineb{\bigline{$\downarrow\,$ $\downarrow\,$}}
\def\biglinem{\bigline{---}}
\def\biglinee{\bigline{$\uparrow\,$ $\uparrow\,$}}
\def\ceil#1{\lceil #1 \rceil}
\def\gdot{{\color{green}{\hskip-.0truecm\rule[-1mm]{4mm}{4mm}\hskip.2truecm}}\hskip-.3truecm}
\def\bluedot{{\color{blue} {\hskip-.0truecm\rule[-1mm]{4mm}{4mm}\hskip.2truecm}}\hskip-.3truecm}
\def\rdot{{\color{red} {\hskip-.0truecm\rule[-1mm]{4mm}{4mm}\hskip.2truecm}}\hskip-.3truecm}
\def\dbar{\bar{\partial}}
\newtheorem{Theorem}{Theorem}[section]
\newtheorem{Corollary}[Theorem]{Corollary}
\newtheorem{Proposition}[Theorem]{Proposition}
\newtheorem{Lemma}[Theorem]{Lemma}
\newtheorem{Remark}[Theorem]{Remark}
\newtheorem{definition}{Definition}[section]
\def\theequation{\thesection.\arabic{equation}}
\def\cmi#1{{\color{red}IK: #1}}
\def\cmj#1{{\color{red}IK: #1}}
\def\cml{\rm \colg Linfeng:~} 

\def\KK{K}
\def\sqrtg{\sqrt{g}}
\def\DD{{\mathcal D}}
\def\OO{\tilde\Omega}
\def\EE{{\mathcal E}}
\def\lot{{\rm l.o.t.}}                       
\def\endproof{\hfill$\Box$\\}
\def\square{\hfill$\Box$\\}
\def\inon#1{\ \ \ \ \text{~~~~~~#1}}                
\def\comma{ {\rm ,\qquad{}} }            
\def\commaone{ {\rm ,\qquad{}} }         
\def\dist{\mathop{\rm dist}\nolimits}    
\def\sgn{\mathop{\rm sgn\,}\nolimits}    
\def\Tr{\mathop{\rm Tr}\nolimits}    
\def\dive{\mathop{\rm div}\nolimits}    
\def\grad{\mathop{\rm grad}\nolimits}    
\def\curl{\mathop{\rm curl}\nolimits}    
\def\det{\mathop{\rm det}\nolimits}    
\def\supp{\mathop{\rm supp}\nolimits}    
\def\re{\mathop{\rm {\mathbb R}e}\nolimits}    
\def\wb{\bar{\omega}}
\def\Wb{\bar{W}}
\def\indeq{\quad{}}                     
\def\indeqtimes{\indeq\indeq\indeq\indeq\times} 
\def\period{.}                           
\def\semicolon{\,;}                      
\newcommand{\cD}{\mathcal{D}}
\newcommand{\cH}{\mathcal{H}}
\newcommand{\imp}{\Rightarrow}
\newcommand{\tr}{\operatorname{tr}}
\newcommand{\vol}{\operatorname{vol}}
\newcommand{\id}{\operatorname{id}}
\newcommand{\p}{\parallel}
\newcommand{\norm}[1]{\Vert#1\Vert}
\newcommand{\abs}[1]{\vert#1\vert}
\newcommand{\nnorm}[1]{\left\Vert#1\right\Vert}
\newcommand{\aabs}[1]{\left\vert#1\right\vert}

\ifnum\showllabel=1
 \def\llabel#1{\marginnote{\color{lightgray}\rm\small(#1)}[-0.0cm]\notag}
\else
 \def\llabel#1{\notag}
\fi


\title[Mach Limits in Analytic Spaces]{Mach Limits in Analytic Spaces} \par \author[J.~Jang]{Juhi Jang}  \address{Department of Mathematics\\ University of Southern California\\ Los Angeles, CA 90089} \email{juhijang@usc.edu} \par \author[I.~Kukavica]{Igor Kukavica} \address{Department of Mathematics\\ University of Southern California\\ Los Angeles, CA 90089} \email{kukavica@usc.edu} \par \author[L.~Li]{Linfeng Li} \address{Department of Mathematics\\ University of Southern California\\ Los Angeles, CA 90089} \email{lli265@usc.edu} \par \begin{abstract} We address the Mach limit problem for the Euler equations in the analytic spaces. We prove that, given analytic data, the solutions to the compressible Euler equations are uniformly bounded in a suitable analytic norm and then show that the convergence toward the incompressible Euler solution holds in the analytic norm.  We also show that the same results hold more generally for Gevrey data with the convergence  in the Gevrey norms. \end{abstract} \par \maketitle \par \startnewsection{Introduction}{sec01}  \par The low Mach number limit problem, which concerns the passage from slightly compressible flows to incompressible flows, is a classical singular limit problem in mathematical fluid dynamics.  The problem has both physical and mathematical importance.  There have been many significant works on the subject and a great deal of progress made in recent decades \cite{A05, A08, Asa87, Ebin77, FKM, Igu97, Iso1, Iso2, Iso3, KM81, KM82, MS01, Sch86, Sch05, Uka}.  The main difficulty of the problem is the presence of different wave speeds, which play a significant role in the limit process.  In particular, one has to address the vanishing of the acoustic waves in the limit.  A study of the low Mach number limit involves two parts: the uniform bounds and existence of slightly compressible flows for a time-independent of Mach numbers and convergence to solutions of the limiting equations.  Interestingly, the analysis of such a singular limit problem significantly changes depending whether compressible fluids are isentropic or non-isentropic, if compressible fluids are inviscid or viscous, if initial data are well-prepared or not, if the problem is set in the whole space or domains with boundaries, or which regularity space of data is considered. In this paper, we address the low Mach number limit of the non-isentropic compressible Euler flows in $\mathbb R^3$ in analytic and, more generally, in Gevrey spaces. \par Before describing the results, we briefly review prior relevant works (cf.~\cite{A05,A08,MS01} for more extensive reviews).  For isentropic flows or well-prepared initial data, it is well-known that solutions of the compressible Euler equations with low Mach numbers exist in Sobolev spaces for a time interval independent of the Mach numbers \cite{KM81, KM82, Sch86}. When initial data are well-prepared, solutions converge to the solutions of the corresponding incompressible Euler equations with the limiting initial data~\cite{KM81, KM82, Sch86}. For the isentropic flows with general initial data, the convergence is not uniform for times close to zero and initial layers are present~\cite{Asa87, Uka, Igu97, Iso1, Iso2, Iso3}. On the other hand, the non-isentropic problem with general initial data is much more involved. In this case, the pressure depends not only on the density but also on the entropy that enters into the coefficients of the linearized equations, and the convergence is more subtle because the acoustic waves are governed by a wave equation with variable coefficients. The first existence and convergence of the non-isentropic problem were given in \cite{MS01} and the existence result for general domains with boundary and the convergence result for exterior domains were obtained in \cite{A05}. The results above were obtained in Sobolev spaces. Recently the low Mach number limit was studied in \cite{FKM} starting from dissipative measure-valued solutions of the isentropic Euler equations.  Also, the Mach limit in the domains with evolving boundary was addressed in~\cite{DE,DL}, while for the dissipative case, see~\cite{A06,D1,D2,DG,DM,F,FN,H,LM,M}.  For other works on analyticity for the equations involving fluids, see \cite{B,BB,BGK,CKV,LO}, while for different approaches to analyticity, cf.~\cite{Bi,BF,BoGK,FT,G,GK,KP,OT}. \par This paper concerns the non-isentropic equations with general analytic or Gevrey initial data in $\mathbb{R}^3$ and convergence holding in these strong norms. The first result provides a uniform in $\epsilon$ bound of the analytic solution, where $\epsilon >0$ represents the Mach number, while the second result asserts the convergence of the solution to the limiting equation as $\epsilon $ tends to zero.  The main difficulty is in obtaining the uniform analytic bound. The Mach limit in an analytic norm is then proven by interpolating the uniform boundedness result and the convergence in the Sobolev space due to M{\'e}tivier and Schochet in~\cite{MS01}. \par For the isentropic case, the standard energy estimate method can be applied to the velocity equation to obtain analytic estimates.  However, for the non-isentropic case, the problem is more difficult since the matrix, $E(\epsilon u^\epsilon, S^\epsilon)$ (cf.~the formulation \eqref{DFGRTHVBSDFRGDFGNCVBSDFGDHDFHDFNCVBDSFGSDFGDSFBDVNCXVBSDFGSDFHDFGHDFTSADFASDFSADFASDXCVZXVSDGHFDGHVBCX01}--\eqref{DFGRTHVBSDFRGDFGNCVBSDFGDHDFHDFNCVBDSFGSDFGDSFBDVNCXVBSDFGSDFHDFGHDFTSADFASDFSADFASDXCVZXVSDGHFDGHVBCX02}), also depends on $S^\epsilon$, and thus spatial derivative bounds cannot be obtained solely by the fundamental energy estimates.  Moreover, the non-isentropic Euler flows feature intriguing wave-transport structure: The divergence component of the modified velocity is governed by nonlinear acoustic equations, while the curl component and entropy are transported, and their interactions are coupled. Thus a careful analysis that captures the coupled structure of the modified velocity and the entropy is required.  To accomplish these, we use the elliptic regularity for the velocity to reduce the spatial derivative to divergence and curl components.  The key to the former is that the divergence equation for the velocity is properly balanced with the analytic energy solution, which motivates us to include time derivatives using $\epsilon \partial_t$ to our analytic norm; for the latter, we appeal to the transport equation of the curl component, which can be treated in a similar way as the entropy.  Thus, the pure time analytic norm needs to be treated differently than the one which also involves the spatial derivatives (cf.~Sections~\ref{sec6.2} and~\ref{sec6.3} respectively).  It is important to include the analytic weight $\kappa$ in \eqref{DFGRTHVBSDFRGDFGNCVBSDFGDHDFHDFNCVBDSFGSDFGDSFBDVNCXVBSDFGSDFHDFGHDFTSADFASDFSADFASDXCVZXVSDGHFDGHVBCX366}, which ultimately balances the time and the spatial derivatives.  The main difficulty in our approach is the handling of the vorticity~$\omega$, which can not be treated directly. Instead, as in \cite{A05}, we need to consider the equation for the modified vorticity $\curl (r_0 v)$, where $r_0$ is a certain function of the entropy (cf.~Section~\ref{sec6.1} below). The product and chain rules then lead to complicated analytic coupling among the entropy, divergence, vorticity, and $\curl(r_0 v)$. \par The paper is organized as follows.  In Section~\ref{sec01a}, we introduce the Mach number limit problem and then formulate the symmetrized version of the compressible Euler equations.  In Section~\ref{sec02}, we define the analytic norm and state the main results.  The first theorem relies on Lemma~\ref{L01}, the proof of which is given at the end of Section~\ref{sec05}.  We present the energy estimate for the transport equation in Section~\ref{sec03}.  Product rule and chain rules in analytic spaces are provided in Section~\ref{sec04}.  In Section~\ref{sec05}, we estimate the curl, divergence, and time-derivative components of the velocity. In Section~\ref{sec06}, we prove the convergence theorem. In Section~\ref{secinitial}, we establish the finiteness of the space-time analytic norm at the initial time under the assumption that the initial data is real-analytic in the spatial variable. In Section~\ref{secgevrey}, we provide the Mach limit theorem in any Gevrey space. \par \startnewsection{Set-up}{sec01a}  \par We consider the compressible Euler equations describing  the motion of an inviscid, non-isentropic gaseous fluid in $\mathbb R^3$   \begin{align}    &\partial_t \rho + v\cdot \nabla \rho + \rho \nabla\cdot v =0\label{Euler1}     \\&    \rho\left( \partial_t v +  v\cdot \nabla v\right) + \nabla P =0 \label{Euler2}    \\&    \partial_t S + v\cdot \nabla S =0 \label{Euler3}    ,   \end{align} where $\rho=\rho(x,t)\in \mathbb R_+$ is the density, $v=v(x,t)\in \mathbb R^3$ is the velocity, $P=P(x,t)\in \mathbb R_+$ is the pressure, and $S=S(x,t)\in \mathbb R$ is the entropy of the fluid.  The system \eqref{Euler1}--\eqref{Euler3} is closed with the equation of state    \begin{equation}    P=P(\rho, S)    .    \label{DFGRTHVBSDFRGDFGNCVBSDFGDHDFHDFNCVBDSFGSDFGDSFBDVNCXVBSDFGSDFHDFGHDFTSADFASDFSADFASDXCVZXVSDGHFDGHVBCX138}   \end{equation} For instance, the equation of state for an ideal gas takes the form   \begin{equation}    P(\rho, S)=\rho^\gamma e^{{S}}      ,    \label{DFGRTHVBSDFRGDFGNCVBSDFGDHDFHDFNCVBDSFGSDFGDSFBDVNCXVBSDFGSDFHDFGHDFTSADFASDFSADFASDXCVZXVSDGHFDGHVBCX15}   \end{equation} where $\gamma>1$ is the adiabatic exponent.  \par To address the low Mach number limit, we introduce the rescalings    \begin{equation}   \tilde t = \epsilon t   \comma   \tilde x = x   \comma \tilde\rho = \rho   \comma \tilde v = \frac{ v}{\epsilon}   \comma \tilde P =P   \comma \tilde S = S   ,    \llabel{8Th sw ELzX U3X7 Ebd1Kd Z7 v 1rN 3Gi irR XG KWK0 99ov BM0FDJ Cv k opY NQ2 aN9 4Z 7k0U nUKa mE3OjU 8D F YFF okb SI2 J9 V9gV lM8A LWThDP nP u 3EL 7HP D2V Da ZTgg zcCC mbvc70 qq P cC9 mt6 0og cr TiA3 HEjw TK8ymK eu J Mc4 q6d Vz2 00 XnYU tLR9 GYjPXv FO V r6W 1zU K1W bP ToaW JJuK nxBLnd 0f t DEb Mmj 4lo HY yhZy MjM9 1zQS4p 7z 8 eKa 9h0 Jrb ac ekci rexG 0z4n3x z0 Q OWS vFj 3jL hW XUIU 21iI AwJtI3 Rb W a90 I7r zAI qI 3UEl UJG7 tLtUXz w4 K QNE TvX zqW au jEMe nYlN IzLGxg B3 A uJ8 6VS 6Rc PJ 8OXW w8im tcKZEz Ho p 84G 1gS As0 PC owMI 2fLK TdD60y nH g 7lk NFj JLq Oo Qvfk fZBN G3o1Dg Cn 9 hyU h5V SP5 z6 1qvQ wceU dVJJsB vX D G4E LHQ HIa PT bMTr sLsm tXGyOB 7p 2 Os4 3US bq5 ik 4Lin 769O TkUxmp I8 u GYn fBK bYI 9A QzCF w3h0 geJftZ ZK U 74r Yle ajm km ZJdi TGHO OaSt1N nl B 7Y7 h0y oWJ ry rVrT zHO8 2S7oub QA W x9d z2X YWB e5 Kf3A LsUF vqgtM2 O2 I dim rjZ 7RN 28 4KGY trVa WW4nTZ XV b RVo Q77 hVL X6 K2kq FWFm aZnsF9 Ch p 8Kx rsc SGP iS tVXB J3xZ cD5IP4 Fu 9 Lcd TR2 Vwb cL DlGK 1ro3 EEyqEA zw 6 sKe Eg2 sFf jz MtrZ 9kbd xNw66c xf t lzD GZh xQA WQ KkSX jqmm rEpNuG 6P y loq 8hH lSf Ma LXm5 RzEX W4Y1Bq ib 3 UOh Yw9 5h6 f6 o8kw 6frZ wg6fIy XP n ae1 TQJ Mt2 TT fWWf jJrX ilpYGr Ul Q 4uM 7Ds p0r Vg 3gIE mQOz TFh9LA KO 8 csQ u6m h25 r8 WqRI DZWg SYkWDu lL 8 Gpt ZW1 0Gd SY FUXL zyDFGRTHVBSDFRGDFGNCVBSDFGDHDFHDFNCVBDSFGSDFGDSFBDVNCXVBSDFGSDFHDFGHDFTSADFASDFSADFASDXCVZXVSDGHFDGHVBCX135}   \end{equation} where $\epsilon>0$ represents the Mach number, the ratio of the typical fluid speed to the typical sound speed. We assume that the typical sound speed is $O(1)$.  For simplicity of notation, we omit  tilde, and obtain the rescaled system    \begin{align}     &\partial_t \rho + v\cdot \nabla \rho + \rho \nabla\cdot v =0          \label{Euler11}      \\& \rho\left( \partial_t v +  v\cdot \nabla v\right) + \frac{1}{\epsilon^2}\nabla P =0           \label{Euler21}      \\& \partial_t S + v\cdot \nabla S =0          \label{Euler31}      .   \end{align} The goal of this paper is to obtain the low Mach number limit of \eqref{Euler11}--\eqref{Euler31} in analytic spaces.  \par \subsection{Reformulation} \label{sec21} Now,  consider $P$, instead of $\rho$, as an independent variable, we may write \eqref{DFGRTHVBSDFRGDFGNCVBSDFGDHDFHDFNCVBDSFGSDFGDSFBDVNCXVBSDFGSDFHDFGHDFTSADFASDFSADFASDXCVZXVSDGHFDGHVBCX138} as   \begin{equation}    \rho=\rho(P,S)    ,   \llabel{QZ hVZMn9 am P 9aE Wzk au0 6d ZghM ym3R jfdePG ln 8 s7x HYC IV9 Hw Ka6v EjH5 J8Ipr7 Nk C xWR 84T Wnq s0 fsiP qGgs Id1fs5 3A T 71q RIc zPX 77 Si23 GirL 9MQZ4F pi g dru NYt h1K 4M Zilv rRk6 B4W5B8 Id 3 Xq9 nhx EN4 P6 ipZl a2UQ Qx8mda g7 r VD3 zdD rhB vk LDJo tKyV 5IrmyJ R5 e txS 1cv EsY xG zj2T rfSR myZo4L m5 D mqN iZd acg GQ 0KRw QKGX g9o8v8 wm B fUu tCO cKc zz kx4U fhuA a8pYzW Vq 9 Sp6 CmA cZL Mx ceBX Dwug sjWuii Gl v JDb 08h BOV C1 pni6 4TTq Opzezq ZB J y5o KS8 BhH sd nKkH gnZl UCm7j0 Iv Y jQE 7JN 9fd ED ddys 3y1x 52pbiG Lc a 71j G3e uli Ce uzv2 R40Q 50JZUB uK d U3m May 0uo S7 ulWD h7qG 2FKw2T JX z BES 2Jk Q4U Dy 4aJ2 IXs4 RNH41s py T GNh hk0 w5Z C8 B3nU Bp9p 8eLKh8 UO 4 fMq Y6w lcA GM xCHt vlOx MqAJoQ QU 1 e8a 2aX 9Y6 2r lIS6 dejK Y3KCUm 25 7 oCl VeE e8p 1z UJSv bmLd Fy7ObQ FN l J6F RdF kEm qM N0Fd NZJ0 8DYuq2 pL X JNz 4rO ZkZ X2 IjTD 1fVt z4BmFI Pi 0 GKD R2W PhO zH zTLP lbAE OT9XW0 gb T Lb3 XRQ qGG 8o 4TPE 6WRc uMqMXh s6 x Ofv 8st jDi u8 rtJt TKSK jlGkGw t8 n FDx jA9 fCm iu FqMW jeox 5Akw3w Sd 8 1vK 8c4 C0O dj CHIs eHUO hyqGx3 Kw O lDq l1Y 4NY 4I vI7X DE4c FeXdFV bC F HaJ sb4 OC0 hu Mj65 J4fa vgGo7q Y5 X tLy izY DvH TR zd9x SRVg 0Pl6Z8 9X z fLh GlH IYB x9 OELo 5loZ x4wag4 cn F aCE KfA 0uz fw HMUV M9Qy eARFe3 Py 6 kQG GFx rPf 6T ZBQR la1a 6Aeker Xg kEoSa}   \end{equation} and \eqref{Euler11} is then replaced by   \begin{equation}    A_0 \left( \partial_t P + v\cdot \nabla P \right)+ \nabla\cdot v =0     ,   \label{Euler111}   \end{equation} where    \begin{equation}    A_0=A_0(S,P)=\frac{1}{\rho(S,P)}\frac{\partial \rho(S,P)}{\partial P}    .    \llabel{ blz nSm mhY jc z3io WYjz h33sxR JM k Dos EAA hUO Oz aQfK Z0cn 5kqYPn W7 1 vCT 69a EC9 LD EQ5S BK4J fVFLAo Qp N dzZ HAl JaL Mn vRqH 7pBB qOr7fv oa e BSA 8TE btx y3 jwK3 v244 dlfwRL Dc g X14 vTp Wd8 zy YWjw eQmF yD5y5l DN l ZbA Jac cld kx Yn3V QYIV v6fwmH z1 9 w3y D4Y ezR M9 BduE L7D9 2wTHHc Do g ZxZ WRW Jxi pv fz48 ZVB7 FZtgK0 Y1 w oCo hLA i70 NO Ta06 u2sY GlmspV l2 x y0X B37 x43 k5 kaoZ deyE sDglRF Xi 9 6b6 w9B dId Ko gSUM NLLb CRzeQL UZ m i9O 2qv VzD hz v1r6 spSl jwNhG6 s6 i SdX hob hbp 2u sEdl 95LP AtrBBi bP C wSh pFC CUa yz xYS5 78ro f3UwDP sC I pES HB1 qFP SW 5tt0 I7oz jXun6c z4 c QLB J4M NmI 6F 08S2 Il8C 0JQYiU lI 1 YkK oiu bVt fG uOeg Sllv b4HGn3 bS Z LlX efa eN6 v1 B6m3 Ek3J SXUIjX 8P d NKI UFN JvP Ha Vr4T eARP dXEV7B xM 0 A7w 7je p8M 4Q ahOi hEVo Pxbi1V uG e tOt HbP tsO 5r 363R ez9n A5EJ55 pc L lQQ Hg6 X1J EW K8Cf 9kZm 14A5li rN 7 kKZ rY0 K10 It eJd3 kMGw opVnfY EG 2 orG fj0 TTA Xt ecJK eTM0 x1N9f0 lR p QkP M37 3r0 iA 6EFs 1F6f 4mjOB5 zu 5 GGT Ncl Bmk b5 jOOK 4yny My04oz 6m 6 Akz NnP JXh Bn PHRu N5Ly qSguz5 Nn W 2lU Yx3 fX4 hu LieH L30w g93Xwc gj 1 I9d O9b EPC R0 vc6A 005Q VFy1ly K7 o VRV pbJ zZn xY dcld XgQa DXY3gz x3 6 8OR JFK 9Uh XT e3xY bVHG oYqdHg Vy f 5kK Qzm mK4 9x xiAp jVkw gzJOdE 4v g hAv 9bV IHe wc Vqcb SUcF 1pHzol Nj T l1B urc Sam IDFGRTHVBSDFRGDFGNCVBSDFGDHDFHDFNCVBDSFGSDFGDSFBDVNCXVBSDFGSDFHDFGHDFTSADFASDFSADFASDXCVZXVSDGHFDGHVBCX168}   \end{equation} The equation of state for an ideal gas in \eqref{DFGRTHVBSDFRGDFGNCVBSDFGDHDFHDFNCVBDSFGSDFGDSFBDVNCXVBSDFGSDFHDFGHDFTSADFASDFSADFASDXCVZXVSDGHFDGHVBCX15} then reads as   \begin{equation}    \rho(P,S)= P^{\frac{1}{\gamma}} e^{-\frac{S}{\gamma }}       .    \llabel{P zkUS 8wwS a7wVWR 4D L VGf 1RF r59 9H tyGq hDT0 TDlooa mg j 9am png aWe nG XU2T zXLh IYOW5v 2d A rCG sLk s53 pW AuAy DQlF 6spKyd HT 9 Z1X n2s U1g 0D Llao YuLP PB6YKo D1 M 0fi qHU l4A Ia joiV Q6af VT6wvY Md 0 pCY BZp 7RX Hd xTb0 sjJ0 Beqpkc 8b N OgZ 0Tr 0wq h1 C2Hn YQXM 8nJ0Pf uG J Be2 vuq Duk LV AJwv 2tYc JOM1uK h7 p cgo iiK t0b 3e URec DVM7 ivRMh1 T6 p AWl upj kEj UL R3xN VAu5 kEbnrV HE 1 OrJ 2bx dUP yD vyVi x6sC BpGDSx jB C n9P Fiu xkF vw 0QPo fRjy 2OFItV eD B tDz lc9 xVy A0 de9Y 5h8c 7dYCFk Fl v WPD SuN VI6 MZ 72u9 MBtK 9BGLNs Yp l X2y b5U HgH AD bW8X Rzkv UJZShW QH G oKX yVA rsH TQ 1Vbd dK2M IxmTf6 wE T 9cX Fbu uVx Cb SBBp 0v2J MQ5Z8z 3p M EGp TU6 KCc YN 2BlW dp2t mliPDH JQ W jIR Rgq i5l AP gikl c8ru HnvYFM AI r Ih7 Ths 9tE hA AYgS swZZ fws19P 5w e JvM imb sFH Th CnSZ HORm yt98w3 U3 z ant zAy Twq 0C jgDI Etkb h98V4u o5 2 jjA Zz1 kLo C8 oHGv Z5Ru Gwv3kK 4W B 50T oMt q7Q WG 9mtb SIlc 87ruZf Kw Z Ph3 1ZA Osq 8l jVQJ LTXC gyQn0v KE S iSq Bpa wtH xc IJe4 SiE1 izzxim ke P Y3s 7SX 5DA SG XHqC r38V YP3Hxv OI R ZtM fqN oLF oU 7vNd txzw UkX32t 94 n Fdq qTR QOv Yq Ebig jrSZ kTN7Xw tP F gNs O7M 1mb DA btVB 3LGC pgE9hV FK Y LcS GmF 863 7a ZDiz 4CuJ bLnpE7 yl 8 5jg Many Thanks, POL OG EPOe Mru1 v25XLJ Fz h wgE lnu Ymq rX 1YKV Kvgm MK7gI4 6h 5 kZB OoJ tfC 5g VvADFGRTHVBSDFRGDFGNCVBSDFGDHDFHDFNCVBDSFGSDFGDSFBDVNCXVBSDFGSDFHDFGHDFTSADFASDFSADFASDXCVZXVSDGHFDGHVBCX139}   \end{equation} To symmetrize the Euler equations, we set    \begin{equation}      P= \bar P e^{\epsilon p}      ,    \llabel{1 kNJr 2o7om1 XN p Uwt CWX fFT SW DjsI wuxO JxLU1S xA 5 ObG 3IO UdL qJ cCAr gzKM 08DvX2 mu i 13T t71 Iwq oF UI0E Ef5S V2vxcy SY I QGr qrB HID TJ v1OB 1CzD IDdW4E 4j J mv6 Ktx oBO s9 ADWB q218 BJJzRy UQ i 2Gp weE T8L aO 4ho9 5g4v WQmoiq jS w MA9 Cvn Gqx l1 LrYu MjGb oUpuvY Q2 C dBl AB9 7ew jc 5RJE SFGs ORedoM 0b B k25 VEK B8V A9 ytAE Oyof G8QIj2 7a I 3jy Rmz yET Kx pgUq 4Bvb cD1b1g KB y oE3 azg elV Nu 8iZ1 w1tq twKx8C LN 2 8yn jdo jUW vN H9qy HaXZ GhjUgm uL I 87i Y7Q 9MQ Wa iFFS Gzt8 4mSQq2 5O N ltT gbl 8YD QS AzXq pJEK 7bGL1U Jn 0 f59 vPr wdt d6 sDLj Loo1 8tQXf5 5u p mTa dJD sEL pH 2vqY uTAm YzDg95 1P K FP6 pEi zIJ Qd 8Ngn HTND 6z6ExR XV 0 ouU jWT kAK AB eAC9 Rfja c43Ajk Xn H dgS y3v 5cB et s3VX qfpP BqiGf9 0a w g4d W9U kvR iJ y46G bH3U cJ86hW Va C Mje dsU cqD SZ 1DlP 2mfB hzu5dv u1 i 6eW 2YN LhM 3f WOdz KS6Q ov14wx YY d 8sa S38 hIl cP tS4l 9B7h FC3JXJ Gp s tll 7a7 WNr VM wunm nmDc 5duVpZ xT C l8F I01 jhn 5B l4Jz aEV7 CKMThL ji 1 gyZ uXc Iv4 03 3NqZ LITG Ux3ClP CB K O3v RUi mJq l5 blI9 GrWy irWHof lH 7 3ZT eZX kop eq 8XL1 RQ3a Uj6Ess nj 2 0MA 3As rSV ft 3F9w zB1q DQVOnH Cm m P3d WSb jst oj 3oGj advz qcMB6Y 6k D 9sZ 0bd Mjt UT hULG TWU9 Nmr3E4 CN b zUO vTh hqL 1p xAxT ezrH dVMgLY TT r Sfx LUX CMr WA bE69 K6XH i5re1f x4 G DKk iB7 f2D Xz Xez2 k2Yc Yc4QjU DFGRTHVBSDFRGDFGNCVBSDFGDHDFHDFNCVBDSFGSDFGDSFBDVNCXVBSDFGSDFHDFGHDFTSADFASDFSADFASDXCVZXVSDGHFDGHVBCX143}   \end{equation} for a positive constant $\bar P$ which represents the reference state at the spatial infinity so that $P= \bar P + O(\epsilon)$. Using  $\partial_t P = \epsilon P \partial_t p$ and $\nabla P=\epsilon P \nabla p$, we rewrite \eqref{Euler111} and \eqref{Euler21} as    \begin{align}    &  a\left( \partial_t p + v\cdot \nabla p\right) + \frac{1}{\epsilon} \nabla\cdot v =0    \label{DFGRTHVBSDFRGDFGNCVBSDFGDHDFHDFNCVBDSFGSDFGDSFBDVNCXVBSDFGSDFHDFGHDFTSADFASDFSADFASDXCVZXVSDGHFDGHVBCX137}     \\    & r\left( \partial_t v +  v\cdot \nabla v\right) +    \frac{1}{\epsilon}\nabla p =0     ,   \label{DFGRTHVBSDFRGDFGNCVBSDFGDHDFHDFNCVBDSFGSDFGDSFBDVNCXVBSDFGSDFHDFGHDFTSADFASDFSADFASDXCVZXVSDGHFDGHVBCX136}   \end{align} respectively, where    \begin{equation}     a = a(S, \epsilon p) = A_0 (S, \bar P e^{\epsilon p}) \bar P e^{\epsilon p}     \label{DFGRTHVBSDFRGDFGNCVBSDFGDHDFHDFNCVBDSFGSDFGDSFBDVNCXVBSDFGSDFHDFGHDFTSADFASDFSADFASDXCVZXVSDGHFDGHVBCX118}   \end{equation} and   \begin{equation}     r = r(S,\epsilon p) = \frac{\rho(S, \bar P e^{\epsilon p})}{\bar P e^{\epsilon p}}.     \label{DFGRTHVBSDFRGDFGNCVBSDFGDHDFHDFNCVBDSFGSDFGDSFBDVNCXVBSDFGSDFHDFGHDFTSADFASDFSADFASDXCVZXVSDGHFDGHVBCX119}   \end{equation} In the case of an ideal gas, from $\rho(P,S)= P^{\frac{1}{\gamma}} e^{-\frac{S}{\gamma }}$, we have  the expression   \begin{equation}     a = \frac{1}{\gamma}    \llabel{yM Y R1o DeY NWf 74 hByF dsWk 4cUbCR DX a q4e DWd 7qb Ot 7GOu oklg jJ00J9 Il O Jxn tzF VBC Ft pABp VLEE 2y5Qcg b3 5 DU4 igj 4dz zW soNF wvqj bNFma0 am F Kiv Aap pzM zr VqYf OulM HafaBk 6J r eOQ BaT EsJ BB tHXj n2EU CNleWp cv W JIg gWX Ksn B3 wvmo WK49 Nl492o gR 6 fvc 8ff jJm sW Jr0j zI9p CBsIUV of D kKH Ub7 vxp uQ UXA6 hMUr yvxEpc Tq l Tkz z0q HbX pO 8jFu h6nw zVPPzp A8 9 61V 78c O2W aw 0yGn CHVq BVjTUH lk p 6dG HOd voE E8 cw7Q DL1o 1qg5TX qo V 720 hhQ TyF tp TJDg 9E8D nsp1Qi X9 8 ZVQ N3s duZ qc n9IX ozWh Fd16IB 0K 9 JeB Hvi 364 kQ lFMM JOn0 OUBrnv pY y jUB Ofs Pzx l4 zcMn JHdq OjSi6N Mn 8 bR6 kPe klT Fd VlwD SrhT 8Qr0sC hN h 88j 8ZA vvW VD 03wt ETKK NUdr7W EK 1 jKS IHF Kh2 sr 1RRV Ra8J mBtkWI 1u k uZT F2B 4p8 E7 Y3p0 DX20 JM3XzQ tZ 3 bMC vM4 DEA wB Fp8q YKpL So1a5s dR P fTg 5R6 7v1 T4 eCJ1 qg14 CTK7u7 ag j Q0A tZ1 Nh6 hk Sys5 CWon IOqgCL 3u 7 feR BHz odS Jp 7JH8 u6Rw sYE0mc P4 r LaW Atl yRw kH F3ei UyhI iA19ZB u8 m ywf 42n uyX 0e ljCt 3Lkd 1eUQEZ oO Z rA2 Oqf oQ5 Ca hrBy KzFg DOseim 0j Y BmX csL Ayc cC JBTZ PEjy zPb5hZ KW O xT6 dyt u82 Ia htpD m75Y DktQvd Nj W jIQ H1B Ace SZ KVVP 136v L8XhMm 1O H Kn2 gUy kFU wN 8JML Bqmn vGuwGR oW U oNZ Y2P nmS 5g QMcR YHxL yHuDo8 ba w aqM NYt onW u2 YIOz eB6R wHuGcn fi o 47U PM5 tOj sz QBNq 7mco fCNjou 83 e mcY 81s vDFGRTHVBSDFRGDFGNCVBSDFGDHDFHDFNCVBDSFGSDFGDSFBDVNCXVBSDFGSDFHDFGHDFTSADFASDFSADFASDXCVZXVSDGHFDGHVBCX161}   \end{equation} for $a$, and   \begin{equation}     r = ( \bar P e^{\epsilon p})^{\frac{1}{\gamma}-1} e^{-\frac{S}{\gamma }}    \llabel{sI 2Y DS3S yloB Nx5FBV Bc 9 6HZ EOX UO3 W1 fIF5 jtEM W6KW7D 63 t H0F CVT Zup Pl A9aI oN2s f1Bw31 gg L FoD O0M x18 oo heEd KgZB Cqdqpa sa H Fhx BrE aRg Au I5dq mWWB MuHfv9 0y S PtG hFF dYJ JL f3Ap k5Ck Szr0Kb Vd i sQk uSA JEn DT YkjP AEMu a0VCtC Ff z 9R6 Vht 8Ua cB e7op AnGa 7AbLWj Hc s nAR GMb n7a 9n paMf lftM 7jvb20 0T W xUC 4lt e92 9j oZrA IuIa o1Zqdr oC L 55L T4Q 8kN yv sIzP x4i5 9lKTq2 JB B sZb QCE Ctw ar VBMT H1QR 6v5srW hR r D4r wf8 ik7 KH Egee rFVT ErONml Q5 L R8v XNZ LB3 9U DzRH ZbH9 fTBhRw kA 2 n3p g4I grH xd fEFu z6RE tDqPdw N7 H TVt cE1 8hW 6y n4Gn nCE3 MEQ51i Ps G Z2G Lbt CSt hu zvPF eE28 MM23ug TC d j7z 7Av TLa 1A GLiJ 5JwW CiDPyM qa 8 tAK QZ9 cfP 42 kuUz V3h6 GsGFoW m9 h cfj 51d GtW yZ zC5D aVt2 Wi5IIs gD B 0cX LM1 FtE xE RIZI Z0Rt QUtWcU Cm F mSj xvW pZc gl dopk 0D7a EouRku Id O ZdW FOR uqb PY 6HkW OVi7 FuVMLW nx p SaN omk rC5 uI ZK9C jpJy UIeO6k gb 7 tr2 SCY x5F 11 S6Xq OImr s7vv0u vA g rb9 hGP Fnk RM j92H gczJ 660kHb BB l QSI OY7 FcX 0c uyDl LjbU 3F6vZk Gb a KaM ufj uxp n4 Mi45 7MoL NW3eIm cj 6 OOS e59 afA hg lt9S BOiF cYQipj 5u N 19N KZ5 Czc 23 1wxG x1ut gJB4ue Mx x 5lr s8g VbZ s1 NEfI 02Rb pkfEOZ E4 e seo 9te NRU Ai nujf eJYa Ehns0Y 6X R UF1 PCf 5eE AL 9DL6 a2vm BAU5Au DD t yQN 5YL LWw PW GjMt 4hu4 FIoLCZ Lx e BVY 5lZ DCD 5Y yBwO IJeDFGRTHVBSDFRGDFGNCVBSDFGDHDFHDFNCVBDSFGSDFGDSFBDVNCXVBSDFGSDFHDFGHDFTSADFASDFSADFASDXCVZXVSDGHFDGHVBCX159}   \end{equation} for $r$. Thus we have obtained the symmetrized version of the compressible Euler equation for non-isentropic fluids in $\mathbb{R}^3$, which reads   \begin{align}    &    E(S, \epsilon u) (\partial_t u + v\cdot \nabla u) + \frac{1}{\epsilon} L(\partial_x) u = 0,    \label{DFGRTHVBSDFRGDFGNCVBSDFGDHDFHDFNCVBDSFGSDFGDSFBDVNCXVBSDFGSDFHDFGHDFTSADFASDFSADFASDXCVZXVSDGHFDGHVBCX01}    \\&    \partial_t S + v \cdot \nabla S = 0,    \label{DFGRTHVBSDFRGDFGNCVBSDFGDHDFHDFNCVBDSFGSDFGDSFBDVNCXVBSDFGSDFHDFGHDFTSADFASDFSADFASDXCVZXVSDGHFDGHVBCX02}   \end{align} where $u = (p, v)$ and    \begin{align}   \begin{split}    E(S, \epsilon u)     =     \begin{pmatrix}    a(S, \epsilon u)  & 0 \\    0 & r(S, \epsilon u)I_3 \\    \end{pmatrix}, \qquad    L(\partial_x) =    \begin{pmatrix}    0 & \dive \\    \nabla & 0\\    \end{pmatrix}.    \label{DFGRTHVBSDFRGDFGNCVBSDFGDHDFHDFNCVBDSFGSDFGDSFBDVNCXVBSDFGSDFHDFGHDFTSADFASDFSADFASDXCVZXVSDGHFDGHVBCX03}   \end{split}   \end{align} After transforming \eqref{Euler1}--\eqref{Euler3} to the symmetrized form \eqref{DFGRTHVBSDFRGDFGNCVBSDFGDHDFHDFNCVBDSFGSDFGDSFBDVNCXVBSDFGSDFHDFGHDFTSADFASDFSADFASDXCVZXVSDGHFDGHVBCX01}--\eqref{DFGRTHVBSDFRGDFGNCVBSDFGDHDFHDFNCVBDSFGSDFGDSFBDVNCXVBSDFGSDFHDFGHDFTSADFASDFSADFASDXCVZXVSDGHFDGHVBCX02}, we now focus on the formulation~\eqref{DFGRTHVBSDFRGDFGNCVBSDFGDHDFHDFNCVBDSFGSDFGDSFBDVNCXVBSDFGSDFHDFGHDFTSADFASDFSADFASDXCVZXVSDGHFDGHVBCX01}--\eqref{DFGRTHVBSDFRGDFGNCVBSDFGDHDFHDFNCVBDSFGSDFGDSFBDVNCXVBSDFGSDFHDFGHDFTSADFASDFSADFASDXCVZXVSDGHFDGHVBCX02}. In view of \eqref{DFGRTHVBSDFRGDFGNCVBSDFGDHDFHDFNCVBDSFGSDFGDSFBDVNCXVBSDFGSDFHDFGHDFTSADFASDFSADFASDXCVZXVSDGHFDGHVBCX118} and \eqref{DFGRTHVBSDFRGDFGNCVBSDFGDHDFHDFNCVBDSFGSDFGDSFBDVNCXVBSDFGSDFHDFGHDFTSADFASDFSADFASDXCVZXVSDGHFDGHVBCX119}, we assume   \begin{equation}    a(S, \epsilon u) = f_1(S) g_1(\epsilon u)       \label{DFGRTHVBSDFRGDFGNCVBSDFGDHDFHDFNCVBDSFGSDFGDSFBDVNCXVBSDFGSDFHDFGHDFTSADFASDFSADFASDXCVZXVSDGHFDGHVBCX115}   \end{equation} and    \begin{equation}    r(S, \epsilon u) = f_2(S) g_2(\epsilon u)       ,    \label{DFGRTHVBSDFRGDFGNCVBSDFGDHDFHDFNCVBDSFGSDFGDSFBDVNCXVBSDFGSDFHDFGHDFTSADFASDFSADFASDXCVZXVSDGHFDGHVBCX116}   \end{equation} where $f_1$, $f_2$, $g_1$, and $g_2$ are positive entire real-analytic functions. \par \startnewsection{The main results}{sec02} We assume that the initial data $(p_0^\epsilon, v_0^\epsilon, S_0^\epsilon)$ satisfies     \begin{align}     \Vert (p_0^\epsilon, v_0^\epsilon, S_0^\epsilon) \Vert_{H^{5}}     &     \leq      M_0     \label{DFGRTHVBSDFRGDFGNCVBSDFGDHDFHDFNCVBDSFGSDFGDSFBDVNCXVBSDFGSDFHDFGHDFTSADFASDFSADFASDXCVZXVSDGHFDGHVBCX540}   \end{align} and   \begin{align}    \sum_{m=0}^\infty \sum_{\vert \alpha \vert =m}     \Vert \partial^\alpha  (p_0^\epsilon, v_0^\epsilon, S_0^\epsilon) \Vert_{L^2} \frac{\tau_0^{(m-3)_+}}{(m-3)!}     &     \leq     M_0     ,     \label{DFGRTHVBSDFRGDFGNCVBSDFGDHDFHDFNCVBDSFGSDFGDSFBDVNCXVBSDFGSDFHDFGHDFTSADFASDFSADFASDXCVZXVSDGHFDGHVBCX36}    \end{align} where $\tau_0, M_0>0$ are fixed constants. For $\tau>0$, define the mixed weighted analytic space   \begin{align}    A(\tau) = \{ u \in C^\infty (\mathbb{R}^3): \Vert u \Vert_{A(\tau)} < \infty \},    \llabel{H VQsKob Yd q fCX 1to mCb Ej 5m1p Nx9p nLn5A3 g7 U v77 7YU gBR lN rTyj shaq BZXeAF tj y FlW jfc 57t 2f abx5 Ns4d clCMJc Tl q kfq uFD iSd DP eX6m YLQz JzUmH0 43 M lgF edN mXQ Pj Aoba 07MY wBaC4C nj I 4dw KCZ PO9 wx 3en8 AoqX 7JjN8K lq j Q5c bMS dhR Fs tQ8Q r2ve 2HT0uO 5W j TAi iIW n1C Wr U1BH BMvJ 3ywmAd qN D LY8 lbx XMx 0D Dvco 3RL9 Qz5eqy wV Y qEN nO8 MH0 PY zeVN i3yb 2msNYY Wz G 2DC PoG 1Vb Bx e9oZ GcTU 3AZuEK bk p 6rN eTX 0DS Mc zd91 nbSV DKEkVa zI q NKU Qap NBP 5B 32Ey prwP FLvuPi wR P l1G TdQ BZE Aw 3d90 v8P5 CPAnX4 Yo 2 q7s yr5 BW8 Hc T7tM ioha BW9U4q rb u mEQ 6Xz MKR 2B REFX k3ZO MVMYSw 9S F 5ek q0m yNK Gn H0qi vlRA 18CbEz id O iuy ZZ6 kRo oJ kLQ0 Ewmz sKlld6 Kr K JmR xls 12K G2 bv8v LxfJ wrIcU6 Hx p q6p Fy7 Oim mo dXYt Kt0V VH22OC Aj f deT BAP vPl oK QzLE OQlq dpzxJ6 JI z Ujn TqY sQ4 BD QPW6 784x NUfsk0 aM 7 8qz MuL 9Mr Ac uVVK Y55n M7WqnB 2R C pGZ vHh WUN g9 3F2e RT8U umC62V H3 Z dJX LMS cca 1m xoOO 6oOL OVzfpO BO X 5Ev KuL z5s EW 8a9y otqk cKbDJN Us l pYM JpJ jOW Uy 2U4Y VKH6 kVC1Vx 1u v ykO yDs zo5 bz d36q WH1k J7Jtkg V1 J xqr Fnq mcU yZ JTp9 oFIc FAk0IT A9 3 SrL axO 9oU Z3 jG6f BRL1 iZ7ZE6 zj 8 G3M Hu8 6Ay jt 3flY cmTk jiTSYv CF t JLq cJP tN7 E3 POqG OKe0 3K3WV0 ep W XDQ C97 YSb AD ZUNp 81GF fCPbj3 iq E t0E NXy pLv fo Iz6z oFoF 9lkIun Xj Y DFGRTHVBSDFRGDFGNCVBSDFGDHDFHDFNCVBDSFGSDFGDSFBDVNCXVBSDFGSDFHDFGHDFTSADFASDFSADFASDXCVZXVSDGHFDGHVBCX13}   \end{align} where   \begin{align}   \begin{split}    \Vert u \Vert_{A(\tau)}    &    =    \sum_{m=1}^\infty \sum_{j=0}^m \sum_{\vert \alpha \vert =j} \Vert \partial^\alpha (\epsilon \partial_t)^{m-j} u \Vert_{L^2} \frac{\kappa^{(j-3)_+}\tau(t)^{(m-3)_+}}{(m-3)!}   ;   \end{split}   \label{DFGRTHVBSDFRGDFGNCVBSDFGDHDFHDFNCVBDSFGSDFGDSFBDVNCXVBSDFGSDFHDFGHDFTSADFASDFSADFASDXCVZXVSDGHFDGHVBCX366}   \end{align} here, $\tau\in(0,1]$ represents the mixed space-time analyticity radius and where $\kappa>0$. It is convenient that the term with $\Vert u\Vert_{L^2}$ is not included in the norm. In \eqref{DFGRTHVBSDFRGDFGNCVBSDFGDHDFHDFNCVBDSFGSDFGDSFBDVNCXVBSDFGSDFHDFGHDFTSADFASDFSADFASDXCVZXVSDGHFDGHVBCX36} and below we use the convention $n!=1$ when $n\in-{\mathbb N}$. As shown in Section~\ref{secinitial} below, \eqref{DFGRTHVBSDFRGDFGNCVBSDFGDHDFHDFNCVBDSFGSDFGDSFBDVNCXVBSDFGSDFHDFGHDFTSADFASDFSADFASDXCVZXVSDGHFDGHVBCX36}  implies that with $\kappa=1$    \begin{equation}    \Vert (p_0^\epsilon, v_0^\epsilon, S_0^\epsilon)\Vert_{A(\tilde\tau_0)}    \leq    Q(M_0)    \label{DFGRTHVBSDFRGDFGNCVBSDFGDHDFHDFNCVBDSFGSDFGDSFBDVNCXVBSDFGSDFHDFGHDFTSADFASDFSADFASDXCVZXVSDGHFDGHVBCX166}   \end{equation} for some function~$Q$, where $\tilde\tau_0 = \tau_0/ Q(M_0)$ is a sufficiently small constant. Note that the time derivatives of the initial data are defined iteratively by differentiating the equations \eqref{DFGRTHVBSDFRGDFGNCVBSDFGDHDFHDFNCVBDSFGSDFGDSFBDVNCXVBSDFGSDFHDFGHDFTSADFASDFSADFASDXCVZXVSDGHFDGHVBCX01}--\eqref{DFGRTHVBSDFRGDFGNCVBSDFGDHDFHDFNCVBDSFGSDFGDSFBDVNCXVBSDFGSDFHDFGHDFTSADFASDFSADFASDXCVZXVSDGHFDGHVBCX02} and evaluating at $t=0$ (cf.~Section~\ref{secinitial} below for details). Also observe that the norm in \eqref{DFGRTHVBSDFRGDFGNCVBSDFGDHDFHDFNCVBDSFGSDFGDSFBDVNCXVBSDFGSDFHDFGHDFTSADFASDFSADFASDXCVZXVSDGHFDGHVBCX366} is an increasing function of $\kappa$, and thus \eqref{DFGRTHVBSDFRGDFGNCVBSDFGDHDFHDFNCVBDSFGSDFGDSFBDVNCXVBSDFGSDFHDFGHDFTSADFASDFSADFASDXCVZXVSDGHFDGHVBCX166} holds for any $\kappa\in(0,1]$. We define the analyticity radius function        \begin{align}         \tau(t) = \tau(0) - Kt,         \label{DFGRTHVBSDFRGDFGNCVBSDFGDHDFHDFNCVBDSFGSDFGDSFBDVNCXVBSDFGSDFHDFGHDFTSADFASDFSADFASDXCVZXVSDGHFDGHVBCX27}         \end{align} where $\tau(0) \leq \min\{1, \tilde{\tau}_0\}$ is a sufficiently small parameter  (different from $\tilde{\tau}_0$), and $\KK\geq1$ is a sufficiently large parameter, both to be determined below. \par The first theorem provides a uniform in $\epsilon$ boundedness of  the analytic norm on a  time interval, which is independent of~$\epsilon$. \par \cole \begin{Theorem} \label{T01} Assume that the initial data $(p_0^\epsilon, v_0^\epsilon, S_0^\epsilon) $ satisfies \eqref{DFGRTHVBSDFRGDFGNCVBSDFGDHDFHDFNCVBDSFGSDFGDSFBDVNCXVBSDFGSDFHDFGHDFTSADFASDFSADFASDXCVZXVSDGHFDGHVBCX540}--\eqref{DFGRTHVBSDFRGDFGNCVBSDFGDHDFHDFNCVBDSFGSDFGDSFBDVNCXVBSDFGSDFHDFGHDFTSADFASDFSADFASDXCVZXVSDGHFDGHVBCX36}, where $M_0, \tau_0>0$. There exist sufficiently small constants $\kappa, \tau(0), \epsilon_0,T_0>0$, depending on $M_0$, such that   \begin{align}    \Vert (p^\epsilon, v^\epsilon, S^\epsilon)(t) \Vert_{A(\tau)}    \leq    M    \comma 0<\epsilon\leq \epsilon_0    \commaone t\in[0,T_0]    ,    \label{DFGRTHVBSDFRGDFGNCVBSDFGDHDFHDFNCVBDSFGSDFGDSFBDVNCXVBSDFGSDFHDFGHDFTSADFASDFSADFASDXCVZXVSDGHFDGHVBCX17}   \end{align} where $\tau$ is as in \eqref{DFGRTHVBSDFRGDFGNCVBSDFGDHDFHDFNCVBDSFGSDFGDSFBDVNCXVBSDFGSDFHDFGHDFTSADFASDFSADFASDXCVZXVSDGHFDGHVBCX27} and  $\KK$ and $M$ are sufficiently large constants depending on $M_0$. \end{Theorem} \colb \par We now turn to the Mach limit for solutions of \eqref{DFGRTHVBSDFRGDFGNCVBSDFGDHDFHDFNCVBDSFGSDFGDSFBDVNCXVBSDFGSDFHDFGHDFTSADFASDFSADFASDXCVZXVSDGHFDGHVBCX01}--\eqref{DFGRTHVBSDFRGDFGNCVBSDFGDHDFHDFNCVBDSFGSDFGDSFBDVNCXVBSDFGSDFHDFGHDFTSADFASDFSADFASDXCVZXVSDGHFDGHVBCX02} in $\mathbb{R}^3$ as $\epsilon \to 0$. Denote   \begin{align}    \delta =\frac{\kappa\tau(0)}{C_0},    \label{DFGRTHVBSDFRGDFGNCVBSDFGDHDFHDFNCVBDSFGSDFGDSFBDVNCXVBSDFGSDFHDFGHDFTSADFASDFSADFASDXCVZXVSDGHFDGHVBCX205}   \end{align} where $\tau(0), \kappa \in(0,1]$ are fixed constants chosen in the proof of Theorem~\ref{T01}, and $C_0>1$ is a sufficiently large constant to be chosen in Section~\ref{sec06}. We introduce the spatial analytic norm        \begin{align}        \Vert u \Vert_{X_{\delta}}        =        \sum_{m=1}^\infty \sum_{\vert \alpha \vert = m} \Vert \partial^\alpha u \Vert_{L^2} \frac{\delta^{(m-3)_+}}{(m-3)!},    \label{DFGRTHVBSDFRGDFGNCVBSDFGDHDFHDFNCVBDSFGSDFGDSFBDVNCXVBSDFGSDFHDFGHDFTSADFASDFSADFASDXCVZXVSDGHFDGHVBCX169}        \end{align} where $\delta>0$ is as in \eqref{DFGRTHVBSDFRGDFGNCVBSDFGDHDFHDFNCVBDSFGSDFGDSFBDVNCXVBSDFGSDFHDFGHDFTSADFASDFSADFASDXCVZXVSDGHFDGHVBCX205}.  Note that this is a part of our main analytic $A$-norm, \eqref{DFGRTHVBSDFRGDFGNCVBSDFGDHDFHDFNCVBDSFGSDFGDSFBDVNCXVBSDFGSDFHDFGHDFTSADFASDFSADFASDXCVZXVSDGHFDGHVBCX36}.  \par By Theorem~\ref{T01}, for a given $M_0$ and $\tau_0>0$,  the solutions $(p^\epsilon, v^\epsilon, S^\epsilon)$ are uniformly bounded by $M$ in the norm of $C^0([0,T_0], X_\delta)$ for fixed parameters $\kappa$, $T_0$, and $\epsilon\in (0,\epsilon_0]$.  The second main theorem shows that solutions of \eqref{DFGRTHVBSDFRGDFGNCVBSDFGDHDFHDFNCVBDSFGSDFGDSFBDVNCXVBSDFGSDFHDFGHDFTSADFASDFSADFASDXCVZXVSDGHFDGHVBCX01}--\eqref{DFGRTHVBSDFRGDFGNCVBSDFGDHDFHDFNCVBDSFGSDFGDSFBDVNCXVBSDFGSDFHDFGHDFTSADFASDFSADFASDXCVZXVSDGHFDGHVBCX02}  converge to the solution of the stratified  incompressible Euler equations   \begin{align}    &r(S, 0) (\partial_t v + v\cdot \nabla v) + \nabla \pi = 0,    \label{DFGRTHVBSDFRGDFGNCVBSDFGDHDFHDFNCVBDSFGSDFGDSFBDVNCXVBSDFGSDFHDFGHDFTSADFASDFSADFASDXCVZXVSDGHFDGHVBCX202}    \\&    \dive v = 0,    \label{DFGRTHVBSDFRGDFGNCVBSDFGDHDFHDFNCVBDSFGSDFGDSFBDVNCXVBSDFGSDFHDFGHDFTSADFASDFSADFASDXCVZXVSDGHFDGHVBCX203}    \\&    \partial_t S + v \cdot \nabla S = 0    ,    \label{DFGRTHVBSDFRGDFGNCVBSDFGDHDFHDFNCVBDSFGSDFGDSFBDVNCXVBSDFGSDFHDFGHDFTSADFASDFSADFASDXCVZXVSDGHFDGHVBCX204}   \end{align} as $\epsilon \to 0$. \par \begin{Theorem} 	\label{T02} 	Assume that the initial data $(p_0^\epsilon, v_0^\epsilon, S_0^\epsilon)$ satisfy \eqref{DFGRTHVBSDFRGDFGNCVBSDFGDHDFHDFNCVBDSFGSDFGDSFBDVNCXVBSDFGSDFHDFGHDFTSADFASDFSADFASDXCVZXVSDGHFDGHVBCX540}--\eqref{DFGRTHVBSDFRGDFGNCVBSDFGDHDFHDFNCVBDSFGSDFGDSFBDVNCXVBSDFGSDFHDFGHDFTSADFASDFSADFASDXCVZXVSDGHFDGHVBCX36} uniformly for fixed $\tau_0, M_0 >0$. Also, suppose 	that the initial data $(v_0^\epsilon, S_0^\epsilon)$ converges to  	$(v_0, S_0)$ in  	$H^3(\mathbb{R}^3)$ as $\epsilon \to 0$, and $S_0^\epsilon$ decays sufficiently rapidly at infinity in the sense  	\begin{align} 		\vert S_0^\epsilon (x) \vert 		\leq  		C \vert x \vert^{-1-\zeta}, 		\indeq\indeq 		|\nabla S_0^\epsilon (x) \vert          		\leq  		C |x \vert^{-2-\zeta}, 		\llabel{yYL 52U bRB jx kQUS U9mm XtzIHO Cz 1 KH4 9ez 6Pz qW F223 C0Iz 3CsvuT R9 s VtQ CcM 1eo pD Py2l EEzL U0USJt Jb 9 zgy Gyf iQ4 fo Cx26 k4jL E0ula6 aS I rZQ HER 5HV CE BL55 WCtB 2LCmve TD z Vcp 7UR gI7 Qu FbFw 9VTx JwGrzs VW M 9sM JeJ Nd2 VG GFsi WuqC 3YxXoJ GK w Io7 1fg sGm 0P YFBz X8eX 7pf9GJ b1 o XUs 1q0 6KP Ls MucN ytQb L0Z0Qq m1 l SPj 9MT etk L6 KfsC 6Zob Yhc2qu Xy 9 GPm ZYj 1Go ei feJ3 pRAf n6Ypy6 jN s 4Y5 nSE pqN 4m Rmam AGfY HhSaBr Ls D THC SEl UyR Mh 66XU 7hNz pZVC5V nV 7 VjL 7kv WKf 7P 5hj6 t1vu gkLGdN X8 b gOX HWm 6W4 YE mxFG 4WaN EbGKsv 0p 4 OG0 Nrd uTe Za xNXq V4Bp mOdXIq 9a b PeD PbU Z4N Xt ohbY egCf xBNttE wc D YSD 637 jJ2 ms 6Ta1 J2xZ PtKnPw AX A tJA Rc8 n5d 93 TZi7 q6Wo nEDLwW Sz e Sue YFX 8cM hm Y6is 15pX aOYBbV fS C haL kBR Ks6 UO qG4j DVab fbdtny fi D BFI 7uh B39 FJ 6mYr CUUT f2X38J 43 K yZg 87i gFR 5R z1t3 jH9x lOg1h7 P7 W w8w jMJ qH3 l5 J5wU 8eH0 OogRCv L7 f JJg 1ug RfM XI GSuE Efbh 3hdNY3 x1 9 7jR qeP cdu sb fkuJ hEpw MvNBZV zL u qxJ 9b1 BTf Yk RJLj Oo1a EPIXvZ Aj v Xne fhK GsJ Ga wqjt U7r6 MPoydE H2 6 203 mGi JhF nT NCDB YlnP oKO6Pu XU 3 uu9 mSg 41v ma kk0E WUpS UtGBtD e6 d Kdx ZNT FuT i1 fMcM hq7P Ovf0hg Hl 8 fqv I3R K39 fn 9MaC Zgow 6e1iXj KC 5 lHO lpG pkK Xd Dxtz 0HxE fSMjXY L8 F vh7 dmJ kE8 QA KDo1 FqML HOZ2iL 9i I m3L Kva YiN K9DFGRTHVBSDFRGDFGNCVBSDFGDHDFHDFNCVBDSFGSDFGDSFBDVNCXVBSDFGSDFHDFGHDFTSADFASDFSADFASDXCVZXVSDGHFDGHVBCX157} 	\end{align} 	for $0<\epsilon \leq \epsilon_0$ and some constants $C$ and $\zeta>0$. Then $(v^\epsilon, p^\epsilon, S^\epsilon)$ converges to  	$(v^{({\rm inc})}, 0, S^{({\rm inc})}) \in L^\infty ([0, T_0], X_\delta)$ in $L^2 ([0,T_0], X_\delta)$, where $\delta \in (0,\tau_0]$ is a sufficiently small constant and 	 $(v^{({\rm inc})}, S^{({\rm inc})})$ is the solution to \eqref{DFGRTHVBSDFRGDFGNCVBSDFGDHDFHDFNCVBDSFGSDFGDSFBDVNCXVBSDFGSDFHDFGHDFTSADFASDFSADFASDXCVZXVSDGHFDGHVBCX202}--\eqref{DFGRTHVBSDFRGDFGNCVBSDFGDHDFHDFNCVBDSFGSDFGDSFBDVNCXVBSDFGSDFHDFGHDFTSADFASDFSADFASDXCVZXVSDGHFDGHVBCX204} 	with the initial data $(w_0, S_0)$, and $w_0$ is the unique solution of  	\begin{align} 		&\dive w_0 = 0,  		\llabel{ sb48 NxwY NR0nx2 t5 b WCk x2a 31k a8 fUIa RGzr 7oigRX 5s m 9PQ 7Sr 5St ZE Ymp8 VIWS hdzgDI 9v R F5J 81x 33n Ne fjBT VvGP vGsxQh Al G Fbe 1bQ i6J ap OJJa ceGq 1vvb8r F2 F 3M6 8eD lzG tX tVm5 y14v mwIXa2 OG Y hxU sXJ 0qg l5 ZGAt HPZd oDWrSb BS u NKi 6KW gr3 9s 9tc7 WM4A ws1PzI 5c C O7Z 8y9 lMT LA dwhz Mxz9 hjlWHj bJ 5 CqM jht y9l Mn 4rc7 6Amk KJimvH 9r O tbc tCK rsi B0 4cFV Dl1g cvfWh6 5n x y9Z S4W Pyo QB yr3v fBkj TZKtEZ 7r U fdM icd yCV qn D036 HJWM tYfL9f yX x O7m IcF E1O uL QsAQ NfWv 6kV8Im 7Q 6 GsX NCV 0YP oC jnWn 6L25 qUMTe7 1v a hnH DAo XAb Tc zhPc fjrj W5M5G0 nz N M5T nlJ WOP Lh M6U2 ZFxw pg4Nej P8 U Q09 JX9 n7S kE WixE Rwgy Fvttzp 4A s v5F Tnn MzL Vh FUn5 6tFY CxZ1Bz Q3 E TfD lCa d7V fo MwPm ngrD HPfZV0 aY k Ojr ZUw 799 et oYuB MIC4 ovEY8D OL N URV Q5l ti1 iS NZAd wWr6 Q8oPFf ae 5 lAR 9gD RSi HO eJOW wxLv 20GoMt 2H z 7Yc aly PZx eR uFM0 7gaV 9UIz7S 43 k 5Tr ZiD Mt7 pE NCYi uHL7 gac7Gq yN 6 Z1u x56 YZh 2d yJVx 9MeU OMWBQf l0 E mIc 5Zr yfy 3i rahC y9Pi MJ7ofo Op d enn sLi xZx Jt CjC9 M71v O0fxiR 51 m FIB QRo 1oW Iq 3gDP stD2 ntfoX7 YU o S5k GuV IGM cf HZe3 7ZoG A1dDmk XO 2 KYR LpJ jII om M6Nu u8O0 jO5Nab Ub R nZn 15k hG9 4S 21V4 Ip45 7ooaiP u2 j hIz osW FDu O5 HdGr djvv tTLBjo vL L iCo 6L5 Lwa Pm vD6Z pal6 9Ljn11 re T 2CP mvj rL3 xH mDYK uv5T npCDFGRTHVBSDFRGDFGNCVBSDFGDHDFHDFNCVBDSFGSDFGDSFBDVNCXVBSDFGSDFHDFGHDFTSADFASDFSADFASDXCVZXVSDGHFDGHVBCX155}    		\\& 		\curl (r_0 w_0) = \curl (r_0 v_0) 		\llabel{1fM oU R RTo Loi lk0 FE ghak m5M9 cOIPdQ lG D LnX erC ykJ C1 0FHh vvnY aTGuqU rf T QPv wEq iHO vO hD6A nXuv GlzVAv pz d Ok3 6ym yUo Fb AcAA BItO es52Vq d0 Y c7U 2gB t0W fF VQZh rJHr lBLdCx 8I o dWp AlD S8C HB rNLz xWp6 ypjuwW mg X toy 1vP bra uH yMNb kUrZ D6Ee2f zI D tkZ Eti Lmg re 1woD juLB BSdasY Vc F Uhy ViC xB1 5y Ltql qoUh gL3bZN YV k orz wa3 650 qW hF22 epiX cAjA4Z V4 b cXx uB3 NQN p0 GxW2 Vs1z jtqe2p LE B iS3 0E0 NKH gY N50v XaK6 pNpwdB X2 Y v7V 0Ud dTc Pi dRNN CLG4 7Fc3PL Bx K 3Be x1X zyX cj 0Z6a Jk0H KuQnwd Dh P Q1Q rwA 05v 9c 3pnz ttzt x2IirW CZ B oS5 xlO KCi D3 WFh4 dvCL QANAQJ Gg y vOD NTD FKj Mc 0RJP m4HU SQkLnT Q4 Y 6CC MvN jAR Zb lir7 RFsI NzHiJl cg f xSC Hts ZOG 1V uOzk 5G1C LtmRYI eD 3 5BB uxZ JdY LO CwS9 lokS NasDLj 5h 8 yni u7h u3c di zYh1 PdwE l3m8Xt yX Q RCA bwe aLi N8 qA9N 6DRE wy6gZe xs A 4fG EKH KQP PP KMbk sY1j M4h3Jj gS U One p1w RqN GA grL4 c18W v4kchD gR x 7Gj jIB zcK QV f7gA TrZx Oy6FF7 y9 3 iuu AQt 9TK Rx S5GO TFGx 4Xx1U3 R4 s 7U1 mpa bpD Hg kicx aCjk hnobr0 p4 c ody xTC kVj 8t W4iP 2OhT RF6kU2 k2 o oZJ Fsq Y4B FS NI3u W2fj OMFf7x Jv e ilb UVT ArC Tv qWLi vbRp g2wpAJ On l RUE PKh j9h dG M0Mi gcqQ wkyunB Jr T LDc Pgn OSC HO sSgQ sR35 MB7Bgk Pk 6 nJh 01P Cxd Ds w514 O648 VD8iJ5 4F W 6rs 6Sy qGz MK fXop oe4e o52UNB 4Q 8 f8N UDFGRTHVBSDFRGDFGNCVBSDFGDHDFHDFNCVBDSFGSDFGDSFBDVNCXVBSDFGSDFHDFGHDFTSADFASDFSADFASDXCVZXVSDGHFDGHVBCX156} 		, 	\end{align} 	with $r_0 = r(S_0, 0)$. \end{Theorem} \par In the rest of the paper, the constant $C$ depends only on $M_0$ and $\tau_0$, and it may vary from relation to relation; we omit the superscript $\epsilon$, and we write $S$, $u$ for $S^\epsilon$, $u^\epsilon$. \par Theorem~\ref{T02} is proven in Section~\ref{sec06} below as a consequence of Theorem~\ref{T01}.  The proof of Theorem~\ref{T01} consists of a~priori estimates performed on the solutions. The a~priori estimates are easily justified by simply restricting the sum \eqref{DFGRTHVBSDFRGDFGNCVBSDFGDHDFHDFNCVBDSFGSDFGDSFBDVNCXVBSDFGSDFHDFGHDFTSADFASDFSADFASDXCVZXVSDGHFDGHVBCX36} to $m\leq m_0$ where $m_0\in \{6,7,\ldots\}$ is arbitrary. The estimates on the finite sums are justified since boundedness of solutions in any Sobolev norm is known by~\cite{A05}. \par The proof of Theorem~\ref{T01} relies on analytic a~priori estimates on the entropy $S$ and the  (modified) velocity $u$. The a~priori estimate needed to prove Theorem~\ref{T01} is the following. \par \cole \begin{Lemma} \rm \label{L01} Let $M_0>0$. For any $\kappa \leq 1$, there exist constants $C$, $\tau_1, \epsilon_0, T_0$  and a nonnegative continuous function $Q$ such that for all $\epsilon \in (0,\epsilon_0]$, the norm    \begin{align}    M_{\epsilon, \kappa} (T) = \sup_{t\in [0,T]} (\Vert S(t) \Vert_{A(\tau(t))}     + \Vert u(t) \Vert_{A(\tau(t))})    \label{DFGRTHVBSDFRGDFGNCVBSDFGDHDFHDFNCVBDSFGSDFGDSFBDVNCXVBSDFGSDFHDFGHDFTSADFASDFSADFASDXCVZXVSDGHFDGHVBCX18}   \end{align} satisfies the estimate   \begin{align}        M_{\epsilon, \kappa} (t)     \leq      C    +     \left( t +\epsilon +\kappa + \tau(0) \right) Q(M_{\epsilon,\kappa} (t))    ,    \label{DFGRTHVBSDFRGDFGNCVBSDFGDHDFHDFNCVBDSFGSDFGDSFBDVNCXVBSDFGSDFHDFGHDFTSADFASDFSADFASDXCVZXVSDGHFDGHVBCX56}   \end{align} for $t\in[0,T_0]$ and $\tau(0) \in (0, \tau_1]$, provided   \begin{equation}    \KK \geq   Q(M_{\epsilon,\kappa}(T_0))    \label{DFGRTHVBSDFRGDFGNCVBSDFGDHDFHDFNCVBDSFGSDFGDSFBDVNCXVBSDFGSDFHDFGHDFTSADFASDFSADFASDXCVZXVSDGHFDGHVBCX149}   \end{equation} holds. \end{Lemma} \colb \par With $\tau=\tau(t)$ as in \eqref{DFGRTHVBSDFRGDFGNCVBSDFGDHDFHDFNCVBDSFGSDFGDSFBDVNCXVBSDFGSDFHDFGHDFTSADFASDFSADFASDXCVZXVSDGHFDGHVBCX27}, we use the notation \eqref{DFGRTHVBSDFRGDFGNCVBSDFGDHDFHDFNCVBDSFGSDFGDSFBDVNCXVBSDFGSDFHDFGHDFTSADFASDFSADFASDXCVZXVSDGHFDGHVBCX18}.  The constant  $\KK$ depends on $M$ (and thus ultimately on~$M_0$), i.e., $\KK=Q(M)$.  We shall work on an interval of time such that    \begin{equation}    T_0\leq \frac{\tau(0)}{2\KK}    .    \label{DFGRTHVBSDFRGDFGNCVBSDFGDHDFHDFNCVBDSFGSDFGDSFBDVNCXVBSDFGSDFHDFGHDFTSADFASDFSADFASDXCVZXVSDGHFDGHVBCX148}   \end{equation} Thus we have $\tau(0)/2\leq \tau(t) \leq \tau(0)$ for $t\in [0, T_0]$. \par From here on, we denote by $Q$  a positive increasing continuous function, which may change from inequality to inequality; importantly, the function $Q$ does not depend on $\epsilon$, $\kappa$, and $t$.  The estimates are performed on an interval  of time $[0,T]$ where \eqref{DFGRTHVBSDFRGDFGNCVBSDFGDHDFHDFNCVBDSFGSDFGDSFBDVNCXVBSDFGSDFHDFGHDFTSADFASDFSADFASDXCVZXVSDGHFDGHVBCX27} holds and is such that   \begin{equation}    T\leq \frac{\tau(0)}{2K}     .    \llabel{z8 u2n GO AXHW gKtG AtGGJs bm z 2qj vSv GBu 5e 4JgL Aqrm gMmS08 ZF s xQm 28M 3z4 Ho 1xxj j8Uk bMbm8M 0c L PL5 TS2 kIQ jZ Kb9Q Ux2U i5Aflw 1S L DGI uWU dCP jy wVVM 2ct8 cmgOBS 7d Q ViX R8F bta 1m tEFj TO0k owcK2d 6M Z iW8 PrK PI1 sX WJNB cREV Y4H5QQ GH b plP bwd Txp OI 5OQZ AKyi ix7Qey YI 9 1Ea 16r KXK L2 ifQX QPdP NL6EJi Hc K rBs 2qG tQb aq edOj Lixj GiNWr1 Pb Y SZe Sxx Fin aK 9Eki CHV2 a13f7G 3G 3 oDK K0i bKV y4 53E2 nFQS 8Hnqg0 E3 2 ADd dEV nmJ 7H Bc1t 2K2i hCzZuy 9k p sHn 8Ko uAR kv sHKP y8Yo dOOqBi hF 1 Z3C vUF hmj gB muZq 7ggW Lg5dQB 1k p Fxk k35 GFo dk 00YD 13qI qqbLwy QC c yZR wHA fp7 9o imtC c5CV 8cEuwU w7 k 8Q7 nCq WkM gY rtVR IySM tZUGCH XV 9 mr9 GHZ ol0 VE eIjQ vwgw 17pDhX JS F UcY bqU gnG V8 IFWb S1GX az0ZTt 81 w 7En IhF F72 v2 PkWO Xlkr w6IPu5 67 9 vcW 1f6 z99 lM 2LI1 Y6Na axfl18 gT 0 gDp tVl CN4 jf GSbC ro5D v78Cxa uk Y iUI WWy YDR w8 z7Kj Px7C hC7zJv b1 b 0rF d7n Mxk 09 1wHv y4u5 vLLsJ8 Nm A kWt xuf 4P5 Nw P23b 06sF NQ6xgD hu R GbK 7j2 O4g y4 p4BL top3 h2kfyI 9w O 4Aa EWb 36Y yH YiI1 S3CO J7aN1r 0s Q OrC AC4 vL7 yr CGkI RlNu GbOuuk 1a w LDK 2zl Ka4 0h yJnD V4iF xsqO00 1r q CeO AO2 es7 DR aCpU G54F 2i97xS Qr c bPZ 6K8 Kud n9 e6SY o396 Fr8LUx yX O jdF sMr l54 Eh T8vr xxF2 phKPbs zr l pMA ubE RMG QA aCBu 2Lqw Gasprf IZ O iKV Vbu Vae 6a baufDFGRTHVBSDFRGDFGNCVBSDFGDHDFHDFNCVBDSFGSDFGDSFBDVNCXVBSDFGSDFHDFGHDFTSADFASDFSADFASDXCVZXVSDGHFDGHVBCX44}   \end{equation} In the rest of the paper, we allow all the constants to depend on $\tau_0$. \par \begin{proof}[Proof of Theorem~\ref{T01} given Lemma~\ref{L01}] Let $M_0>0$ be as in \eqref{DFGRTHVBSDFRGDFGNCVBSDFGDHDFHDFNCVBDSFGSDFGDSFBDVNCXVBSDFGSDFHDFGHDFTSADFASDFSADFASDXCVZXVSDGHFDGHVBCX540}--\eqref{DFGRTHVBSDFRGDFGNCVBSDFGDHDFHDFNCVBDSFGSDFGDSFBDVNCXVBSDFGSDFHDFGHDFTSADFASDFSADFASDXCVZXVSDGHFDGHVBCX36}. Also, fix  $C_0$ and $Q_0$ as the constant $C$ and the function $Q$ appearing in the statement of Lemma~\ref{L01}, respectively. Now, choose and fix   \begin{equation}    M_1 > \max\{C_0,Q_0(M_0)\}    .    \llabel{ y9Kc Fk6cBl Z5 r KUj htW E1C nt 9Rmd whJR ySGVSO VT v 9FY 4uz yAH Sp 6yT9 s6R6 oOi3aq Zl L 7bI vWZ 18c Fa iwpt C1nd Fyp4oK xD f Qz2 813 6a8 zX wsGl Ysh9 Gp3Tal nr R UKt tBK eFr 45 43qU 2hh3 WbYw09 g2 W LIX zvQ zMk j5 f0xL seH9 dscinG wu P JLP 1gE N5W qY sSoW Peqj MimTyb Hj j cbn 0NO 5hz P9 W40r 2w77 TAoz70 N1 a u09 boc DSx Gc 3tvK LXaC 1dKgw9 H3 o 2kE oul In9 TS PyL2 HXO7 tSZse0 1Z 9 Hds lDq 0tm SO AVqt A1FQ zEMKSb ak z nw8 39w nH1 Dp CjGI k5X3 B6S6UI 7H I gAa f9E V33 Bk kuo3 FyEi 8Ty2AB PY z SWj Pj5 tYZ ET Yzg6 Ix5t ATPMdl Gk e 67X b7F ktE sz yFyc mVhG JZ29aP gz k Yj4 cEr HCd P7 XFHU O9zo y4AZai SR O pIn 0tp 7kZ zU VHQt m3ip 3xEd41 By 7 2ux IiY 8BC Lb OYGo LDwp juza6i Pa k Zdh aD3 xSX yj pdOw oqQq Jl6RFg lO t X67 nm7 s1l ZJ mGUr dIdX Q7jps7 rc d ACY ZMs BKA Nx tkqf Nhkt sbBf2O BN Z 5pf oqS Xtd 3c HFLN tLgR oHrnNl wR n ylZ NWV NfH vO B1nU Ayjt xTWW4o Cq P Rtu Vua nMk Lv qbxp Ni0x YnOkcd FB d rw1 Nu7 cKy bL jCF7 P4dx j0Sbz9 fa V CWk VFo s9t 2a QIPK ORuE jEMtbS Hs Y eG5 Z7u MWW Aw RnR8 FwFC zXVVxn FU f yKL Nk4 eOI ly n3Cl I5HP 8XP6S4 KF f Il6 2Vl bXg ca uth8 61pU WUx2aQ TW g rZw cAx 52T kq oZXV g0QG rBrrpe iw u WyJ td9 ooD 8t UzAd LSnI tarmhP AW B mnm nsb xLI qX 4RQS TyoF DIikpe IL h WZZ 8ic JGa 91 HxRb 97kn Whp9sA Vz P o85 60p RN2 PS MGMM FK5X W52OnW IDFGRTHVBSDFRGDFGNCVBSDFGDHDFHDFNCVBDSFGSDFGDSFBDVNCXVBSDFGSDFHDFGHDFTSADFASDFSADFASDXCVZXVSDGHFDGHVBCX144}   \end{equation} Then select $\kappa \leq 1$ sufficiently small, $\tau(0) \leq \min\{1, \tilde{\tau}_0, \tau_1\}$, $T_1\in(0,T_0]$, and $\epsilon\in(0,\epsilon_0]$ sufficiently small, so that   \begin{align}     C_0 + \left( T_1+\epsilon + \kappa + \tau(0) \right) Q_0(M_1)      < M_1         .    \llabel{y o Yng xWn o86 8S Kbbu 1Iq1 SyPkHJ VC v seV GWr hUd ew Xw6C SY1b e3hD9P Kh a 1y0 SRw yxi AG zdCM VMmi JaemmP 8x r bJX bKL DYE 1F pXUK ADtF 9ewhNe fd 2 XRu tTl 1HY JV p5cA hM1J fK7UIc pk d TbE ndM 6FW HA 72Pg LHzX lUo39o W9 0 BuD eJS lnV Rv z8VD V48t Id4Dtg FO O a47 LEH 8Qw nR GNBM 0RRU LluASz jx x wGI BHm Vyy Ld kGww 5eEg HFvsFU nz l 0vg OaQ DCV Ez 64r8 UvVH TtDykr Eu F aS3 5p5 yn6 QZ UcX3 mfET Exz1kv qE p OVV EFP IVp zQ lMOI Z2yT TxIUOm 0f W L1W oxC tlX Ws 9HU4 EF0I Z1WDv3 TP 4 2LN 7Tr SuR 8u Mv1t Lepv ZoeoKL xf 9 zMJ 6PU In1 S8 I4KY 13wJ TACh5X l8 O 5g0 ZGw Ddt u6 8wvr vnDC oqYjJ3 nF K WMA K8V OeG o4 DKxn EOyB wgmttc ES 8 dmT oAD 0YB Fl yGRB pBbo 8tQYBw bS X 2lc YnU 0fh At myR3 CKcU AQzzET Ng b ghH T64 KdO fL qFWu k07t DkzfQ1 dg B cw0 LSY lr7 9U 81QP qrdf H1tb8k Kn D l52 FhC j7T Xi P7GF C7HJ KfXgrP 4K O Og1 8BM 001 mJ PTpu bQr6 1JQu6o Gr 4 baj 60k zdX oD gAOX 2DBk LymrtN 6T 7 us2 Cp6 eZm 1a VJTY 8vYP OzMnsA qs 3 RL6 xHu mXN AB 5eXn ZRHa iECOaa MB w Ab1 5iF WGu cZ lU8J niDN KiPGWz q4 1 iBj 1kq bak ZF SvXq vSiR bLTriS y8 Q YOa mQU ZhO rG HYHW guPB zlAhua o5 9 RKU trF 5Kb js KseT PXhU qRgnNA LV t aw4 YJB tK9 fN 7bN9 IEwK LTYGtn Cc c 2nf Mcx 7Vo Bt 1IC5 teMH X4g3JK 4J s deo Dl1 Xgb m9 xWDg Z31P chRS1R 8W 1 hap 5Rh 6Jj yT NXSC Uscx K4275D 72 g pRW xcf AbDFGRTHVBSDFRGDFGNCVBSDFGDHDFHDFNCVBDSFGSDFGDSFBDVNCXVBSDFGSDFHDFGHDFTSADFASDFSADFASDXCVZXVSDGHFDGHVBCX57}   \end{align} Next, set   \begin{equation}    T_2=\min\left\{             T_1,             \frac{\tau(0)}{2Q_0(M_1)}            \right\}    .    \label{DFGRTHVBSDFRGDFGNCVBSDFGDHDFHDFNCVBDSFGSDFGDSFBDVNCXVBSDFGSDFHDFGHDFTSADFASDFSADFASDXCVZXVSDGHFDGHVBCX150}   \end{equation} In view of \eqref{DFGRTHVBSDFRGDFGNCVBSDFGDHDFHDFNCVBDSFGSDFGDSFBDVNCXVBSDFGSDFHDFGHDFTSADFASDFSADFASDXCVZXVSDGHFDGHVBCX149}, this last condition ensures   \begin{equation}    \frac{\tau(0)}{2}\leq \tau(t) \leq \tau(0)    \comma    t\in [0, T_2].    \llabel{Z Y7 Apto 5SpT zO1dPA Vy Z JiW Clu OjO tE wxUB 7cTt EDqcAb YG d ZQZ fsQ 1At Hy xnPL 5K7D 91u03s 8K 2 0ro fZ9 w7T jx yG7q bCAh ssUZQu PK 7 xUe K7F 4HK fr CEPJ rgWH DZQpvR kO 8 Xve aSB OXS ee XV5j kgzL UTmMbo ma J fxu 8gA rnd zS IB0Y QSXv cZW8vo CO o OHy rEu GnS 2f nGEj jaLz ZIocQe gw H fSF KjW 2Lb KS nIcG 9Wnq Zya6qA YM S h2M mEA sw1 8n sJFY Anbr xZT45Z wB s BvK 9gS Ugy Bk 3dHq dvYU LhWgGK aM f Fk7 8mP 20m eV aQp2 NWIb 6hVBSe SV w nEq bq6 ucn X8 JLkI RJbJ EbwEYw nv L BgM 94G plc lu 2s3U m15E YAjs1G Ln h zG8 vmh ghs Qc EDE1 KnaH wtuxOg UD L BE5 9FL xIp vu KfJE UTQS EaZ6hu BC a KXr lni r1X mL KH3h VPrq ixmTkR zh 0 OGp Obo N6K LC E0Ga Udta nZ9Lvt 1K Z eN5 GQc LQL L0 P9GX uakH m6kqk7 qm X UVH 2bU Hga v0 Wp6Q 8JyI TzlpqW 0Y k 1fX 8gj Gci bR arme Si8l w03Win NX w 1gv vcD eDP Sa bsVw Zu4h aO1V2D qw k JoR Shj MBg ry glA9 3DBd S0mYAc El 5 aEd pII DT5 mb SVuX o8Nl Y24WCA 6d f CVF 6Al a6i Ns 7GCh OvFA hbxw9Q 71 Z RC8 yRi 1zZ dM rpt7 3dou ogkAkG GE 4 87V ii4 Ofw Je sXUR dzVL HU0zms 8W 2 Ztz iY5 mw9 aB ZIwk 5WNm vNM2Hd jn e wMR 8qp 2Vv up cV4P cjOG eu35u5 cQ X NTy kfT ZXA JH UnSs 4zxf Hwf10r it J Yox Rto 5OM FP hakR gzDY Pm02mG 18 v mfV 11N n87 zS X59D E0cN 99uEUz 2r T h1F P8x jrm q2 Z7ut pdRJ 2DdYkj y9 J Yko c38 Kdu Z9 vydO wkO0 djhXSx Sv H wJo XE7 9f8 qh iBr8 KYTxDFGRTHVBSDFRGDFGNCVBSDFGDHDFHDFNCVBDSFGSDFGDSFBDVNCXVBSDFGSDFHDFGHDFTSADFASDFSADFASDXCVZXVSDGHFDGHVBCX151}   \end{equation} Note that $M_{\epsilon,\kappa}(0) \leq M_0$. By \eqref{DFGRTHVBSDFRGDFGNCVBSDFGDHDFHDFNCVBDSFGSDFGDSFBDVNCXVBSDFGSDFHDFGHDFTSADFASDFSADFASDXCVZXVSDGHFDGHVBCX56}--\eqref{DFGRTHVBSDFRGDFGNCVBSDFGDHDFHDFNCVBDSFGSDFGDSFBDVNCXVBSDFGSDFHDFGHDFTSADFASDFSADFASDXCVZXVSDGHFDGHVBCX150} and the continuation principle, we get   \begin{equation}    M_{\epsilon,\kappa}(t) \leq M_1    \comma t\in[0,T_2]    ,    \llabel{ OfcYYF sM y j0H vK3 ayU wt 4nA5 H76b wUqyJQ od O u8U Gjb t6v lc xYZt 6AUx wpYr18 uO v 62v jnw FrC rf Z4nl vJuh 2SpVLO vp O lZn PTG 07V Re ixBm XBxO BzpFW5 iB I O7R Vmo GnJ u8 Axol YAxl JUrYKV Kk p aIk VCu PiD O8 IHPU ndze LPTILB P5 B qYy DLZ DZa db jcJA T644 Vp6byb 1g 4 dE7 Ydz keO YL hCRe Ommx F9zsu0 rp 8 Ajz d2v Heo 7L 5zVn L8IQ WnYATK KV 1 f14 s2J geC b3 v9UJ djNN VBINix 1q 5 oyr SBM 2Xt gr v8RQ MaXk a4AN9i Ni n zfH xGp A57 uA E4jM fg6S 6eNGKv JL 3 tyH 3qw dPr x2 jFXW 2Wih pSSxDr aA 7 PXg jK6 GGl Og 5PkR d2n5 3eEx4N yG h d8Z RkO NMQ qL q4sE RG0C ssQkdZ Ua O vWr pla BOW rS wSG1 SM8I z9qkpd v0 C RMs GcZ LAz 4G k70e O7k6 df4uYn R6 T 5Du KOT say 0D awWQ vn2U OOPNqQ T7 H 4Hf iKY Jcl Rq M2g9 lcQZ cvCNBP 2B b tjv VYj ojr rh 78tW R886 ANdxeA SV P hK3 uPr QRs 6O SW1B wWM0 yNG9iB RI 7 opG CXk hZp Eo 2JNt kyYO pCY9HL 3o 7 Zu0 J9F Tz6 tZ GLn8 HAes o9umpy uc s 4l3 CA6 DCQ 0m 0llF Pbc8 z5Ad2l GN w SgA XeN HTN pw dS6e 3ila 2tlbXN 7c 1 itX aDZ Fak df Jkz7 TzaO 4kbVhn YH f Tda 9C3 WCb tw MXHW xoCC c4Ws2C UH B sNL FEf jS4 SG I4I4 hqHh 2nCaQ4 nM p nzY oYE 5fD sX hCHJ zTQO cbKmvE pl W Und VUo rrq iJ zRqT dIWS QBL96D FU d 64k 5gv Qh0 dj rGlw 795x V6KzhT l5 Y FtC rpy bHH 86 h3qn Lyzy ycGoqm Cb f h9h prB CQp Fe CxhU Z2oJ F3aKgQ H8 R yIm F9t Eks gP FMMJ TAIy z3ohWj Hx M RDFGRTHVBSDFRGDFGNCVBSDFGDHDFHDFNCVBDSFGSDFGDSFBDVNCXVBSDFGSDFHDFGHDFTSADFASDFSADFASDXCVZXVSDGHFDGHVBCX147}   \end{equation} and Theorem~\ref{T01} is proven. \end{proof} \par Sections~\ref{sec03}--\ref{sec05} are devoted to the proof of Lemma~\ref{L01}, thus completing the proof of Theorem~\ref{T01}. \par \begin{Remark}[Boundedness of Sobolev norms] \label{R04} {\rm By \cite[Theorem~1.1]{A05} the $H^5$ norm of  $(p^\epsilon, v^\epsilon, S^\epsilon)$ can be estimated by  a constant on a time interval $[0,T_0]$, where $T_0$ only depends on the $H^5$ norm of the initial data. More precisely, for given initial data satisfying \eqref{DFGRTHVBSDFRGDFGNCVBSDFGDHDFHDFNCVBDSFGSDFGDSFBDVNCXVBSDFGSDFHDFGHDFTSADFASDFSADFASDXCVZXVSDGHFDGHVBCX540}, there exists $T_0 >0$ and a constant $C$ such that   \begin{equation}    \sup_{0\leq m\leq 5,0 \leq j \leq m, \vert \alpha \vert = j} \Vert \partial^\alpha (\epsilon \partial_t)^{m-j} (p^\epsilon, v^\epsilon, S^\epsilon)(t) \Vert_{L^2}       \leq     C    \comma t\in[0,T_0]    \comma    \epsilon \in (0,1]    .    \llabel{86 KJO NKT c3 uyRN nSKH lhb11Q 9C w rf8 iiX qyY L4 zh9s 8NTE ve539G zL g vhD N7F eXo 5k AWAT 6Vrw htDQwy tu H Oa5 UIO Exb Mp V2AH puuC HWItfO ru x YfF qsa P8u fH F16C EBXK tj6ohs uv T 8BB PDN gGf KQ g6MB K2x9 jqRbHm jI U EKB Im0 bbK ac wqIX ijrF uq9906 Vy m 3Ve 1gB dMy 9i hnbA 3gBo 5aBKK5 gf J SmN eCW wOM t9 xutz wDkX IY7nNh Wd D ppZ UOq 2Ae 0a W7A6 XoIc TSLNDZ yf 2 XjB cUw eQT Zt cuXI DYsD hdAu3V MB B BKW IcF NWQ dO u3Fb c6F8 VN77Da IH E 3MZ luL YvB mN Z2wE auXX DGpeKR nw o UVB 2oM VVe hW 0ejG gbgz Iw9FwQ hN Y rFI 4pT lqr Wn Xzz2 qBba lv3snl 2j a vzU Snc pwh cG J0Di 3Lr3 rs6F23 6o b LtD vN9 KqA pO uold 3sec xqgSQN ZN f w5t BGX Pdv W0 k6G4 Byh9 V3IicO nR 2 obf x3j rwt 37 u82f wxwj SmOQq0 pq 4 qfv rN4 kFW hP HRmy lxBx 1zCUhs DN Y INv Ldt VDG 35 kTMT 0ChP EdjSG4 rW N 6v5 IIM TVB 5y cWuY OoU6 Sevyec OT f ZJv BjS ZZk M6 8vq4 NOpj X0oQ7r vM v myK ftb ioR l5 c4ID 72iF H0VbQz hj H U5Z 9EV MX8 1P GJss Wedm hBXKDA iq w UJV Gj2 rIS 92 AntB n1QP R3tTJr Z1 e lVo iKU stz A8 fCCg Mwfw 4jKbDb er B Rt6 T8O Zyn NO qXc5 3Pgf LK9oKe 1p P rYB BZY uui Cw XzA6 kaGb twGpmR Tm K viw HEz Rjh Te frip vLAX k3PkLN Dg 5 odc omQ j9L YI VawV mLpK rto0F6 Ns 7 Mmk cTL 9Tr 8f OT4u NNJv ZThOQw CO C RBH RTx hSB Na Iizz bKIB EcWSMY Eh D kRt PWG KtU mo 26ac LbBn I4t2P1 1e R iPP 99n j4q Q3 DFGRTHVBSDFRGDFGNCVBSDFGDHDFHDFNCVBDSFGSDFGDSFBDVNCXVBSDFGSDFHDFGHDFTSADFASDFSADFASDXCVZXVSDGHFDGHVBCX121}   
\end{equation} In the rest of the paper, we always work on an interval of time $[0,T]$ such that $0<T\leq T_0$. } \end{Remark} \par \begin{Remark}  \label{R05} {\rm (Boundedness of functions of solutions). If $F$ is a smooth function of $u$ and $S$, then from Remark~\ref{R04} there exists some constant $C$ depending on the function $F$ such that 	\begin{align} 	\Vert F(\epsilon u(t),S(t)) \Vert_{L^\infty} 	\leq 	C 	\comma 	t\in [0,T_0] 	\comma 	\epsilon \in (0,1] 	.    \llabel{62UN AQaH JPPY1O gL h N8s ta9 eJz Pg mE4z QgB0 mlAWBa 4E m u7m nfY gbN Lz ddGp hhJV 9hyAOG CN j xJ8 3Hg 6CA UT nusW 9pQr Wv1DfV lG n WxM Bbe 9Ww Lt OdwD ERml xJ8LTq KW T tsR 0cD XAf hR X1zX lAUu wzqnO2 o7 r toi SMr OKL Cq joq1 tUGG iIxusp oi i tja NRn gtx S0 r98r wXF7 GNiepz Ef A O2s Ykt Idg H1 AGcR rd2w 89xoOK yN n LaL RU0 3su U3 JbS8 dok8 tw9NQS Y4 j XY6 25K CcP Ly FRlS p759 DeVbY5 b6 9 jYO mdf b99 j1 5lvL vjsk K2gEwl Rx O tWL ytZ J1y Z5 Pit3 5SOi ivz4F8 tq M JIg QQi Oob Sp eprt 2vBV qhvzkL lf 7 HXA 4so MXj Wd MS7L eRDi ktUifL JH u kes trv rl7 mY cSOB 7nKW MD0xBq kb x FgT TNI wey VI G6Uy 3dL0 C3MzFx sB E 7zU hSe tBQ cX 7jn2 2rr0 yL1Erb pL R m3i da5 MdP ic dnMO iZCy Gd2MdK Ub x saI 9Tt nHX qA QBju N5I4 Q6zz4d SW Y Urh xTC uBg BU T992 uczE mkqK1o uC a HJB R0Q nv1 ar tFie kBu4 9ND9kK 9e K BOg PGz qfK J6 7NsK z3By wIwYxE oW Y f6A Kuy VPj 8B 9D6q uBkF CsKHUD Ck s DYK 3vs 0Ep 3g M2Ew lPGj RVX6cx lb V OfA ll7 g6y L9 PWyo 58h0 e07HO0 qz 8 kbe 85Z BVC YO KxNN La4a FZ7mw7 mo A CU1 q1l pfm E5 qXTA 0QqV MnRsbK zH o 5vX 1tp MVZ XC znmS OM73 CRHwQP Tl v VN7 lKX I06 KT 6MTj O3Yb 87pgoz ox y dVJ HPL 3k2 KR yx3b 0yPB sJmNjE TP J i4k m2f xMh 35 MtRo irNE 9bU7lM o4 b nj9 GgY A6v sE sONR tNmD FJej96 ST n 3lJ U2u 16o TE Xogv Mqwh D0BKr1 Ci s VYb A2w kfX 0n 4hD5 Lbr8 l7ErDFGRTHVBSDFRGDFGNCVBSDFGDHDFHDFNCVBDSFGSDFGDSFBDVNCXVBSDFGSDFHDFGHDFTSADFASDFSADFASDXCVZXVSDGHFDGHVBCX170} 	\end{align} } \end{Remark} \par \startnewsection{Analytic estimate of the entropy}{sec03} The following statement provides an analytic estimate for the entropy~$S$. \par \cole \begin{Lemma} \label{L02} Let $M_0>0$. For any $\kappa\in (0,1]$, there exists  $\tau_1\in (0,1]$ such that if $0<\tau(0)\leq \tau_1$, then   \begin{align}   \Vert S(t) \Vert_{A(\tau(t))}     \leq     C + t Q(M_{\epsilon,\kappa} (t))    \comma t\in (0,T_0]    ,    \label{DFGRTHVBSDFRGDFGNCVBSDFGDHDFHDFNCVBDSFGSDFGDSFBDVNCXVBSDFGSDFHDFGHDFTSADFASDFSADFASDXCVZXVSDGHFDGHVBCX19}   \end{align} for all $\epsilon\in(0,1]$, provided $\KK$ in \eqref{DFGRTHVBSDFRGDFGNCVBSDFGDHDFHDFNCVBDSFGSDFGDSFBDVNCXVBSDFGSDFHDFGHDFTSADFASDFSADFASDXCVZXVSDGHFDGHVBCX27} satisfies   \begin{equation}     \KK \geq     Q(M_{\epsilon,\kappa}(T_0))     ,    \llabel{fu N8 O cUj qeq zCC yx 6hPA yMrL eB8Cwl kT h ixd Izv iEW uw I8qK a0VZ EqOroD UP G phf IOF SKZ 3i cda7 Vh3y wUSzkk W8 S fU1 yHN 0A1 4z nyPU Ll6h pzlkq7 SK N aFq g9Y hj2 hJ 3pWS mi9X gjapmM Z6 H V8y jig pSN lI 9T8e Lhc1 eRRgZ8 85 e NJ8 w3s ecl 5i lCdo zV1B oOIk9g DZ N Y5q gVQ cFe TD VxhP mwPh EU41Lq 35 g CzP tc2 oPu gV KOp5 Gsf7 DFBlek to b d2y uDt ElX xm j1us DJJ6 hj0HBV Fa n Tva bFA VwM 51 nUH6 0GvT 9fAjTO 4M Q VzN NAQ iwS lS xf2p Q8qv tdjnvu pL A TIw ym4 nEY ES fMav UgZo yehtoe 9R T N15 EI1 aKJ SC nr4M jiYh B0A7vn SA Y nZ1 cXO I1V 7y ja0R 9jCT wxMUiM I5 l 2sT XnN RnV i1 KczL G3Mg JoEktl Ko U 13t saq jrH YV zfb1 yyxu npbRA5 6b r W45 Iqh fKo 0z j04I cGrH irwyH2 tJ b Fr3 leR dcp st vXe2 yJle kGVFCe 2a D 4XP OuI mtV oa zCKO 3uRI m2KFjt m5 R GWC vko zi7 5Y WNsb hORn xzRzw9 9T r Fhj hKb fqL Ab e2v5 n9mD 2VpNzl Mn n toi FZB 2Zj XB hhsK 8K6c GiSbRk kw f WeY JXd RBB xy qjEV F5lr 3dFrxG lT c sby AEN cqA 98 1IQ4 UGpB k0gBeJ 6D n 9Jh kne 5f5 18 umOu LnIa spzcRf oC 0 StS y0D F8N Nz F2Up PtNG 50tqKT k2 e 51y Ubr szn Qb eIui Y5qa SGjcXi El 4 5B5 Pny Qtn UO MHis kTC2 KsWkjh a6 l oMf gZK G3n Hp h0gn NQ7q 0QxsQk gQ w Kwy hfP 5qF Ww NaHx SKTA 63ClhG Bg a ruj HnG Kf4 6F QtVt SPgE gTeY6f JG m B3q gXx tR8 RT CPB1 8kQa jtt6GD rK b 1VY LV3 RgW Ir AyZf 69V8 VM7jHO b7 z Lva XTDFGRTHVBSDFRGDFGNCVBSDFGDHDFHDFNCVBDSFGSDFGDSFBDVNCXVBSDFGSDFHDFGHDFTSADFASDFSADFASDXCVZXVSDGHFDGHVBCX152}   \end{equation} where $T_0>0$ is a sufficiently small constant depending on $M_0$. \end{Lemma} \colb \par \begin{proof}[Proof of Lemma~\ref{L02}] Fix $m\in{\mathbb N}$ and  $|\alpha|=j$ where  $0\leq j \leq m$. We apply $\partial^\alpha (\epsilon \partial_t)^{m-j}$ to the equation \eqref{DFGRTHVBSDFRGDFGNCVBSDFGDHDFHDFNCVBDSFGSDFGDSFBDVNCXVBSDFGSDFHDFGHDFTSADFASDFSADFASDXCVZXVSDGHFDGHVBCX02}  and take the $L^2$-inner product with $\partial^\alpha (\epsilon \partial_t)^{m-j} S$ obtaining   \begin{align}    \frac{1}{2} \frac{d}{dt} \Vert \partial^\alpha (\epsilon \partial_t)^{m-j} S \Vert_{L^2}^2     +    \bigl\langle  v \cdot \nabla \partial^\alpha (\epsilon \partial_t)^{m-j} S, \partial^\alpha (\epsilon \partial_t)^{m-j} S \bigr\rangle     =     \bigl\langle [v\cdot \nabla , \partial^\alpha (\epsilon \partial_t)^{m-j} ] S, \partial^\alpha (\epsilon \partial_t)^{m-j} S \bigr\rangle    ,    \llabel{T VI0 ON KMBA HOwO Z7dPky Cg U S74 Hln FZM Ha br8m lHbQ NSwwdo mO L 6q5 wvR exV ej vVHk CEdX m3cU54 ju Z SKn g8w cj6 hR 1FnZ Jbkm gKXJgF m5 q Z5S ubX vPK DB OCGf 4srh 1a5FL0 vY f RjJ wUm 2sf Co gRha bxyc 0Rgava Rb k jzl teR GEx bE MMhL Zbh3 axosCq u7 k Z1P t6Y 8zJ Xt vmvP vAr3 LSWDjb VP N 7eN u20 r8B w2 ivnk zMda 93zWWi UB H wQz ahU iji 2T rXI8 v2HN ShbTKL eK W 83W rQK O4T Zm 57yz oVYZ JytSg2 Wx 4 Yaf THA xS7 ka cIPQ JGYd Dk0531 u2 Q IKf REW YcM KM UT7f dT9E kIfUJ3 pM W 59Q LFm u02 YH Jaa2 Er6K SIwTBG DJ Y Zwv fSJ Qby 7f dFWd fT9z U27ws5 oU 5 MUT DJz KFN oj dXRy BaYy bTvnhh 2d V 77o FFl t4H 0R NZjV J5BJ pyIqAO WW c efd R27 nGk jm oEFH janX f1ONEc yt o INt D90 ONa nd awDR Ki2D JzAqYH GC T B0p zdB a3O ot Pq1Q VFva YNTVz2 sZ J 6ey Ig2 N7P gi lKLF 9Nzc rhuLeC eX w b6c MFE xfl JS E8Ev 9WHg Q1Brp7 RO M ACw vAn ATq GZ Hwkd HA5f bABXo6 EW H soW 6HQ Yvv jc ZgRk OWAb VA0zBf Ba W wlI V05 Z6E 2J QjOe HcZG Juq90a c5 J h9h 0rL KfI Ht l8tP rtRd qql8TZ GU g dNy SBH oNr QC sxtg zuGA wHvyNx pM m wKQ uJF Kjt Zr 6Y4H dmrC bnF52g A0 3 28a Vuz Ebp lX Zd7E JEEC 939HQt ha M sup Tcx VaZ 32 pPdb PIj2 x8Azxj YX S q8L sof qmg Sq jm8G 4wUb Q28LuA ab w I0c FWN fGn zp VzsU eHsL 9zoBLl g5 j XQX nR0 giR mC LErq lDIP YeYXdu UJ E 0Bs bkK bjp dc PLie k8NW rIjsfa pH h 4GY vMF bA6 7q yex7 DFGRTHVBSDFRGDFGNCVBSDFGDHDFHDFNCVBDSFGSDFGDSFBDVNCXVBSDFGSDFHDFGHDFTSADFASDFSADFASDXCVZXVSDGHFDGHVBCX21}   \end{align} where $\langle \cdot,\cdot\rangle$  denotes the scalar product in $L^2$. Using the Cauchy-Schwarz inequality and summing over $\vert \alpha \vert = j$, we obtain   \begin{align}    \frac{d}{dt} \sum_{\vert \alpha \vert = j}  \Vert \partial^\alpha (\epsilon \partial_t)^{m-j}S \Vert_{L^2}     \leq     C \Vert \nabla v \Vert_{L_x^\infty} \sum_{\vert \alpha \vert = j}\Vert \partial^\alpha (\epsilon \partial_t)^{m-j} S \Vert_{L^2}     +    C \sum_{\vert \alpha \vert = j} \Vert [v\cdot \nabla, \partial^\alpha (\epsilon \partial_t)^{m-j} ]S \Vert_{L^2}    .    \llabel{sHgH G3GlW0 y1 W D35 mIo 5gE Ub Obrb knjg UQyko7 g2 y rEO fov QfA k6 UVDH Gl7G V3LvQm ra d EUO Jpu uzt BB nrme filt 1sGSf5 O0 a w2D c0h RaH Ga lEqI pfgP yNQoLH p2 L AIU p77 Fyg rj C8qB buxB kYX8NT mU v yT7 YnB gv5 K7 vq5N efB5 ye4TMu Cf m E2J F7h gqw I7 dmNx 2CqZ uLFthz Il B 1sj KA8 WGD Kc DKva bk9y p28TFP 0r g 0iA 9CB D36 c8 HLkZ nO2S 6Zoafv LX b 8go pYa 085 EM RbAb QjGt urIXlT E0 G z0t YSV Use Cj DvrQ 2bvf iIJCdf CA c WyI O7m lyc s5 Rjio IZt7 qyB7pL 9p y G8X DTz JxH s0 yhVV Ar8Z QRqsZC HH A DFT wvJ HeH OG vLJH uTfN a5j12Z kT v GqO yS8 826 D2 rj7r HDTL N7Ggmt 9M z cyg wxn j4J Je Qb7e MmwR nSuZLU 8q U NDL rdg C70 bh EPgp b7zk 5a32N1 Ib J hf8 XvG RmU Fd vIUk wPFb idJPLl NG e 1RQ RsK 2dV NP M7A3 Yhdh B1R6N5 MJ i 5S4 R49 8lw Y9 I8RH xQKL lAk8W3 Ts 7 WFU oNw I9K Wn ztPx rZLv NwZ28E YO n ouf xz6 ip9 aS WnNQ ASri wYC1sO tS q Xzo t8k 4KO z7 8LG6 GMNC ExoMh9 wl 5 vbs mnn q6H g6 WToJ un74 JxyNBX yV p vxN B0N 8wy mK 3reR eEzF xbK92x EL s 950 SNg Lmv iR C1bF HjDC ke3Sgt Ud C 4cO Nb4 EF2 4D 1VDB HlWA Tyswjy DO W ibT HqX t3a G6 mkfG JVWv 40lexP nI c y5c kRM D3o wV BdxQ m6Cv LaAgxi Jt E sSl ZFw DoY P2 nRYb CdXR z5HboV TU 8 NPg NVi WeX GV QZ7b jOy1 LRy9fa j9 n 2iE 1S0 mci 0Y D3Hg UxzL atb92M hC p ZKL JqH TSF RM n3KV kpcF LUcF0X 66 i vdq 01c Vqk oQ qu1u 2Cpi p5EV7A gMDFGRTHVBSDFRGDFGNCVBSDFGDHDFHDFNCVBDSFGSDFGDSFBDVNCXVBSDFGSDFHDFGHDFTSADFASDFSADFASDXCVZXVSDGHFDGHVBCX14}   \end{align} With the notation \eqref{DFGRTHVBSDFRGDFGNCVBSDFGDHDFHDFNCVBDSFGSDFGDSFBDVNCXVBSDFGSDFHDFGHDFTSADFASDFSADFASDXCVZXVSDGHFDGHVBCX36}, the above estimate implies   \begin{align}   \begin{split}    \frac{d}{dt} \Vert S \Vert_{A(\tau)}    &    =    \dot{\tau}(t) \Vert S \Vert_{\tilde A(\tau)}     +     \sum_{m=1}^\infty \sum_{j=0}^m \sum_{\vert \alpha \vert = j} \frac{\kappa^{(j-3)_+}\tau^{(m-3)_+}}{(m-3)!}      \frac{d}{dt} \Vert \partial^\alpha (\epsilon \partial_t)^{m-j} S\Vert_{L^2}    \\    &    \leq     \dot{\tau}(t) \Vert S \Vert_{\tilde A(\tau)}    +    C \Vert \nabla v \Vert_{L_x^\infty} \Vert S \Vert_{A(\tau)}     +    C \sum_{m=1}^\infty        \sum_{j=0}^m \sum_{l=0}^j       \sum_{\vert \alpha \vert = j}       \sum_{\substack{\beta \leq \alpha\\ \vert \beta \vert = l}}       \sum_{\substack{k=0\\ 1 \leq l + k}}^{m-j}       \mathcal{C}_{m,j,l,\alpha,\beta,k}   ,   \end{split}   \label{DFGRTHVBSDFRGDFGNCVBSDFGDHDFHDFNCVBDSFGSDFGDSFBDVNCXVBSDFGSDFHDFGHDFTSADFASDFSADFASDXCVZXVSDGHFDGHVBCX05} \end{align} where   \begin{align}    \mathcal{C}_{m,j,l,\alpha,\beta,k} =\frac{\kappa^{(j-3)_+}\tau^{(m-3)_+}}{(m-3)!} \binom{\alpha}{\beta} \binom{m-j}{k}\Vert \partial^\beta (\epsilon \partial_t)^k v \cdot \partial^{\alpha - \beta}  (\epsilon \partial_t)^{m-j-k} \nabla S \Vert_{L^2}    \llabel{ O Rcf ZjL x7L cv 9lXn 6rS8 WeK3zT LD P B61 JVW wMi KE uUZZ 4qiK 1iQ8N0 83 2 TS4 eLW 4ze Uy onzT Sofn a74RQV Ki u 9W3 kEa 3gH 8x diOh AcHs IQCsEt 0Q i 2IH w9v q9r NP lh1y 3wOR qrJcxU 4i 5 5ZH TOo GP0 zE qlB3 lkwG GRn7TO oK f GZu 5Bc zGK Fe oyIB tjNb 8xfQEK du O nJV OZh 8PU Va RonX BkIj BT9WWo r7 A 3Wf XxA 2f2 Vl XZS1 Ttsa b4n6R3 BK X 0XJ Tml kVt cW TMCs iFVy jfcrze Jk 5 MBx wR7 zzV On jlLz Uz5u LeqWjD ul 7 OnY ICG G9i Ry bTsY JXfr Rnub3p 16 J BQd 0zQ OkK ZK 6DeV gpXR ceOExL Y3 W KrX YyI e7d qM qanC CTjF W71LQ8 9m Q w1g Asw nYS Me WlHz 7ud7 xBwxF3 m8 u sa6 6yr 0nS ds Ywuq wXdD 0fRjFp eL O e0r csI uMG rS OqRE W5pl ybq3rF rk 7 YmL URU SSV YG ruD6 ksnL XBkvVS 2q 0 ljM PpI L27 Qd ZMUP baOo Lqt3bh n6 R X9h PAd QRp 9P I4fB kJ8u ILIArp Tl 4 E6j rUY wuF Xi FYaD VvrD b2zVpv Gg 6 zFY ojS bMB hr 4pW8 OwDN Uao2mh DT S cei 90K rsm wa BnNU sHe6 RpIq1h XF N Pm0 iVs nGk bC Jr8V megl 416tU2 nn o llO tcF UM7 c4 GC8C lasl J0N8Xf Cu R aR2 sYe fjV ri JNj1 f2ty vqJyQN X1 F YmT l5N 17t kb BTPu F471 AH0Fo7 1R E ILJ p4V sqi WT TtkA d5Rk kJH3Ri RN K ePe sR0 xqF qn QjGU IniV gLGCl2 He 7 kmq hEV 4PF dC dGpE P9nB mcvZ0p LY G idf n65 qEu Df Mz2v cq4D MzN6mB FR t QP0 yDD Fxj uZ iZPE 3Jj4 hVc2zr rc R OnF PeO P1p Zg nsHA MRK4 ETNF23 Kt f Gem 2kr 5gf 5u 8Ncu wfJC av6SvQ 2n 1 8P8 RcI kmMDFGRTHVBSDFRGDFGNCVBSDFGDHDFHDFNCVBDSFGSDFGDSFBDVNCXVBSDFGSDFHDFGHDFTSADFASDFSADFASDXCVZXVSDGHFDGHVBCX23}   \end{align} with   \begin{align}   \begin{split}    \Vert u \Vert_{\tilde A(\tau)}     &    =    \sum_{m=4}^\infty \sum_{j=0}^m \sum_{\vert \alpha \vert =j} \Vert \partial^\alpha (\epsilon \partial_t)^{m-j} u \Vert_{L^2} \frac{\kappa^{(j-3)_+} (m-3)\tau(t)^{m-4}}{(m-3)!}\\   \end{split}   \llabel{ SD 0wrV R1PY x7kEkZ Js J 7Wb 6XI WDE 0U nqtZ PAqE ETS3Eq NN f 38D Ek6 NhX V9 c3se vM32 WACSj3 eN X uq9 GhP OPC hd 7v1T 6gqR inehWk 8w L oaa wHV vbU 49 02yO bCT6 zm2aNf 8x U wPO ilr R3v 8R cNWE k7Ev IAI8ok PA Y xPi UlZ 4mw zs Jo6r uPmY N6tylD Ee e oTm lBK mnV uB B7Hn U7qK n353Sn dt o L82 gDi fcm jL hHx3 gi0a kymhua FT z RnM ibF GU5 W5 x651 0NKi 85u8JT LY c bfO Mn0 auD 0t vNHw SAWz E3HWcY TI d 2Hh XML iGi yk AjHC nRX4 uJJlct Q3 y Loq i9j u7K j8 4EFU 49ud eA93xZ fZ C BW4 bSK pyc f6 nncm vnhK b0HjuK Wp 6 b88 pGC 3U7 km CO1e Y8jv Ebu59z mG Z sZh 93N wvJ Yb kEgD pJBj gQeQUH 9k C az6 ZGp cpg rH r79I eQvT Idp35m wW m afR gjD vXS 7a FgmN IWmj vopqUu xF r BYm oa4 5jq kR gTBP PKLg oMLjiw IZ 2 I4F 91C 6x9 ae W7Tq 9CeM 62kef7 MU b ovx Wyx gID cL 8Xsz u2pZ TcbjaK 0f K zEy znV 0WF Yx bFOZ JYzB CXtQ4u xU 9 6Tn N0C GBh WE FZr6 0rIg w2f9x0 fW 3 kUB 4AO fct vL 5I0A NOLd w7h8zK 12 S TKy 2Zd ewo XY PZLV Vvtr aCxAJm N7 M rmI arJ tfT dd DWE9 At6m hMPCVN UO O SZY tGk Pvx ps GeRg uDvt WTHMHf 3V y r6W 3xv cpi 0z 2wfw Q1DL 1wHedT qX l yoj GIQ AdE EK v7Ta k7cA ilRfvr lm 8 2Nj Ng9 KDS vN oQiN hng2 tnBSVw d8 P 4o3 oLq rzP NH ZmkQ Itfj 61TcOQ PJ b lsB Yq3 Nul Nf rCon Z6kZ 2VbZ0p sQ A aUC iMa oRp FW fviT xmey zmc5Qs El 1 PNO Z4x otc iI nwc6 IFbp wsMeXx y8 l J4A 6OV 0qR zr St3P MbvR DFGRTHVBSDFRGDFGNCVBSDFGDHDFHDFNCVBDSFGSDFGDSFBDVNCXVBSDFGSDFHDFGHDFTSADFASDFSADFASDXCVZXVSDGHFDGHVBCX102}   \end{align}    denoting the dissipative analytic norm corresponding to \eqref{DFGRTHVBSDFRGDFGNCVBSDFGDHDFHDFNCVBDSFGSDFGDSFBDVNCXVBSDFGSDFHDFGHDFTSADFASDFSADFASDXCVZXVSDGHFDGHVBCX36}. In the above sums as well as below, the multiindices  $\alpha,\beta,\ldots$ are assumed to belong to ${\mathbb N}_0^{3}$. The third term on the far right side of \eqref{DFGRTHVBSDFRGDFGNCVBSDFGDHDFHDFNCVBDSFGSDFGDSFBDVNCXVBSDFGSDFHDFGHDFTSADFASDFSADFASDXCVZXVSDGHFDGHVBCX05} equals   \begin{align}   \begin{split}    \mathcal{C}     &=    C\sum_{m=1}^4 \sum_{j=0}^m \sum_{l=0}^j \sum_{\vert \alpha \vert = j} \sum_{\substack{\beta \leq \alpha\\\vert \beta \vert = l}} \sum_{\substack{k=0\\  1 \leq l+ k }}^{m-j}  \mathcal{C}_{m,j,l,\alpha, \beta, k}    \\&\indeq    +    C\sum_{m=5}^\infty \sum_{j=0}^m \sum_{l=0}^j \sum_{\vert \alpha \vert = j} \sum_{\substack{\beta \leq \alpha\\\vert \beta \vert = l}} \sum_{\substack{k=0\\  1 \leq l+ k \leq [m/2]}}^{m-j}  \mathcal{C}_{m,j,l,\alpha, \beta, k}    \\&\indeq    +    C\sum_{m=7}^\infty \sum_{j=0}^m \sum_{l=0}^j \sum_{\vert \alpha \vert = j} \sum_{\substack{\beta \leq \alpha\\\vert \beta \vert = l}} \sum_{\substack{k=0\\  [m/2]+1\leq l+ k \leq m-3}}^{m-j}  \mathcal{C}_{m,j,l,\alpha, \beta, k}\\    &\indeq    +    C\sum_{m=5}^\infty \sum_{j=0}^m \sum_{l=0}^j \sum_{\vert \alpha \vert = j} \sum_{\substack{\beta \leq \alpha\\\vert \beta \vert = l}} \sum_{k=0}^{m-j}  \mathcal{C}_{m,j,l,\alpha, \beta, k} \mathbbm{1}_{\{m-2\leq l+k\leq m\}}    \\&    =    \mathcal{C}_1    +    \mathcal{C}_2    +    \mathcal{C}_3    +    \mathcal{C}_4    ,   \end{split}   \llabel{gOS5ob ka F U9p OdM Pdj Fz 1KRX RKDV UjveW3 d9 s hi3 jzK BTq Zk eSXq bzbo WTc5yR RM o BYQ PCa eZ2 3H Wk9x fdxJ YxHYuN MN G Y4X LVZ oPU Qx JAli DHOK ycMAcT pG H Ikt jlI V25 YY oRC7 4thS sJClD7 6y x M6B Rhg fS0 UH 4wXV F0x1 M6Ibem sT K SWl sG9 pk9 5k ZSdH U31c 5BpQeF x5 z a7h WPl LjD Yd KH1p OkMo 1Tvhxx z5 F LLu 71D UNe UX tDFC 7CZ2 473sjE Re b aYt 2sE pV9 wD J8RG UqQm boXwJn HK F Mps XBv AsX 8N YRZM wmZQ ctltsq of i 8wx n6I W8j c6 8ANB wz8f 4gWowk mZ P Wlw fKp M1f pd o0yT RIKH MDgTl3 BU B Wr6 vHU zFZ bq xnwK kdmJ 3lXzIw kw 7 Jku JcC kgv FZ 3lSo 0ljV Ku9Syb y4 6 zDj M6R XZI DP pHqE fkHt 9SVnVt Wd y YNw dmM m7S Pw mqhO 6FX8 tzwYaM vj z pBS NJ1 z36 89 00v2 i4y2 wQjZhw wF U jq0 UNm k8J 8d OOG3 QlDz p8AWpr uu 4 D9V Rlp VVz QQ g1ca Eqev P0sFPH cw t KI3 Z6n Y79 iQ abga 0i9m RVGbvl TA g V6P UV8 Eup PQ 6xvG bcn7 dQjV7C kw 5 7NP WUy 9Xn wF 9ele bZ8U YJDx3x CB Y CId PCE 2D8 eP 90u4 9NY9 Jxx9RI 4F e a0Q Cjs 5TL od JFph ykcz Bwoe97 Po h Tql 1LM s37 cK hsHO 5jZx qpkHtL bF D nvf Txj iyk LV hpwM qobq DM9A0f 1n 4 i5S Bc6 trq VX wgQB EgH8 lISLPL O5 2 EUv i1m yxk nL 0RBe bO2Y Ww8Jhf o1 l HlU Mie sst dW w4aS WrYv Osn5Wn 3w f wzH RHx Fg0 hK FuNV hjzX bg56HJ 9V t Uwa lOX fT8 oi FY1C sUCg CETCIv LR 0 AgT hCs 9Ta Zl 6ver 8hRt edkAUr kI n Sbc I8n yEj Zs VOSz tBbh 7WjBgf aA F t4DFGRTHVBSDFRGDFGNCVBSDFGDHDFHDFNCVBDSFGSDFGDSFBDVNCXVBSDFGSDFHDFGHDFTSADFASDFSADFASDXCVZXVSDGHFDGHVBCX37}   \end{align} where we split the sum according to the low and high values of $l + k$ and $m$. We claim that there exists $T_0>0$, such that for any $\kappa \in (0,1]$,  there is $\tau_1\in (0,1]$ such that if $0<\tau(0)\leq \tau_1$, then   \begin{align}   &   \mathcal{C}_1   \leq   C,   \label{DFGRTHVBSDFRGDFGNCVBSDFGDHDFHDFNCVBDSFGSDFGDSFBDVNCXVBSDFGSDFHDFGHDFTSADFASDFSADFASDXCVZXVSDGHFDGHVBCX315}   \\    &    \mathcal{C}_2    \leq    C\Vert v \Vert_{A(\tau)} \Vert S \Vert_{\tilde A(\tau)},    \label{DFGRTHVBSDFRGDFGNCVBSDFGDHDFHDFNCVBDSFGSDFGDSFBDVNCXVBSDFGSDFHDFGHDFTSADFASDFSADFASDXCVZXVSDGHFDGHVBCX79}    \\    &    \mathcal{C}_3    \leq    C \Vert v \Vert_{A(\tau)} \Vert S \Vert_{\tilde A(\tau)},    \label{DFGRTHVBSDFRGDFGNCVBSDFGDHDFHDFNCVBDSFGSDFGDSFBDVNCXVBSDFGSDFHDFGHDFTSADFASDFSADFASDXCVZXVSDGHFDGHVBCX80}    \\    &    \mathcal{C}_4    \leq    C \Vert v \Vert_{A(\tau)}    .    \label{DFGRTHVBSDFRGDFGNCVBSDFGDHDFHDFNCVBDSFGSDFGDSFBDVNCXVBSDFGSDFHDFGHDFTSADFASDFSADFASDXCVZXVSDGHFDGHVBCX316}   \end{align}    Proof of \eqref{DFGRTHVBSDFRGDFGNCVBSDFGDHDFHDFNCVBDSFGSDFGDSFBDVNCXVBSDFGSDFHDFGHDFTSADFASDFSADFASDXCVZXVSDGHFDGHVBCX315}: Using H\"older's and the Sobolev inequalities, $\mathcal{C}_1$ may be estimated by low-order mixed space-time derivatives, and \eqref{DFGRTHVBSDFRGDFGNCVBSDFGDHDFHDFNCVBDSFGSDFGDSFBDVNCXVBSDFGSDFHDFGHDFTSADFASDFSADFASDXCVZXVSDGHFDGHVBCX315} follows by appealing to Remark~\ref{R04}.    Proof of \eqref{DFGRTHVBSDFRGDFGNCVBSDFGDHDFHDFNCVBDSFGSDFGDSFBDVNCXVBSDFGSDFHDFGHDFTSADFASDFSADFASDXCVZXVSDGHFDGHVBCX79}: Using H\"older's and the Sobolev inequalities we arrive at   \begin{align}   \begin{split}   \mathcal{C}_2   &    \leq    C\sum_{m=5}^\infty \sum_{j=0}^m \sum_{l=0}^j \sum_{\vert \alpha \vert = j} \sum_{\substack{\beta \leq \alpha\\\vert \beta \vert = l}} \sum_{\substack{k=0\\  1 \leq l+ k \leq [m/2]}}^{m-j} \frac{\kappa^{(j-3)_+}\tau^{m-3}}{(m-3)!} \binom{\alpha}{\beta} \binom{m-j}{k}  \Vert \partial^\beta (\epsilon \partial_t)^k v \Vert_{L^2}^{1/4}    \\&\indeqtimes    \Vert D^2 \partial^\beta (\epsilon \partial_t)^k v \Vert_{L^2}^{3/4}     \Vert \partial^{\alpha - \beta} (\epsilon \partial_t)^{m-j-k} \nabla S \Vert_{L^2}   ,   \end{split}    \llabel{J 6CT UCU 54 3rba vpOM yelWYW hV B RGo w5J Rh2 nM fUco BkBX UQ7UlO 5r Y fHD Mce Wou 3R oFWt baKh 70oHBZ n7 u nRp Rh3 SIp p0 Btqk 5vhX CU9BHJ Fx 7 qPx B55 a7R kO yHmS h5vw rDqt0n F7 t oPJ UGq HfY 5u At5k QLP6 ppnRjM Hk 3 HGq Z0O Bug FF xSnA SHBI 7agVfq wf g aAl eH9 DMn XQ QTAA QM8q z9trz8 6V R 2gO MMV uMg f6 tGLZ WEKq vkMEOg Uz M xgN 4Cb Q8f WY 9Tk7 3Gg9 0jy9dJ bO v ddV Zmq Jjb 5q Q5BS Ffl2 tNPRC8 6t I 0PI dLD UqX KO 1ulg XjPV lfDFkF h4 2 W0j wkk H8d xI kjy6 GDge M9mbTY tU S 4lt yAV uor 6w 7Inw Ch6G G9Km3Y oz b uVq tsX TNZ aq mwkz oKxE 9O0QBQ Xh x N5L qr6 x7S xm vRwT SBGJ Y5uo5w SN G p3h Ccf QNa fX Wjxe AFyC xUfM8c 0k K kwg psv wVe 4t FsGU IzoW FYfnQA UT 9 xcl Tfi mLC JR XFAm He7V bYOaFB Pj j eF6 xI3 CzO Vv imZ3 2pt5 uveTrh U6 y 8wj wAy IU3 G1 5HMy bdau GckOFn q6 a 5Ha R4D Ooj rN Ajdh SmhO tphQpc 9j X X2u 5rw PHz W0 32fi 2bz1 60Ka4F Dj d 1yV FSM TzS vF 1YkR zdzb YbI0qj KM N XBF tXo CZd j9 jD5A dSrN BdunlT DI a A4U jYS x6D K1 X16i 3yiQ uq4zoo Hv H qNg T2V kWG BV A4qe o8HH 70FflA qT D BKi 461 GvM gz d7Wr iqtF q24GYc yi f YkW Hv7 EI0 aq 5JKl fNDC NmWom3 Vy X JsN t4W P8y Gg AoAT OkVW Z4ODLt kz a 9Pa dGC GQ2 FC H6EQ ppks xFKMWA fY 0 Jda SYg o7h hG wHtt bb4z 5qrcdc 9C n Amx qY6 m8u Gf 7DZQ 6FBU PPiOxg sQ 0 CZl PYP Ba7 5O iV6t ZOBp fYuNcb j4 V Upb TKX ZRJ f3 6DFGRTHVBSDFRGDFGNCVBSDFGDHDFHDFNCVBDSFGSDFGDSFBDVNCXVBSDFGSDFHDFGHDFTSADFASDFSADFASDXCVZXVSDGHFDGHVBCX45}   \end{align} and thus   \begin{align}   \begin{split}   \mathcal{C}_2   &    \leq    C\sum_{m=5}^\infty \sum_{j=0}^m \sum_{l=0}^j \sum_{\vert \alpha \vert = j} \sum_{\substack{\beta \leq \alpha\\\vert \beta \vert = l}} \sum_{\substack{k=0\\  1 \leq l+ k \leq [m/2]}}^{m-j}      \kappa^a \tau^b      \left(\Vert \partial^\beta (\epsilon \partial_t)^k v \Vert_{L^2} \frac{\kappa^{(l-3)_+}\tau^{(l+k-3)_+}}{(l+k-3)!}\right)^{1/4}\\    &\indeqtimes     \left(\Vert D^2 \partial^\beta (\epsilon \partial_t)^k v \Vert_{L^2}\frac{\kappa^{(l-1)_+}\tau^{(l+k-1)_+}}{(l+k-1)!}\right)^{3/4}    \\&\indeqtimes    \left(\Vert \partial^{\alpha - \beta} (\epsilon \partial_t)^{m-j-k} \nabla S \Vert_{L^2} \frac{\kappa^{(j-l-2)_+}(m-k-l-2)\tau^{m-k-l-3}}{(m-k-l-2)!}\right)    \mathcal{A}_{m,j,l,\alpha, \beta, k}   ,   \end{split}    \label{DFGRTHVBSDFRGDFGNCVBSDFGDHDFHDFNCVBDSFGSDFGDSFBDVNCXVBSDFGSDFHDFGHDFTSADFASDFSADFASDXCVZXVSDGHFDGHVBCX04}   \end{align} where    \begin{align}    \mathcal{A}_{m,j,l,\alpha,\beta,k}     &    =     \binom{\alpha}{\beta} \binom{m-j}{k} \frac{(l+k-3)!^{1/4} (l+k-1)!^{3/4} (m-k-l-2)!}{(m-k-l-2)(m-3)!}    \label{DFGRTHVBSDFRGDFGNCVBSDFGDHDFHDFNCVBDSFGSDFGDSFBDVNCXVBSDFGSDFHDFGHDFTSADFASDFSADFASDXCVZXVSDGHFDGHVBCX53}   \end{align} and   \begin{align}   \begin{split}        a        &        =        (j-3)_+ - \left(\frac{l-3}{4}\right)_+ - \left(\frac{3l-3}{4}\right)_+ - (j-l-2)_+,        \\        b         &        =         m-3 -\left(\frac{l+k-3}{4}\right)_+ - \left(\frac{3l+3k-3}{4}\right)_+ -(m-k-l-3).   \end{split}    \label{DFGRTHVBSDFRGDFGNCVBSDFGDHDFHDFNCVBDSFGSDFGDSFBDVNCXVBSDFGSDFHDFGHDFTSADFASDFSADFASDXCVZXVSDGHFDGHVBCX378}   \end{align} For simplicity, we omitted indicating the dependence of $a$ and $b$ on $j$, $k$, and $l$. Since $l+k \geq 1$ and $0\leq l \leq j$, one can readily check that $-3/2 \leq a \leq 3/2$ and $ 1 \leq b \leq 3/2 $, which implies    \begin{align}    \kappa^a \tau^b \leq C    \label{DFGRTHVBSDFRGDFGNCVBSDFGDHDFHDFNCVBDSFGSDFGDSFBDVNCXVBSDFGSDFHDFGHDFTSADFASDFSADFASDXCVZXVSDGHFDGHVBCX310}   \end{align} if    \begin{equation}    \tau(0) \leq \kappa^3       .    \label{DFGRTHVBSDFRGDFGNCVBSDFGDHDFHDFNCVBDSFGSDFGDSFBDVNCXVBSDFGSDFHDFGHDFTSADFASDFSADFASDXCVZXVSDGHFDGHVBCX54}   \end{equation} Recall the combinatorial inequality   \begin{align}    \binom{\alpha}{\beta}     \leq    \binom{\vert \alpha \vert}{\vert \beta \vert},    \label{DFGRTHVBSDFRGDFGNCVBSDFGDHDFHDFNCVBDSFGSDFGDSFBDVNCXVBSDFGSDFHDFGHDFTSADFASDFSADFASDXCVZXVSDGHFDGHVBCX59}   \end{align} which may also be written as   \begin{align}    \binom{j}{l} \binom{m-j}{k}     \leq     \binom{m}{l+k}     ,    \label{DFGRTHVBSDFRGDFGNCVBSDFGDHDFHDFNCVBDSFGSDFGDSFBDVNCXVBSDFGSDFHDFGHDFTSADFASDFSADFASDXCVZXVSDGHFDGHVBCX113}   \end{align} from where we obtain  \begin{align}   \begin{split}    \mathcal{A}_{m,j,l,\alpha,\beta,k}     &    \leq     \frac{Cm!}{(l+k)!(m-l-k)!} \frac{(l+k-3)!(l+k)^{3/2} (m-k-l-3)!}{(m-3)!}\\    &    \leq    \frac{Cm^3}{(m-l-k)^3}    \leq    C,   \end{split}   \label{DFGRTHVBSDFRGDFGNCVBSDFGDHDFHDFNCVBDSFGSDFGDSFBDVNCXVBSDFGSDFHDFGHDFTSADFASDFSADFASDXCVZXVSDGHFDGHVBCX58}   \end{align} since $l+k \leq [m/2]$. Using   \begin{equation}    \sum_{|\alpha|=j}    \sum_{\substack{ \beta\leq \alpha\\ \vert \beta \vert = l}}     x_{\beta} y_{\alpha-\beta}     =     \left(      \sum_{|\beta|=l} x_{\beta}     \right)     \left(      \sum_{|\gamma|=j-l} y_{\gamma}     \right)    \label{DFGRTHVBSDFRGDFGNCVBSDFGDHDFHDFNCVBDSFGSDFGDSFBDVNCXVBSDFGSDFHDFGHDFTSADFASDFSADFASDXCVZXVSDGHFDGHVBCX114}   \end{equation} from \cite[Lemma~4.2]{KV}, together with \eqref{DFGRTHVBSDFRGDFGNCVBSDFGDHDFHDFNCVBDSFGSDFGDSFBDVNCXVBSDFGSDFHDFGHDFTSADFASDFSADFASDXCVZXVSDGHFDGHVBCX04}, \eqref{DFGRTHVBSDFRGDFGNCVBSDFGDHDFHDFNCVBDSFGSDFGDSFBDVNCXVBSDFGSDFHDFGHDFTSADFASDFSADFASDXCVZXVSDGHFDGHVBCX53}--\eqref{DFGRTHVBSDFRGDFGNCVBSDFGDHDFHDFNCVBDSFGSDFGDSFBDVNCXVBSDFGSDFHDFGHDFTSADFASDFSADFASDXCVZXVSDGHFDGHVBCX310}, \eqref{DFGRTHVBSDFRGDFGNCVBSDFGDHDFHDFNCVBDSFGSDFGDSFBDVNCXVBSDFGSDFHDFGHDFTSADFASDFSADFASDXCVZXVSDGHFDGHVBCX58} and the discrete H\"older inequality, we obtain   \begin{align}   \begin{split}    \mathcal{C}_2    &    \leq     C \sum_{m=5}^\infty \sum_{j=0}^m \sum_{l=0}^j \sum_{\substack{k=0\\ 1 \leq l+k \leq [m/2]}}^{m-j}   \left(\sum_{\vert\beta \vert= l} \Vert \partial^\beta (\epsilon \partial_t)^k v \Vert_{L^2}\frac{\kappa^{(l-3)_+}\tau^{(l+k-3)_+}}{(l+k-3)!}\right)^{1/4}\\    &\indeqtimes    \left(\sum_{\vert\beta \vert = l}\Vert D^2 \partial^\beta (\epsilon \partial_t)^k v \Vert_{L^2} \frac{\kappa^{(l-1)_+}\tau^{(l+k-1)_+}}{(l+k-1)!}\right)^{3/4}    \\&\indeqtimes    \left(\sum_{\vert \gamma \vert = j-l} \Vert \partial^{\gamma} (\epsilon \partial_t)^{m-j-k} \nabla S \Vert_{L^2} \frac{\kappa^{(j-l-2)_+}(m-k-l-2)\tau^{m-k-l-3}}{(m-k-l-2)!}\right)    \\&    \leq    C\Vert v \Vert_{A(\tau)} \Vert S \Vert_{\tilde A(\tau)},   \end{split}   \label{DFGRTHVBSDFRGDFGNCVBSDFGDHDFHDFNCVBDSFGSDFGDSFBDVNCXVBSDFGSDFHDFGHDFTSADFASDFSADFASDXCVZXVSDGHFDGHVBCX06}   \end{align} where the last inequality follows from the discrete Young inequality. \par Proof of \eqref{DFGRTHVBSDFRGDFGNCVBSDFGDHDFHDFNCVBDSFGSDFGDSFBDVNCXVBSDFGSDFHDFGHDFTSADFASDFSADFASDXCVZXVSDGHFDGHVBCX80}: We reverse the roles of $l+k$ and $m-l-k$ and proceed as above, arriving at   \begin{align}   \begin{split}    \mathcal{C}_3    &    \leq    C\sum_{m=7}^\infty \sum_{j=0}^m \sum_{l=0}^j \sum_{\vert \alpha \vert = j} \sum_{\substack{\beta \leq \alpha\\\vert \beta \vert = l}} \sum_{\substack{k=0\\  [m/2]+1\leq l+ k \leq m-3}}^{m-j}    \kappa^a \tau^b    \left(\Vert \partial^\beta (\epsilon \partial_t)^k v \Vert_{L^2} \frac{\kappa^{(l-3)_+}\tau^{l+k-3}}{(l+k-3)!}\right) \\    &\indeqtimes    \left(\Vert \partial^{\alpha-\beta} (\epsilon \partial_t)^{m-j-k} \nabla S \Vert_{L^2}\frac{\kappa^{(j-l-2)_+}(m-l-k-2)\tau^{m-l-k-3}}{(m-l-k-2)!}\right)^{1/4}    \\&\indeqtimes    \left(\Vert D^2 \partial^{\alpha - \beta} (\epsilon \partial_t)^{m-j-k} \nabla S \Vert_{L^2} \frac{\kappa^{(j-l)_+}(m-l-k)\tau^{m-l-k-1}}{(m-l-k)!}\right)^{3/4} \mathcal{B}_{m,j,l,\alpha, \beta, k},    \label{DFGRTHVBSDFRGDFGNCVBSDFGDHDFHDFNCVBDSFGSDFGDSFBDVNCXVBSDFGSDFHDFGHDFTSADFASDFSADFASDXCVZXVSDGHFDGHVBCX60}   \end{split}   \end{align} where we denote   \begin{align}    \mathcal{B}_{m,j,l,\alpha, \beta, k}         &    =    \binom{\alpha}{\beta} \binom{m-j}{k}\frac{(l+k-3)!(m-l-k-2)!^{1/4}(m-l-k)!^{3/4}}{(m-l-k-2)^{1/4}(m-l-k)^{3/4}(m-3)!}    \llabel{EA0 LDgA dfdOpS bg 1 ynC PUV oRW xe WQMK Smuh 3JHqX1 5A P JJX 2v0 W6l m0 llC8 hlss 1NLWaN hR B Aqf Iuz kx2 sp 01oD rYsR ywFrNb z1 h Gpq 99F wUz lf cQkT sbCv GIIgmf Hh T rM1 ItD gCM zY ttQR jzFx XIgI7F MA p 1kl lwJ sGo dX AT2P goIp 9VonFk wZ V Qif q9C lAQ 4Y BwFR 4nCy RAg84M LJ u nx8 uKT F3F zl GEQt l32y 174wLX Zm 6 2xX 5xG oaC Hv gZFE myDI zj3q10 RZ r ssw ByA 2Wl OA DDDQ Vin8 PTFLGm wi 6 pgR ZQ6 A5T Ll mnFV tNiJ bnUkLy vq 9 zSB P6e JJq 7P 6RFa im6K XPWaxm 6W 7 fM8 3uK D6k Nj 7vhg 4ppZ 4ObMaS aP H 0oq xAB G8v qr qT6Q iRGH BCCN1Z bl T Y4z q8l FqL Ck ghxD UuZw 7MXCD4 ps Z cEX 9Rl Cwf 0C CG8b gFti Uv3mQe LW J oyF kv6 hcS nM mKbi QukL FpYAqo 5F j f9R RRt qS6 XW VoIY VDMl a5c7cW KJ L Uqc vti IOe VC U7xJ dC5W 5bk3fQ by Z jtU Dme gbg I1 79dl U3u3 cvWoAI ow b EZ0 xP2 FBM Sw azV1 XfzV i97mmy 5s T JK0 hz9 O6p Da Gcty tmHT DYxTUB AL N vQe fRQ uF2 Oy okVs LJwd qgDhTT Je R 7Cu Pcz NLV j1 HKml 8mwL Fr8Gz6 6n 4 uA9 YTt 9oi JG clm0 EckA 9zkElO B9 J s7G fwh qyg lc 2RQ9 d52a YQvC8A rK 7 aCL mEN PYd 27 XImG C6L9 gOfyL0 5H M tgR 65l BCs WG wFKG BIQi IRBiT9 5N 7 8wn cbk 7EF ei BRB2 16Si HoHJSk Ng x qup JmZ 1px Eb Wcwi JX5N fiYPGD 6u W sXT P94 uaF VD ZuhJ H2d0 PLOY24 3x M K47 VP6 FTy T3 5zpL xRC6 tN89as 3k u 8eG rdM KWo MI U946 FBjk sOTe0U xZ D 4av bTw 5mQ 3R y9Af JFjP gvLFKDFGRTHVBSDFRGDFGNCVBSDFGDHDFHDFNCVBDSFGSDFGDSFBDVNCXVBSDFGSDFHDFGHDFTSADFASDFSADFASDXCVZXVSDGHFDGHVBCX167}   \end{align} and   \begin{align}   \begin{split}    a    &= (j-3)_+ - (l-3)_+ - \left(\frac{j-l-2}{4}\right)_+ - \left(\frac{3j-3l}{4}\right)_+,    \\    b    &    =m-3 - (l+k-3) - \frac{(m-l-k-3)}{4} - \frac{3(m-l-k-1)}{4}    .   \end{split}    \llabel{z 0o l fZd j3O 07E av pWfb M3rB GSyOiu xp I 4o8 2JJ 42X 1G Iux8 QFh3 PhRtY9 vj i SL6 x76 W9y 2Z z3YA SGRM p7kDhr gm a 8fW GG0 qKL sO 5oQr 42t1 jP1crM 2f C lRb ETd qra 5l VG1l Kitb XqbdPK ca U V0l v4L alo 8V TXcl aUqh 5GWCzA nR n lNN cmw aF8 Er bwX3 2rji Hleb4g XS j LRO JgG 2yb 8O CAxN 4uy4 RsLQjD 7U 7 enw cYC nZx iK dju7 4vpj BKKjRR l3 6 kXX zvn X2J rD 8aPD UWGs tgb8CT WY n HRs 6y6 JCp 8L x1jz CI1m tG26y5 zr J 1nF hX6 7wC zq F8uZ QIS0 dnYxPe XD y jBz 1aY wzD Xa xaMI ZzJ3 C3QRra hp w 8sW Lxr AsS qZ P5Wv v1QF 7JPAVQ wu W u69 YLw NHU PJ 0wjs 7RSi VaPrEG gx Y aVm Sk3 Yo1 wL n0q0 PVeX rzoCIH 7v x q5z tOm q6m p4 drAp dzhw SOlRPD ps C lr8 FoZ UG7 vD UYhb ScJ6 gJb8Q8 em G 2JG 9Oj a83 ow Ywjo zLa3 DB500s iG j EHo lPu qe4 p7 T1kQ JmU6 cHnOo2 9o r oOz Ta3 j31 n8 mDL7 CIvC pKZUs0 jV r b7v HIH 7NT tY Y7JK vVdG LhA1ON CW o QW1 fvj mlH 7l SlIm 8T1Q SdUWhT iM P KDZ mm4 V7o fR W1dn lqg0 Ah1QRj dt K ZVz EBN E1e Xi RRSL LQPE SEDeXb iM M Ffx C5F I1z vi yNsY HPsG xfGiIu hD P Di0 OIH uBT TH OCHy CTkA BxuCjg OZ s 965 wfe Fwv fR pNLL T3Ev gKgkO9 jy y vot RRl pDT dn 9H5Z nqwW r4OUkI lx t sk0 RZd ODn so Yid6 ctgw wQrxQk 1S 8 ajp PiZ Jlp 5p IAT1 t482 KxtvQ6 D1 T VzQ 7F3 xoz 6H w2ph WDlC Jg7VcE ix 6 XFI dlO lcN bg ODKp 86tC HVGrzE cV n Bk9 9sq 5XG d1 DNFA Negg JYjfBW jA b JSc hyEDFGRTHVBSDFRGDFGNCVBSDFGDHDFHDFNCVBDSFGSDFGDSFBDVNCXVBSDFGSDFHDFGHDFTSADFASDFSADFASDXCVZXVSDGHFDGHVBCX64}   \end{align} Since $0 \leq l \leq j$, it is readily seen that  $-5/2 \leq a \leq 1/2$ and $b= 3/2$, which implies         \begin{align}        \kappa^a \tau^b        \leq        C        \label{DFGRTHVBSDFRGDFGNCVBSDFGDHDFHDFNCVBDSFGSDFGDSFBDVNCXVBSDFGSDFHDFGHDFTSADFASDFSADFASDXCVZXVSDGHFDGHVBCX314}        \end{align} if  \eqref{DFGRTHVBSDFRGDFGNCVBSDFGDHDFHDFNCVBDSFGSDFGDSFBDVNCXVBSDFGSDFHDFGHDFTSADFASDFSADFASDXCVZXVSDGHFDGHVBCX54} holds. Using \eqref{DFGRTHVBSDFRGDFGNCVBSDFGDHDFHDFNCVBDSFGSDFGDSFBDVNCXVBSDFGSDFHDFGHDFTSADFASDFSADFASDXCVZXVSDGHFDGHVBCX59}--\eqref{DFGRTHVBSDFRGDFGNCVBSDFGDHDFHDFNCVBDSFGSDFGDSFBDVNCXVBSDFGSDFHDFGHDFTSADFASDFSADFASDXCVZXVSDGHFDGHVBCX113}, we obtain        \begin{align}        \begin{split}    \mathcal{B}_{m,j,l,\alpha, \beta, k}         &   \leq    \frac{Cm!}{(l+k)!(m-l-k)!}\frac{(l+k-3)!(m-l-k-2)! (m-l-k-1)^{1/2}}{(m-3)!}\\    &    \leq    \frac{Cm^3}{(l+k)^3}    \leq    C,    \label{DFGRTHVBSDFRGDFGNCVBSDFGDHDFHDFNCVBDSFGSDFGDSFBDVNCXVBSDFGSDFHDFGHDFTSADFASDFSADFASDXCVZXVSDGHFDGHVBCX313}        \end{split}        \end{align} since $[m/2]+1 \leq l+k$. Combining \eqref{DFGRTHVBSDFRGDFGNCVBSDFGDHDFHDFNCVBDSFGSDFGDSFBDVNCXVBSDFGSDFHDFGHDFTSADFASDFSADFASDXCVZXVSDGHFDGHVBCX60}--\eqref{DFGRTHVBSDFRGDFGNCVBSDFGDHDFHDFNCVBDSFGSDFGDSFBDVNCXVBSDFGSDFHDFGHDFTSADFASDFSADFASDXCVZXVSDGHFDGHVBCX313} and proceeding as in \eqref{DFGRTHVBSDFRGDFGNCVBSDFGDHDFHDFNCVBDSFGSDFGDSFBDVNCXVBSDFGSDFHDFGHDFTSADFASDFSADFASDXCVZXVSDGHFDGHVBCX06}, we obtain   \begin{align}    \mathcal{C}_3     \leq     C \Vert v \Vert_{A(\tau)} \Vert S \Vert_{\tilde A(\tau)}.   \llabel{ uVl EN awP0 DWoZ WKuP4I Pt v Zbm nRL 047 2K 3bBQ IH5S pPxtXy 5N J joW ceA 7Fe T7 Iwpi vQdq LaeZE0 Qf i MW1 Koz kdU tR sGH6 ryob MpDbfL t0 Z 2FA XbR 3QQ wu Iizg ZFQ4 Gh4lY5 pt 9 RMT ieq BIk dX I979 BGU2 yYtJSa nO M sDL Wyd CQf ol xJWb bIdb EggZLB Kb F mKX oRM cUy M8 NlGn WyuE RUtbAs 4Z R PHd IWt lbJ Rt Qwod dmlZ hI3I8A 9K 8 Syf lGz cVj Cq GkZn aZrx HNxIcM ae G QdX XxG HFi 6A eYBA lo4Q 9HZIjJ jt O hl4 VLm Vvc ph mMES M8lt xHQQUH jJ h Yyf 5Nd c0i 8m HOTN S7yx 5hNrJC yJ 1 ZFj 4Qe Iom 7w czw9 8Bn6 SxxoqP tn X p4F yiE b2M Cy j2AH aB8F ejdIRh qQ V fR8 rEt z0m q5 4IZt bSlX dBmEvC uv A f5b YxZ 3LE sJ YEX8 eNmo tV2IHl hJ E 70c s45 KVw JR 1riF MPEs P3srHa 8p q wVN AHu soh YI rkNw ekfR bDVLm2 ax u 6ca KkT Xrg Bg nQhU A1z8 X6Mtqv ks U fAF VLg Tmq Pn trgI ggjf JfMGfC uB y BS7 njW fYR Nh pHsj FCzM 4f6cRD gj P Zkb SUH QBn zQ wEnS 9CxS fn00xm Af w lTv 4HI ZIZ Ay XIs4 hPOP jQ3v93 iT L 0Jt NJ8 baB BW cY18 vifU iGKvSQ 4g E kZ1 0yS 5lX Cw I4oX 2gPB isFp7T jK u pgV n5o i4u xK t2QP 4kbr ChS5Zn uW X Wep 0mO jW1 r2 IaXv Hle8 ksF2XQ 52 9 gTL s3u vAO f6 4HOV Iqrb LoG5I2 n0 X skv cKY FIV 8y P9tf MEVP R7F0ip Da q wgQ xro 5Et IW r3tE aSs5 CjzfRR AL g vmy MhI ztV Kj StP7 44RC 0TTPQp n8 g LVt zpL zEQ e2 Rck9 WuM7 XHGA7O 7K G wfm ZHL hJR NU DEQe Brqf KIt0Y4 RW 4 9GK EHY ptg LH 4F8r ZDFGRTHVBSDFRGDFGNCVBSDFGDHDFHDFNCVBDSFGSDFGDSFBDVNCXVBSDFGSDFHDFGHDFTSADFASDFSADFASDXCVZXVSDGHFDGHVBCX07}   \end{align}    Proof of  \eqref{DFGRTHVBSDFRGDFGNCVBSDFGDHDFHDFNCVBDSFGSDFGDSFBDVNCXVBSDFGSDFHDFGHDFTSADFASDFSADFASDXCVZXVSDGHFDGHVBCX316}:  We split ${\mathcal C}_4$ into three sums according to the value of $l+k$ being equal to $m-2$, $m-1$, or $m$, 
and denote them by ${\mathcal C}_{41}$, ${\mathcal C}_{42}$, and ${\mathcal C}_{43}$, respectively. \par For $\mathcal{C}_{41}$, we use H\"older's and the Sobolev inequalities and obtain        \begin{align}        \begin{split}        \mathcal{C}_{41}        &        \leq     C\sum_{m=5}^\infty \sum_{j=0}^m \sum_{l=0}^j \sum_{\vert \alpha \vert = j} \sum_{\substack{\beta \leq \alpha\\\vert \beta \vert = l}} \sum_{k=0}^{m-j} \frac{\kappa^{(j-3)_+}\tau^{m-3}}{(m-3)!} \binom{\alpha}{\beta} \binom{m-j}{k} \mathbbm{1}_{\{l+k=m-2\}}\\     &\indeqtimes     \Vert \partial^\beta (\epsilon \partial_t)^k v \Vert_{L^2} \Vert \partial^{\alpha-\beta} (\epsilon \partial_t)^{m-j-k} \nabla S \Vert_{L^\infty} \\     &     \leq     C  \sum_{m=5}^\infty \sum_{j=0}^m \sum_{l=0}^j \sum_{k=0}^{m-j} \left(\sum_{\vert \beta \vert = l} \Vert \partial^\beta (\epsilon \partial_t)^k v \Vert_{L^2} \frac{\kappa^{(l-3)_+}\tau^{m-5}}{(m-5)!} \right)      \left( \sum_{\vert \gamma \vert= j-l} \Vert  D^2\partial^\gamma (\epsilon \partial_t)^{m-j-k} \nabla S \Vert_{L^2}      \right)^{3/4}     \\     &\indeqtimes      \left( \sum_{\vert \gamma \vert= j-l} \Vert \partial^\gamma (\epsilon \partial_t)^{m-j-k} \nabla S \Vert_{L^2}      \right)^{1/4}     \frac{m!}{(m-2)!} \frac{(m-5)!}{(m-3)!}    \mathbbm{1}_{\{l+k=m-2\}}\\     &     \leq         C \Vert v \Vert_{A(\tau)}     ,     \llabel{fYC vcf1pO yj k 8iT ES0 ujR vF pipc wIvL DgikPu qq k 9RE dH9 YjR UM kr9b yFJK LBex0S gD J 2gB IeC X2C UZ yyRt GNY3 eGOaDp 3m w QyV 1Aj tGL gS C1dD pQCB cocMSM 4j q bSW bvx 6aS nu MtD0 5qpw NDlW0t Z1 c bjz wU5 bUd CG AghC w0nI CDFKHR kp h btA 6nY ld6 c5 TSkD q3Qx o2jhDx Qb m b8n Pq3 zNZ QF JJyu Vm1C 6rzRDC B1 m eQy 4Tt Yr5 jQ VWoO fbrY Q6qakZ ep H b2b 5w4 KN3 mE HtQK AXsI ycbaky ID 9 O8Y CmR lEW 7f GISs 6xaz bM6PSB N2 B jtb 65z z2N uY o4kU lpIq JVBC4D zu Z ZN6 Zkz 0oo mm nswe bstF mlxkKE QE L 6bs oYz xx0 8I Q5Ma 7Inf dXLQ9j eH S Tmi gtt k4v P7 778H p1o6 7atRbf cr S 2CW zwQ 9j0 Rj r0VL 9vlv kkk6J9 bM 1 Xgi Yla y8Z Eq 39Z5 3jRn Xh5mKP Pa 5 tFw 7E0 nE7 Cu FIoV lFxg uxB1hq lH e OLd b7R Kfl 0S KJiY ekpv RSYnNF f7 U VOW Bvw pN9 mt gGwh 2NJC Y53IdJ XP p YAZ 1B1 AgS xn 61oQ Vtg7 W7QcPC 42 e cSA 5jG 4K5 H1 tQs6 TNph OKTBId Gk F SGm V0k zAx av Qzje XGbi Sjg3kY Z5 L xzF 3JN Hkn rm y4sm J70w hEtBeX kS T WEu jcA uS0 Nk Hloa 7wYg Ma5j8g 4g i 7WZ 77D s5M ZZ MtN5 iJEa CfHJ0s D6 z VuX 06B P99 Fg a9Gg YMv6 YFVOBE Ry 3 Xw2 SBY ZDx ix xWHr rlxj KA3fok Ph 9 Y75 8fG XEh gb Bw82 C4JC StUeoz Jf I uGj Ppw p7U xC E5ah G5EG JF3nRL M8 C Qc0 0Tc mXI SI yZNJ WKMI zkF5u1 nv D 8GW YqB t2l Nx dvzb Xj00 EEpUTc w3 z vyf ab6 yQo Rj HWRF JzPB uZ61G8 w0 S Abz pNL IVj WH kWfj ylXj 6VZvjs Tw DFGRTHVBSDFRGDFGNCVBSDFGDHDFHDFNCVBDSFGSDFGDSFBDVNCXVBSDFGSDFHDFGHDFTSADFASDFSADFASDXCVZXVSDGHFDGHVBCX319}     \end{split}        \end{align} where in the second inequality we applied \eqref{DFGRTHVBSDFRGDFGNCVBSDFGDHDFHDFNCVBDSFGSDFGDSFBDVNCXVBSDFGSDFHDFGHDFTSADFASDFSADFASDXCVZXVSDGHFDGHVBCX59}--\eqref{DFGRTHVBSDFRGDFGNCVBSDFGDHDFHDFNCVBDSFGSDFGDSFBDVNCXVBSDFGSDFHDFGHDFTSADFASDFSADFASDXCVZXVSDGHFDGHVBCX113}, \eqref{DFGRTHVBSDFRGDFGNCVBSDFGDHDFHDFNCVBDSFGSDFGDSFBDVNCXVBSDFGSDFHDFGHDFTSADFASDFSADFASDXCVZXVSDGHFDGHVBCX114} and we used $\tau, \kappa\leq C$; in the last inequality, we estimated the low-order mixed space-time Sobolev norm of $S$ by $C$ using Remark~\ref{R04}.    For  $\mathcal{C}_{42}$ and $\mathcal{C}_{43}$, we proceed as in above, by writing         \begin{align}        \begin{split}        \mathcal{C}_{42}        &        \leq        C\sum_{m=5}^\infty \sum_{j=0}^m \sum_{l=0}^j \sum_{\vert \alpha \vert = j} \sum_{\substack{\beta \leq \alpha\\\vert \beta \vert = l}} \sum_{k=0}^{m-j} \frac{\kappa^{(j-3)_+}\tau^{m-3}}{(m-3)!} \binom{\alpha}{\beta} \binom{m-j}{k} \mathbbm{1}_{\{l+k=m-1\}}\\            &\indeqtimes            \Vert \partial^\beta (\epsilon \partial_t)^k v \Vert_{L^2} \Vert \partial^{\alpha-\beta} (\epsilon \partial_t)^{m-j-k} \nabla S \Vert_{L^\infty} \\          &          \leq      C \sum_{m=5}^\infty \sum_{j=0}^m \sum_{l=0}^j \sum_{k=0}^{m-j}       \left(\sum_{\vert \beta \vert = l} \Vert \partial^\beta (\epsilon \partial_t)^k v \Vert_{L^2} \frac{\kappa^{(l-3)_+}\tau^{m-4}}{(m-4)!} \right)         \left( \sum_{\vert \gamma \vert = j-l} \Vert D^2 \partial^{\gamma} (\epsilon \partial_t)^{m-j-k} \nabla S \Vert_{L^2}         \right)^{3/4}           \\     &\indeqtimes     \left( \sum_{\vert \gamma \vert = j-l} \Vert \partial^{\gamma} (\epsilon \partial_t)^{m-j-k} \nabla S \Vert_{L^2}         \right)^{1/4}         \frac{m!}{(m-1)!} \frac{(m-4)!}{(m-3)!}    \mathbbm{1}_{\{l+k=m-1\}}\\     &     \leq     C\Vert v \Vert_{A(\tau)}     \end{split}    \llabel{O 3Uz Bos Q7e rX yGsd vcKr YzZGQe AM 1 u1T Nky bHc U7 1Kmp yaht wKEj7O u0 A 7ep b7v 4Fd qS AD7c 02cG vsiW44 4p F eh8 Odj wM7 ol sSQo eyZX ota8wX r6 N SG2 sFo GBe l3 PvMo Ggam q3Ykaa tL i dTQ 84L YKF fA F15v lZae TTxvru 2x l M2g FBb V80 UJ Qvke bsTq FRfmCS Ve 3 4YV HOu Kok FX YI2M TZj8 BZX0Eu D1 d Imo cM9 3Nj ZP lPHq Ell4 z66IvF 3T O Mb7 xuV RYj lV EBGe PNUg LqSd4O YN e Xud aDQ 6Bj KU rIpc r5n8 QTNztB ho 3 LC3 rc3 0it 5C N2Tm N88X YeTdqT LP l S97 uLM w0N As MphO uPNi sXNIlW fX B Gc2 hxy kg5 0Q TN75 t5JN wZR3NH 1M n VRZ j2P rUY ve HPEl jGaT Ix4sCF zK B 0qp 3Pl eK6 8p 85w4 4l5z Zl07br v6 1 Kki AuT SA5 dk wYS3 F3YF 3e1xKE JW o AvV OZV bwN Yg F7CK bSi9 2R0rlW h2 a khC oEp pr6 O2 PZJD ZN8Z ZD4IhH PT M vSD TgO y1l Z0 Y86n 9aMg kWdeuO Zj O i2F g3z iYa SR Cjlz XdQK bcnb5p KT q rJp 1P6 oGy xc 9vZZ RZeF r5TsSZ zG l 7HW uIG M0y Re YDw3 lMux gAdF6d pp 8 ZVR cl7 uqH 8O BMbz L6dK BflWCW dl V hyc V5n Epv 2J SkD0 ccMp oIR38Q pe Z j9j 0Zo Pmq XR TxBs 8w9Q 5epR3t N5 j bvb rbS K7U 4W 4PJ0 ovnB 0opRpC YN P so8 34P wtS Rq vir4 DRqu jaJq32 QU T G1P gbp 6nJ M2 CUnE NdJC r3ZGBH Eg B tds Td8 4gM 22 gKBN 7Qnm RtJgKU IG E eKx 64y AGK Ge zeJN mpeQ kLR389 HH 9 fXL BcE 6T4 Gj VZLI dLQI iQtkBk 9G 9 FzH WIG m91 M7 SW02 9tzN UX3HLr OU t vG5 QZn Dqy M6 ESTx foUV ylEQ99 nT C SkH A8s fxr DFGRTHVBSDFRGDFGNCVBSDFGDHDFHDFNCVBDSFGSDFGDSFBDVNCXVBSDFGSDFHDFGHDFTSADFASDFSADFASDXCVZXVSDGHFDGHVBCX171}        \end{align} and        \begin{align}        \begin{split}        \mathcal{C}_{43}        &        \leq        C\sum_{m=5}^\infty \sum_{j=0}^m \sum_{l=0}^j \sum_{\vert \alpha \vert = j} \sum_{\substack{\beta \leq \alpha\\\vert \beta \vert = l}} \sum_{k=0}^{m-j} \frac{\kappa^{(j-3)_+}\tau^{m-3}}{(m-3)!} \binom{\alpha}{\beta} \binom{m-j}{k} \mathbbm{1}_{\{l+k=m\}}\\            &\indeqtimes            \Vert \partial^\beta (\epsilon \partial_t)^k v \Vert_{L^2} \Vert \partial^{\alpha-\beta} (\epsilon \partial_t)^{m-j-k} \nabla S \Vert_{L^\infty} \\          &          \leq      C\sum_{m=5}^\infty \sum_{j=0}^m \sum_{\vert \beta \vert = j} \left( \Vert \partial^\beta (\epsilon \partial_t)^{m-j} v \Vert_{L^2} \frac{\kappa^{(j-3)_+}\tau^{m-3}}{(m-3)!} \right)     \Vert D^2 \nabla S \Vert_{L^2}^{3/4} \Vert \nabla S \Vert_{L^2}^{1/4}       \\     &     \leq     C \Vert v \Vert_{A(\tau)}     .     \end{split}    \llabel{ON eFp9 QLDn hLBPib iu j cJc 8Qz Z2K zD oDHg 252c lhDcaQ continuous n xG9 aJl jFq mA DsfD FA0w DO3CZr Q1 a 2IG tqK bjc iq zRSd 0fjS JA1rsi e9 i qOr 5xg Vlj y6 afNu ooOy IVlT21 vJ W fKU deL bcq 1M wF9N R9xQ np6Tqg El S k50 p43 Hsd Cl 7VKk Zd12 Ijx43v I7 2 QyQ vUm 77B V2 3a6W h6IX dP9n67 St l Zll bRi DyG Nr 0g9S 4AHA Vga0Xo fk X FZw gGt sW2 J4 92NC 7FAd 8AVzIE 0S w EaN EI8 v9e le 8EfN Yg3u WVH3JM gi 7 vGf 4N0 akx mB AIjp x4dX lxQRGJ Ze r TMz BxY 9JA tm ZCjH 9064 Q4uzKx gm p CQg 8x0 6NY x0 2vkn EtYX 5O2vgP 3g c spG swF qhX 3a pbPW sf1Y OzHivD ia 1 eOD MIL TC2 mP ojef mEVB 9hWwMa Td I Gjm 9Pd pHV WG V4hX kfK5 Rtci05 ek z j0L 8Tm e2J PX pDI8 Ebcq V4Fdxv rH I eP8 CdO RJp Ti MVEb AunS GsUMWP ts 4 uBv 2QS iXI b7 B8zo 7bp9 voEwNR uX J 4Zx uRZ Yhc 1h 339T HRXV Fw5XVW 8g a B39 mFS v6M ze znkb LHrt Z73hUu aq L vPh gTl NnV po 1Zgg mnRA qM3X31 OR Y Sj8 Rkt S8V GO jrz1 iblt 3uOuEs 8Q 3 xJ1 cA2 NKo F8 o6U3 mW2H q5y6jp os x Jgw WZ4 Exd 79 Jvlc wauo RDCYZz mp a bV0 9jg ume bz cbug patf 9yU9iB Ey v 3Uh S79 XdI mP NEhN 64Rs 9iHQ84 7j X UCA ufF msn Uu dD4S g3FM LMWbcB Ys 4 JFy Yzl rSf nk xPjO Hhsq lbV5eB ld 5 H6A sVt rHg CN Yn5a C028 FEqoWa KS s 9uu 8xH rbn 1e RIp7 sL8J rFQJat og Z c54 yHZ vPx Pk nqRq Gw7h lG6oBk zl E dJS Eig f0Q 1B oCMa nS1u LzlQ3H nA u qHG Plc Iad FL RkdDFGRTHVBSDFRGDFGNCVBSDFGDHDFHDFNCVBDSFGSDFGDSFBDVNCXVBSDFGSDFHDFGHDFTSADFASDFSADFASDXCVZXVSDGHFDGHVBCX172}        \end{align} Combining \eqref{DFGRTHVBSDFRGDFGNCVBSDFGDHDFHDFNCVBDSFGSDFGDSFBDVNCXVBSDFGSDFHDFGHDFTSADFASDFSADFASDXCVZXVSDGHFDGHVBCX05}--\eqref{DFGRTHVBSDFRGDFGNCVBSDFGDHDFHDFNCVBDSFGSDFGDSFBDVNCXVBSDFGSDFHDFGHDFTSADFASDFSADFASDXCVZXVSDGHFDGHVBCX316} and Remark~\ref{R04} to bound $\Vert\nabla v\Vert_{L_x^\infty}$, we get   \begin{align}    \frac{d}{dt} \Vert S \Vert_{A(\tau)} \leq \Vert S \Vert_{\tilde A(\tau)}(\dot{\tau} + C \Vert v \Vert_{A(\tau)}) + C\Vert S \Vert_{A(\tau)} + C \Vert v \Vert_{A(\tau)} + C    .    \label{DFGRTHVBSDFRGDFGNCVBSDFGDHDFHDFNCVBDSFGSDFGDSFBDVNCXVBSDFGSDFHDFGHDFTSADFASDFSADFASDXCVZXVSDGHFDGHVBCX26}   \end{align} Now, determine $K$ in \eqref{DFGRTHVBSDFRGDFGNCVBSDFGDHDFHDFNCVBDSFGSDFGDSFBDVNCXVBSDFGSDFHDFGHDFTSADFASDFSADFASDXCVZXVSDGHFDGHVBCX27} to be sufficiently large so that    \begin{align}   \dot{\tau}(t) + C\Vert v \Vert_{A(\tau)} \leq 0   \comma   0 \leq t \leq T_0,   \label{DFGRTHVBSDFRGDFGNCVBSDFGDHDFHDFNCVBDSFGSDFGDSFBDVNCXVBSDFGSDFHDFGHDFTSADFASDFSADFASDXCVZXVSDGHFDGHVBCX38}   \end{align} where $T_0>0$ satisfies \eqref{DFGRTHVBSDFRGDFGNCVBSDFGDHDFHDFNCVBDSFGSDFGDSFBDVNCXVBSDFGSDFHDFGHDFTSADFASDFSADFASDXCVZXVSDGHFDGHVBCX148}. The lemma is then proven by integrating \eqref{DFGRTHVBSDFRGDFGNCVBSDFGDHDFHDFNCVBDSFGSDFGDSFBDVNCXVBSDFGSDFHDFGHDFTSADFASDFSADFASDXCVZXVSDGHFDGHVBCX26} on $[0, T_0]$, using \eqref{DFGRTHVBSDFRGDFGNCVBSDFGDHDFHDFNCVBDSFGSDFGDSFBDVNCXVBSDFGSDFHDFGHDFTSADFASDFSADFASDXCVZXVSDGHFDGHVBCX38}, and applying the Gronwall lemma. \end{proof} \par After Section~\ref{sec04}, we work with derivatives of the solution and thus instead of the norms \eqref{DFGRTHVBSDFRGDFGNCVBSDFGDHDFHDFNCVBDSFGSDFGDSFBDVNCXVBSDFGSDFHDFGHDFTSADFASDFSADFASDXCVZXVSDGHFDGHVBCX36} we use   \begin{align}    \begin{split}     \Vert u \Vert_{B(\tau)}      &     =     \sum_{m=1}^\infty \sum_{j=0}^m \sum_{\vert \alpha \vert =j}      \Vert \partial^\alpha (\epsilon \partial_t)^{m-j} u \Vert_{L^2} \frac{\kappa^{(j-2)_+}\tau(t)^{(m-2)_+}}{(m-2)!}    \end{split}    \label{DFGRTHVBSDFRGDFGNCVBSDFGDHDFHDFNCVBDSFGSDFGDSFBDVNCXVBSDFGSDFHDFGHDFTSADFASDFSADFASDXCVZXVSDGHFDGHVBCX103}   \end{align} and the corresponding dissipative analytic norm   \begin{align}   \begin{split}    \Vert u \Vert_{\tilde B(\tau)}     &    =    \sum_{m=3}^\infty \sum_{j=0}^m \sum_{\vert \alpha \vert =j} \Vert \partial^\alpha (\epsilon \partial_t)^{m-j} u \Vert_{L^2} \frac{\kappa^{(j-2)_+}(m-2)\tau(t)^{m-3}}{(m-2)!}    .   \end{split}    \label{DFGRTHVBSDFRGDFGNCVBSDFGDHDFHDFNCVBDSFGSDFGDSFBDVNCXVBSDFGSDFHDFGHDFTSADFASDFSADFASDXCVZXVSDGHFDGHVBCX133}   \end{align}    It turns out that the curl component of the velocity  satisfies an equation similar to the one for the entropy, but with the nonzero right-hand side. Thus we now consider the inhomogeneous transport equation        \begin{align}        \partial_t \tilde{S} + v\cdot \nabla \tilde{S} =  G,        \llabel{j aLg0 VAPAn7 c8 D qoV 8bR CvO zq k5e0 Zh3t zJBWBO RS w Zs9 CgF bGo 1E FAK7 EesL XYWaOP F4 n XFo GQl h3p G7 oNtG 4mpT MwEqV4 pO 8 fMF jfg ktn kw IB8N P60f wfEhjA DF 3 bMq EPV 9U0 o7 fcGq UUL1 0f65lT hL W yoX N4v uSY es 96Sc 2HbJ 0hugJM eB 5 hVa EdL TXr No 2L78 fJme hCMd6L SW q ktp Mgs kNJ q6 tvZO kgp1 GBBqG4 mA 7 tMV p8F n60 El QGMx joGW CrvQUY V1 K YKL pPz Vhh uX VnWa UVqL xeS9ef sA i 7Lm HXC ARg 4Y JnvB e46D UuQYkd jd z 5Mf PLH oWI TM jUYM 7Qry u7W8Er 0O g j2f KqX Scl Gm IgqX Tam7 J8UHFq zv b Vvx Niu j6I h7 lxbJ gMQY j5qtga xb M Hwb JT2 tlB si b8i7 zj6F MTLbwJ qH V IiQ 3O0 LNn Ly pZCT VUM1 bcuVYT ej G 3bf hcX 0BV Ql 6Dc1 xiWV K4S4RW 5P y ZEV W8A Yt9 dN VSXa OkkG KiLHhz FY Y K1q NGG EEU 4F xdja S2NR REnhHm B8 V y44 6a3 VCe Ck wjCM e3DG fMiFop vl z Lp5 r0z dXr rB DZQv 9HQ7 XJMJog kJ n sDx WzI N7F Uf veeL 0ljk 83TxrJ FD T vEX LZY pEq 5e mBaw Z8VA zvvzOv CK m K2Q ngM MBA Wc UH8F jSJt hocw4l 9q J TVG sq8 yRw 5z qVSp d9Ar UfVDcD l8 B 1o5 iyU R4K Nq b84i OkIQ GIczg2 nc t txd WfL QlN ns g3BB jX2E TiPrpq ig M OSw 4Cg dGP fi G2HN ZhLe aQwyws ii A WrD jo4 LDb jB ZFDr LMuY dt6k6H n9 w p4V k7t ddF rz CKid QPfC RKUedz V8 z ISv ntB qpu 3c p5q7 J4Fg Bq59pS Md E onG 7PQ CzM cW lVR0 iNJh WHVugW PY d IMg tXB 2ZS ax azHe Wp7r fhk4qr Ab J FFG 0li i9M WI l44j s9gN lu46Cf DFGRTHVBSDFRGDFGNCVBSDFGDHDFHDFNCVBDSFGSDFGDSFBDVNCXVBSDFGSDFHDFGHDFTSADFASDFSADFASDXCVZXVSDGHFDGHVBCX142}        \end{align} where $\tilde{S}=\tilde{S}(x,t)$, $v=v(x,t)$, and $G=G(x,t)$. \par \cole \begin{Lemma}  \label{L08} For any $\kappa \in (0,1]$, there exists $\tau_1 \in(0,1]$  such that if $0<\tau(0)\leq \tau_1$, then        \begin{align}         \Vert \tilde{S} \Vert_{A(\tau)}         \leq    \Vert \tilde{S}(0) \Vert_{A(\tau)}         +         C\int_0^t \left( \Vert G(s)\Vert_{A(\tau)}         +        \Vert v(s) \Vert_{A(\tau)} \right)         ds          +         Ct         \comma         t\in [0,T_0]         ,        \label{DFGRTHVBSDFRGDFGNCVBSDFGDHDFHDFNCVBDSFGSDFGDSFBDVNCXVBSDFGSDFHDFGHDFTSADFASDFSADFASDXCVZXVSDGHFDGHVBCX141}        \end{align} for some constant $C$ and sufficiently small $T_0 > 0$, provided $\KK$ in \eqref{DFGRTHVBSDFRGDFGNCVBSDFGDHDFHDFNCVBDSFGSDFGDSFBDVNCXVBSDFGSDFHDFGHDFTSADFASDFSADFASDXCVZXVSDGHFDGHVBCX27} satisfies   \begin{equation}    \KK\geq C \Vert v(t) \Vert_{A(\tau)}    \comma    t\in [0, T_0]    ,    \label{DFGRTHVBSDFRGDFGNCVBSDFGDHDFHDFNCVBDSFGSDFGDSFBDVNCXVBSDFGSDFHDFGHDFTSADFASDFSADFASDXCVZXVSDGHFDGHVBCX153}   \end{equation} where $T_0$ is chosen sufficiently small so that \eqref{DFGRTHVBSDFRGDFGNCVBSDFGDHDFHDFNCVBDSFGSDFGDSFBDVNCXVBSDFGSDFHDFGHDFTSADFASDFSADFASDXCVZXVSDGHFDGHVBCX148} holds. Similarly, for any $\kappa \leq 1$, there exists $\tau(0)>0$ such that        \begin{align}         \Vert \tilde{S} \Vert_{B(\tau)}         \leq        \Vert \tilde{S}(0) \Vert_{B(\tau)}         +        C         \int_0^t \left( \Vert G(s)\Vert_{B(\tau)}                        +                        \Vert v(s) \Vert_{B(\tau)}                    \right)         ds         +         Ct      \comma      t\in [0,T_0]      ,        \label{DFGRTHVBSDFRGDFGNCVBSDFGDHDFHDFNCVBDSFGSDFGDSFBDVNCXVBSDFGSDFHDFGHDFTSADFASDFSADFASDXCVZXVSDGHFDGHVBCX145}        \end{align} for some constant $C$ and sufficiently small $T_0 > 0$, provided $\KK$ satisfies   \begin{equation}    \KK\geq  C\Vert v(t)\Vert_{B(\tau)}    \comma    t\in [0,T_0]    ,    \label{DFGRTHVBSDFRGDFGNCVBSDFGDHDFHDFNCVBDSFGSDFGDSFBDVNCXVBSDFGSDFHDFGHDFTSADFASDFSADFASDXCVZXVSDGHFDGHVBCX154}   \end{equation} where $T_0$ is chosen sufficiently small so that \eqref{DFGRTHVBSDFRGDFGNCVBSDFGDHDFHDFNCVBDSFGSDFGDSFBDVNCXVBSDFGSDFHDFGHDFTSADFASDFSADFASDXCVZXVSDGHFDGHVBCX148} holds.   \end{Lemma} \colb \par Note that from definitions \eqref{DFGRTHVBSDFRGDFGNCVBSDFGDHDFHDFNCVBDSFGSDFGDSFBDVNCXVBSDFGSDFHDFGHDFTSADFASDFSADFASDXCVZXVSDGHFDGHVBCX36} and \eqref{DFGRTHVBSDFRGDFGNCVBSDFGDHDFHDFNCVBDSFGSDFGDSFBDVNCXVBSDFGSDFHDFGHDFTSADFASDFSADFASDXCVZXVSDGHFDGHVBCX103}, we have   \begin{equation}    \Vert v\Vert_{B(\tau)}    \leq \Vert  v\Vert_{A(\tau)}    \llabel{P3 H vS8 vQx Yw9 cE yGYX i3wi 41aIuU eQ X EjG 3XZ IUl 8V SPJV gCJ3 ZOliZQ LO R zOF VKq lyz 8D 4NB6 M5TQ onmBvi kY 8 8TJ ONa DfE 2u zbcv fL67 bnJUz8 Sd 7 yx5 jWr oXd Jp 0lSy mIK8 bkKzql jN n 4Kx luF hYL g0 FrO6 yRzt wFTK7Q RN 0 1O2 1Zc HNK gR M7GZ 9nB1 Etq8sq lA s fxo tsl 927 c6 Y8IY 8T4x 0DRhoh 07 1 8MZ Joo 1oe hV Lr8A EaLK hyw6Sn Dt h g2H Mt9 D1j UF 5b4w cjll AvvOSh tK 8 06u jYa 0TY O4 pcVX hkOO JVtHN9 8Q q q0J 1Hk Ncm LS 3MAp Q75A lAkdnM yJ M qAC erD l5y Py s44a 7cY7 sEp6Lq mG 3 V53 pBs 2uP NU M7pX 6sy9 5vSv7i IS 8 VGJ 08Q KhA S3 jIDN TJsf bhIiUN fe H 9Xf 8We Cxm BL gzJT IN5N LhvdBO zP m opx YqM 4Vh ky btYg a3XV TTqLyA Hy q Yqo fKP 58n 8q R9AY rRRe tBFxHG g7 p duM 8gm 1Td pl RKIW 9gi5 ZxEEAH De A sfP 5hb xAx bW CvpW k9ca qNibi5 A5 N Y5I lVA S3a hA aB8z zUTu yK55gl DL 5 XO9 CpO RXw rE V1IJ G7wE gpOag9 zb J iGe T6H Emc Ma QpDf yDxh eTNjwf wM x 2Ci pkQ eUj RU VhCf NMo5 DZ4h2a dE j ZTk Ox9 46E eU IZv7 rFL6 dj2dwg Rx g bOb qJs Yms Dq QAss n9g2 kCb1Ms gK f x0Y jK0 Glr XO 7xI5 WmQH ozMPfC XT m Dk2 Tl0 oRr nZ vAsF r7wY EJHCd1 xz C vMm jeR 4ct k7 cS2f ncvf aN6AO2 nI h 6nk VkN 8tT 8a Jdb7 08jZ ZqvL1Z uT 5 lSW Go0 8cL J1 q3Tm AZF8 qhxaoY JC 6 FWR uXH Mx3 Dc w8uJ 87Q4 kXVac6 OO P DZ4 vRt sP0 1h KUkd aCLB iPSAtL u9 W Loy xMa Bvi xH yadn qQSJ WgSCkF 7l H aO2 yGR IDFGRTHVBSDFRGDFGNCVBSDFGDHDFHDFNCVBDSFGSDFGDSFBDVNCXVBSDFGSDFHDFGHDFTSADFASDFSADFASDXCVZXVSDGHFDGHVBCX163}   \end{equation} for all $v$, and thus \eqref{DFGRTHVBSDFRGDFGNCVBSDFGDHDFHDFNCVBDSFGSDFGDSFBDVNCXVBSDFGSDFHDFGHDFTSADFASDFSADFASDXCVZXVSDGHFDGHVBCX153} implies \eqref{DFGRTHVBSDFRGDFGNCVBSDFGDHDFHDFNCVBDSFGSDFGDSFBDVNCXVBSDFGSDFHDFGHDFTSADFASDFSADFASDXCVZXVSDGHFDGHVBCX154}. \par \begin{proof} We proceed exactly as in the proof of Lemma~\ref{L02}. Using the Cauchy-Schwarz inequality with the inhomogeneous part $G$, we obtain   \begin{align}    \frac{d}{dt} \Vert \tilde{S} \Vert_{A(\tau)} \leq \Vert \tilde{S} \Vert_{\tilde A(\tau)}(\dot{\tau}    + C\Vert v \Vert_{A(\tau)})     +     C( \Vert \tilde{S} \Vert_{A(\tau)} + \Vert v \Vert_{A(\tau)} + \Vert G \Vert_{A(\tau)}) + C    .    \llabel{lK 3a FZen CWqO 9EyRof Yb k idH Qh1 G2v oh cMPo EUzp 6f14Ni oa r vW8 OUc 426 Ar sSo7 HiBU KdVs7c Oj a V9K EUt Kne 4V IPuZ c4bP RFB9AB fq c lU2 ct6 PDQ ud t4VO zMMU NrnzJX px k E2N B8p fJi M4 UNg4 Oi1g chfOU6 2a v Nrp cc8 IJm 2W nVXL D672 ltZTf8 RD w qTv BXE WuH 2c JtO1 INQU lOmEPv j3 O OvQ SHx iKc 8R vNnJ NNCC 3KXp3J 8w 5 0Ws OTX HHh vL 5kBp Kr5u rqvVFv 8u p qgP RPQ bjC xm e33u JUFh YHBhYM Od 0 1Jt 7yS fVp F0 z6nC K8gr RahMJ6 XH o LGu 4v2 o9Q xO NVY8 8aum 7cZHRN XH p G1a 8KY XMa yT xXIk O5vV 5PSkCp 8P B oBv 9dB mep ms 7DDU aicX Y8Lx8I Bj F Btk e2y ShN GE 7a0o EMFy AUUFkR WW h eDb HhA M6U h3 73Lz TTTx xm6ybD Bs I IIA PHh i83 7r a970 Fam4 O7afXU Gr f 0vW e52 e8E Py BFZ0 wxBz ptJf8L iZ k dTZ SSP pSz rb GEpx b4KX LHLg1V Pa f 7ys vYs FJb 8r DpAM Knzq Dg7g2H wC r uQN DBz Z5S NM ayKB 6RIe PFIHFQ aw r RHA x38 CHh oB GVIR vxSM Yf0g8h ac i bKG 3Cu Sl5 jT Kl42 o6gA OYYHUB 2S V O3R c4w hR8 pw krrA NA4j 7MfcEM al 4 HwK PTg ZaZ 9G 8sev uwIA hkhR8W ga f zJA 0FV NmS Cw UB0Q JDgR jCSVSr sG M bWA Bxv zOM My cNSO ylZz wFiNPc mf Q hwZ Pan Lp0 1E UVHM A2dE 0nLuNV xK x co7 opb QzR aA lowo Vtor qU5eUX tE l 1qh 1IP CPE uV Pxcn TwZv KJTpHZ pq x XOF aGs rQN Pn uqc9 CMD8 mJZoMO Cy 6 XHj WAf EqI 95 Zjgc PdV8 maWwkH lM 3 0Vw DjX lx1 Qf gyLF xe5i wJDnJI rU K L6e CVt h3g X8 RLAC CC2DFGRTHVBSDFRGDFGNCVBSDFGDHDFHDFNCVBDSFGSDFGDSFBDVNCXVBSDFGSDFHDFGHDFTSADFASDFSADFASDXCVZXVSDGHFDGHVBCX140}   \end{align} The estimate~\eqref{DFGRTHVBSDFRGDFGNCVBSDFGDHDFHDFNCVBDSFGSDFGDSFBDVNCXVBSDFGSDFHDFGHDFTSADFASDFSADFASDXCVZXVSDGHFDGHVBCX141} then follows by using \eqref{DFGRTHVBSDFRGDFGNCVBSDFGDHDFHDFNCVBDSFGSDFGDSFBDVNCXVBSDFGSDFHDFGHDFTSADFASDFSADFASDXCVZXVSDGHFDGHVBCX38} and the Gronwall inequality. Analogously, we use the analytic shift $(m-2)!$ instead of $(m-3)!$ and proceed as in the proof of Lemma~\ref{L02}, we conclude   \begin{align}    \frac{d}{dt} \Vert \tilde{S} \Vert_{B(\tau)}         \leq \Vert \tilde{S} \Vert_{\tilde B(\tau)}(\dot{\tau} + C\Vert v \Vert_{B(\tau)}) + C( \Vert \tilde{S} \Vert_{B(\tau)} + \Vert v\Vert_{B(\tau)}  +         \Vert G \Vert_{B(\tau)} ) + C    .    \llabel{q 5kMYXo 8N s DfA n3O sap ZT 8U3F dMIE DKqMo0 Yv b wCG MR6 6XT yI OtLU uC6c mcOFsv pW T niQ mu0 PeH EF 9Imo lIuT hHWHwh 8J z 4hC 0rK 2Gd Nz LXiE Y7Vu QfRbXp iQ n Pps 9gM A8m Wk yXsY FLoi Rtl2Kl 2p I 9bS nyi 07m UZ qhEs BOCg I4F5AF Fd j X3w f0W u2Y qd dp2Z Ukje FMAxnD ls u t9q zby RgD Wr HldN Zewz EK1cSw WJ Z ywl oSo f6z VD AB6e r0o2 HZY1tr Zh B uL5 zYz rAU dM KXVK GWKI HOqqx1 zj 8 tlp xuU D83 eL Uerj xfHN MZlaqZ Vh T 6Jk u15 FdL vd eo08 7AsG C8WdoM nf 4 dTo Hw4 7hg lT qjKt AlwR 9ufOhL KT D gWZ hxH FX5 gU 5uN2 S6es PlxKpX zB m gyW Uy5 D01 WD 88a4 YmWR fdmev1 dB v HOm hTB qur Ag TC6y rrRB Pn9QfZ 9T 4 mwI h3x jAt ki MlAl Td6f SQ5iQB BY 6 OEr T3g f0D Ke ECnj gcTX AL8grK Bp f cJv q4f pIh WG FSdh 6LOq g0ao9A ja k qEZ Kgv 95B Aq vCSJ Jgo1 Lzsv5y hP Q kMp PWn sXv HN ZAUQ t1Dp O47V7A R6 J CTR 9fn H6Q VY wjdZ TR3T ZSCdOc id D YXY izt VEg NK 6hZW tLoo E11Miq yw C o9k Ujd nt1 CC c80e JHxM Wn7GlD I9 y Cp9 xs9 zkn Cb Fsjl Kydv YatKpv KJ p ySV eTP 1zR B9 N91v A0XX SacVlN zZ X 3jR gbD jcs cB Bvea dZer kNNT2n i9 P Do6 HyD NuA vU oYNa IuQ6 oUCEi5 k5 k Bw1 fwk tQD SG 4Ky7 U2nX SKlOez 0P J 0v1 SNn MkU dm xxN5 t7vh LQxWUu lV u tc6 EMF Pa0 mI kRDV wluK itmTnc yT 5 8CD jRP unX B7 4DYS JKWE PYi0Yx ly A d7H GWy yCe Jz d8ht 2NnH GfOmsD Lb W qhY k2v 3J7 j2 nbB4 taYD MN0OkJ Td k DFGRTHVBSDFRGDFGNCVBSDFGDHDFHDFNCVBDSFGSDFGDSFBDVNCXVBSDFGSDFHDFGHDFTSADFASDFSADFASDXCVZXVSDGHFDGHVBCX146}   \end{align} The assertion \eqref{DFGRTHVBSDFRGDFGNCVBSDFGDHDFHDFNCVBDSFGSDFGDSFBDVNCXVBSDFGSDFHDFGHDFTSADFASDFSADFASDXCVZXVSDGHFDGHVBCX145} may then be obtained by  setting   \begin{align}   \dot{\tau}(t) + C\Vert v \Vert_{B(\tau)} \leq 0    \llabel{NPO 7Jv kTR FY wud2 MZ91 SZPVQc Ll v rOc IN9 2CO u4 QpaM 7ShS sg1qs8 ui j WXM MnX 976 0o tPu5 BJwt xkMVH4 wu j 37t RdB 7Za 2F eTvX LlkC 0kZ1ZR CZ v bcV w9S RuU im ZYbI yqO0 qKkkir gp v LzB S44 Rwj 1N ZRJH Oafv DTKV3e PS J J0w uXj Kzg eX a11G uCRi RVPRSU Nx S qio nXM k3f 8c KO4i nK7I fRUJ0M W6 Z dcM T3L aSj Z4 IqtQ IFDu kYcYz7 00 u nb8 AtS mg2 kM KAAL 7DB4 cgDSBF HX 5 GPc HAq Cyt P3 oJRr yBay lYND2H GJ Z rTg UUR yUw zC li2F 9vvy flNg1n R5 Y 1nx JC2 35J xy SQAm anUo Zh1VvD oM y 3RL 9pU Iem C7 9uRd oMV0 hz3sKl 3u S B9W EO7 EFb VY WJMe DYZe 6UuJQ2 rb l oFI c6c a1h FE dw2d 4sST rHnhAI lQ 1 o9R phH tZ9 C4 DInJ MbOm YJatZf wE A KGD pQb 2Mx 5o 8ndJ nyvr UaP1lk NO G dle O90 C3Q pE 1fEX gQ5Y 2APzGV PR J n6r PVX g9U 8d 9upd t0YZ pJ4i1h 1W T mix B5H 0nd Nf 7UYb KCXm 1RHvRl 6I g FtP gkh xdX 95 jIOZ 0qtx 9xEmzd F1 L 5sN PY0 CKo PB SK6S 7faL CutDrB VC t B3N ykr yNU qA CBCt o7OV cAjjsN Mj f pDA 9w2 r6C zm Qhf2 xmwR 0HvIhi Ij h HoZ 4ob ATs Wk JgUB X5Jb TtFr3f Dq 6 DjR CxJ o6k ZK TUfJ Nw8u CuyAAx qj B Csv Law MMS 37 OhxI QcaW 0SnnMP dZ Q Ljt IIp nIn eW sWiU QpO6 9palKX NR l v9b YG0 6pt JE 7WLp WiXn eaUhEs zo G mQq uLL n5b Xd yFia 5iLr g7PABB N4 v 8Qa pYA v8h Gh Md7E 10KV uQl78K Re 3 8xd UEF e15 K7 2PTL wYDk utECGA DC E BMb cFv Lgn nr bdwl dq6C S8wVyB zP G ZaC 3fG krv mhDFGRTHVBSDFRGDFGNCVBSDFGDHDFHDFNCVBDSFGSDFGDSFBDVNCXVBSDFGSDFHDFGHDFTSADFASDFSADFASDXCVZXVSDGHFDGHVBCX134}   \end{align} with $C$ sufficiently large and using the Gronwall inequality. \end{proof} \par \startnewsection{Analytic estimate of $\partial_t E$}{sec04}  In order to bound the velocity, we first need to obtain an analytic estimate for $\partial_{t} E$, which in turn requires the bound on the entropy. We first provide a product rule for the type $B$ norm. \par \cole \begin{Lemma} \label{L03} Let $k\in\{2,3,\ldots\}$ and $\tau>0$. For $f_1, \ldots, f_k \in B(\tau)$, and any $\kappa \in(0,1]$, there exists  $\tau_1\in(0,1]$ and $T_0>0$ such that if $0<\tau(0)\leq \tau_1$, then   \begin{align}   \begin{split}   \biggl\Vert \prod_{i=1}^k f_i \biggr\Vert_{B(\tau)}   &   \leq   C^{k}  \sum_{i=1}^{k} \left(   \Vert f_i\Vert_{B(\tau)}   \prod_{1\leq j\leq k; j\neq i}   (\Vert f_j\Vert_{B(\tau)}+ \Vert f_j\Vert_{L^2})\right)   ,   \end{split}   \llabel{ Mj9u vRdn SYXXg1 2I k XVN SN3 601 pW dyki U6ka UwDUZ2 G8 r 5Xi ZXM Q7A Gr plYA Plwn 11dm00 Jo d c1h 7zF n5r Lb VOHM Dh0Q ggiSOK ll 3 vzZ 0A6 hDO 56 OyuN Bgfz TkNNyR 28 P sJU Dso fPa Xg qBUK o9tX hTwgIF Ax g S43 mPT Rh5 QL fBBr Lyi8 we8dXm Jn k R7D cCI I4D f1 yrov twK6 Zq8Fay 5D r bFo lZg iNN UK Qko1 99y4 3VW46d Uh w t4d EWE Cfq sz QjuZ cFNq RA8EBK Tj x Xvf BSu Pr9 Ie lzgC VxKd owNHtb T6 j KdP L1M YPf jf mZiL HDZU fPBhjt 64 X adg QnI 2yl Wj ZqHT p3H9 LIGIds X0 l iHs 3aE 1qH cN YwAf L2aJ AMnqOi Sd x F6c G4b Ex1 aZ jpeJ sDcY CD98Z7 vM Y v77 Vfx 3eV Qa kI2a Pi3U hSEKfs DN O ley 7xr pP4 S9 FoFg 8deO ZeMJ5P QW S Mlo ZjH qNX rt Ch7p qHNF uq1MF5 tP b BVg GwF tRY hA i2q5 2Rw4 dFk76z Gc R OdF IXf t1L K3 fxk0 xMVm qt2h7r qf 9 OlF 4gj jR2 B8 FxKi 9pwg 5yRY8X MX I WIS Ojb csR Kn iRqL wJzk AU7oq6 tB K pEv 4ES lNO y9 u1tc iXJC AYUOps 6O A h6W 4Zx V8l o2 dueB iLmZ rIw6dF UG 4 w8P cdi JSF mS wE9K hkgS Z75dyc S1 7 zsk wXx P9U Vl mzJw i5E0 CZMvlQ Le F Zbl h53 SAB Dv gFzm lZH5 lJ42UO cR o jWm p7F tO4 at vrDj VQbv SNkhDb Nl 8 R21 vnh X3I LY S65p r9OA dnA3j6 KE L 5tN aNV msK vl Gmy2 kMyV vsb73R cL U NGz i6w lc1 uF WkYr pXEm fkPv4t Qo z W2C HUX S2k CG 3Cfb Z7cL MohJXI uk a sc5 F0v F8G oN ZkeI 6DEv 2OlWLJ pu u OiS Xad leN continuous gyW4 u0bp 8TDtYF QI 8 kp2 9nL jdv VT PfrK DFGRTHVBSDFRGDFGNCVBSDFGDHDFHDFNCVBDSFGSDFGDSFBDVNCXVBSDFGSDFHDFGHDFTSADFASDFSADFASDXCVZXVSDGHFDGHVBCX68}   \end{align} for $k\geq 2$, where the constant is independent of $k$. \end{Lemma} \colb \par \begin{proof}[proof of Lemma~\ref{L03}] By induction, it is sufficient to prove the inequality   \begin{align}   &   \label{DFGRTHVBSDFRGDFGNCVBSDFGDHDFHDFNCVBDSFGSDFGDSFBDVNCXVBSDFGSDFHDFGHDFTSADFASDFSADFASDXCVZXVSDGHFDGHVBCX67}   \Vert fg\Vert_{B(\tau)}   \leq   C \Vert f \Vert_{B(\tau)} (\Vert g \Vert_{B(\tau)} + \Vert g \Vert_{L^2})   +   C (\Vert f \Vert_{B(\tau)} + \Vert f \Vert_{L^2}) \Vert g \Vert_{B(\tau)}   ,   \end{align} for $f$ and $g$ such that the respective right hand sides are finite. To prove the estimate \eqref{DFGRTHVBSDFRGDFGNCVBSDFGDHDFHDFNCVBDSFGSDFGDSFBDVNCXVBSDFGSDFHDFGHDFTSADFASDFSADFASDXCVZXVSDGHFDGHVBCX67}, we use the Leibniz rule and write   \begin{align}   \begin{split}   \Vert fg \Vert_{B(\tau)}   &   =   \sum_{m=1}^\infty \sum_{j=0}^m \sum_{\vert \alpha \vert = j}          \frac{\kappa^{(j-2)_+}\tau^{(m-2)_+}}{(m-2)!}\Vert \partial^\alpha (\epsilon \partial_t)^{m-j} (fg) \Vert_{L^2}    \\&        \leq   \sum_{m=1}^\infty \sum_{j=0}^m \sum_{l=0}^j \sum_{\vert \alpha \vert =j} \sum_{\vert \beta \vert = l, \beta \leq \alpha} \sum_{k=0}^{m-j} \mathcal{H}_{m,j,l,\alpha,\beta, k}   ,   \end{split}   \label{DFGRTHVBSDFRGDFGNCVBSDFGDHDFHDFNCVBDSFGSDFGDSFBDVNCXVBSDFGSDFHDFGHDFTSADFASDFSADFASDXCVZXVSDGHFDGHVBCX47}   \end{align} where    \begin{align}        \mathcal{H}_{m,j,l,\alpha,\beta, k}     = \binom{\alpha}{\beta} \binom{m-j}{k} \frac{\kappa^{(j-2)_+}\tau^{(m-2)_+}}{(m-2)!} \Vert \partial^\beta (\epsilon \partial_t)^k  f\partial^{\alpha - \beta} (\epsilon \partial_t)^{m-j-k} g \Vert_{L^2}    .    \llabel{xCVr 7pdxnV wd H HQu bf5 O4i xY rddM brhd 60rN8G L1 T Gfy eCQ mNa JN 3fJg n0we 17GkQA g1 W Aj6 l87 vzm Oz dKQ1 HZ8q ATPMo1 KA U LCH IjK sRX wK T5xY B9iw CmzcM5 nl f bkM Bhe 7ON I5 U2Nr E5WX sFl5mk w3 T zbJ icL nBr Aj cGa9 wJZG Tj7Ymb bE r 0cg xOx s75 sh L8m7 RUxH RQBPih Q1 Z g0p 7UI yPs op iOae hpYI 4z5HbV Qx s th4 RUe V3B sg K0xx zeyE FCVmbO LA 5 2no FxB k8r 1B QcFe K5PI KE9uvU NX 3 B77 LAv k5G du 4dVu Pux0 h7zEtO xl v Inn vDy Ge1 qJ iATO RZ29 Nb6q8R I4 V 3Dv 4fB sJD vp 6ago 5Fal ZhU1Yx 2e v ycy bq7 Jw4 eJ 9oew gCma 6lFCjs Oy z eoX OyI agD o4 rJPD vdVd AIPfva xO I sle 7l6 0zf IT nPR5 IE34 RJ1dqT Xj 0 SVu TpT rmk FS n2gI WtUv MdtZIW IZ T o2a Jpe FRG mL Ia5Y G6yn 0Lboer wM P eyv JZH roC E3 8u15 8qCn rhQWMU 6v 8 vnc MbF GS2 vM 3vW5 qwbO 6UKlUB S9 Y 2oL 2ju yOl k3 XJ7m fyQM GcRyAV Sc 0 Yk8 biz EBI GD YdOG oW6e mVaisr Fm 5 Ehk 6nw h98 Pn 1pr7 Od6q GjlJOb Lu D e0U qZT HY6 w0 iZ67 Kfw5 cPZ1ZF pv G 095 DYK QTH J7 t5HP bc8W evwqlB tM E FsB gDS acc QO NR6s BXLp 8nu9yg Z4 7 NBU Ckn TcX gs syAg Ke8n TCWZjy 1Y g l83 LTL Rtp 6i SFID k1Bt U6O9px cN W t3F HEP LmD 7T XtkN Zr0r YhB8fr pu q ccV bXr PCH jJ bJoq MK7O 6CvwuA Te r cnN 2Sp jTA aM yFNp a9Ge Kl351n Ds 3 KVr 6WM CTY S0 zsAB O0Sw HhCqG0 qk 6 2kp IM5 YjY Ce M76V cZ0c FJZTEH Zy 5 Ljz lsD Rtf vN E1e4 QcGK Xy7y7H p2DFGRTHVBSDFRGDFGNCVBSDFGDHDFHDFNCVBDSFGSDFGDSFBDVNCXVBSDFGSDFHDFGHDFTSADFASDFSADFASDXCVZXVSDGHFDGHVBCX173}        \end{align} We split the sum on the right side of \eqref{DFGRTHVBSDFRGDFGNCVBSDFGDHDFHDFNCVBDSFGSDFGDSFBDVNCXVBSDFGSDFHDFGHDFTSADFASDFSADFASDXCVZXVSDGHFDGHVBCX47} according to the low and high values of $l+k$ and $m$, and we claim        \begin{align}        &        \sum_{m=1}^2 \sum_{j=0}^m \sum_{l=0}^j \sum_{\vert \alpha \vert =j} \sum_{\vert \beta \vert = l, \beta \leq \alpha} \sum_{k=0}^{m-j} \mathcal{H}_{m,j,l,\alpha,\beta, k}        \leq        C (\Vert f \Vert_{B(\tau)} + \Vert f \Vert_{L^2} ) \Vert g \Vert_{B(\tau)}        +        C (\Vert g \Vert_{B(\tau)} + \Vert g \Vert_{L^2} ) \Vert f \Vert_{B(\tau)}        ,        \label{DFGRTHVBSDFRGDFGNCVBSDFGDHDFHDFNCVBDSFGSDFGDSFBDVNCXVBSDFGSDFHDFGHDFTSADFASDFSADFASDXCVZXVSDGHFDGHVBCX379}        \\        &        \sum_{m=3}^\infty \sum_{j=0}^m \sum_{l=0}^j \sum_{\vert \alpha \vert =j} \sum_{\vert \beta \vert = l, \beta \leq \alpha} \sum_{k=0}^{m-j} \mathcal{H}_{m,j,l,\alpha,\beta, k} \mathbbm{1}_{\{l+k = 0\}}        \leq        C (\Vert f \Vert_{B(\tau)} + \Vert f \Vert_{L^2})        \Vert g \Vert_{B(\tau)}        ,        \label{DFGRTHVBSDFRGDFGNCVBSDFGDHDFHDFNCVBDSFGSDFGDSFBDVNCXVBSDFGSDFHDFGHDFTSADFASDFSADFASDXCVZXVSDGHFDGHVBCX400}        \\        &        \sum_{m=3}^\infty \sum_{j=0}^m \sum_{l=0}^j \sum_{\vert \alpha \vert =j} \sum_{\vert \beta \vert = l, \beta \leq \alpha} \sum_{\substack{k=0\\1 \leq l+k \leq [m/2]}}^{m-j}\mathcal{H}_{m,j,l,\alpha,\beta, k}        \leq        C (\Vert f \Vert_{B(\tau)}+ \Vert f \Vert_{L^2}) \Vert g \Vert_{B(\tau)}        ,        \label{DFGRTHVBSDFRGDFGNCVBSDFGDHDFHDFNCVBDSFGSDFGDSFBDVNCXVBSDFGSDFHDFGHDFTSADFASDFSADFASDXCVZXVSDGHFDGHVBCX372}        \\        &        \sum_{m=3}^\infty \sum_{j=0}^m \sum_{l=0}^j \sum_{\vert \alpha \vert =j} \sum_{\vert \beta \vert = l, \beta \leq \alpha} \sum_{\substack{k=0\\ [m/2]+1\leq l+k \leq m-1}}^{m-j} \mathcal{H}_{m,j,l,\alpha,\beta, k}        \leq        C \Vert f \Vert_{B(\tau)} (\Vert g \Vert_{B(\tau)} + \Vert g \Vert_{L^2})        .        \label{DFGRTHVBSDFRGDFGNCVBSDFGDHDFHDFNCVBDSFGSDFGDSFBDVNCXVBSDFGSDFHDFGHDFTSADFASDFSADFASDXCVZXVSDGHFDGHVBCX303}        \\        &        \sum_{m=3}^\infty \sum_{j=0}^m \sum_{l=0}^j \sum_{\vert \alpha \vert =j} \sum_{\vert \beta \vert = l, \beta \leq \alpha} \sum_{k=0}^{m-j} \mathcal{H}_{m,j,l,\alpha,\beta, k} \mathbbm{1}_{\{l+k = m\}}        \leq        C \Vert f \Vert_{B(\tau)}         (\Vert g \Vert_{B(\tau)} + \Vert g \Vert_{L^2})        ,        \label{DFGRTHVBSDFRGDFGNCVBSDFGDHDFHDFNCVBDSFGSDFGDSFBDVNCXVBSDFGSDFHDFGHDFTSADFASDFSADFASDXCVZXVSDGHFDGHVBCX401}        \end{align} \par Proof of \eqref{DFGRTHVBSDFRGDFGNCVBSDFGDHDFHDFNCVBDSFGSDFGDSFBDVNCXVBSDFGSDFHDFGHDFTSADFASDFSADFASDXCVZXVSDGHFDGHVBCX379}: For $m=1$, we use H\"older's and the Sobolev inequalities and arrive at        \begin{align}        \begin{split}       &       \sum_{j=0}^1 \sum_{l=0}^j \sum_{\vert \alpha \vert =j} \sum_{\vert \beta \vert = l, \beta \leq \alpha} \sum_{k=0}^{1-j} \mathcal{H}_{1,j,l,\alpha,\beta, k}        \leq        C \Vert f \epsilon \partial_t g \Vert_{L^2}        +        C \Vert f Dg \Vert_{L^2}        +        C \Vert g \epsilon \partial_t f \Vert_{L^2}        +        C \Vert g Df \Vert_{L^2}        \\&\indeq        \leq        C \Vert D^2 f \Vert_{L^2}^{3/4} \Vert f \Vert_{L^2}^{1/4} \Vert \epsilon \partial_t g \Vert_{L^2}        +        C \Vert D^2 f \Vert_{L^2}^{3/4} \Vert f \Vert_{L^2}^{1/4} \Vert D g \Vert_{L^2}\\        &\indeq\indeq        +        C \Vert D^2 g \Vert_{L^2}^{3/4} \Vert g \Vert_{L^2}^{1/4}  \Vert  \epsilon \partial_t f \Vert_{L^2}        +        C \Vert D^2 g \Vert_{L^2}^{3/4} \Vert g \Vert_{L^2}^{1/4}   \Vert  D f \Vert_{L^2}        \\&\indeq        \leq        C (\Vert f \Vert_{B(\tau)} + \Vert f \Vert_{L^2} )\Vert g \Vert_{B(\tau)}         +        C (\Vert g \Vert_{B(\tau)} + \Vert g \Vert_{L^2} )\Vert f \Vert_{B(\tau)}         .        \label{DFGRTHVBSDFRGDFGNCVBSDFGDHDFHDFNCVBDSFGSDFGDSFBDVNCXVBSDFGSDFHDFGHDFTSADFASDFSADFASDXCVZXVSDGHFDGHVBCX380}        \end{split}        \end{align} For $m=2$, by Leibniz rule we write        \begin{align}        \begin{split}        &        \sum_{j=0}^2 \sum_{l=0}^j \sum_{\vert \alpha \vert =j} \sum_{\vert \beta \vert = l, \beta \leq \alpha} \sum_{k=0}^{2-j} \mathcal{H}_{2,j,l,\alpha,\beta, k}        \\&\indeq        \leq        C \Vert f (\epsilon \partial_t)^2 g \Vert_{L^2}        +        C \Vert f D \epsilon \partial_t g \Vert_{L^2}        +        C \Vert f D^2 g \Vert_{L^2}        \\&\indeq\indeq        +        C \Vert Df \epsilon \partial_t g \Vert_{L^2}        +        C \Vert Df Dg \Vert_{L^2}        +        C \Vert \epsilon \partial_t f (\epsilon \partial_t) g \Vert_{L^2}        +        C \Vert \epsilon \partial_t f D g \Vert_{L^2}        \\&\indeq\indeq        +        C \Vert g (\epsilon \partial_t)^2 f \Vert_{L^2}        +        C \Vert g D \epsilon \partial_t f \Vert_{L^2}        +        C \Vert g D^2 f \Vert_{L^2}        .        \label{DFGRTHVBSDFRGDFGNCVBSDFGDHDFHDFNCVBDSFGSDFGDSFBDVNCXVBSDFGSDFHDFGHDFTSADFASDFSADFASDXCVZXVSDGHFDGHVBCX385}        \end{split}        \end{align} All the terms in \eqref{DFGRTHVBSDFRGDFGNCVBSDFGDHDFHDFNCVBDSFGSDFGDSFBDVNCXVBSDFGSDFHDFGHDFTSADFASDFSADFASDXCVZXVSDGHFDGHVBCX385} are estimated using H\"older and Sobolev inequalities. For illustration, we treat the fifth term, for which we write        \begin{align}        \begin{split}        \Vert Df Dg \Vert_{L^2}        &        \leq        \Vert Df \Vert_{L^4}  \Vert Dg \Vert_{L^4}        \leq        C \Vert D^2 f \Vert_{L^2}^{3/4} \Vert Df \Vert_{L^2}^{1/4} \Vert D^2 g \Vert_{L^2}^{3/4} \Vert Dg \Vert_{L^2}^{1/4}        \leq        C\Vert f \Vert_{B(\tau)} \Vert g \Vert_{B(\tau)}        .        \label{DFGRTHVBSDFRGDFGNCVBSDFGDHDFHDFNCVBDSFGSDFGDSFBDVNCXVBSDFGSDFHDFGHDFTSADFASDFSADFASDXCVZXVSDGHFDGHVBCX386}        \end{split}        \end{align} Collecting the estimates \eqref{DFGRTHVBSDFRGDFGNCVBSDFGDHDFHDFNCVBDSFGSDFGDSFBDVNCXVBSDFGSDFHDFGHDFTSADFASDFSADFASDXCVZXVSDGHFDGHVBCX380}--\eqref{DFGRTHVBSDFRGDFGNCVBSDFGDHDFHDFNCVBDSFGSDFGDSFBDVNCXVBSDFGSDFHDFGHDFTSADFASDFSADFASDXCVZXVSDGHFDGHVBCX386}, we obtain~\eqref{DFGRTHVBSDFRGDFGNCVBSDFGDHDFHDFNCVBDSFGSDFGDSFBDVNCXVBSDFGSDFHDFGHDFTSADFASDFSADFASDXCVZXVSDGHFDGHVBCX379}. \par Proof of \eqref{DFGRTHVBSDFRGDFGNCVBSDFGDHDFHDFNCVBDSFGSDFGDSFBDVNCXVBSDFGSDFHDFGHDFTSADFASDFSADFASDXCVZXVSDGHFDGHVBCX400}: Using H\"older and Sobolev inequalities, we obtain   \begin{align}   \begin{split}   &        \sum_{m=3}^\infty \sum_{j=0}^m \sum_{l=0}^j \sum_{\vert \alpha \vert =j} \sum_{\vert \beta \vert = l, \beta \leq \alpha} \sum_{k=0}^{m-j}\mathcal{H}_{m,j,l,\alpha,\beta, k} \mathbbm{1}_{\{ l+k =0\}}        \\    &\indeq    \leq   C\sum_{m=3}^\infty \sum_{j=0}^m \sum_{\vert \alpha \vert =j}  \Vert f \Vert_{L^\infty}   \Vert \partial^\alpha (\epsilon \partial_t)^{m-j} g \Vert_{L^2} \frac{\kappa^{(j-2)_+} \tau^{(m-2)_+}}{(m-2)!}    \\   &\indeq   \leq   C( \Vert f \Vert_{B(\tau)} + \Vert f \Vert_{L^2} )\Vert g \Vert_{B(\tau)}   .   \end{split}   \label{DFGRTHVBSDFRGDFGNCVBSDFGDHDFHDFNCVBDSFGSDFGDSFBDVNCXVBSDFGSDFHDFGHDFTSADFASDFSADFASDXCVZXVSDGHFDGHVBCX402}   \end{align} \par Proof of \eqref{DFGRTHVBSDFRGDFGNCVBSDFGDHDFHDFNCVBDSFGSDFGDSFBDVNCXVBSDFGSDFHDFGHDFTSADFASDFSADFASDXCVZXVSDGHFDGHVBCX372}: Using H\"older and Sobolev inequalities, we obtain   \begin{align}   \begin{split}   &        \sum_{m=3}^\infty \sum_{j=0}^m \sum_{l=0}^j \sum_{\vert \alpha \vert =j} \sum_{\vert \beta \vert = l, \beta \leq \alpha} \sum_{\substack{k=0\\0 \leq l+k \leq [m/2]}}^{m-j}\mathcal{H}_{m,j,l,\alpha,\beta, k}        \\    &\indeq    \leq   C\sum_{m=3}^\infty \sum_{j=0}^m \sum_{l=0}^j \sum_{\vert \alpha \vert =j} \sum_{\vert \beta \vert = l, \beta \leq \alpha} \sum_{\substack{k=0\\0 \leq l+k \leq [m/2]}}^{m-j}    \left( \Vert \partial^\beta (\epsilon \partial_t)^k f \Vert_{L^2} \frac{\kappa^{(l-2)_+}\tau^{(l+k-2)_+}}{(l+k-2)!}\right)^{1/4}   \\   &   \indeqtimes   \left(\Vert D^2 \partial^\beta (\epsilon \partial_t)^k f \Vert_{L^2} \frac{\kappa^{l_+}\tau^{(l+k)_+}}{(l+k)!}\right)^{3/4}   \left(\Vert \partial^{\alpha -\beta} (\epsilon \partial_t)^{m-j-k} g \Vert_{L^2} \frac{\kappa^{(j-l-2)_+} \tau^{(m-l-k-2)_+}}{(m-l-k-2)!}\right)    ,   \end{split}   \llabel{ o Xfz X6D Zbs 8n mLkL REmO 32NF7S ez N 3kn 7UZ I29 gS Ew2z k5YB a0aVY6 um S 1AB aKt 5Ah kx 6Qz8 B2Sp qdGFQV 8b q ASL CGw cdk XB aIL4 4gxb itC1tx jg f Q6A KOR WGJ KF 2ym4 xY9E DRne7O DE x bjd iFd VEl um lJwt cNTi KoLEff m6 Q I1h mfE Xqg 4S d3b8 GBBQ G7Fsyq 6g 8 ISR 17d ioq Xq fE2j RYpB JkWP2H ft h wNr hDs CJ6 Xj 2O6O Ni9j WzM1HH Ol F m2T fMm ujP hi Vujh pDSS 5vLdDj ay X 8Bp Bee Knw z6 Po6K Uz01 etrb2z 8j w iJy 9GB 3F7 RN KgEU WkFV uxBJpn tm c qZF 7NI VIi SW XEy2 B7tR n6afnG tz p 09o O1U sLv of ZRiQ cxV7 tFgjBZ 9m E ssT sKS ENx 0N 4gQ0 ubFQ GfRGM1 dM T XJ5 3WW ic7 Al LkC7 6gLZ m4Sn9z os C Be9 fDl mE6 Lq 9lJY 0heW K3oKvF iz S vWq Owh 8Nv AI coWO 4QfR 4CxSOG YX B Yo2 zig pP0 nY ii2u ivvj hp9zpn fs R Otk Km3 afL xS h6D4 eOoQ 0TcRSc Nt 9 ZGU vxH 0aR Pr HhuZ rJtG LHt533 m8 j 3jJ eSb CYI 84 loyk M7ia Y1pC09 AJ J 8pu hnm 6yT QO 1Wjg C5Nv JJcAuK 0I b 3Yx CEl eEH jd Ll8L 677t ACt5w4 MT f sZX C73 0oY 7L m9tG fKZQ 3UFhPK 7L q BEB 1tn de3 A0 5ZPC aVVM O4Wp1v l5 N WYV AS3 A6p pl 8D70 2knd T3RzBT oC o eQz gww kUc Hv 6n8y GW6A DNujYL GN T MO4 Orl G24 pC bTFb aEWw LfvUlF 7k p NGq 0km Jd9 vo UPmd lK4y ovDYNY PQ h xO4 g1s AZR KO sjpI 9Nva BCW4Es sG G iWL GR9 RPw oE 9NRO T90c KFKyx5 Zw z WBE nqO yIT IP yL1I Lo7e JfQdKi R1 V Jwi Lx7 0Oj SK HISh GIXv mm10FD xl I NtL dHG zZ4DFGRTHVBSDFRGDFGNCVBSDFGDHDFHDFNCVBDSFGSDFGDSFBDVNCXVBSDFGSDFHDFGHDFTSADFASDFSADFASDXCVZXVSDGHFDGHVBCX31}   \end{align} where we bound the constant coefficient by $C$ analogously as in \eqref{DFGRTHVBSDFRGDFGNCVBSDFGDHDFHDFNCVBDSFGSDFGDSFBDVNCXVBSDFGSDFHDFGHDFTSADFASDFSADFASDXCVZXVSDGHFDGHVBCX58}, and bound the $\tau$ and $\kappa$ term by $C$ analogously as in \eqref{DFGRTHVBSDFRGDFGNCVBSDFGDHDFHDFNCVBDSFGSDFGDSFBDVNCXVBSDFGSDFHDFGHDFTSADFASDFSADFASDXCVZXVSDGHFDGHVBCX378}--\eqref{DFGRTHVBSDFRGDFGNCVBSDFGDHDFHDFNCVBDSFGSDFGDSFBDVNCXVBSDFGSDFHDFGHDFTSADFASDFSADFASDXCVZXVSDGHFDGHVBCX310}. Therefore, using the discrete H\"older and Young inequalities, we obtain~\eqref{DFGRTHVBSDFRGDFGNCVBSDFGDHDFHDFNCVBDSFGSDFGDSFBDVNCXVBSDFGSDFHDFGHDFTSADFASDFSADFASDXCVZXVSDGHFDGHVBCX372}. \par Proof of \eqref{DFGRTHVBSDFRGDFGNCVBSDFGDHDFHDFNCVBDSFGSDFGDSFBDVNCXVBSDFGSDFHDFGHDFTSADFASDFSADFASDXCVZXVSDGHFDGHVBCX303}: We reverse the roles of $l$ and $m-l-k$ and proceed as in the above argument, obtaining   \begin{align}   \begin{split}   &   \sum_{m=3}^\infty \sum_{j=0}^m \sum_{l=0}^j \sum_{\vert \alpha \vert =j} \sum_{\vert \beta \vert = l, \beta \leq \alpha} \sum_{\substack{k=0\\ [m/2]+1\leq l+k \leq m}}^{m-j} \mathcal{H}_{m,j,l,\alpha,\beta, k}   \\    &\indeq   \leq   C\sum_{m=3}^\infty \sum_{j=0}^m \sum_{l=0}^j \sum_{\vert \alpha \vert =j} \sum_{\vert \beta \vert = l, \beta \leq \alpha} \sum_{\substack{k=0\\ [m/2]+1\leq l+k \leq m}}^{m-j} \left(\Vert \partial^\beta (\epsilon \partial_t)^k f \Vert_{L^2} \frac{\kappa^{(l-2)_+}\tau^{(l+k-2)_+}}{(l+k-2)!}\right)   \\   &   \indeqtimes   \left(\Vert \partial^{\alpha -\beta} (\epsilon \partial_t)^{m-j-k} g \Vert_{L^2} \frac{\kappa^{(j-l-2)_+}\tau^{(m-l-k-2)_+}}{(m-l-k-2)!}\right)^{1/4}   \\   &   \indeqtimes    \left(\Vert D^2 \partial^{\alpha -\beta} (\epsilon \partial_t)^{m-j-k} g \Vert_{L^2} \frac{\kappa^{(j-l)_+}\tau^{(m-l-k)_+}}{(m-l-k)!}\right)^{3/4}     .   \end{split}   \llabel{ OB mclD 6VUm JvbO5I 3E I inertial manifold d1h jxM Lf 9WNn q0cf LTeFzx KE Q HL5 u9r 94V qD SQs1 OhLa dtGVbT mq R Ct4 RQK TJX Db 5X7v c4gR O9owai Ud w Ac4 uIK r2e A0 fjhH u7Hu kD1pbw m7 Y zSr pRx dBp LY AGi7 aIh0 aZot6y MJ x sfB pw3 Jbb H4 lIXf 9s3j FDyyYS C8 w NFP YZi O2A JB ZKda AD3h u4WI0h 1z M ut1 AyX XJ1 IW tVOU vgjM Edr3H6 ZA S xZe GCN EG9 U8 XLLC 8otn qMrr5l lw n o3y V0t SlP W1 LjWv kj8W 91VA8T 6C X qWM s7W jvw Zr e71B uv1R kHBPRk Bw h suc Ryo Hn8 BL 0fGP m3AN nQX1Mq Dg L rFJ mPZ 1sT Lo 46Zh ffTF KGFDSI Mo V 0Ut Bu6 d0T Ik EwRX 13S4 5chg6J yD 2 aHO EAb D98 fD OYbX ljuO 3gIzi2 Ba E Ycz mnu Ox9 4Y NAzG eE0i HiqoXE tl 8 Ljg xh3 GO5 du EBzl 0uoV OZLI0x iE C VeN Fhc QXg qq N2eS obAT 4IXufl DO a nkx zPU d37 eY wmvr ud8e luLCBk wt T omG SSp Ww7 m1 0SlO pCJq qm6hQL au 4 qM8 oyx lOR GC T1cs PUcG qV6zE9 cP o OKW CTP vz9 PM UTRd 9e0f 5g8B7n PU R fC1 t0u 4HP Vz eEgM RZ9Z Vv1rWP Is P Jf1 FLi Ia2 Ty Zq8h aK9V 3qN8Qs Zt 7 HRO zBA heG 35 fAkp gPKm IuZsYW bS Q w6D g25 gM5 HG 2RWp uipX gEHDco 7U e 6aF NVd Pf5 d2 7rYW JCOo z1ystn gc 9 xlb cXZ Blz AM 9czX n1Eu n1GxIN AX z 44Q h8h pfe Xq dkdz uIRY MNMcWh CW i qZ5 sIV yj9 rh 1A47 uLkT wfyfry oi D x1e nHS eip 1v sO4K M6I5 nql8i6 XM K 0rk s0J 59i 2g zeHr jL1G bnITIB UV 3 3u1 He2 frw h8 SlMF fxUR XLEeSH 3m R UZP o4q NIEDFGRTHVBSDFRGDFGNCVBSDFGDHDFHDFNCVBDSFGSDFGDSFBDVNCXVBSDFGSDFHDFGHDFTSADFASDFSADFASDXCVZXVSDGHFDGHVBCX48}   \end{align} Therefore, using the discrete H\"older and Young inequalities, we obtain~\eqref{DFGRTHVBSDFRGDFGNCVBSDFGDHDFHDFNCVBDSFGSDFGDSFBDVNCXVBSDFGSDFHDFGHDFTSADFASDFSADFASDXCVZXVSDGHFDGHVBCX303}. \par Proof of \eqref{DFGRTHVBSDFRGDFGNCVBSDFGDHDFHDFNCVBDSFGSDFGDSFBDVNCXVBSDFGSDFHDFGHDFTSADFASDFSADFASDXCVZXVSDGHFDGHVBCX401}: We proceed as in \eqref{DFGRTHVBSDFRGDFGNCVBSDFGDHDFHDFNCVBDSFGSDFGDSFBDVNCXVBSDFGSDFHDFGHDFTSADFASDFSADFASDXCVZXVSDGHFDGHVBCX402}, obtaining   \begin{align}   \begin{split}   &        \sum_{m=3}^\infty \sum_{j=0}^m \sum_{l=0}^j \sum_{\vert \alpha \vert =j} \sum_{\vert \beta \vert = l, \beta \leq \alpha} \sum_{k=0}^{m-j}\mathcal{H}_{m,j,l,\alpha,\beta, k} \mathbbm{1}_{\{ l+k =m\}}        \\    &\indeq    \leq   C\sum_{m=3}^\infty \sum_{j=0}^m \sum_{\vert \alpha \vert =j}  \Vert g \Vert_{L^\infty}   \left(\Vert \partial^\alpha (\epsilon \partial_t)^{m-j} f \Vert_{L^2} \frac{\kappa^{(j-2)_+} \tau^{(m-2)_+}}{(m-2)!}\right)    \\   &\indeq   \leq   C \Vert f \Vert_{B(\tau)}   ( \Vert g \Vert_{B(\tau)} + \Vert g \Vert_{L^2})   .   \end{split}   \llabel{ iC DpZP g9jU AiCbJ3 yZ n Yvi ypD Y6U dI 0D0U vh8B 5pzuby bL z K6O kxi v1n 6J FvXn AnLZ kB2NBL 7X r JXM MKs 8CZ sG 212r nKHz isWyrb Rd h Aev TuC XC2 38 cSHD CDRu SfLZ6O cD 9 ILJ inl aki 39 ZA5i 8M5P HQYRu9 c4 V 1mh LVz eTT XV pgYH oa1X j0tDpT KK r zAm UdJ KYB Vz Yf2f uZHO iTzct8 8p O cWn FTv pWY 7r mUX9 mStl Dy20kd 9A u sOJ f9f lOW hZ wlWi HSoU yCdlxp PL d WIU hiT 3LH ok LLeF WnzM Xpp1rc Nz M dnm qvO Oqq u8 r9aH jZug ynu1WR 2K t 1bL Am4 xFI J5 6unz ROpY f9QL0M lw k oJU IPq s7J 4F uANa H5PL Cga54b D9 T mxt zHi 82i Dy kpSW vkcv xuDk7r WA a Ujb hIK Iz4 yM L1wz 7iFC xXO8ni s8 u Kjb 3ug lgI mL MaYS JSdA KU9AVH Z8 G 1n2 uex pWz yZ dj2g vCTe T5702u Se Z 9BC uRQ VhF JQ NMZ3 S2es CCUqV3 CH g WPG x9g zqE 69 aQz5 VOSu RPnXm5 7M l lKG Vza zF7 XD lGQJ 2Gaw ooq71n MV y ltP DG7 L0L d0 OErw 0kLA YXZzR6 Xa e ile VHu k7x cw yBOl V2Jr QxWjPb eL a iIY qHq pbc vo LZVz 1ytu fBpcJ4 MV e cfe ep7 bry 2U jHFL wzWF uMMdhV Hx D X8G sv0 RQL K8 YqbA THFE X0Ukxj Li S KH3 Rr5 umo xS EQd0 Myjy 6DdKVF mC U 7Tf BI0 xls OW k20H afdx vncOzF EK s kwm IHN HlU Nh KYaD nErM s9hNUp c2 L 8ip cs8 Xwv jO eX35 ehok gG6nyz ye v qmZ YDf dhQ XE Fk8q gF3J 8SySTL 1e m Z1h qVG Ls1 rR D8Fs 6u4N CoeyUC ER C 0xh 4x6 7er Cg 8lf8 J1Go 29fRN6 HM u UIw acx m2A f9 sfbN ov4C qQKcFl oP 7 BRW 85r UNd 2B gxWk OpDr DFGRTHVBSDFRGDFGNCVBSDFGDHDFHDFNCVBDSFGSDFGDSFBDVNCXVBSDFGSDFHDFGHDFTSADFASDFSADFASDXCVZXVSDGHFDGHVBCX403}   \end{align} \par Combining \eqref{DFGRTHVBSDFRGDFGNCVBSDFGDHDFHDFNCVBDSFGSDFGDSFBDVNCXVBSDFGSDFHDFGHDFTSADFASDFSADFASDXCVZXVSDGHFDGHVBCX379}--\eqref{DFGRTHVBSDFRGDFGNCVBSDFGDHDFHDFNCVBDSFGSDFGDSFBDVNCXVBSDFGSDFHDFGHDFTSADFASDFSADFASDXCVZXVSDGHFDGHVBCX401}, we obtain~\eqref{DFGRTHVBSDFRGDFGNCVBSDFGDHDFHDFNCVBDSFGSDFGDSFBDVNCXVBSDFGSDFHDFGHDFTSADFASDFSADFASDXCVZXVSDGHFDGHVBCX47}.  \end{proof} \par Similarly 
to \eqref{DFGRTHVBSDFRGDFGNCVBSDFGDHDFHDFNCVBDSFGSDFGDSFBDVNCXVBSDFGSDFHDFGHDFTSADFASDFSADFASDXCVZXVSDGHFDGHVBCX67} and Lemma~\ref{L03}, with analytic shift $(m-3)!$ rather than $(m-2)!$, we also have   \begin{align}   &   \Vert fg\Vert_{A(\tau)}   \leq   C \Vert f \Vert_{A(\tau)} (\Vert g \Vert_{A(\tau)} + \Vert g \Vert_{L^2})   +   C (\Vert f \Vert_{A(\tau)} + \Vert f \Vert_{L^2}) \Vert g \Vert_{A(\tau)}    .    \label{DFGRTHVBSDFRGDFGNCVBSDFGDHDFHDFNCVBDSFGSDFGDSFBDVNCXVBSDFGSDFHDFGHDFTSADFASDFSADFASDXCVZXVSDGHFDGHVBCX100}   \end{align} In the case when $f$ belongs to $L^\infty$ but is not square integrable, we have variant formulas   \begin{align}   &   \Vert fg\Vert_{A(\tau)}   \leq   C \Vert f \Vert_{A(\tau)} (\Vert g \Vert_{A(\tau)} + \Vert g \Vert_{L^2})   +   C \Vert f \Vert_{L^\infty} \Vert g \Vert_{A(\tau)}    ,    \label{DFGRTHVBSDFRGDFGNCVBSDFGDHDFHDFNCVBDSFGSDFGDSFBDVNCXVBSDFGSDFHDFGHDFTSADFASDFSADFASDXCVZXVSDGHFDGHVBCX328}    \end{align} and    \begin{align}    &      \Vert fg\Vert_{B(\tau)}   \leq   C \Vert f \Vert_{B(\tau)} (\Vert g \Vert_{B(\tau)} + \Vert g \Vert_{L^2})   +   C \Vert f \Vert_{L^\infty} \Vert g \Vert_{B(\tau)}    \label{DFGRTHVBSDFRGDFGNCVBSDFGDHDFHDFNCVBDSFGSDFGDSFBDVNCXVBSDFGSDFHDFGHDFTSADFASDFSADFASDXCVZXVSDGHFDGHVBCX330}    .      \end{align} The proofs are similar to \eqref{DFGRTHVBSDFRGDFGNCVBSDFGDHDFHDFNCVBDSFGSDFGDSFBDVNCXVBSDFGSDFHDFGHDFTSADFASDFSADFASDXCVZXVSDGHFDGHVBCX67}, where the modification of the proof for the variant formula \eqref{DFGRTHVBSDFRGDFGNCVBSDFGDHDFHDFNCVBDSFGSDFGDSFBDVNCXVBSDFGSDFHDFGHDFTSADFASDFSADFASDXCVZXVSDGHFDGHVBCX328} is to treat the term  $\Vert f \partial^\alpha (\epsilon \partial_t)^{m-j} g \Vert_{L^2}$ by H\"older's inequality with exponents $\infty$ and $2$. \par The next lemma provides  an analytic estimate for composition of functions. \cole \begin{Lemma} \label{L05} Assume that $f$ is an entire real-analytic function. Then there exists a function $Q$ such that   \begin{align}   &   \Vert f(S(t))\Vert_{B(\tau)}   \leq    Q( \Vert S(t)\Vert_{A(\tau)}       +\Vert S(t)\Vert_{L^2})    ,    \label{DFGRTHVBSDFRGDFGNCVBSDFGDHDFHDFNCVBDSFGSDFGDSFBDVNCXVBSDFGSDFHDFGHDFTSADFASDFSADFASDXCVZXVSDGHFDGHVBCX70}    \end{align} and  \begin{align}     \Vert f(S(t))\Vert_{A(\tau)}   \leq    Q( \Vert S(t)\Vert_{A(\tau)}       +\Vert S(t)\Vert_{L^2})      ,   \label{DFGRTHVBSDFRGDFGNCVBSDFGDHDFHDFNCVBDSFGSDFGDSFBDVNCXVBSDFGSDFHDFGHDFTSADFASDFSADFASDXCVZXVSDGHFDGHVBCX388}   \end{align} where $Q$ also depends on $f$. \end{Lemma} \colb \par \begin{proof}  First we prove \eqref{DFGRTHVBSDFRGDFGNCVBSDFGDHDFHDFNCVBDSFGSDFGDSFBDVNCXVBSDFGSDFHDFGHDFTSADFASDFSADFASDXCVZXVSDGHFDGHVBCX70}. Since $f$ is entire, for every $R>0$ there exists $N(R)>0$ such that   \begin{align}   \vert f^{(k)}(x) \vert    \leq   \frac{N k!}{R^k}    \comma x\in {\mathbb R}    \commaone k\in {\mathbb N}_0   \llabel{11KRKZ la H XQo N7O aoh ao tKxJ GOUt y96bCD K3 3 hIE KUo gH2 Rw EYIS uKzS 4pVz7B zd H rIj RgC vSJ Us 1x1D 3RRq k9icqL UQ 9 TeT iCn ifS 9N pGBJ FlPN UFbB5I U9 5 6UG 9Pr Sw9 oN NnvK 4g3G FnUxtW PN X Vvu e5r AIN 0e oRq2 XpZN DRpQUw Ij h rAX b13 L7G 4O Sy8f znjO JBbaVL Pa z vFW 8uC Y5i f8 VkV0 vydX HDDIU8 ga 2 Epi AyY YWh IK YvQ5 Xjd8 50u0Ay rP I bkO Eax PZD Hm R249 GiAn rkNORV g4 5 eQC 6Gz WVp lx UGMU R37s seSzNr la v QzM RpU HdQ 1j eJFh RqYl OTMQ1s Ti c KRH oxs D02 As e6eb YrNx UezPV1 Ih R wGx Dgj slQ db WcoI CqaG YMERvu gV 6 kyV Ql1 l68 rT m1ok gKup WNIfnH Z4 R mbd WEJ GBa Q4 CIXF Dzck kQ1BqW uy q xN3 N6k j6L ox cvV4 6tpl SOTvyA Q4 v Q0Z ZVe CxC Bv u6O4 OkIG j6IXwL Bi a ouA trS lFz 79 5kew 70Vn FBnjEf Mf i Sep 8xi Xv8 3I 36SZ goH8 sv2Brt lO u 8jT J0u gcT 08 dORq nA2U zjWI4L 5J j M3L Qdg sTy V3 hNC7 2Q89 Qb5bxu r3 w 6rj yNl 9kD te 492P Jf3T xTDd17 Vb n APK ptp usR Qu HiBw vO39 DUmCif ho w 4RU vFo ejB w1 JOq8 mHPw CfIfh3 uU x OVX Loz 1E1 d0 V2Kb eCLL 9M3pMA kL u Acb 6KU HnO Xq dPGA TF1L rsh5tp x3 O ZIA D0H 5nD GH lXBq Tqgy AEA6pb VZ c NjR E1B kH3 om JFFj m9TJ A7NUBm tg 5 ppA wIh 190 lJ CmYe ih6J WiCfyD WA B yAb g61 BS1 4Q NzvZ SQgq LjGvIW 7V z 37v T4d Z92 l3 Ddxf toju poeKIT ks c 2uF 2Yn BJt Mr rNsD J0hP EK2h7K Fi Q mbE zr1 TCl t5 0d3L R9HD yUIetg wm y KvDFGRTHVBSDFRGDFGNCVBSDFGDHDFHDFNCVBDSFGSDFGDSFBDVNCXVBSDFGSDFHDFGHDFTSADFASDFSADFASDXCVZXVSDGHFDGHVBCX69}   \end{align} and    \begin{align}   f(S(t))    =    \sum_{k=0}^\infty \frac{f^{(k)}(0) S(t)^k}{k!}    .   \label{DFGRTHVBSDFRGDFGNCVBSDFGDHDFHDFNCVBDSFGSDFGDSFBDVNCXVBSDFGSDFHDFGHDFTSADFASDFSADFASDXCVZXVSDGHFDGHVBCX66}   \end{align} By Lemma~\ref{L03}, we obtain   \begin{align}   \left\Vert \frac{f^{(k)}(0)}{k!} S^k \right\Vert_{B(\tau)}   \leq   \frac{N C^k k}{ R^{k} }   \Vert S \Vert_{B(\tau)} (\Vert S \Vert_{B(\tau)}+ \Vert S \Vert_{L^2})^{k-1}    .   \label{DFGRTHVBSDFRGDFGNCVBSDFGDHDFHDFNCVBDSFGSDFGDSFBDVNCXVBSDFGSDFHDFGHDFTSADFASDFSADFASDXCVZXVSDGHFDGHVBCX73}   \end{align} Summing \eqref{DFGRTHVBSDFRGDFGNCVBSDFGDHDFHDFNCVBDSFGSDFGDSFBDVNCXVBSDFGSDFHDFGHDFTSADFASDFSADFASDXCVZXVSDGHFDGHVBCX73} in $k\in {\mathbb N}$ and using the Taylor expansion \eqref{DFGRTHVBSDFRGDFGNCVBSDFGDHDFHDFNCVBDSFGSDFGDSFBDVNCXVBSDFGSDFHDFGHDFTSADFASDFSADFASDXCVZXVSDGHFDGHVBCX66}, we arrive at   \begin{align}   \begin{split}   \Vert f(S(t))\Vert_{B(\tau)}   &\leq   \sum_{k=1}^\infty \left\Vert \frac{f^{(k)}(0)}{k!} S(t)^k \right\Vert_{B(\tau)}   \leq   C N \frac{\Vert S(t) \Vert_{B(\tau)}}{\Vert S \Vert_{B(\tau)}+ \Vert  S \Vert_{L^2}}    \sum_{k=1}^\infty \left(\frac{C(\Vert S \Vert_{B(\tau)} + \Vert S \Vert_{L^2})}{R} \right)^{k}    \\&   \leq   C N \sum_{k=1}^\infty \left(\frac{C(\Vert S \Vert_{B(\tau)} + \Vert S \Vert_{L^2})}{R} \right)^{k}    .   \end{split}    \llabel{6 NMD AzF Dv MoxI YoMi Kt9ZPA da D Yug 53l gYe rH dqgX X70K pLETmz D2 4 Cry GGR gzr NZ Vs78 R7S0 k0ji59 Hq H YiB 0Kv tZr qj tqBx Rqbc Lz69t0 O7 c eIl LvM kDc 72 5QVg UMo6 uO72ft jK N mef 1dr 5yi ZK ljQg HKgl iZZQtk he h 8wZ 7ZR oSG Va x9RK HQwg BEvzRw dG N cEP CGY hat b2 jZ6X mUdN pYdg7l 7s B frh Gml ueY kH A4vI esnr WIqxLS Pz S nDF x7R PCf sG eziz 44Bx QVkm9e dB A EhM dru UNI a0 3TUv nhkR mygOHT cT f iLk fUg Yhr ux fmyq LiFR ozIcHk CC 0 kXQ Jo5 C3j sf LMJe 0vbW 57SYQb bW W 1ov oL2 RRm 7y Idmw pXqn eEKrwB vj v or3 pvU VDD 8B Bsuq 9Oob ZWmwy1 ql 3 38i vRE eCO sJ K7Tj ax80 Tdcrxr U0 W krj idN QHL Ob sVEM R5ih 9SmKjr gZ C MDU D28 M5T L5 qTCS bIe9 MYP3AR qO L xQH 60d vw1 XC vSnp TBca BAMObs pr u 0BD xjV UWW fM p10G NLtf RogR2T lC 4 KKh CbY sRC Gh 7FNp jgdS CHUzEK Rs A cK2 iZG 9wF pg wCPU i7jh 7fw5Of IJ 6 iw8 EU7 ECM Bz zhKB EAOX FSz3qs SM v 5pZ ldB J6p QZ cQ8I SDMz Gh2GQD m2 R co6 9Z0 zEE PS ymYZ K9vz tjRhCx x9 8 0eq RB3 C9D im ZOyq 936z i2X9pj xa f 4T8 T1y NdI Fs dPMi U7gX deHctk QG k Jfj Xs4 uoB xZ A2vY 4sWx hqgNJo UN v lSs Dq3 xWs Oj rc1V oi6e ienlnG Hf 8 UYC Uex IgL a3 wPm7 A37T KTNSTs v5 s 2pL uyK 6jX Kq clcp nPEm ofpXpX zV t uwp FO8 Usp pu ucs3 ZyVM GgG4zK Xz 1 STn gNj SRR eL 07cW jObW U2d1wq 5k 2 UTA Hgu ogq zV j9zR 5clt RIRgGa Jg Q 2RR lMv dik H6 UDFGRTHVBSDFRGDFGNCVBSDFGDHDFHDFNCVBDSFGSDFGDSFBDVNCXVBSDFGSDFHDFGHDFTSADFASDFSADFASDXCVZXVSDGHFDGHVBCX174}   \end{align} Choosing  $     R= 2C \Vert S(t) \Vert_{B(\tau)} + 2C\Vert S(t) \Vert_{L^2}  $, we obtain $  \Vert f(S(t))\Vert_{B(\tau)}\leq C N$, where $N$ depends on $\Vert S(t) \Vert_{B(\tau)} + \Vert S(t) \Vert_{L^2}$. Finally, observe that $    \Vert S(t)\Vert_{B(\tau)}    \leq    \Vert S(t)\Vert_{A(\tau)} $, by the definition of the norms, concluding the proof of \eqref{DFGRTHVBSDFRGDFGNCVBSDFGDHDFHDFNCVBDSFGSDFGDSFBDVNCXVBSDFGSDFHDFGHDFTSADFASDFSADFASDXCVZXVSDGHFDGHVBCX70}. The estimate \eqref{DFGRTHVBSDFRGDFGNCVBSDFGDHDFHDFNCVBDSFGSDFGDSFBDVNCXVBSDFGSDFHDFGHDFTSADFASDFSADFASDXCVZXVSDGHFDGHVBCX388} is proven analogously by using \eqref{DFGRTHVBSDFRGDFGNCVBSDFGDHDFHDFNCVBDSFGSDFGDSFBDVNCXVBSDFGSDFHDFGHDFTSADFASDFSADFASDXCVZXVSDGHFDGHVBCX100}, and we omit the details here. \end{proof} \par For the next two lemmas, assume that $\tilde e$ is one of components of the matrix $E$ in \eqref{DFGRTHVBSDFRGDFGNCVBSDFGDHDFHDFNCVBDSFGSDFGDSFBDVNCXVBSDFGSDFHDFGHDFTSADFASDFSADFASDXCVZXVSDGHFDGHVBCX03}, i.e., either  $r$ or one of the components of $a$. By the assumptions \eqref{DFGRTHVBSDFRGDFGNCVBSDFGDHDFHDFNCVBDSFGSDFGDSFBDVNCXVBSDFGSDFHDFGHDFTSADFASDFSADFASDXCVZXVSDGHFDGHVBCX115} and \eqref{DFGRTHVBSDFRGDFGNCVBSDFGDHDFHDFNCVBDSFGSDFGDSFBDVNCXVBSDFGSDFHDFGHDFTSADFASDFSADFASDXCVZXVSDGHFDGHVBCX116}, we have   \begin{equation}    \tilde e(S, \epsilon u) = f(S) g(\epsilon u)    ,    \label{DFGRTHVBSDFRGDFGNCVBSDFGDHDFHDFNCVBDSFGSDFGDSFBDVNCXVBSDFGSDFHDFGHDFTSADFASDFSADFASDXCVZXVSDGHFDGHVBCX117}   \end{equation} where $f$ and $g$ are positive entire real-analytic functions. \par The first lemma gives the estimate of the derivative of the component of the matrix $E$. \cole \begin{Lemma} \label{L06} Given $M_0>0$, and \eqref{DFGRTHVBSDFRGDFGNCVBSDFGDHDFHDFNCVBDSFGSDFGDSFBDVNCXVBSDFGSDFHDFGHDFTSADFASDFSADFASDXCVZXVSDGHFDGHVBCX117}, where $f$ and $g$ are as above. Then   \begin{align}   \begin{split}   \Vert \partial_t \tilde e\Vert_{B(\tau)}   &   \leq   Q( \Vert u \Vert_{A(\tau)} + \Vert u \Vert_{L^2}   ,   \Vert S \Vert_{A(\tau)} + \Vert S \Vert_{L^2}   )   \end{split}   \label{DFGRTHVBSDFRGDFGNCVBSDFGDHDFHDFNCVBDSFGSDFGDSFBDVNCXVBSDFGSDFHDFGHDFTSADFASDFSADFASDXCVZXVSDGHFDGHVBCX95} \end{align} for some function  $Q$. \end{Lemma} \colb \par \begin{proof} By \eqref{DFGRTHVBSDFRGDFGNCVBSDFGDHDFHDFNCVBDSFGSDFGDSFBDVNCXVBSDFGSDFHDFGHDFTSADFASDFSADFASDXCVZXVSDGHFDGHVBCX02}, the chain rule, and product rule, we obtain   \begin{align}   \begin{split}   \partial_t \tilde e    =   f'(S) \partial_t S g(\epsilon u)    +    f(S)\nabla g(\epsilon u)\cdot \epsilon \partial_t u    =   -f'(S) v\cdot \nabla S g(\epsilon u)    +    f(S) \nabla g(\epsilon u)\cdot \epsilon \partial_t u   .   \llabel{QiH fOJo TsCM5W 2y 6 iNn 1xp obq mp gwjy xoRs dmPboo zE A yUo Tmw 1tQ mq P9Ew yL8H aGxAc0 Uh K 2zn vVW cv4 xF 5xAY XXjp JmE5fV 2h U n6j t2L t2n pf eIdg LKL0 xSUDMR 51 R zg6 r6q UCM cV 24DA Shdp mrmven e8 e SZa P5B bbP PS B6eO sVJO V4dmUi Z2 m Ewi wTN t1J DD wWjT 4mfF 4jZC9h ko 7 gie fw6 w1v ml Czxs 6ijE v9M3Ot rO x W2p kMW i7z FK YpA0 DcFZ pnNS1L k1 r POT NRZ Npe 0O WgJK 5agU 63Lxqc Hw 0 iSB cYq NRa D3 8nPv GHse lXU3x8 bU Z uje 0xi Rm3 iR Abwt JBla V4M2PJ Mq 7 jb3 KYD NMb l3 CVKr ZZDR nV4A65 Ht Z iRo dJP 0T4 NF m1G3 ukOB MuK52l jg 5 v9t BsF jgC Ra 68Hj Ya69 4ZjFJ6 CV A rUC 2OV qxQ cn Otix 4DXk XaU86p 3k 1 dQq fJU cTY g9 MU1R GZFJ PXYBCZ pg F VfY WG2 vVe Uq L8QD 5Ikc RtJuuS 2j V 2j9 re9 nJq 1j 1LPn GXB7 7qxq0Y 2a z EuZ 4u4 41Z Dh P2c6 ltZf e7sMqQ EK x Xmd Tmc 3NG continuous bV4n mp3b GwOzqX PV e 7sE Dvo eN9 Dm lBeZ GcrM yDBgyv yB b 0Ya eHb b2P lE CWXT Tua3 s6XI7f t2 A h6e EUi Uur 3a iIji IQx2 c3fDbv dy B 7wo 4oR 3i1 lr vQS8 HjwH sbZzQ4 7M 5 Upn d0q 5kK W6 ZWnG PJtI EdnmGA Ln e p3h uZw DAN Uo 8G1v cBzx XVdpxF Bp U JlI tAE BDs jv lold GbWb qs2ZYg wI Z 5UR D3k PbY Dm bFN9 fu02 6aB6pz 5D g pDw rEG Q0F qd 0SE0 sQdQ Ar8OTg O4 f pav WhG Zyg Rh reK5 l6IF BWjy60 Vk e HEn XbH 5mS L4 5RJe 4oS9 aMVTca RL S 768 nDf gaC eH 9JEv oHAo 0guEz8 XS c EWR DNN 53w Jb B2uF hfDFGRTHVBSDFRGDFGNCVBSDFGDHDFHDFNCVBDSFGSDFGDSFBDVNCXVBSDFGSDFHDFGHDFTSADFASDFSADFASDXCVZXVSDGHFDGHVBCX65}   \end{split}   \end{align} Therefore,   \begin{align}   \begin{split}   \Vert \partial_t \tilde e\Vert_{B(\tau)}   &   \leq   \Vert f'(S) v\cdot \nabla S g(\epsilon u) \Vert_{B(\tau)}   +   \Vert f(S) \nabla g(\epsilon u) \cdot \epsilon \partial_t u  \Vert_{B(\tau)}   = \mathcal{G}_1 + \mathcal{G}_2.   .   \end{split}   \label{DFGRTHVBSDFRGDFGNCVBSDFGDHDFHDFNCVBDSFGSDFGDSFBDVNCXVBSDFGSDFHDFGHDFTSADFASDFSADFASDXCVZXVSDGHFDGHVBCX71}   \end{align} By repeated use of \eqref{DFGRTHVBSDFRGDFGNCVBSDFGDHDFHDFNCVBDSFGSDFGDSFBDVNCXVBSDFGSDFHDFGHDFTSADFASDFSADFASDXCVZXVSDGHFDGHVBCX67} and \eqref{DFGRTHVBSDFRGDFGNCVBSDFGDHDFHDFNCVBDSFGSDFGDSFBDVNCXVBSDFGSDFHDFGHDFTSADFASDFSADFASDXCVZXVSDGHFDGHVBCX330} and Remark~\ref{R05}, we arrive at        \begin{align}        \begin{split}        \mathcal{G}_1        &        \leq        \Vert f'(S) \Vert_{B(\tau)} (\Vert v \cdot \nabla S g(\epsilon u) \Vert_{B(\tau)} + \Vert v \cdot \nabla S g(\epsilon u) \Vert_{L^2} )         +        \Vert f'(S) \Vert_{L^\infty} \Vert v \cdot \nabla S g(\epsilon u) \Vert_{B(\tau)}        \\        &        \leq        \Vert f'(S) \Vert_{B(\tau)}        \left( \Vert g(\epsilon u) \Vert_{B(\tau)} ( \Vert v\cdot \nabla S \Vert_{B(\tau)} + \Vert v\cdot \nabla S \Vert_{L^2})  + C\Vert v \cdot \nabla S \Vert_{B(\tau)} + \Vert v \cdot \nabla S \Vert_{L^2} \right)        \\        &        \indeq        +        C \Vert g(\epsilon u) \Vert_{B(\tau)} \left( \Vert v \cdot \nabla S \Vert_{B(\tau)} + \Vert v \cdot \nabla S \Vert_{L^2} \right)        +        C \Vert v \cdot \nabla S \Vert_{B(\tau)}        .        \label{DFGRTHVBSDFRGDFGNCVBSDFGDHDFHDFNCVBDSFGSDFGDSFBDVNCXVBSDFGSDFHDFGHDFTSADFASDFSADFASDXCVZXVSDGHFDGHVBCX389}        \end{split}        \end{align} For the term $\Vert v \cdot \nabla S \Vert_{B(\tau)}$, we again appeal to \eqref{DFGRTHVBSDFRGDFGNCVBSDFGDHDFHDFNCVBDSFGSDFGDSFBDVNCXVBSDFGSDFHDFGHDFTSADFASDFSADFASDXCVZXVSDGHFDGHVBCX67}, obtaining   \begin{align}   \begin{split}   \Vert v\cdot \nabla S \Vert_{B(\tau)}   &   \leq   C\Vert v \Vert_{B(\tau)}    \left(\Vert \nabla S \Vert_{B(\tau)} + \Vert \nabla S \Vert_{L^2} \right)   +   C\Vert \nabla S \Vert_{B(\tau)}    \left(\Vert v \Vert_{B(\tau)} + \Vert v \Vert_{L^2} \right)   .   \end{split}   \llabel{L6 zuTGzM bi t UjW dxZ 6a2 TI fhtI yuR5 IgpZCA 5H Z MZT 9Sv LIq Sw De8g BxlE dFkyj2 Qy u Yzw lwp xYa s2 Xz8v Cmy2 027Jm0 LH D Mi8 X5I 255 a2 UfDR 4mco TUjWxZ 3Z J 03a lP1 5KP Wf fdCP QvOf akMNcp uV D t4c jR1 oY7 dq RGJv GT0i kreblM Gb a NXG b8O Ex6 aI sztj 7eeT t9OKLB Fu M QbP Lyq Tdk jP dF7D NGeS MBOEYZ 0a H FSU Gw6 GNI l3 rCv2 gKZv onEoA4 Ii x K1D 6Rm zZg ME tWfX dAB9 HMHCjO 0w o SnF R5o uCA 2u g9Qr 9w68 xDKat9 r3 p nFj uBu aLm uZ kyxC OZpf Abl4tX og g aAa ws7 Te4 IO Hrgz IvJq 03hLQ5 85 C RSL UZA xXz 6R jOpe Sx7B RT7txk A3 d Yoi pdO pjM YA SRDU XR02 w2BAyx Hn x Bn5 huR 3Mg Qz B8mV lx3o BYOTTU My g vzK k6x rgi JZ BGwQ eoe3 ToBsNh gz l 64F VnO df7 4c inPC eqU7 hVu03C Sp k pK0 SZ9 faB xR dRWa Iz7Q lt63qk h0 b brE Z2F zUf vS D29l 3cwm 9uA1zY lT A ONV dcO OfY cD S6Hg l4QN mCuPdg FE p uc4 nz4 vO8 N6 HC52 kSn7 ZMrzuz 2z K 9tk 6DA jaU W6 vghs pP3M DTE9TD 8W b Ire HTM 8Va qw E1kM M3vD cFVic5 wD M XTh n1N dqg IJ ZEwP Xjaz E6sa9U MI l sAZ S4E F2v eh cxfo AuLV xunTvR TK a wGc Hov iaX bt CuyM VT7D QdbuGB 6W e mA8 AO9 XEF kO 8f6A rkAb i397pw 9U o R7F CtP Gyt Kq 01TE mnbb NVQOuW DU J dgz saN JLP MA g1EY Ks7v NI45FZ xQ z abl oG5 bT4 UM eFFs gh4P zviDrN Um f uDE umN FpN Wb gBds tqQk xVgebW nk X xwE x5m nT5 Cl bHNY 3dv7 bvSXqV AK w Zq9 OZC Vhc nF plqS 46GI DIS6gB yG MDFGRTHVBSDFRGDFGNCVBSDFGDHDFHDFNCVBDSFGSDFGDSFBDVNCXVBSDFGSDFHDFGHDFTSADFASDFSADFASDXCVZXVSDGHFDGHVBCX96}   \end{align} By the definition of the analytic norms in \eqref{DFGRTHVBSDFRGDFGNCVBSDFGDHDFHDFNCVBDSFGSDFGDSFBDVNCXVBSDFGSDFHDFGHDFTSADFASDFSADFASDXCVZXVSDGHFDGHVBCX36} and \eqref{DFGRTHVBSDFRGDFGNCVBSDFGDHDFHDFNCVBDSFGSDFGDSFBDVNCXVBSDFGSDFHDFGHDFTSADFASDFSADFASDXCVZXVSDGHFDGHVBCX103}, we have        \begin{align}        \begin{split}        \Vert \nabla S\Vert_{B(\tau)}        &        =        \sum_{m=1}^\infty \sum_{j=0}^m \sum_{\vert \alpha \vert =j} \Vert \partial^\alpha (\epsilon \partial_t)^{m-j} \nabla S \Vert_{L^2} \frac{\kappa^{(j-2)_+} \tau^{(m-2)_+}}{(m-2)!}\\        &        \leq        C\sum_{m=1}^\infty \sum_{j=0}^m \sum_{\vert \alpha \vert =j +1} \Vert \partial^\alpha (\epsilon \partial_t)^{m-j} S \Vert_{L^2} \frac{\kappa^{(j-2)_+} \tau^{(m-2)_+}}{(m-2)!}        \leq        C\Vert S \Vert_{A(\tau)}.        \label{DFGRTHVBSDFRGDFGNCVBSDFGDHDFHDFNCVBDSFGSDFGDSFBDVNCXVBSDFGSDFHDFGHDFTSADFASDFSADFASDXCVZXVSDGHFDGHVBCX327}        \end{split}        \end{align}   Collecting estimates \eqref{DFGRTHVBSDFRGDFGNCVBSDFGDHDFHDFNCVBDSFGSDFGDSFBDVNCXVBSDFGSDFHDFGHDFTSADFASDFSADFASDXCVZXVSDGHFDGHVBCX389}--\eqref{DFGRTHVBSDFRGDFGNCVBSDFGDHDFHDFNCVBDSFGSDFGDSFBDVNCXVBSDFGSDFHDFGHDFTSADFASDFSADFASDXCVZXVSDGHFDGHVBCX327}, we obtain        \begin{align}        \mathcal{G}_1        \leq        Q(\Vert u \Vert_{A(\tau)} + \Vert u \Vert_{L^2},         \Vert S \Vert_{A(\tau)} + \Vert S \Vert_{L^2}        )        .        \label{DFGRTHVBSDFRGDFGNCVBSDFGDHDFHDFNCVBDSFGSDFGDSFBDVNCXVBSDFGSDFHDFGHDFTSADFASDFSADFASDXCVZXVSDGHFDGHVBCX391}        \end{align} Using analogous arguments, we also get   \begin{align}   \begin{split}   \mathcal{G}_2   \leq   Q(\Vert u \Vert_{A(\tau)} + \Vert u \Vert_{L^2},         \Vert S \Vert_{A(\tau)} + \Vert S \Vert_{L^2}        )   \label{DFGRTHVBSDFRGDFGNCVBSDFGDHDFHDFNCVBDSFGSDFGDSFBDVNCXVBSDFGSDFHDFGHDFTSADFASDFSADFASDXCVZXVSDGHFDGHVBCX97}   \end{split}   \end{align} since by definition        \begin{align}        \begin{split}         \Vert \epsilon \partial_t u \Vert_{B(\tau)}         &         =         \sum_{m=1}^\infty \sum_{j=0}^m \sum_{\vert \alpha \vert = j} \Vert \partial^\alpha (\epsilon \partial_t)^{m-j+1} u \Vert_{L^2} \frac{\kappa^{(j-2)_+} \tau^{(m-2)_+}}{(m-2)!}          \\         &         \leq         C\sum_{m=1}^\infty \sum_{j=0}^m \sum_{\vert \alpha \vert = j} \Vert \partial^\alpha (\epsilon \partial_t)^{m-j} \Vert_{L^2} \frac{\kappa^{(j-3)_+} \tau^{(m-3)_+}}{(m-3)!}         \leq         C\Vert u \Vert_{A(\tau)}         .         \end{split}         \llabel{ KMu ec2 Qsa XE t6nq tuzW xEMcmW cY 7 Yo1 TKs EhX Il Bg4d U6fd JT1HYh Gy N pqv l2L Qaz FB oUQp dn8O iN1cGW 8l F I3s KQR 4Ep X4 Sfs7 d6xx t8hoY8 Ft a iiw WRO I3r 1s hgHg Y6CF K2738b WZ C kOb 3oN dMQ 3E x3PG eYrE 1574bD 7y G iOp fQe HH5 HO 7RQt OJ9Z ZpbtYp aH l L9C 0pb mVt RS VOzK ACwQ Mj4Q4o 4e J ArR dci atp hU eJAD tQdk Wk6m70 KA p mTV 15q Ynb oD Bktn k4uk g8SPIj Om 3 loj 895 BhI mn 8bXA UAmL WRxp2c OF i WNj XPo qOH Nn ZN9W qkEU NHh13J HC p YI5 qYr Owv Gf gUxh i4UT dN3VMS iR T vvH 7xs Etl I5 jkmH a35r e1gHeq 4L V iMp iUN oWh Tp 9zvY hGul ekd0rs CB Z fje HQG yhW iE QX1J grj0 L9iprp Dq F NiZ sFS BY0 Jz 81oj F3V4 dZnehU R0 J XtN Xsw cRq 7B spAp 8od9 9NsP3s sZ e x8a SlC gDd ur UaYH I1Sd qCnaMK Pr T WQC ROz huZ qH xPEG LM9B BfabHU iS q 4Qy 1LP hQW aI WYo9 IakJ bRNCxH hp 6 eiW v7p 7JU Jg SBNg USVZ sHUJ3o eV 1 uWw 9Ae eix bp 8Axt 5ZgQ PREAQs A0 i wxl NK0 td3 Jr PygK rvUR 0YXdIl Ww J dEa o9W MuR eY fQUq Uoeb EgHTLI cL 6 2xV JZh nGx EH yBRf f4wT 9LuePp wj r pRo gWJ 0Iv 6s VWMI DiLz otM3Zr Y4 U KtD EJo 2Jo tO UhTM blm3 oNAr4F UY d dau g68 TX9 Bg oI0N 0C5y SekTx6 Jx U vud b3K mMo cy 1WtA yC1u aC9rD3 l5 7 9WJ J59 LQa O7 hxpD HsE1 LtilFT P6 j hAS QeP rSi os f5Ts UVjO PX5qZZ o0 m 0dz wu6 ctA w6 a3bo SxFq DWJauv pS g cgU o3a PoB wW E2BJ 9GAl K3c0uj cs Q FNb 4yC zpR UDFGRTHVBSDFRGDFGNCVBSDFGDHDFHDFNCVBDSFGSDFGDSFBDVNCXVBSDFGSDFHDFGHDFTSADFASDFSADFASDXCVZXVSDGHFDGHVBCX305}        \end{align} Therefore, \eqref{DFGRTHVBSDFRGDFGNCVBSDFGDHDFHDFNCVBDSFGSDFGDSFBDVNCXVBSDFGSDFHDFGHDFTSADFASDFSADFASDXCVZXVSDGHFDGHVBCX95} is proven by combining \eqref{DFGRTHVBSDFRGDFGNCVBSDFGDHDFHDFNCVBDSFGSDFGDSFBDVNCXVBSDFGSDFHDFGHDFTSADFASDFSADFASDXCVZXVSDGHFDGHVBCX71}, \eqref{DFGRTHVBSDFRGDFGNCVBSDFGDHDFHDFNCVBDSFGSDFGDSFBDVNCXVBSDFGSDFHDFGHDFTSADFASDFSADFASDXCVZXVSDGHFDGHVBCX391}, and \eqref{DFGRTHVBSDFRGDFGNCVBSDFGDHDFHDFNCVBDSFGSDFGDSFBDVNCXVBSDFGSDFHDFGHDFTSADFASDFSADFASDXCVZXVSDGHFDGHVBCX97}. \end{proof} \par The second lemma gives the analytic estimates for the component of the matrix $E$. \cole \begin{Lemma} \label{L07} Assume  \eqref{DFGRTHVBSDFRGDFGNCVBSDFGDHDFHDFNCVBDSFGSDFGDSFBDVNCXVBSDFGSDFHDFGHDFTSADFASDFSADFASDXCVZXVSDGHFDGHVBCX117}, where $f$ and $g$ are as above. Then        \begin{align}        \Vert \tilde e(t) \Vert_{A(\tau)}        \leq        Q(\Vert u \Vert_{A(\tau)} + \Vert u \Vert_{L^2}, \Vert S \Vert_{A(\tau)} + \Vert S \Vert_{L^2})      ,        \label{DFGRTHVBSDFRGDFGNCVBSDFGDHDFHDFNCVBDSFGSDFGDSFBDVNCXVBSDFGSDFHDFGHDFTSADFASDFSADFASDXCVZXVSDGHFDGHVBCX190}        \end{align} for some function $Q$. \end{Lemma} \colb \par \begin{proof} Since $\tilde e(t) = f(S) g(\epsilon u)$, the proof of the estimate \eqref{DFGRTHVBSDFRGDFGNCVBSDFGDHDFHDFNCVBDSFGSDFGDSFBDVNCXVBSDFGSDFHDFGHDFTSADFASDFSADFASDXCVZXVSDGHFDGHVBCX190} may be carried out by appealing to Lemmas~\ref{L03} and \ref{L05} in the $A$-norm. \end{proof} \par \startnewsection{Estimates on the velocity}{sec05} \par Recall that   \begin{equation}      E(S,\epsilon u) =       \begin{pmatrix}     a(S, \epsilon u) & 0\\     0 & r(S,\epsilon u) \mathbb{I}_3    \end{pmatrix}    ,    \llabel{i QeZE DLIk zqkkEG LK C 45B ngC Xim 0y QL5U Y0jW zkX2gv w2 u i3S 30G XpR Ru RIOY ERuR MD0bBN Ob 6 yvI peN pLK yV 1Dqn Tur1 ckV4yM 5Z 0 dh7 u9K msr PG rxRk RPb0 9I1YXI 9I M f6k PUD YOz 9P 535C U39t LxBPHI q5 g L4w d3S 87h iG Y0di FePn JBIRNw R6 0 scp aOS 5Gk Ep CyWM EXyM H0v86F Rc r 8aQ Yqu O5Z Zd a1If uGzD 9TN1PQ LN Y Ftx B6N X3y OI AULF Eswa 9VuVUz vo N 7i8 Au6 7gd Rn yN3n tXN7 5CDXmh DV W SDz clZ WQ3 D7 FniE s4KR wvAMdk gX K S8f eUR SIS DF 3KoW BCeW 8fkYNg H8 E Sno nm8 fw8 Ah Culf SYR4 LvuF5x YN u BFM 0s8 Q9f Vr VjS6 qkEd I9B7Aa AB n FEH yQ4 FE3 p2 zNYb eCWW sjICZd Mi Q xwE xry PWH PE dyQb 0vtY 4GOLLA XS Y FDd vhb woj up 5CPr RfoT WAfFde Qy 4 Qwg GS8 Ss4 K6 Urdk hKQB OIMpMA Ps j THP wz4 WqI Tm hrrr rh2Z J2AmbT sa J Hj3 QxI xTN 0H yXIg rXl7 sQdnBW rz o lZx CCw SlP 6b 52rq IT8y OeZKUW 8i A Sz8 WKi u3b jh g1zC xce2 fpagwS yo i D3Y xMR rbr pr B8oF pJMt ItcgQC JD 8 qQP X3t FyO yK 9kqY sp1H e43FFj fT Y AoH 9Tr sZu lJ Lfz5 PN9M wOwlhK XL n 4BW KHp c5r 7r V3Iv f2qA k0ckyW iP t 2Oo 9BT fM5 cX jalt IM7P PtTpRX 1O y sUM Bp8 XCu zg Ad0N KAvd il19Gw ZV P rGL oGO Nxe U0 3Sm0 BA1c xyfKGI TF L Nfg WPM pfR mO 63LF MDUm z7qGhj 19 0 n34 Oc9 gie JV zMNO gAhf GTqBpw 53 q 4ey b5J 8Ze UN T8R3 8gKs 708pFf j1 E LgK iz0 EAC dA lPNN 4okg hcRHpm Uu o eX4 JJ2 jtE XN XPRf b9Fb N2DFGRTHVBSDFRGDFGNCVBSDFGDHDFHDFNCVBDSFGSDFGDSFBDVNCXVBSDFGSDFHDFGHDFTSADFASDFSADFASDXCVZXVSDGHFDGHVBCX125}   \end{equation} where $a(S, \epsilon u)$ and $r(S, \epsilon u)$ are as in Section~\ref{sec21}. \par \subsection{Estimate on the curl} \label{sec6.1} We first need to rewrite the equation \eqref{DFGRTHVBSDFRGDFGNCVBSDFGDHDFHDFNCVBDSFGSDFGDSFBDVNCXVBSDFGSDFHDFGHDFTSADFASDFSADFASDXCVZXVSDGHFDGHVBCX136} so to be able to estimate the curl of the velocity~$v$. Introduce $r_0(x) = r(x,0)$ (i.e., $r_0(S) = r(S,0)$), and note that, by our assumptions,    \begin{equation}     r_0 (S) = f_2(S) g_2(0)        .    \llabel{9ZtY Ue X 5EK 9Lv 22f vW 5Q0V I0sq 1IR1a5 ao j pZW yEK cOL aR Ud2w xsHX ws4XSr IV g WKy wyt yA0 0F oMN8 est5 2u9fFK ES C iP7 WBv bjO hW uZIU 5qDY QZOYVc 1m 2 uyy WrA TV9 JI VCcI Lzr0 XfXvLR jH k Q7d vQ5 qQk Kb UddV j2Ng D9SbOe 2N D ztn Ija LQE En LY4K 0Do6 6oKZnE 5t g 9p8 WgF Jry mn IR4J uqan KB0Tu9 gT T dhm FWJ r7S kc mRgS VZBH XMKPdv e2 T KbY we1 w6L jL t3fq NqPN rkL9xZ 31 2 OlY PJz YHI R9 FOwt Oda3 azKYIZ mf O piD P2U Y1x dX Mxnx 0FHt 0Z7KK5 me I iZq gEI AVq kX uCIa u0ge 7syt7w PL 1 Som n8t rVN AP 5vOe Apza 4OmcKk 2K S IDK bjS H9R t0 UKyy FShx SWRzxS su M N2v FJg dTw YN CC7f YPYG 52l9Du Z3 0 jM3 nxA AjY ns bz3Z REXu oEuVh6 vU i KQX l9x plK 4p v4r1 SOF3 gu9JSq uV 3 aSx xP4 r2s qm h8Vn c5lA nLI692 OJ 3 I9h fl4 tXz aJ e6GP jnSE JW9IKl Rk 7 tDl XhK dgK qL 5vJL R7DG F4USAI tt B cO5 DLE LRk ZR 6VJC CuPG SsCMuM Uq I ge2 P0Y mcB 39 mmPK 7m3i No3LPl KT 6 Dxn Guw VqM Zr P9zd W0MI XlRsCv y2 x 5hW n7U xgt JM GkQS m5W7 JEQfvm B2 G i6c Xhy Qt1 Wj Tdz4 Samd 0QLh4l fR n Gnt TOT 6Ya Qs AyNo ew52 Km6kid Uy k KV0 tnt PTp By gKcO hcx7 7EkXSL oc d pou Owa CQ3 fN M4vt wIJv Nz7dwF Nv i TOv 3Me dg8 iL haM3 kdcj DT30pa Bs w xFV IaX yjw 0j OSbR OdY3 e3Phl4 qm 1 wta M08 Qrm GB ISKq yoyg oFISt0 tQ p 0pZ 1QQ Jkn Im Qg43 6gpz dYuwl5 ms o oOh 9l4 Z25 dX GcSB aavB Celg09 qW H 49v DFGRTHVBSDFRGDFGNCVBSDFGDHDFHDFNCVBDSFGSDFGDSFBDVNCXVBSDFGSDFHDFGHDFTSADFASDFSADFASDXCVZXVSDGHFDGHVBCX130}   \end{equation} Define    \begin{equation}     \tilde f(x,y) = 1- \frac{r_0(x)}{ r(x, y)} = 1 - \frac{g_2(0)}{g_2(y)}     .    \llabel{DRF PE6 SR gMSr OEiF DqSmv7 Mq C 8n0 B4O ujj ip uVia 053h gXhQqE Iz H 7cb OAB dHw Od TEXv oykC 4wrajh zq n quD ojX l0P z7 Utyb iOOs EINnWE Xe h c70 MvZ 7n6 V9 qYzr yOYu iqBEmX ak t OJI 8ZA dNd vl eWiA h0gw DT6fWH 6x 8 uSA jGk KuV 6b wiUf cfOW jyzWdR n7 K pFu CAH Iog Hd U3Wk bDq1 yV8oMf Js k TSJ AFV 4eh tk SEpC WQaS gnZkk3 ae e sYU 7TC 7Dj dy cEy6 YhXS t7PPC7 5V y XVX 003 3yw Eo y37O VWeF B6f5A4 zI 5 U5F YFT vRF mZ bF16 anRn JcMwFJ uX C gDV U7F kan wU Zqfh ivmv YiUAZa iJ L eDm ZbM Kcv xR NJpw xkgI OqH6pT R8 2 1OC O97 sij Rl LagX kCrx nH5nGq 2b 8 TD1 QLj wMU G7 wLB8 hV0t cfNtxr tA K TbT V70 Cus 2Z 9zWb i3fB hUBueL ct D Xuk 88j Ypx Jz uARu 5GQS MoTAOB uu B yUh ESR 0IV rV GFcH 0rtS LRCVou oU 5 QbO UKP 8Wa Fj Z0RL HqyX 8xXBLO wU t kiE wWG 9fm n7 Ifh4 MwTp uLWjOV Cm p vGn oyH PV3 Q2 xzY5 iDGu MIyY2z Hh L TMR UHz 9Pr Ip 7eqf jbQM hsqm7f 21 z apb nCN GWx jG JAHf CykE OXMfRk QM S OnR mHD slD Bl TsX1 OYIh 6VnhtO RV g XFc lHw 0Ms U8 taIo G1w4 0uBD1s Oo H WK3 Ing jGM dI jSUJ arrb i4PVOm zD g PpU Kxj Kv1 BO ITyQ Es2w ltfw1Q mq U 5Ng Zs0 44d Vr 2Fik dG6e qDEgva jQ F WQP mQe Nsc dx SAyh qL2L fTUWHJ oE s YAr e3L vak eY HXXZ EbdJ odq2jj uf Z Fwk FwM LCi BA zKsl RU0w ObDZ08 O1 I xMM ACd T6C uj 3WOg HbMf WXYKct Ro I ofd 0EE CQV bF Hq3G Vsjj JoPiGc rg 9 Oa2 u5s 6j2 W5 AphDFGRTHVBSDFRGDFGNCVBSDFGDHDFHDFNCVBDSFGSDFGDSFBDVNCXVBSDFGSDFHDFGHDFTSADFASDFSADFASDXCVZXVSDGHFDGHVBCX126}   \end{equation}  Since $\tilde f$ is a function of $y$ only and vanishes at $y=0$, there exists a bounded entire function $h$ such that   \begin{equation*}    \tilde f(x,y) = y h(y)    .   \end{equation*} Denoting   \begin{equation}    \tilde h(x) = x h(\epsilon x)    ,    \llabel{H hdSr Lt2Ctx 8D H zuZ XSE jHg y3 97e8 6ALC 5AbRC7 FQ 4 wEL lpD vHt ou mXBq q92L gd2nSR cR G 0a4 Dbm yLR pM SH3B XeD4 ty54XX qE M bCk Huu ytF 38 almr JcZe HlP8cg wN X GTZ XeL avo hy QaSE h1B7 TEz54d RZ G zhP JZt YAF K4 PoB5 vc5J Dphrxf vS F cW2 hSR zpn yv lqRW whzh Z7fcd7 66 G h9D EWg 5Oc wF 3FcG MUe4 2nZF1y w0 4 7XI G22 Jd1 IN GkQL rXcv IW54aU 33 t nhD 9Nq KOL tM Kkt5 UCl2 CmLLGk DZ s 8pO Rcl J3B qG 5hKn eVXE UFbnER g2 k 1yz fok QoC sn OZVS Zk0X 7sUjP5 Rr u AQi Lb6 j2y 3X Ri9u zWFZ ZQnfC3 Zi s nU2 4uq cES J9 ATwm RCmb jfWBbj vQ b Z1G maY OiF K2 s7Z2 0eOV Lynbcw Rt e lbg InP M0f 6B d17E oyU9 WOzSvN Jh 6 ruX 82b A5F vG C0sj hJsR 110KLy Nn r wYe Vxg GE2 h6 LZS6 AuVv mZ22sB OA V LsD yC1 Hx9 Aq VK0x jKEn BwZAEn jX a HBe Phl GAr tT YDqq 5AI3 n843Vc 1e s s4d N0e hvk iA Gcyt H6lp oRT3A2 eY L TK0 WLJ Opp 5m bSzo aMNx jCJvze Ts 9 fXW iIo cB0 eJ 8ZLK 7UuR shiZWs 4g Q 6t6 tZS g2P bl DsYo LBhV b2CfR4 Xp T ORO qto ZGj 6i 5W89 NZGb FuP2gH 5B m GsI bWw rDv r9 oomz XTFB 1vDSZr Pg 9 5mC Qf9 OyH YQ f47u QYJ5 xoMCpH 9H d H5W 5Ls A0S HB cWAW Go3B iXzdcJ xB I SVM LxH KyS oL JcJs hYQy ybWSSG dw c dEZ 6bu 0qf mm j7WM sVRY qJD7Wc iZ q Ari mTf 6cs ur ioA8 Ow5x ADChPG 7s l xY1 gye geB 9O C39K FnGn 7aB1yr 4v o 4hH wIC 93L q9 KQvI gFnB oS04NZ d5 a y1r lEr Ppb vx pseY Dufg nOiibG DFGRTHVBSDFRGDFGNCVBSDFGDHDFHDFNCVBDSFGSDFGDSFBDVNCXVBSDFGSDFHDFGHDFTSADFASDFSADFASDXCVZXVSDGHFDGHVBCX131}   \end{equation} we then have   \begin{equation}    \tilde f(S,\epsilon u) = \epsilon \tilde h(u)    .    \llabel{hY O CBw wv4 pxm 37 IgUW 9XY4 zxXMVQ DN m Ywk HaG 258 v6 MBSe rLf0 Fya6yC LC z Gm3 P8C OVD g7 Rpp3 FaWd G0kIQY RI s 0ie Iu2 tgM mi YoJO mYt3 fBSHj2 9S q 8ab 4N9 4WE hm 8BC7 9jHL 9KmhK9 Zf r 6Al pyf Ihp G9 3BnO Mt16 adBKKu 3o N cfI hWj Mfk eG dnzs sVHK cphPZl Bg c ewS d0u 1Pe PY EZow ysGu GBvc6e a5 i WB3 XgS zyH nO aDp8 VRDV qmVy30 99 p j5g FkD pqz PQ 96m5 MyJS Q7bj2L yL R n1P svR qAa N5 QOIc JKVP gGovU3 Q6 9 LN4 W4v 8wb S0 jsi6 krY7 uOj9j1 aB S Z6u onf 5hs Ij 6BQ8 Nnb2 fcKcWD zL 7 nMs D7N xvu Y9 5USb CwaH EBadGI qH g ori u5e sF7 Ab YAYt Jodo qq8VAy c8 J ApV 0lJ lYw yn 6Sc6 qTU7 O1ek3r 6y 8 zE0 Nvq FxO e0 etCQ fFQY U6wydA K5 x sgt Wyo Hf4 wk LJE7 vsDt aSeyQY ne c GAE E9Y ieh g1 bbi1 8Qic 5rJyA6 E1 G Nzn KFe O8F 23 Ren7 C5ml sRUt5m oB Q E0P 0vM b5s NR JiRE NqaZ l1S0XQ CL A WKO fnh sEc rp 3BDn d8xR XBqged xR V 9ds 7hx DYs Xf uUOY UpPd 6nkpnq U1 B Zyl ozB iHl OZ qUwQ Hw6u yx1Cev AK J pX8 j3F DHK dk FGe1 PLPw a3zuUl rn P llZ svR 42F oz OOJ2 xv7S OMypK6 I9 d xqW RHM dk9 0P Zn1C JiHj FiQOpy Pn F Wcn PNL zFC UT ZbYK KM8Q yLgFXt Fq K zGF o8P nWk Yx Yg2T sGQZ fmqqDq iz H d93 oVT k6Q GD Ulju HwS8 8DFPcV 0P f Y1P qwp sQS bM rVVN eZDe 1F6mt7 Pg e lfB t5i ta8 qX Mizw KxCN oigAqP ys i tAb I7g 4R8 XE 7tAK Nexk g5a2ur Ri g 9qM qHS kER bI p9hr W5jz TcJlTg pM e LWF zzS VDFGRTHVBSDFRGDFGNCVBSDFGDHDFHDFNCVBDSFGSDFGDSFBDVNCXVBSDFGSDFHDFGHDFTSADFASDFSADFASDXCVZXVSDGHFDGHVBCX132}   \end{equation} Since $\partial_t S + v \cdot \nabla S = 0$, the equation \eqref{DFGRTHVBSDFRGDFGNCVBSDFGDHDFHDFNCVBDSFGSDFGDSFBDVNCXVBSDFGSDFHDFGHDFTSADFASDFSADFASDXCVZXVSDGHFDGHVBCX136} for $v$ is equivalent to the nonlinear transport equation   \begin{align}   (\partial_t + v\cdot \nabla) (r_0 v) + \frac{1}{\epsilon} \nabla p = \tilde{h} \nabla p.    \llabel{Ng Dt ecg2 YfSf NjVFP1 55 D AWM bnt nVm 17 177l YTOb MWmVtw fP q J0n F6b aQx uS P13Z Jubf mgRHsW KV n CRD cMk LRV Xd 3PEv Zme7 xKUmMU Rk b if3 lAO uTe mZ qkO5 fId0 htO7xU pL z tA9 RO7 gys Af qj75 1DjT sfImVI PQ n e9F o9a Rex 4T 4oyp enSz EDID7K Wc h KsI rTo lag XL h4f0 vi4y l25MHP pb k 9zY Ifp FVV 19 CRh1 tkqo Q1bQGt zG L 0f4 xN6 GCH P7 RXOn RbpX 2WlP56 rN Q lEy ARF 84J ep rN20 CZxi yjplnN Ji E 3Cf UxU jGx Vf ClOs 8Vwv DylP65 4x A Dcv WYt PPl qy DzXN 83tJ gydFQh Hu e uqU 6Zy Vij YQ jtdJ f4J9 3VVpNw za g iEm sHT Ins FJ aiog OY7G nLhHqU yL x rAq 8it XPs Mt yXO3 2mVY U2GMV7 29 m Gz8 vXx 4rN FR 84Vg cS6B mn15nK dR E fkh XGu f2c 9C rXBq sEVO hPCASy f9 J s77 D9k qMz nb cTe2 Uu2J nq8DBs 8r f U7r d0B qQ4 zz XOxT g7Cr ZuthjY W4 a Ltr FcC jwG Zs 7MWM mSpH FHfLiB 7k X 34M WAu 5LI j3 gvVG Akt1 HUEE8U BY W VMJ zMs v5P 5J ZvQn VHI1 ZMg6Af sQ j xl2 VB9 XKg EZ rlgD 5Gqi FivOJg x0 u YgE 71I Uqg jK nLSe lC6w 6tFI8f KM i Hsq HDS eNX K9 vtSm VcR8 qXEEOq dh 6 WiO bbu Bi1 dL 1TGy Uler X0GfVw JX Z cO4 PCi e5n 0g NVUs 6zGE iK4Qu0 eV O BbG Iv3 1f7 Mm ZWGU ui0X sJeIKr c6 e 1Hd z1A 1uA iE J1cB UWYP soTyxI E5 B V9E yuj ob6 iZ VLOi NN6O dNQmB1 lI j dwi BRH 8d2 G2 I62d Vz4I wScfAf Pt O MSH 1H1 XAL jF IXix vFrN 9eesLS kD 9 2y7 poO Y5z fX T21p hWyN 6C3iGT 3j H Rxn Lx4 ShH FP Seag 0pIDFGRTHVBSDFRGDFGNCVBSDFGDHDFHDFNCVBDSFGSDFGDSFBDVNCXVBSDFGSDFHDFGHDFTSADFASDFSADFASDXCVZXVSDGHFDGHVBCX34}   \end{align} Applying curl to the above equation and using $\curl \nabla p = 0$, we arrive at   \begin{align}   (\partial_t + v\cdot \nabla) \curl(r_0 v) = [v\cdot \nabla, \curl](r_0 v) + [\curl, \tilde{h}] \nabla p   .   \label{DFGRTHVBSDFRGDFGNCVBSDFGDHDFHDFNCVBDSFGSDFGDSFBDVNCXVBSDFGSDFHDFGHDFTSADFASDFSADFASDXCVZXVSDGHFDGHVBCX110}   \end{align} To treat \eqref{DFGRTHVBSDFRGDFGNCVBSDFGDHDFHDFNCVBDSFGSDFGDSFBDVNCXVBSDFGSDFHDFGHDFTSADFASDFSADFASDXCVZXVSDGHFDGHVBCX110}, we would like to use \eqref{DFGRTHVBSDFRGDFGNCVBSDFGDHDFHDFNCVBDSFGSDFGDSFBDVNCXVBSDFGSDFHDFGHDFTSADFASDFSADFASDXCVZXVSDGHFDGHVBCX145} from Lemma~\ref{L08} and thus we need to estimate the forcing term    \begin{equation}    G = [v\cdot \nabla, \curl](r_0 v) + [\curl, \tilde{h}] \nabla p = G_1 + G_2       \llabel{B SVsPSR Rg q oHF PB9 xiM 4L zsaa on5q dO9eSI ck w 1TW pA0 dPR xm 0mJS 4CO2 0oacZN Rh y Xd9 vF8 myO tV O1yy 8IPX Ri69P7 nE i VT9 tqw zd7 Si 8yQL Jx6H izMwkK qP K Vou NGw hrs qK iGAL uz34 Qk0yTD Fs J vXI YjM YHS 9X WlgE wN6g CwyVdn Ix K 4B0 5Xc hlw sf 7kVs FFU6 TUiBVU ra V 7Wr TLb ySa hd qdcZ gGtK tv0Uqp 86 j B63 0pH aqG wD nGpg 0Zpr B4fFDm tx 2 J23 Ed3 7CN W5 LkvQ FWEo uXhTIq dU m rBx 7BG PME D1 7hd2 4zjC EavON1 DJ o xfm Zad kTV q0 VaQa DqE0 vOGEaz PB E KbU LHq xEU gd PZyH W7xH 2QWZqj xA a P0G szJ TMS Ww xJ2s ZwiY A2Hu5w NN p JNQ r1x 3ed vP 8QVE szZC f2893t zS k 3sL WLv XG2 Ro 4lH9 AvQb 5Ty4O1 FT O AYS cPc PIh JB PFPu HDAt dSw9x8 vu 9 hrW f4X Oyp Gs NQvx f9ko lHSNEk ME G aZM Pg1 AwX Ya RSUo t7Fm p3SE4Z MR F bqq 9KC jfC 9L Hi1X 5nEh Uw8q2Y 6H L ATR c4T NlN Wk stGQ 2RiL jSdaiz Ob 6 P84 8n3 HLW Hy Aptp fQyX Ecd7x0 zq V NmD Kzl 5Nt EV G2Ev bbiv LHBtna s3 T 4dd SM0 QJy x3 fU3C uYD2 A5yhtc sh T 2kt VFK XXU DH oEs9 pq3V 1LmreP Cw x 4R8 R51 mNs W3 ViX6 mzJo ctVZW1 os 7 k61 bwJ oz5 nO Esii h2VH u4D0bf y5 R NqG QbI H8l XV 9pKs HGGj m4xCyp KP P TAU wxs x7S Lj PjQd zHCV uarB6X Uv e 32Z 5b1 lfF SR yfI0 SaOd aYTBXN 9q m ijD O3a z1z o9 yKuq nB2U Shf0Gn Wh Q Eht 1eH PAd 4a 9l5p 3Yi8 1r25uZ ww R 8G3 OIz hTO GY sITG 5Ywt jFM4r2 gL K n87 OPh oVn 11 UEIn 32O8 sGpRnb Nn R DFGRTHVBSDFRGDFGNCVBSDFGDHDFHDFNCVBDSFGSDFGDSFBDVNCXVBSDFGSDFHDFGHDFTSADFASDFSADFASDXCVZXVSDGHFDGHVBCX127}   \end{equation} in the analytic norm \eqref{DFGRTHVBSDFRGDFGNCVBSDFGDHDFHDFNCVBDSFGSDFGDSFBDVNCXVBSDFGSDFHDFGHDFTSADFASDFSADFASDXCVZXVSDGHFDGHVBCX103}. Since $([v_j \partial_j, \curl] w)_i = \epsilon_{ikm} \partial_m v_j \partial_j w_k$, where $\epsilon_{ikm}$ is the permutation symbol, we may apply Lemma~\ref{L03} and obtain   \begin{align}    \Vert G_1 \Vert_{B(\tau)}     &    \leq    C (\Vert \nabla (r_0 v) \Vert_{B(\tau)} + \Vert \nabla (r_0 v) \Vert_{L^2} )     (\Vert \nabla v \Vert_{B(\tau)} + \Vert \nabla v \Vert_{L^2} )     .    \label{DFGRTHVBSDFRGDFGNCVBSDFGDHDFHDFNCVBDSFGSDFGDSFBDVNCXVBSDFGSDFHDFGHDFTSADFASDFSADFASDXCVZXVSDGHFDGHVBCX39}   \end{align} From \eqref{DFGRTHVBSDFRGDFGNCVBSDFGDHDFHDFNCVBDSFGSDFGDSFBDVNCXVBSDFGSDFHDFGHDFTSADFASDFSADFASDXCVZXVSDGHFDGHVBCX328}, \eqref{DFGRTHVBSDFRGDFGNCVBSDFGDHDFHDFNCVBDSFGSDFGDSFBDVNCXVBSDFGSDFHDFGHDFTSADFASDFSADFASDXCVZXVSDGHFDGHVBCX327}, \eqref{DFGRTHVBSDFRGDFGNCVBSDFGDHDFHDFNCVBDSFGSDFGDSFBDVNCXVBSDFGSDFHDFGHDFTSADFASDFSADFASDXCVZXVSDGHFDGHVBCX190} and \eqref{DFGRTHVBSDFRGDFGNCVBSDFGDHDFHDFNCVBDSFGSDFGDSFBDVNCXVBSDFGSDFHDFGHDFTSADFASDFSADFASDXCVZXVSDGHFDGHVBCX39}, we get $ \Vert G_1 \Vert_{B(\tau)} \leq Q(M_{\epsilon, \kappa}(t)) $. The term $G_2$  may be estimated in an analogous way since $([\curl, \tilde{h}]\nabla p)_i = \epsilon_{ijk} \partial_j \tilde{h} \partial_k p$, leading to   \begin{align}   \begin{split}    \Vert G_2 \Vert_{B(\tau)}    &\leq    C (\Vert \nabla p \Vert_{B(\tau)}  + \Vert \nabla p \Vert_{L^2} )(\Vert \nabla \tilde{h} \Vert_{B(\tau)} + \Vert \nabla \tilde{h} \Vert_{L^2})    \leq    Q(M_{\epsilon, \kappa}(t))    .   \end{split}    \llabel{IZe 0JQ 29v ld LUrB 00Fd 3eabLW FP v 66h Xwj kEc RT o4b0 deLW v7rsDc 7j i veH 75j 8F3 2V z9lk cNph uH21hO E9 L V28 bc6 f8a 6s zNZX RpJx DwCkqB rr 6 ewI LJZ FhH e8 QzMV DJtJ ewQZJU I5 5 1nJ 9Pe ovN SZ Zali eWJs ew3kjd km K n3Q 6GK kfs nK eFlM nycL 7fD9Ia Z2 z gk6 dlq d01 Z1 rwEc bdMR H5I49K qE F Rtg 0tv kEx ZC RtMC kHE5 snrnPg 6I j vso b8K Thi Os lyOF Ya4m bysgIg a7 p zCd iXn ubn tk exBb Wpwp Mo4k8y xe s QnX cG8 mEF NW aegq bClc Z5bzZy Nj E Vtx ND8 cqb Df 7Y7C ZmIa 140Y9s QM l jx4 EzB M6U 0P HNOB sii1 AKPov8 wB S PX2 VPf 8xV 3D 3TDI WXly 4kOf59 po 0 Skn 07H CFx P9 bjyZ BmyH MTfaFd JG q zVF PGN y6p oI BElR sWC4 85AFty 6t 6 0NF gbM hZa jD ei2x vrcz UQK9iM Ph t iC1 JBi BOt S9 5oFW T7jV 0Q25YS eG 3 jID 3SX k9x JI J1Qg YMR2 hjzlnm TN t fqC 5QI 5Iw pe qxKG 1t8w q4sthw H6 n BDB Ucl kzZ Wz 0B2C qsO4 fDgXc3 Ou M GJo nsZ M2E KI UOKr 73Yv 7JL9Cd yg u Qqt IUr rhk Hn HJ5J igfR rnKyNp 0T 3 7DB JiF zVq yP nQSY 5QBv 3LYZ5L rC d Tqa zec zlm GV Acx7 uExX Yu8ik7 BI X 3DE paI yld Wv AtDx vl70 fkOcQ7 Rr e yXD JE1 gj3 cL SNQT mK6l qOE4SV Gx g DJO 9mu KDv vM aHGe evgM M7lWB1 5W q 9lF oUx FcP Sy jHSW X5kg Pup5yv w2 8 JWV Z9A t0W jf JuLf hFQd 6aE7Vn Pq 7 ow2 C6h K2l oM kgJn 03Vz 5EteIM oi H WS2 Htw 4f9 9H Nj73 18ga IvJxja kN j unH NNM 8kp By 8bp9 KEQc lHhTNF gu 2 uEk b6l PKa hMDFGRTHVBSDFRGDFGNCVBSDFGDHDFHDFNCVBDSFGSDFGDSFBDVNCXVBSDFGSDFHDFGHDFTSADFASDFSADFASDXCVZXVSDGHFDGHVBCX43}   \end{align} Proceeding as in Lemma~\ref{L08}, we obtain        \begin{align}        \begin{split}        \frac{d}{dt} \Vert \curl(r_0 v) \Vert_{B(\tau)}         &        =        \dot{\tau} \Vert \curl (r_0 v) \Vert_{\tilde{B}(\tau)}        +        \sum_{m=1}^\infty \sum_{j=0}^m \sum_{\vert \alpha \vert = j}  \frac{\kappa^{(j-2)_+} \tau^{(m-2)_+}}{(m-2)!} \frac{d}{dt}\Vert \partial^\alpha (\epsilon \partial_t)^{m-j} \curl(r_0 v) \Vert_{L^2},        \end{split}        \label{DFGRTHVBSDFRGDFGNCVBSDFGDHDFHDFNCVBDSFGSDFGDSFBDVNCXVBSDFGSDFHDFGHDFTSADFASDFSADFASDXCVZXVSDGHFDGHVBCX309}        \end{align} where we used the notation \eqref{DFGRTHVBSDFRGDFGNCVBSDFGDHDFHDFNCVBDSFGSDFGDSFBDVNCXVBSDFGSDFHDFGHDFTSADFASDFSADFASDXCVZXVSDGHFDGHVBCX103}--\eqref{DFGRTHVBSDFRGDFGNCVBSDFGDHDFHDFNCVBDSFGSDFGDSFBDVNCXVBSDFGSDFHDFGHDFTSADFASDFSADFASDXCVZXVSDGHFDGHVBCX133}. By \eqref{DFGRTHVBSDFRGDFGNCVBSDFGDHDFHDFNCVBDSFGSDFGDSFBDVNCXVBSDFGSDFHDFGHDFTSADFASDFSADFASDXCVZXVSDGHFDGHVBCX145} from Lemma~\ref{L08}, we get   \begin{align}   \begin{split}   \Vert \curl(r_0 v)(t)\Vert_{B(\tau)}   &   \leq   \Vert \curl (r_0 v)(0) \Vert_{B(\tau)}   +   C   t\sup_{s\in (0,t)} \Vert G(s) \Vert_{B(\tau)}   +   C   t\sup_{s\in (0,t)}\Vert v(s) \Vert_{B(\tau)}   +   Ct   \\   &   \leq   C   +   t Q(M_{\epsilon,\kappa}(t))   .    \label{DFGRTHVBSDFRGDFGNCVBSDFGDHDFHDFNCVBDSFGSDFGDSFBDVNCXVBSDFGSDFHDFGHDFTSADFASDFSADFASDXCVZXVSDGHFDGHVBCX51}    \end{split}   \end{align}   Next, we  estimate $\curl v$ in the analytic norm $B(\tau)$. Denoting   \begin{equation}    R_0 = \frac{1}{r_0}    ,    \llabel{ Nekv nAo0 mMoeLX Y0 s pWI DgT gba mn kCTS UvKL j6dtN3 xY Q r7k ycT ga6 P3 EVDP R1kf tVZdHM la C iTw UNJ Aw1 pf ZrEt ytoz vY85OB 0F t LYp fKq Fko dC G2d5 ezkg 1rqSaY 8E J dLo UYF DML sa WIIT SAph uqtx20 vt b dYD k8R eLv 27 4NHf 3aKJ U0coov 2E c Mcw omv uDx Zk oHa4 L8uh voKCiM y7 U zui N29 wfP fG 8TFt jOdq lG4M0l uR t 0lx ZM3 YA5 uY oGkN kBzh bHo4aO EE 5 rxC Up1 ArJ jB N1h5 cWs1 F6qQVa pG N gBg x7M 5As wf EMyM gPso VexGPi mA 4 yXP ma5 dTj J2 BN3F 3SMY KxqSeg aB 7 69y 5PV O8Q yJ TRBe 2vVe 0pkNSy On h kXc 9d0 clS Xl bsqp 8tCv KLTo5a lz i gOO RUl P3q MG vTJA RYoz YDJLvT Uu q MOO NYN P82 Ps 9OiD s17x 75SajM oJ G Bdp 885 oix sz l6oK VEaW AZWwF9 6X f zlO d9z CCg ma 4P67 OY59 PgbGtz VZ 4 KWP NxA jy0 Fp 94LW BmVq jqW3vX H4 r Cz4 IuA L6F dN eNtl HCh4 NRn6Ux nB Y Spy wkH C2I 0G ictz w30l LmQmNT Z8 H juy bH1 W0z wX r9JE ACx0 pSDpGu O6 0 FM2 Qme wNR Fh P9fi Im1F eGcDRw 2o v hfk 2fS x8N UA IwPv 4CmF CE8lJz 5I e Hrs Pwi uDa T6 9r6M j4ke BfBhvg BX O KJd Tbj UFN A5 9p2f Dj1W pyNbiY AL 9 Oyw osN Ciw 4t SyNk 8avA m4NgAp m0 X gUy 3yA hP3 TW ptUe 3Ejd LgM1O9 lt 3 2Xb pFz K4A vi 5WwF dPyH AYdZrf Fh r dkY rdf nJA ep d29E 66Cj si1JRZ 6l J UOA gxA p9M os ewoB EV4n 6aH75q qk v nXy bkm eyY j4 TaXP Wupq UzKbkp qy o 78o 4zw XEr s8 jZMp W6ij 9AP4qU 6B o fbe 06q WG1 LY 6kbl lcrP mrXDFGRTHVBSDFRGDFGNCVBSDFGDHDFHDFNCVBDSFGSDFGDSFBDVNCXVBSDFGSDFHDFGHDFTSADFASDFSADFASDXCVZXVSDGHFDGHVBCX83}   \end{equation} we rewrite   \begin{align}      \begin{split}       \Vert \curl v\Vert_{B(\tau)}       \leq       \Vert R_0 \curl (r_0 v)\Vert_{B(\tau)}       +       \Vert [\curl,R_0] r_0 v\Vert_{B(\tau)}             =\xi_1+\xi_2      .      \end{split}    \label{DFGRTHVBSDFRGDFGNCVBSDFGDHDFHDFNCVBDSFGSDFGDSFBDVNCXVBSDFGSDFHDFGHDFTSADFASDFSADFASDXCVZXVSDGHFDGHVBCX165}    \end{align} For the term $\xi_1$, we use \eqref{DFGRTHVBSDFRGDFGNCVBSDFGDHDFHDFNCVBDSFGSDFGDSFBDVNCXVBSDFGSDFHDFGHDFTSADFASDFSADFASDXCVZXVSDGHFDGHVBCX330} and the curl estimate \eqref{DFGRTHVBSDFRGDFGNCVBSDFGDHDFHDFNCVBDSFGSDFGDSFBDVNCXVBSDFGSDFHDFGHDFTSADFASDFSADFASDXCVZXVSDGHFDGHVBCX51}, obtaining        \begin{align}        \begin{split}        \xi_1        &        \leq        C \Vert R_0 \Vert_{B(\tau)}         (\Vert \curl (r_0 v) \Vert_{B(\tau)} + \Vert \curl (r_0 v) \Vert_{L^2} )        +        C \Vert R_0 \Vert_{L^\infty} \Vert \curl(r_0 v) \Vert_{B(\tau)}        .        \label{DFGRTHVBSDFRGDFGNCVBSDFGDHDFHDFNCVBDSFGSDFGDSFBDVNCXVBSDFGSDFHDFGHDFTSADFASDFSADFASDXCVZXVSDGHFDGHVBCX324}        \end{split}        \end{align} Since $R_0$ satisfies the homogeneous transport equation $    \partial_t R_0 + v\cdot \nabla R_0 = 0 $, the inequality \eqref{DFGRTHVBSDFRGDFGNCVBSDFGDHDFHDFNCVBDSFGSDFGDSFBDVNCXVBSDFGSDFHDFGHDFTSADFASDFSADFASDXCVZXVSDGHFDGHVBCX145} from Lemma~\ref{L08} implies        \begin{align}        \Vert R_0 (S(t)) \Vert_{B(\tau)}        \leq        \Vert R_0 (S(0)) \Vert_{B(\tau)}        +         C        t\sup_{s\in (0,t)} \Vert v(s) \Vert_{B(\tau)}         +        C t        \leq        C        +        tQ(M_{\epsilon, \kappa}(t)).        \label{DFGRTHVBSDFRGDFGNCVBSDFGDHDFHDFNCVBDSFGSDFGDSFBDVNCXVBSDFGSDFHDFGHDFTSADFASDFSADFASDXCVZXVSDGHFDGHVBCX321}        \end{align} Combining \eqref{DFGRTHVBSDFRGDFGNCVBSDFGDHDFHDFNCVBDSFGSDFGDSFBDVNCXVBSDFGSDFHDFGHDFTSADFASDFSADFASDXCVZXVSDGHFDGHVBCX51}, \eqref{DFGRTHVBSDFRGDFGNCVBSDFGDHDFHDFNCVBDSFGSDFGDSFBDVNCXVBSDFGSDFHDFGHDFTSADFASDFSADFASDXCVZXVSDGHFDGHVBCX324}, and \eqref{DFGRTHVBSDFRGDFGNCVBSDFGDHDFHDFNCVBDSFGSDFGDSFBDVNCXVBSDFGSDFHDFGHDFTSADFASDFSADFASDXCVZXVSDGHFDGHVBCX321}, we obtain 	\begin{align} 	\xi_1 	\leq 	C  	+  	tQ(M_{\epsilon, \kappa}(t)). 	\label{DFGRTHVBSDFRGDFGNCVBSDFGDHDFHDFNCVBDSFGSDFGDSFBDVNCXVBSDFGSDFHDFGHDFTSADFASDFSADFASDXCVZXVSDGHFDGHVBCX338} 	\end{align} For $\xi_2$, we first rewrite it as   \[  [\curl,R_0] r_0 v = R_1 \nabla S \times v   , \] where $R_1 = -\fractext{r_0'}{r_0}$.  Applying Lemma \ref{L03} and \eqref{DFGRTHVBSDFRGDFGNCVBSDFGDHDFHDFNCVBDSFGSDFGDSFBDVNCXVBSDFGSDFHDFGHDFTSADFASDFSADFASDXCVZXVSDGHFDGHVBCX330}, we get 	\begin{align}    \begin{split} 	\xi_2=\Vert [\curl,R_0] r_0 v\Vert_{B(\tau)}    & \leq        C \Vert R_1 \Vert_{B(\tau)}         ( \Vert \nabla S \Vert_{B(\tau)} + \Vert \nabla S \Vert_{L^2}  )        ( \Vert v \Vert_{B(\tau)} + \Vert  v\Vert_{L^2}  )\\   &\indeq        +C \Vert \nabla S \Vert_{B(\tau)}         ( \Vert R_1 \Vert_{B(\tau)} + \Vert R_1 \Vert_{L^\infty}  )        ( \Vert v \Vert_{B(\tau)} + \Vert  v\Vert_{L^2}  )\\        &\indeq        +C \Vert v \Vert_{B(\tau)}         ( \Vert R_1 \Vert_{B(\tau)} + \Vert R_1 \Vert_{L^\infty}  )        ( \Vert \nabla S \Vert_{B(\tau)} + \Vert \nabla S \Vert_{L^2}  )     .    \end{split}    \llabel{jpU LG 6 mkq qGO qtb vM OO3t XZlH fb2xBr Xk K XCf bTj xRv Db KlDr Z1aD F352Sg 5A a v3l zOX Mja r4 a7qG oYXu 3m60XP lQ B E1P UxO ivi Bc glpt UEyz 5EDs7e DQ p xJv DhG 8DJ zv ScH9 gTnX UuUftb xS m Ttd b33 oeQ 6a iHUn cNe4 HecijP Ld r nbG zeA rry pC 5pWB qjIs jCp4x2 5H W Zwo 8j2 MNb dX 2Cfv qMUt l0eT39 Lz G oZG UXc C1c xR KQmr XFI2 A5XLOp xH g NP1 vXn Z35 YT T3VG gG0Y 6loTkg X0 L jtD CaB SFg BK Qvpo DC9H obyECQ km x BrJ yBu Z3W 39 tu4D ea8x sYqE4O xT 1 5FI kmY w0K TT YcNj y8Vv mxWdyr 6u N TV1 3Or OhZ wA eZav nDhc EmBSwe Wl J ps7 N31 mM2 jX y92S dSas 7K9c3M Bc W yUS HdI XC9 V5 oGtz aEY8 nMBJ5D NZ H wBD OyW amR PV YnoV 2qlh DcqpB8 tH m M0e 4ja 5pt jD RAUa AgTE bYwBr4 ra p jLt CHQ ORU 6i 4GrY kFzx 91EBrq RG h mlS u3R Kzh 48 9F6L 6b4b xcdxYh lb k 26D 4DP Djm cj 3Gm6 MqaX GeNphB OS V Yr0 mOQ s30 RQ c5wE ge7a eVYf4g vJ B A2e Woh 292 cH SKXV NgRN 8T0oSN D8 7 zYU 63e pzr MI O2au IWoZ e6SOII LT p n4d MOk 7Xt fv enMl UF0r 5GvYIa uS A ak2 0YF PB9 jO xRu8 zz3D q32eq6 iU e KOX gJ0 zCW 18 zX2n 0RhK CtqtZJ Jx 2 v2P olu 2Sx jg lz9a eKoC 1db0On Ka o uVp Cm4 lvK 41 H5Ge v2sr uVbkFY Dj m 7Mw fiT 1bl qc 93lF jHNw UtfPPN vy G dZO kim HRH TN NNzk w71l IGsYTS WR 3 ulX vok Sf1 KE 3N9Q LhfH 7OEAEL Tr e 0BX Lhe x4q yp TkhH sZBV nrgnHa tA i EJv 6Mz 3X5 fP 5iiB I7m9 q5VzUO Do 9 YgA 9DFGRTHVBSDFRGDFGNCVBSDFGDHDFHDFNCVBDSFGSDFGDSFBDVNCXVBSDFGSDFHDFGHDFTSADFASDFSADFASDXCVZXVSDGHFDGHVBCX177}        	\end{align} To bound the right hand side, it suffices to estimate $ \Vert R_1 \Vert_{B(\tau)} $, $ \Vert \nabla S \Vert_{B(\tau)} $, and $\Vert v \Vert_{B(\tau)} $, as the rest are  bounded by $C$  (cf. Remark \ref{R04}).  For $ \Vert R_1 \Vert_{B(\tau)}$, since $R_1=-\fractext{r_0'}{r_0}$ depends only on the entropy $S$, it satisfies the homogeneous transport equation $\partial_t R_1 + v\cdot \nabla R_1= 0$ and thus by  Lemma~\ref{L08},         \begin{align}        \Vert R_1 (S(t)) \Vert_{B(\tau)}        \leq        \Vert R_1 (S(0)) \Vert_{B(\tau)}        +        C t        +         C        t\sup_{s\in (0,t)} \Vert v(s) \Vert_{B(\tau)}         \leq        C        +        tQ(M_{\epsilon, \kappa}(t)).    \llabel{uF rhG KO YZH4 NNng GYwPYw LP r tmj 63d kYx 01 d8zw broo n3LtCq pS e eQH TdV yRd z2 GIxy jxKP pLa43Q qk z 9bJ VgN uiv cK NtiU 4ujW unorFD Ic K zhi cQZ NU6 Pz 2uuR dQuK Omeh3m gf f nVs f6M JMj BM sQhP H3o3 2OSx2M DE Z MH5 Hnp WJi bL 5OFk DOUw iRxT3v 06 s 2BY w4I JjX 1z Nhb7 vTWS Rff0ks ZR B BOD c5b Cv1 kW 9Imj wmTR JwtKdl nq P n05 tyg ZWv iI Egd8 s1Mt Aoe0mE jE I z3E afd FHN RW D785 FlER fwQnJK Gn j mzf VU2 4Ng tQ rxrO BtTK wBpqF7 DJ J JAJ CF8 1h2 eR hZvu StrU 9VrPvj Dz s cPP w3r RiP iJ gkT7 MfVP TQHq12 Jw v Twx 4cN eEe gD TNyQ WJX1 0wYClR DL G m1x e7j cIA SC 2Jgx 3cfR zO9Ig6 Mx c dNR wZL 0ny w5 ufeP nsNP 09sz5Z IW r MQ3 uh0 NiF 92 ae0v XnX0 5RreQz CA 5 jvD dHX nSH mk TwKQ aSFK KKQC5J eH M TGv F4p eEM Rx WeVY XcqN HhqFBv Bu E XPx Cf8 0MN sv 3RRJ VhcM yUPRU8 9o E 0z7 TCg QFf 4W Pet8 YMqf aL8Bmy uo D rCS c5y Z5D Jj 9icv gYu0 SwkO9e xA q 2B2 SIi 96G D0 UY9M 6YcT lOvAje T1 A VE9 R01 xsV SO KOLx XtRz Lb4LO1 iX m yBm OTs iN6 yY d9By 2Jhh Zs7aek 4d s jUZ gAN w1V 9Q 3EQH il8o z1jQVr 4A b JQe 7CP nRs um j5y1 JvKb 3ZeVPO x3 9 tp3 IaP yxe zp 10fA v5r5 0op2rV lf s C8S zD4 oXa OP ZdzK v8bQ rCMmkO r6 b XwV exQ ZPK XA yF29 DYhb tcMcnW Nv l Sqr TKy 2EZ 76 xdl2 I7t2 UznuGa cB u rSE NCW 7zN ZH 7j5Y y6Dq UUG300 pP F drv eXd hom N4 iura J8k2 FbP40Z xt 9 u69 P5w CLb gZ 6kTGDFGRTHVBSDFRGDFGNCVBSDFGDHDFHDFNCVBDSFGSDFGDSFBDVNCXVBSDFGSDFHDFGHDFTSADFASDFSADFASDXCVZXVSDGHFDGHVBCX178}        \end{align} For $ \Vert \nabla S \Vert_{B(\tau)}$, by \eqref{DFGRTHVBSDFRGDFGNCVBSDFGDHDFHDFNCVBDSFGSDFGDSFBDVNCXVBSDFGSDFHDFGHDFTSADFASDFSADFASDXCVZXVSDGHFDGHVBCX327} and Lemma \ref{L02}, we obtain  	\begin{align} 	\Vert \nabla S \Vert_{B(\tau)}  	\leq  	C \Vert  S \Vert_{A(\tau)}  	+      tQ(M_{\epsilon, \kappa}(t)).     \llabel{ SR8T hvau6f Xj u SS6 ERp Rdk Xh a4tn 5qaD KBVosb Hi 4 mzu WFP FWe Nz pRqE SErU DkhmdF 26 M ZJU 6AO Dd2 ao Iazs QyxN wMW0x7 4T 0 cSd SnC Txa SD ymsI tYD6 PfKil2 Dp E 7Oc 8jd 2hf 9U 345G P8r8 8CjSMa 0r q r0a DPg 0Fc lz FatV BeiW Fnnul8 x6 M ekg gEH NtG wo Lbsv lctw jdtVt5 DR 8 YU1 hMV Gst vf AzJ5 QTGI 1sBM6C T8 P JZ2 i1o ylP tO OxN1 AOAE vni73h ol w UaO OVC nKs U2 VP4F Jwi7 ZjBSlb fz c pPu bts JSf u1 8LQj coy2 nmBvKL Sw T 0uq TVU lJY 1k 9kIX tCl4 6vGilS Eg I kdy 9ku 8Ty n5 C1le 0Fg3 khMNu7 AS q Rth gPQ Iv8 ql ADHU Jj7e Gzm6T5 8H c kDk ZB9 5DF Kr Ux95 ZWuF aoxKky HU c kpK g0h j57 6z okEX uGZO KIqhFu Z3 6 Rzl x8J VFM Op exwp QB1J HH18HU 1z v hqz ykp bXr hu nclg W8x6 lvAe7j Z6 k 5Oz cXh 0G4 XD PAoF IC3c 51hWt1 T3 u 2Kg 9lk 0aM MK iCIK oXN1 X7S6fp Yb m 8vr 4GW b23 41 o5oX TXbo c0HigO Ge P 7Gr Dju cYc u5 dKxR uKqe LRj4Jv v2 A Txo HOH Yx9 Hl uGjp NXLJ gPCNna Ob g 0zv 8TF hT4 LF 1VPP 1gF4 EDNXBT IA Z Trk Hu7 jRE L5 9J8x 4Q1L CCgsRq jO y 953 1BZ rk8 Lj enbj nZxz XbyhTj qS W Bnc nmW 6RC 3x jhxg clBi yrZBSN pe y sqV Y8d Ukg 1E CGXR wYhM kWlx0Y MC M uLz 4kP nCc Bz WAGB gkS1 AkqJys tP T HYq 0jn Auj 2i zfOt 4c7T 2XyQw8 LD B H1S 4Oq MrT PQ evzq ImDi 1lu7Aq zn M T06 GdW G18 Yk mZMA prJ8 0gqCMS oS a SK1 040 1RJ 3a Aaiv MlE5 hnLj3u mY P L9Z DUl 2K4 8o EP5O fcQp aGuZxi ODFGRTHVBSDFRGDFGNCVBSDFGDHDFHDFNCVBDSFGSDFGDSFBDVNCXVBSDFGSDFHDFGHDFTSADFASDFSADFASDXCVZXVSDGHFDGHVBCX179} 	\end{align} For $\Vert v \Vert_{B(\tau)}$, by the norm relation we have  \begin{align}        \begin{split}        \Vert v \Vert_{B(\tau)}        &        \leq        \sum_{m=1}^3 \sum_{j=0}^m \sum_{\vert \alpha \vert =j}        \Vert \partial^j (\epsilon \partial_t)^{m-j} v \Vert_{L^2}        +        \sum_{m=4}^\infty \sum_{j=0}^m \sum_{\vert \alpha \vert =j}        \Vert \partial^j (\epsilon \partial_t)^{m-j} v \Vert_{L^2} \frac{\kappa^{(j-2)_+}\tau(t)^{m-2}}{(m-2)!}\\        &        \leq        C        +        C\tau \Vert v \Vert_{A(\tau)}        \leq        C        +        \tau Q(M_{\epsilon, \kappa}(t))        .        \end{split}    \llabel{u C d4H N2g zwz 26 bd82 1FDg e0Wi9A Ur I a9i 1GJ T9b mD dxfE 68pN o8cGYv u7 U YzY oN5 CGJ T8 Oslp 4rLB ezUcDD 4d B TFj Omz d4Q rW 6O9v olY9 siRkzc MN y RL5 8pt XTE Ms MvkW FYGU xutZLG 4S L 2he tQ2 O6q G0 gfak 95iq 7fphqo lr y vnO zFQ DPh Ep L4dY I0sa kvwh4y lH P KA7 m3M YtU Sm Tx15 Qt0X 54AVmg dC s HmU rA7 DAb pN XCtJ 0P1e b0Q6Vo 2J a ff0 EOA QUf 05 10tf Wogc c6p0dV ua W tW0 uRs j2X oO J2OW mCpX baFvys WR S nkW o6m EyK 3G HLNk xe3B 5TF1qB nO p ngL qyt bym b6 8q36 Z5f6 EOm9ID EZ k Yyc jnL peL i1 xpyI wueT 2WP1xW O3 9 QiM 5YR vaH 8I 1gYr fa3c iagDow q9 T nek s4w b92 M8 lXtz 9Tel MyPkGi wP O vlk 4wS ZZN Vw vJCD 1vhy 3wL4tZ 7j W EJf yJL XLl mE ZjCL uhUh E8Dzne vG x Qse 4yP Smj 8m ZWQ3 9zhn VKLXer ve M 9vX S8N oyy Is flj0 xzdc 4tGEME 5e M YGE ERC 5uO p5 VbWS eXJs Glmo4O KC 3 kEd mZX peX z9 CdSM tecc bWFEbp r9 E Z9g gzF D5b tT rFto Nmgf G2S9TE Zv w dWL K7S S5V D9 1VYW QEkA Ey3fDC ZF x Ar3 hOX O7S um pBsZ 1fDp n5wGzW 9Q g oFH sLs L0Q Gw 5m0J S7L0 T9IdJ6 oF R Hfr 8V3 nT3 mB WBbw 43Jn CoAmgj gv Y o39 JNa sqE bl Q7LJ fims l7hE6s u9 i IWX lxC GQd PE v5la V3TG cQODda I2 0 rJT Srg UHC 4o CjW3 ybyp A1Tqzn oO 4 tiE 7aG kqt 0W 0Mea fqkP M7cnkv hJ Y 0BX zMb QTO 71 ZsW2 zPjK Of80k1 Sq 9 5G2 MTb vnL YE iGu0 QNGT bEFJZm 6d m EzN 3b9 wKi uw Vk5Y FNRT xRuiAi qT d oXn 15E 4VDFGRTHVBSDFRGDFGNCVBSDFGDHDFHDFNCVBDSFGSDFGDSFBDVNCXVBSDFGSDFHDFGHDFTSADFASDFSADFASDXCVZXVSDGHFDGHVBCX180}        \end{align} By combining the above estimates, we deduce that  	\begin{align} 	\xi_2  	\leq  	C + (t+\tau)  Q(M_{\epsilon, \kappa}(t))    .    \llabel{t xa otyY O9yf H7tT3V om p Vnd iZP 46i cL lIxB JCnp YsFUuM NQ b vnG 9hY ies b9 N47f 43mm Ba9OTZ tD H BTl fZW AgV Rj oEt8 R9hJ t51Ltx xk S rew mHE 8O4 a7 EM7N LdqZ bHtrVr 0w Q Eze Nqh Psy b3 Ve0O Cyna amXp8c ha E CFT H3T 5ik wA 3bSV gy7u E9coQZ T3 q OIU Dx4 2Nq gD MFNT dMvw gD1fZF PR k vHK WJf vGM Zi Dn7S Ksz4 i8Y7q6 rg p j1l NQ6 bj3 WR ZSDt Sy0X pePGkz v1 D rV2 ojU 7Fy Hd 97XY 5LvK Eu8TxF pQ z rUz mNf 9Lh 1X hK5L IJYt AROxR4 h0 g sFu Uoc RIE HV uHnU mln7 Xzc2Sn jc Q mAA mnl UoI Xh kWG9 BfRt d6GlFd Zp H HvX Gey Brn Gv jder N5Od LReh4E 5b l wZC dd0 1DS jo pD4I Goye LL5Wvr Nh b yGZ g5N 9NB MI i0t9 GBWq aEXzII Xe e 67N Kdm y5g Fl iPWO zcvx iNA7y4 cz l k4z eQK WY5 Yu hXpC YJX2 T1lRzg Rb T jms mqU eYU Lipschitz qVcF feZ5 UxOoAK jO u Nn2 Yvv PBy HI qI23 R7Lr sbkILe 0K 0 5Ad 2qB JQd gd UYw6 NuUH 3Qff5S Kj A Gxx lel eha FQ wnzF 4oUb e2vqhP cr i 2Yp wZr WQ4 bL PVBc tPeM QhFQgZ BY 2 7Ky JN3 gbD AZ fo2X 5KKZ ZPksuk nO O Sk9 M4w 8d2 6K RzXW OnOx KgMphz Ux C 4PO 35h QoD 0R msBY DSw1 51D2Ib 1Z j a6H Dty KlH qA CT8H 88ot i35EFF dW l yiD KxF wym Gd rdvn z24C RPVdcv 8m o wWa Qum LO2 lb YLIP ZjmU e9Ufkl va R ck3 Khq Khd 5o o6xQ 3SJZ lHPlpJ mW a I9G nou xhm s0 FFYO ekUe 45v50i D6 E 0b6 SKd e3z 6M PE2e 5HTH swZgbx ua E tdS pND kM9 H7 rqtV 7pPu MCjAz4 ep 3 3Di UKt e6T bQ 0zDFGRTHVBSDFRGDFGNCVBSDFGDHDFHDFNCVBDSFGSDFGDSFBDVNCXVBSDFGSDFHDFGHDFTSADFASDFSADFASDXCVZXVSDGHFDGHVBCX181} 	\end{align} Therefore, together with \eqref{DFGRTHVBSDFRGDFGNCVBSDFGDHDFHDFNCVBDSFGSDFGDSFBDVNCXVBSDFGSDFHDFGHDFTSADFASDFSADFASDXCVZXVSDGHFDGHVBCX165} and \eqref{DFGRTHVBSDFRGDFGNCVBSDFGDHDFHDFNCVBDSFGSDFGDSFBDVNCXVBSDFGSDFHDFGHDFTSADFASDFSADFASDXCVZXVSDGHFDGHVBCX338} we arrive at   \begin{align}        \Vert \curl v \Vert_{B(\tau)}        \leq        C        +         \left(t + \tau \right)Q(M_{\epsilon, \kappa}(t))        .        \label{DFGRTHVBSDFRGDFGNCVBSDFGDHDFHDFNCVBDSFGSDFGDSFBDVNCXVBSDFGSDFHDFGHDFTSADFASDFSADFASDXCVZXVSDGHFDGHVBCX356}   \end{align} \par \subsection{Energy equation for the pure time derivatives ($A_1$ norm)} \label{sec6.2} In this section, we estimate the pure time-analytic norm    \begin{align}   \begin{split}    \Vert u \Vert_{A_1(\tau)}    &    =    \sum_{m=1}^\infty \Vert (\epsilon \partial_t)^m u \Vert_{L^2} \frac{\tau(t)^{(m-3)_+}}{(m-3)!}   \end{split}    \llabel{Ko sGYf QZfHL3 2d 6 ttw ctf 6er Xg nauY jRkd JD0DS9 ky X CpD g3m IXv Hj kryb oW5i 8yYm1u bQ 6 5gW SeL Prv 3i v0o9 ioWB pGgxr9 wA j 4XY YFc CcQ yj nWGX r8bh 2z28gX q1 I C5P PzM uxL Rl G2gV dIFX c77P8S j6 b Eyt eAZ 1V3 pX gXLh M4Y9 cnuk8D Yh n 9Ls PFO JrD 9Z vr9e vJnq S8Un9B qy 6 exF Ga1 2Bv aS vhfY owVA w5NTM1 W9 K aIC 6Zs oHC Qv pGM7 eGB5 Vtp4RI GB f YiC H8s 4Tg 29 MX0B h5Xo mESjnY LZ L fEH uzu eAD ND SPbT uF57 zpdaFl yN 1 u73 A20 KAQ vM gfjg Nkez AxDz9S AV c 9wi wMQ iD6 6x HJQD JroY ipvIxi RQ j T00 6B3 H6R kI 7tbH aT52 HFO7VT THANKS, p sD2 U2y Aiy N0 Z6V9 hQWc IlJqlT Rb t lHb Sso YlJ JM icYW C7K8 YkJcvw ad u Lx1 Mbv G0p EW u3Az bhRo IkEscz nk 0 e2Q ssu O11 PR n84F QAjI uew3yJ 9E t g9O 0aD iE8 mq qvwy MePq PjuTa1 Dl 4 rAJ Oc0 39C qD wcIV aPfz K0idQF fe r zOI AED 9yb gi q71Z EGsW 46j2Ik JC I Zcn pPw LeH tF HrtX QUZH Vnp21p ot X VUV iym cdb dZ XVvv KMi2 gJ9ZhR sX R 1p9 oIC EMU 0r Ccds YrAC ZR2thZ eB F 72o bwx cCX mF LGNo dosE 5btAuY UI a pWS Faf YFE 8Y M70j kK8D bRnCZd MN F wtf FFR zej 2A awAK tsPJ 9Mtskx bZ S WbZ XG2 wjy MQ w6dH MwtO kR6eLb Yu I Zea hyf VDD X6 zZBg TSnY 3npekn XH 1 O2A 7Px NUr WT 62yp aJFM 5jp2mO nP n SE1 4cu eyJ wj dvmU Ug8Y dgO50J Wt t I2w VCS Wq2 sX Nux8 GLNh 3g573b vX n CCy UXx GMs At msFD AFuv eOQBn1 CZ N YnQ 29u HGL d3 BY3q ekCC zDFGRTHVBSDFRGDFGNCVBSDFGDHDFHDFNCVBDSFGSDFGDSFBDVNCXVBSDFGSDFHDFGHDFTSADFASDFSADFASDXCVZXVSDGHFDGHVBCX92}   \end{align} with the corresponding dissipative analytic norm   \begin{align}   \begin{split}    \Vert u \Vert_{\tilde A_1(\tau)}     =    \sum_{m=4}^\infty \Vert (\epsilon \partial_t)^m u \Vert_{L^2} \frac{(m-3)\tau(t)^{m-4}}{(m-3)!}    .   \end{split}    \llabel{gg5P5 0j 9 b1h hQy 627 jU 6niw i7yW B53cJq Ri S T4t Eyr 16Y jC 7tW3 Oos9 MSxuwn 6N v 8q9 it0 g9s e9 yUwj L3ce FhVhiL mZ V 1ij Fbe wny zi Ujge 7jTj WjA1tx NA 2 2jR aDj enZ 9z mfiF QlYw 6yXu93 81 m ppY fUT pzT lk pm7I edFg 3KIBml on V 50r z2n QXG VK D92h JEh7 iV1YF9 2A g 4Z4 0Sa nXB h4 JIMB Qoaz lDgnlF lh 3 VfA hpd Rk9 lK nkT9 bMY9 nSfF11 dw 2 gXs yK9 3lC FB TA97 P0tX 0sZJwW rk L wbW vA6 6OV G1 Yifx H1uY Xrv6pN Ct K X0h jLS Ffn Bp 7WZl VSGO fwlQJN 1N 9 7Ea Qqw 9kh Ri rFGn OvC4 96pcfs J7 J KzU Sai N0W BR HSOt VMDI aNQnTd Nl R TAM dFV 7xO DZ 9HqW 8rB7 kOaM4C RP C W3F JUE UTo Zh azPA owJb jhJobk oj L qBP DAz MyT mU dmsd jyQS Q0sBh7 Wh A 6if hjd nJ7 RF c3aL 3oWf 7ZhSUR Fa p lLj 1di zgb mq aOC5 qO1g wB6fYW RE m h6E usX KMD dR qgbs TDqW E3usJb 3B t jye L9e oTV FJ mEyZ 64CU P3wXSj P4 a h5l L8k 62v Sb GG6u CUpv KTYhsa dJ f m3s ReI sds Hc OCUP 3gHL xRmk3c mg e COR hCA w9Q aa NWqv NyM3 b2fgKN 2O 0 VgU GUB ds8 Mq 8dvq OlnH xaZIz0 fq n GIX FFw ynY p4 xUz2 MruI FaJTdt FB K MWY XZm AKw mk IJvr Bj38 7HcXfq s9 m 8jQ DCu vto qI T11m YgiL NtBjca NM v Ki5 qeK Nbl Mr m4Vp iyaf U176Nw yj x qRe 3OE V79 89 uPV6 pykd rJWcJT yN r ac0 hpr 8NF lZ R4cE O4JG qODL0K SA 3 J56 3kw 4cf KJ jc02 ZlcW QCe96g rn N j23 z0C npE T7 tKoP pAy2 faNw13 QS R 3DU psG yd2 QK fHIh 6DLo rpU4tl dL 8 pDiDFGRTHVBSDFRGDFGNCVBSDFGDHDFHDFNCVBDSFGSDFGDSFBDVNCXVBSDFGSDFHDFGHDFTSADFASDFSADFASDXCVZXVSDGHFDGHVBCX122}   \end{align}    Consider the partially linearized equation   \begin{align}   E(\partial_t \dot{u} + v \cdot \nabla \dot{u}) +\frac{1}{\epsilon} L(\partial_x ) \dot{u} = F,   \label{DFGRTHVBSDFRGDFGNCVBSDFGDHDFHDFNCVBDSFGSDFGDSFBDVNCXVBSDFGSDFHDFGHDFTSADFASDFSADFASDXCVZXVSDGHFDGHVBCX08}   \end{align} where $\dot{u} = (\dot{p}, \dot{v})$ and $E=E(S, \epsilon u)$. \par The next lemma provides a differential inequality that is used for pure time derivatives of $u$. \cole \begin{Lemma} \label{L2} For all $(\dot{u}, F)$ satisfying \eqref{DFGRTHVBSDFRGDFGNCVBSDFGDHDFHDFNCVBDSFGSDFGDSFBDVNCXVBSDFGSDFHDFGHDFTSADFASDFSADFASDXCVZXVSDGHFDGHVBCX08}, we have   \begin{align}   \frac{d}{dt} \Vert E^{1/2} \dot{u} \Vert_{L^2}    \leq    C(\Vert \dot{u} \Vert_{L^2} + \Vert F \Vert_{L^2}),   \llabel{ n8R WR8 6w feJ7 jbUY QNMCvy Tt y qmC QfS vEB UR tB1P 1tGv xRObO1 52 u jxQ 2x6 2wK 9j lGNn R9CO keXi9D O0 T Pei asn fO5 Wn EEyI XsmW 4BkMKK IY e Zpk Yra RI9 b6 JIsk Sbxc xqa35J Kp R Qgc GRY y22 4x dbvD 1m0X dbDZbQ LU p b2E kMh Qon 5S Wdrr OrK9 D9dbSX Eh s dlZ rZ2 ZhH Gd P9WT LS52 Sm9EUt s5 9 qFx X2t whN Zh qzhq 46bg rxloeb 4s S pU1 xCo yCC c1 ER7S CXuG NTCRH1 ER 5 mwl Q76 EbV 21 tUGF m0VA TYXbi9 ez V yqQ tfy 6iy 92 Z5kq 0Z8W iZhxvK 1h 6 11c Q8E hji 0C A30H q5uV TQVeif Ng i Wfv Ru8 jt8 5L rIxc UMxg 3i7u5a mT L xSc Q0G fKj WV 89OW euwb npXC0W lZ m g7r gnK 65J XV ckuK 1PTL fbxK8c OT g P4k OTV fGO ZW HpXr hRnV 3Vh81I 08 F Jx2 pNi AE1 4I S1VN 1OTP JYNcRP 6g 5 svQ bHc gg3 4t jx86 NWCV 477R2y nz I LTD 8kQ LKu EP ZFxR Ci0h 3XHCXs w9 e SbW agG HVb vf aBH1 MiFt g232mx Vh b xh9 A3F IdJ fC ViRA ccom 54uJpP ko 0 pxM uod 9Xo Pe dVGS zluu kRpg5b lS N EeC XTi 71O Zs Hnnn CY2r 1wRwQw PB H cHG 6d3 PqI C1 eLr0 xQhX bWLEz5 4L e TBw 83O X3B m6 qWpd qPqf aB5QXM LT f lrm UVX Lwi zx 4YsY DGr5 qFsHet vI d w61 f7k FAB Bs poHr tTck nv5oUa ut t yTZ vsW a5Z oV UDvP rAZg o0pm70 qq D mkb NJv QA4 im Euql bCif rfxv8m Kx G 13O yYE YSO TV QDKm yMoN gbNeDV 9n K juB B1i GOn pH J70l McnB Zh8Pxx Lv U bVz mrO uWQ li TMAx SnZc WXDyNc 0x 6 IOO xhU Ckg Ci TH1r UbPi OLmF84 lF u FLV XDd dgq Rc O3DFGRTHVBSDFRGDFGNCVBSDFGDHDFHDFNCVBDSFGSDFGDSFBDVNCXVBSDFGSDFHDFGHDFTSADFASDFSADFASDXCVZXVSDGHFDGHVBCX12}   \end{align} for a constant $C\geq1$. \end{Lemma} \colb \par \begin{proof} We multiply the equation \eqref{DFGRTHVBSDFRGDFGNCVBSDFGDHDFHDFNCVBDSFGSDFGDSFBDVNCXVBSDFGSDFHDFGHDFTSADFASDFSADFASDXCVZXVSDGHFDGHVBCX08} by $\dot{u}$ and integrate in $\mathbb{R}^3$.  Since $L(\partial_x)$ is skew-symmetric, we have   \begin{align}    \frac{1}{\epsilon}\langle  L(\partial_x) \dot{u} , \dot{u} \rangle = 0    ,    \llabel{3J r4n9 a2KNMl 8c l uom BbP 44A XH zw5m KWPf o2Losl bc X JVP Q7v 5QR mh zVZM efgo AjEyxn 3N t ECy sUj wnj 0Z KfVN KCm8 YWb8UX ji 6 tap 8f0 18l qy gV1i dPL4 3gYdme Ld h POJ 5eC tsP dv NZku YMts pEam1w qy c diR x6O sIe Pb tLeo R4Xf eFWPWP n3 m x55 4xO WlD 1g EfQj 1zTO AthIhD 9p v iyh CcD sny Pc vwhA dksA VSgefN 4k Q VPb wZE Ugj E4 6jfY KgMp rkBhjX h2 n PkL 2Eu 2N0 DB 0OfQ pwyk 8ZEAW0 RC 4 0Fh P70 ssh NF qXZG zasQ hiiVmb Pw h XgH I1V QYz GB XmQd lpSu he0j8P FG M XKq eu6 y5Z Ob G0SA EeVx avrD7A ga w e0p raT tfu ZI qrwU gqzB SODq3c br K 0S6 QxA 4fG Ew wCll z6Fo zqfsq0 bD e 0lB q29 Nb4 jy A0PY xIW4 No0f9f gD 5 tYA Akr jIS rH seV6 Ybnt r7l9dw nx T F4X D3w pPg Hj BxuA W2jQ lFd5Rl 9H 0 mXC vBB Olv uT 6jxQ Pbv2 svdkNX mw d KJb Z9f PRo AX 4CyZ NAOa nf2miu Ah Y AjT 8D4 30z Xn RWIR Hs7t ClGeHC u4 8 zin 3r6 Am4 Jd 4b8J 8pAW 3vAtYe Kb l 0Le 0zu iwx c3 84oC 1Key jf0RHS XC 2 UH7 Amr bir Qv Oaxl Tvls sI1a9W RS a gmn eWf LEe s9 pSS9 E7J0 NSqFel UI Y 0FE mHi zVp LJ HBcN 9Gsy XfTxsl P8 Z wco r2K TBY rT RPp0 mbHa a6cdDq EP k JWC OoY jB1 cD FJj6 BLwY hzXr0E iN 4 Wcq zys Jnw o5 7PCp 6Yaq CBa7xR IP C UrB iCi 4bw gJ DmI5 Nj6C iwPMlx Jf S CIM bIs Dt3 FR 9Qcn kAuJ twJxTF cU d 2QE Zb3 nWP Tp Ll0K boxl caO3VJ W3 T RUM aDY 5BF T7 Qcot GbmK uvlGHJ IR H Oz3 yVa CEj Rj f7My tqx7 t6dDxxDFGRTHVBSDFRGDFGNCVBSDFGDHDFHDFNCVBDSFGSDFGDSFBDVNCXVBSDFGSDFHDFGHDFTSADFASDFSADFASDXCVZXVSDGHFDGHVBCX72}   \end{align} i.e., the term with $1/\epsilon$ cancels out.  Using also the Cauchy-Schwarz inequality, we get   \begin{align}    \langle E \partial_t \dot{u} , \dot{u} \rangle     \leq    C\Vert \nabla (Ev) \Vert_{L_x^\infty} \Vert \dot{u} \Vert_{L^2}^2     +     C\Vert F \Vert_{L^2} \Vert \dot{u} \Vert_{L^2}   ,   \label{DFGRTHVBSDFRGDFGNCVBSDFGDHDFHDFNCVBDSFGSDFGDSFBDVNCXVBSDFGSDFHDFGHDFTSADFASDFSADFASDXCVZXVSDGHFDGHVBCX09}   \end{align} and thus by H\"older's inequality and since $E$ is a positive definite symmetric matrix, we obtain from \eqref{DFGRTHVBSDFRGDFGNCVBSDFGDHDFHDFNCVBDSFGSDFGDSFBDVNCXVBSDFGSDFHDFGHDFTSADFASDFSADFASDXCVZXVSDGHFDGHVBCX09}   \begin{align}   \begin{split}   \frac{d}{dt} \Vert E^{1/2} \dot{u} \Vert_{L^2}^2   &   =   \frac{d}{dt} \langle E \dot{u} , \dot{u} \rangle = \langle \partial_t E \dot{u}, \dot{u} \rangle + 2 \langle E \partial_t \dot{u}, \dot{u} \rangle \\   &   \leq   \Vert \partial_t E \Vert_{L_x^\infty} \Vert \dot{u} \Vert_{L^2}^2+C\Vert \nabla (Ev) \Vert_{L_x^\infty} \Vert \dot{u} \Vert_{L^2}^2 + C\Vert F \Vert_{L^2} \Vert \dot{u} \Vert_{L^2}   .   \label{DFGRTHVBSDFRGDFGNCVBSDFGDHDFHDFNCVBDSFGSDFGDSFBDVNCXVBSDFGSDFHDFGHDFTSADFASDFSADFASDXCVZXVSDGHFDGHVBCX10}   \end{split}   \end{align} On the other hand,   \begin{align}   \frac{d}{dt} \Vert E^{1/2} \dot{u} \Vert_{L^2}^2    =    2 \Vert E^{1/2} \dot{u} \Vert_{L^2} \frac{d}{dt} \Vert E^{1/2} \dot{u} \Vert_{L^2}.    \label{DFGRTHVBSDFRGDFGNCVBSDFGDHDFHDFNCVBDSFGSDFGDSFBDVNCXVBSDFGSDFHDFGHDFTSADFASDFSADFASDXCVZXVSDGHFDGHVBCX74}   \end{align} Now, we combine \eqref{DFGRTHVBSDFRGDFGNCVBSDFGDHDFHDFNCVBDSFGSDFGDSFBDVNCXVBSDFGSDFHDFGHDFTSADFASDFSADFASDXCVZXVSDGHFDGHVBCX10}--\eqref{DFGRTHVBSDFRGDFGNCVBSDFGDHDFHDFNCVBDSFGSDFGDSFBDVNCXVBSDFGSDFHDFGHDFTSADFASDFSADFASDXCVZXVSDGHFDGHVBCX74}, and using that the low-order Sobolev norms of $\partial_t E$, $\partial_x E$, $\partial_x v$, and $E^{-1/2}$  may be estimated by $C$ (cf.~Remark~\ref{R04} and~\ref{R05}). We arrive at   \begin{align}   \frac{d}{dt} \Vert E^{1/2} \dot{u} \Vert_{L^2}   \leq   \frac{C\Vert \dot{u} \Vert_{L^2}^2}{\Vert E^{1/2} \dot{u} \Vert_{L^2}}   +   \frac{C \Vert F \Vert_{L^2} \Vert \dot{u} \Vert_{L^2}}{\Vert E^{1/2} \dot{u} \Vert_{L^2}}   \leq   C(\Vert \dot{u} \Vert_{L^2} +\Vert F \Vert_{L^2})   ,    \llabel{ Yw R 5Rq cmG TG1 Ug LQnk NCrc gc0hTE KO 5 DlN ihk G4g Ha NERi FMkt mRwNMl kj j UVR ULq T1i mw QMZr Voe0 q6JMwA Cy Q r0s JKJ kXm KD JdD5 rk3g 6QSRon XT x OJO wCB 7xi 2q PtJp 2g4h yupZ18 KB R Cvk 0Pl 4A3 Dk JC9O 1uwO 7FdxZP 23 M IHH dEY 0dy 4c bsQ0 CYkN YruoqO 48 X wVY 4fa 8YG 22 fZyd wrk0 SumBo1 ZA W Tu6 sjt Tgy ig ZRG8 e1PD 1qQMro RA K nX2 A4F yAh mP hnmj yrx8 2UUh9L 9r X qN9 mYw cwB yy 3U9H ksxc pPfRtJ TP 6 MkF NV1 1AS Qv 0sQe sPGq 0QUqs3 aE d SMI IiO kF0 H2 BW0N E0zs 4IG3Du ye V CMJ Yks DWb Id KwFZ l1xq KeyFxr na R 6kb yQD sPC UX MFos ezxI VH8fTN pr V lT7 kaj WKg 5N BiP2 mmbD IAyVG7 9E k p2w 2XV jQ4 4V 6RD9 w9CI C7QB1Z JE Q EXs eJU Nwe ZY KOUO J7Rn 3WVEZn Lb g tCe qJV D6C l3 jOXA KqEw nZPZl7 Ux x 2Ju Kxq kRO sL yqg8 O9cs oG2F62 0l 2 OqH Xv1 8rK ui j89C qf94 GDjFJp 2f 6 OO5 3Td uXX yG QTVf BWnR rXbrXY cu 9 tio 0pO 7W9 nO GtWS gjJK YfTxpY Rt B 9Fu FXO OmO 9r MmwE q9xV BGfpBL hE f Ggi b8p rW1 E4 ETHJ VIDH APwUtX 29 Y vy4 qIQ wJx Ra Qnpg Uu42 IX0RX3 oB Q dY2 TnL 50j OV AFMI XFXc XThlqY v9 9 z3Q eIT PuQ nP l50g PB1u TUZwEg Oy Z 34i WUg jAt Q2 S5tR poUl QzRZDb Sz B hJ7 R9r E2D ah WTLg VCBY KVgJsk 3U 2 gTf Vcv Oxa Jx gGVF wrpn dztNeU l8 p afI DSk CWU z3 J90J kAW9 qOY8si 1l S 6RQ gvv k98 bN LdCd S3EC S6gMG5 7o s 3UJ iAV I9J Sn tM34 05CF DzaiRW KM o 0Be r7i DFGRTHVBSDFRGDFGNCVBSDFGDHDFHDFNCVBDSFGSDFGDSFBDVNCXVBSDFGSDFHDFGHDFTSADFASDFSADFASDXCVZXVSDGHFDGHVBCX75}   \end{align} where we appealed to   \begin{align}    \Vert \dot{u} \Vert_{L^2} = \Vert E^{-1/2} E^{1/2} \dot{u}\Vert_{L^2}     \leq    C \Vert E^{-1/2} \Vert_{L^\infty} \Vert E^{1/2} \dot{u} \Vert_{L^2}    \leq    C \Vert E^{1/2} \dot{u} \Vert_{L^2}    ,    \llabel{89l kT ulk0 aFKr Tinhkt iH p gtV bvd v73 ga SrPl Zr6O tHDhNl oh R XC1 jH5 LQm NP AGma WBKd KMbdHk cL M Uvj LPZ 2LE JN Kyar HC6Q 9CZnIf aI 8 oCb OBr iA0 Hs sLuM Z7oF ijJFfG wI 1 0ev RFS sYp yQ IeiV rSym o3Dg2K ZM t 4uA U3f Fla YS FR2k 0OtB EqNJkr jK Q ts9 MgE X7r jb XvaI 6WU6 vslYzd 7F 2 WJe ZKv wDK D0 GkHu vPGN PrkE8M Ql 7 bwS Hci 1CH Oi e9qO Fhqx rWDo0d uA A Ztr Owk eUP NG jjna 0v6a 2QnTND w8 1 7K2 oMR FrW ff 0nJQ 9czP ysWid8 Ey F HVQ xgM U3M im teUD pdEk IViQ3a xO y g4N hrd g6q uY Co2A BKbu Cv5b48 6g N MaJ yh7 haW 7v zuLo Jbtg QnqX2K vj S yXZ ApG 9Ts oa acPl HwWz wyn6iD Ow W FRr 154 Fze Cr RIzu 5aF9 WDkLQ5 pa 3 EHW OcL foy Gz LU0r DQxD 1Va1S1 eP F 7ej AMx Mz4 UJ MIGF 9t6k 69KKdh MJ H zsH l2I O6Z DP afro WtMb EkJOQu RZ T e9T PRa KP6 Ao e4Qd N80v 7cs1UG eP g CLL ieT i7s cW KEgu ArFq Fx6fVH Vc Y Er4 b1B bj0 8w rynS eEW2 XNGq8I 6a s 2Hc TDW 9ok SM 5I8w dplE jHr0kq RL C Wko OHT m19 OS FSyr sZ06 fVHBkA hO q k1c Iai zDt 8T uQhG luOM CAHf0X ic p Ncc Yke MIy ZF bhRd xHf1 XYoJa7 LP v 4lE XdV DwD NT qwuy 6DR5 ZfSDTh fL K H2k tl2 JgG Qq Ws1u 7Hwq 4rFrZf 2m K Lmi eNo UzT j8 w9Br NbDY EWqZ1F i1 y hnL 5Rd 0cY Xy uPom whnD EorGly T7 t 7xH 14u 40B mL hrQV S9lG trdfbd 5S j 9ZK bIp iOV T7 E7DH Ht2l PmYC0D ad J eP8 why Naj 0C RHqO wHmH smwXzP Mi J WN7 kL9 zCN qW BpYJ uDDFGRTHVBSDFRGDFGNCVBSDFGDHDFHDFNCVBDSFGSDFGDSFBDVNCXVBSDFGSDFHDFGHDFTSADFASDFSADFASDXCVZXVSDGHFDGHVBCX160}   \end{align} and the lemma is proven. \end{proof} \par Using the previous lemma, the next statement provides a pure time derivative analytic estimate for the solution $u$ in the $A_1(\tau)$ norm.  \par \cole \begin{Lemma} \label{L09} There exist $t_0>0$ sufficiently small  depending only on $M_0$ and $\epsilon_1>0$ sufficiently small depending on $M_{\epsilon, \kappa}(T)$ such that for $t\in (0, t_0)$ and $\epsilon \in (0, \epsilon_1)$, we have   \begin{align}   \Vert u(t) \Vert_{A_1}    \leq    C + t Q(M_{\epsilon, \kappa}(t))    ,    \label{DFGRTHVBSDFRGDFGNCVBSDFGDHDFHDFNCVBDSFGSDFGDSFBDVNCXVBSDFGSDFHDFGHDFTSADFASDFSADFASDXCVZXVSDGHFDGHVBCX76}   \end{align} for a function~$Q$. \end{Lemma} \colb \par \begin{proof} For $m\in {\mathbb N}$, we apply  $(\epsilon \partial_t)^m$ to the equation \eqref{DFGRTHVBSDFRGDFGNCVBSDFGDHDFHDFNCVBDSFGSDFGDSFBDVNCXVBSDFGSDFHDFGHDFTSADFASDFSADFASDXCVZXVSDGHFDGHVBCX01}. Then  $\dot u=(\epsilon \partial_{t})^{m}u$ satisfies \eqref{DFGRTHVBSDFRGDFGNCVBSDFGDHDFHDFNCVBDSFGSDFGDSFBDVNCXVBSDFGSDFHDFGHDFTSADFASDFSADFASDXCVZXVSDGHFDGHVBCX08} with   \begin{align}   F    =  [E ,(\epsilon \partial_t)^m ] \partial_t u + [Ev, (\epsilon \partial_t )^m] \nabla u.   \label{DFGRTHVBSDFRGDFGNCVBSDFGDHDFHDFNCVBDSFGSDFGDSFBDVNCXVBSDFGSDFHDFGHDFTSADFASDFSADFASDXCVZXVSDGHFDGHVBCX104}  \end{align} Denote    \begin{align}   \Vert u \Vert_{A_E}    =    \sum_{m=1}^\infty \Vert E^{1/2} (\epsilon \partial_t)^m u \Vert_{L^2} \frac{\tau^{(m-3)_+}}{(m-3)!}    \label{DFGRTHVBSDFRGDFGNCVBSDFGDHDFHDFNCVBDSFGSDFGDSFBDVNCXVBSDFGSDFHDFGHDFTSADFASDFSADFASDXCVZXVSDGHFDGHVBCX77}   \end{align} with the corresponding dissipative norm   \begin{align}   \Vert u \Vert_{\tilde{A}_E}    =    \sum_{m=4}^\infty \Vert E^{1/2} (\epsilon \partial_t)^m u \Vert_{L^2} \frac{(m-3) \tau^{m-4}}{(m-3)!}    .    \label{DFGRTHVBSDFRGDFGNCVBSDFGDHDFHDFNCVBDSFGSDFGDSFBDVNCXVBSDFGSDFHDFGHDFTSADFASDFSADFASDXCVZXVSDGHFDGHVBCX78}   \end{align} By Lemma~\ref{L2} and using the notation \eqref{DFGRTHVBSDFRGDFGNCVBSDFGDHDFHDFNCVBDSFGSDFGDSFBDVNCXVBSDFGSDFHDFGHDFTSADFASDFSADFASDXCVZXVSDGHFDGHVBCX77}--\eqref{DFGRTHVBSDFRGDFGNCVBSDFGDHDFHDFNCVBDSFGSDFGDSFBDVNCXVBSDFGSDFHDFGHDFTSADFASDFSADFASDXCVZXVSDGHFDGHVBCX78}, we obtain   \begin{align}   \begin{split}   \frac{d}{dt} \Vert u \Vert_{A_E}   &   =    \dot{\tau} \Vert u \Vert_{\tilde{A}_E}    +    \sum_{m=1}^\infty \frac{\tau^{(m-3)_+}}{(m-3)!}\frac{d}{dt}  \Vert E^{1/2}  (\epsilon \partial_t)^m u \Vert_{L^2}\\    &    \leq    \dot{\tau} \Vert u \Vert_{\tilde{A}_E}    +   C \Vert u \Vert_{A_1}   +   C \sum_{m=1}^\infty \frac{\tau^{(m-3)_+}}{(m-3)!} \Vert F \Vert_{L^2},   \end{split}   \label{DFGRTHVBSDFRGDFGNCVBSDFGDHDFHDFNCVBDSFGSDFGDSFBDVNCXVBSDFGSDFHDFGHDFTSADFASDFSADFASDXCVZXVSDGHFDGHVBCX29}   \end{align} where $F$ is given in \eqref{DFGRTHVBSDFRGDFGNCVBSDFGDHDFHDFNCVBDSFGSDFGDSFBDVNCXVBSDFGSDFHDFGHDFTSADFASDFSADFASDXCVZXVSDGHFDGHVBCX104}. Note that   \begin{align}   \begin{split}   \Vert F \Vert_{L^2}    &   \leq   \sum_{j=1}^m \binom{m}{j} \Vert (\epsilon \partial_t)^{j-1} \partial_t E (\epsilon \partial_t)^{m-j+1} u \Vert_{L^2}    +   \sum_{j=1}^m \binom{m}{j} \Vert (\epsilon \partial_t)^j  (Ev) (\epsilon \partial_t)^{m-j} \nabla u \Vert_{L^2} \\   &   =   F_{1,m} + F_{2,m}.   \end{split}   \label{DFGRTHVBSDFRGDFGNCVBSDFGDHDFHDFNCVBDSFGSDFGDSFBDVNCXVBSDFGSDFHDFGHDFTSADFASDFSADFASDXCVZXVSDGHFDGHVBCX40}   \end{align} For the first sum in \eqref{DFGRTHVBSDFRGDFGNCVBSDFGDHDFHDFNCVBDSFGSDFGDSFBDVNCXVBSDFGSDFHDFGHDFTSADFASDFSADFASDXCVZXVSDGHFDGHVBCX40}, we have   \begin{align}   \begin{split}   \sum_{m=1}^\infty \frac{\tau^{(m-3)_+}}{(m-3)!} F_{1,m}    &   =   \sum_{m=1}^4 \sum_{j=1}^m   \frac{\tau^{(m-3)_+}}{(m-3)!} \binom{m}{j} \Vert (\epsilon \partial_t)^{j-1} \partial_t E (\epsilon \partial_t)^{m-j+1} u \Vert_{L^2}\\   &\indeq   +   \sum_{m=5}^\infty \sum_{j=1}^{[m/2]}    \frac{\tau^{(m-3)_+}}{(m-3)!} \binom{m}{j} \Vert (\epsilon \partial_t)^{j-1} \partial_t E (\epsilon \partial_t)^{m-j+1} u \Vert_{L^2}   \\   &\indeq   +   \sum_{m=5}^\infty \sum_{j=[m/2]+1}^m    \frac{\tau^{(m-3)_+}}{(m-3)!} \binom{m}{j} \Vert (\epsilon \partial_t)^{j-1} \partial_t E (\epsilon \partial_t)^{m-j+1} u \Vert_{L^2}   \\&   =   \mathcal{D}_1 + \mathcal{D}_2 + \mathcal{D}_3   ,   \end{split}    \label{DFGRTHVBSDFRGDFGNCVBSDFGDHDFHDFNCVBDSFGSDFGDSFBDVNCXVBSDFGSDFHDFGHDFTSADFASDFSADFASDXCVZXVSDGHFDGHVBCX82}   \end{align} where we split the sum according to the low and high values of $j$. We claim   \begin{align}   &   \mathcal{D}_1   \leq   C,   \label{DFGRTHVBSDFRGDFGNCVBSDFGDHDFHDFNCVBDSFGSDFGDSFBDVNCXVBSDFGSDFHDFGHDFTSADFASDFSADFASDXCVZXVSDGHFDGHVBCX342}   \\   &   \mathcal{D}_2   \leq    C\Vert \partial_t E \Vert_{B(\tau)} \Vert u \Vert_{\tilde A_1(\tau)}    +   C \Vert u \Vert_{\tilde A_1(\tau)},   \label{DFGRTHVBSDFRGDFGNCVBSDFGDHDFHDFNCVBDSFGSDFGDSFBDVNCXVBSDFGSDFHDFGHDFTSADFASDFSADFASDXCVZXVSDGHFDGHVBCX87}   \\   &   \mathcal{D}_3   \leq  C \Vert \partial_t E \Vert_{B(\tau)} \Vert u \Vert_{A(\tau)}.   \label{DFGRTHVBSDFRGDFGNCVBSDFGDHDFHDFNCVBDSFGSDFGDSFBDVNCXVBSDFGSDFHDFGHDFTSADFASDFSADFASDXCVZXVSDGHFDGHVBCX88}   \end{align}    Proof of \eqref{DFGRTHVBSDFRGDFGNCVBSDFGDHDFHDFNCVBDSFGSDFGDSFBDVNCXVBSDFGSDFHDFGHDFTSADFASDFSADFASDXCVZXVSDGHFDGHVBCX342}: Using H\"older's and the Sobolev inequalities, we may estimate $\mathcal{D}_1$ using low order mixed space time derivative of $u$ and $S$, and from Remark~\ref{R04}, we obtain~\eqref{DFGRTHVBSDFRGDFGNCVBSDFGDHDFHDFNCVBDSFGSDFGDSFBDVNCXVBSDFGSDFHDFGHDFTSADFASDFSADFASDXCVZXVSDGHFDGHVBCX342}.    Proof of \eqref{DFGRTHVBSDFRGDFGNCVBSDFGDHDFHDFNCVBDSFGSDFGDSFBDVNCXVBSDFGSDFHDFGHDFTSADFASDFSADFASDXCVZXVSDGHFDGHVBCX87}: Using the  approach as in the estimate for $S$, we have   \begin{align}   \begin{split}    \mathcal{D}_1    &    =      \sum_{m=5}^\infty\sum_{j=1}^{[m/2]} \frac{\tau^{m-3}}{(m-3)!}  \binom{m}{j} \Vert (\epsilon \partial_t)^{j-1} \partial_t E (\epsilon \partial_t)^{m-j+1} u \Vert_{L^2}\\    &    \leq     C\tau^a \sum_{m=5}^\infty \sum_{j=1}^{[m/2]} \left(\frac{\tau^{(j-3)_+}}{(j-3)!}\Vert (\epsilon \partial_t)^{j-1} \partial_t E \Vert_{L^2} \right)^{1/4} \left(\frac{\tau^{(j-1)_+}}{(j-1)!} \Vert D^2 (\epsilon \partial_t)^{j-1} \partial_t E \Vert_{L^2}\right)^{3/4}\\    &\indeq\indeq    \times   \left(\frac{(m-j-2)\tau^{m-j-3}}{(m-j-2)!} \Vert (\epsilon \partial_t)^{m-j+1} u \Vert_{L^2}\right) \mathcal{A}_{j,m}   ,   \end{split}   \label{DFGRTHVBSDFRGDFGNCVBSDFGDHDFHDFNCVBDSFGSDFGDSFBDVNCXVBSDFGSDFHDFGHDFTSADFASDFSADFASDXCVZXVSDGHFDGHVBCX11}   \end{align}  where   \begin{align}   \begin{split}   \mathcal{A}_{j,m}    &   =    \frac{m!}{j!(m-j)! (m-3)!}\frac{(j-3)!^{1/4} (j-1)!^{3/4} (m-j-2)!}{m-j-2}    \leq    \frac{Cm^3}{(m-j)^3}    \leq    C,    \end{split}    \llabel{CJ lkrjej qn J L1j aP9 zhT dx itBH NTPI 4K12QZ gm O D38 mXG MgC Sw U4Oq KZPE QnlJW5 q8 G o4T tIF sNc En H3eV s5Gk YrBZYT 6w l Z1o Oao QDq wL t0VD D13N pv1dp9 Qi C zPm FHJ bGg nn zRF7 IIjz FaHhLN ii E cRD ANl ckV kz dDnQ NXeA LdIzOA Iv e FuJ iYo Z2s 1p QNvc SUTY LZ24K2 w7 M ZSG dw4 OSQ RB XMqJ ggDQ wxzmkn 7u r hkP Ei9 KhG 29 F1R3 IJWW bpWjIv Xq w 9Mx yZE kXt 8h Gcj3 uszI hQJDD5 me O 41s iUm JFs oc Y8wY Hc0l OxuFJI tV 3 wF1 YbZ Zmu XK oRmR 9O1D 7QOOQ0 wR f ytN 6Em pn7 cm Rc0h vnR2 qDzWEt JR v CIu TRU m7v oZ eWsb LaXF i4WTmZ Eo r xN3 arv wfl Ba FSr2 Lgr5 ULf0BV f4 F 4lI iYV AKv 2t 57Ub wPnU rZSBSN Ok e 3hZ 8P8 gQW cn Fdsu 1s37 Cup6PN Dd q gf0 B5w Rxc MS yFu0 Y6BJ 3kAZMf z6 l 1IY esD LZx EU J6Tm x6x5 bj3CZa uP 6 glB afH 7WC 7L lhOk vvsN jwNsD2 JJ Q s3l Ufx cms fP teRi g41V cyLsTk S9 U Lvx 697 m3N W4 bHwg mj7j PnYzFA OU P 2zi F0E 2BZ dK hKHy uaMM cQtslD wm N qi9 QmH KEM bR Hgh3 qkPD 7xQjvD cV G BbG aVC 7hg 9A yPbV HE2T n7fXHi dO Z wot 0FE qfS Ji d8OH XyFI jMki5f TM W 3wf 2au ThQ 7d un76 uC6H qeZvej Kr N tac ujx wbH y2 hjDV TrYk PwKdht Xh e YsP ZuS 8qA Sb vGGt sHPO KJZOK2 XW D qiM HLA nDR M9 njVv In6l vHJkQU RQ I SAr oDW 73K yJ hXb0 C9Ix NPbsrK Ja W HOu I0l GbT X6 t3ow HTTK PypWyM k5 I GEn 9U8 AdU H1 CNCj ygXC cJQJhC uY r Hx5 NJE 9kN 4Z aIWv zd6i te93do Vj VDFGRTHVBSDFRGDFGNCVBSDFGDHDFHDFNCVBDSFGSDFGDSFBDVNCXVBSDFGSDFHDFGHDFTSADFASDFSADFASDXCVZXVSDGHFDGHVBCX61}   \end{align} and         \begin{align}        a= m-3 -\left(\frac{j-3}{4}\right)_+ - \left(\frac{3j-3}{4}\right)_+ - (m-j-3) \geq 1,        \label{DFGRTHVBSDFRGDFGNCVBSDFGDHDFHDFNCVBDSFGSDFGDSFBDVNCXVBSDFGSDFHDFGHDFTSADFASDFSADFASDXCVZXVSDGHFDGHVBCX340}        \end{align} since $1\leq j \leq [m/2]$. By \eqref{DFGRTHVBSDFRGDFGNCVBSDFGDHDFHDFNCVBDSFGSDFGDSFBDVNCXVBSDFGSDFHDFGHDFTSADFASDFSADFASDXCVZXVSDGHFDGHVBCX11}--\eqref{DFGRTHVBSDFRGDFGNCVBSDFGDHDFHDFNCVBDSFGSDFGDSFBDVNCXVBSDFGSDFHDFGHDFTSADFASDFSADFASDXCVZXVSDGHFDGHVBCX340} and the discrete Young inequality, we obtain   \begin{align}   \begin{split}   \mathcal{D}_2   &   \leq   C\sum_{m=5}^\infty \sum_{j=1}^{[m/2]} \left(\frac{\tau^{(j-3)_+}}{(j-3)!} \Vert (\epsilon \partial_t)^{j-1} \partial_t E \Vert_{L^2}\right)\left( \frac{(m-j-2)\tau^{m-j-3}}{(m-j-2)!} \Vert (\epsilon \partial_t)^{m-j+1} u \Vert_{L^2}\right)\\   &\indeq\indeq   +   C\sum_{m=5}^\infty \sum_{j=1}^{[m/2]} \left(\frac{\tau^{(j-1)_+}}{(j-1)!} \Vert D^2 (\epsilon \partial_t)^{j-1} \partial_t E \Vert_{L^2} \right)   \left(\frac{(m-j-2) \tau^{m-j-3}}{(m-j-2)!} \Vert (\epsilon \partial_t)^{m-j+1} u \Vert_{L^2}\right)\\   &   \leq    C \left(\Vert \partial_t E \Vert_{B(\tau)} + \Vert \partial_t E \Vert_{L^2} \right)  \Vert u \Vert_{\tilde A_1(\tau)}\\   &   \leq   C\Vert \partial_t E \Vert_{B(\tau)} \Vert u \Vert_{\tilde A_1(\tau)}    +   C \Vert u \Vert_{\tilde A_1(\tau)}   .   \end{split}   \llabel{ 7Ah RIR UYj 92 cgBg 1W7h al2yZA kr c 7fc RgQ Cy3 o2 Qehj y0C3 ne33cG o8 n r9U cBw xIF Rz M5qw 2The mjz0qJ mU 3 w2T R5N nVt 5D zWMw rkBm JD9Aqu Z5 n 2ed rNa XDd C4 gL1h FweD lIMNkJ 7S 0 GsY KaA ApJ uh AqdX 7noH uYcymu Xx v kwQ rVH AF6 Om I2c4 ctLu 2hRytx pm F mM9 UVv O2K D9 jDkH 4rKw wtvCT6 8o E bQy v5C cdc vb nVWT zYeH OOaUVx 8t 7 mC4 O0K rFS ai Cy0p LYTZ u7cNg7 zT L lXE jNb PI7 ED LtLP 2Dlg 37F2jO 6f T mh2 rbI 7sN v2 bqKL 5zQC 4TIlRs 2y Q nPW sD3 pEA E0 iY68 cNYB x7WrfF Fr 6 smo Lk6 GJL GH Zex8 X0tq 7NORje af i LVL lbn WGU nD EWfQ nXd5 oksWpZ Aq i t72 5zE kcT zF k3Xy Ue2V 56ky7N jy 2 qwF tPR lwi fX FqxW XHOB TJpGaG vm l j7v IBO QFc nt yksn LByb EzH2w6 5Y g GGm Jxo yFb O9 kq2b kmfg VJrFb8 sU 0 luS HVI 7jC kL mToN 3i3r rSnZhz lU h s80 c9p cjD WD VpVU d1Sj HK4YR7 4m p qtZ IDk pdt s6 hzXl ZWvH iDEy0O OG G iqd 5LN Enp rv NgIo 2LVf zCRNXm ZW v vvs ahP d9I 3x o9nO xDpr e7bUuF Gy M TfC 1OB tSu tg AZji acRC YUdVQ8 8P 2 zS1 lyx yVa vW Bks0 wjcg vG4sjr ta T BVO 0Lt TvT CQ Za9d 3Mdk tacxsN va 4 OWQ LFV k5w TP 15ER inE0 cEQw1H CX N yG2 sOz J1x Qg sFSj kziA G4VQVw lx d MHe u9l ONk nW eLVd KuR9 76NErI nG O 01S DLW 85g AI ftRY GpVf RbYRC4 g2 k Me6 hZg MgY 1Z vvbK g1zv 9nwUdu 0B s 2ZI skb AxM Gg uCmY ncKV rj5YLU PN p 8fl KwN sHK Fq 6ys6 PeNr uoZ84F xr 5 iuH zuS Hle uDFGRTHVBSDFRGDFGNCVBSDFGDHDFHDFNCVBDSFGSDFGDSFBDVNCXVBSDFGSDFHDFGHDFTSADFASDFSADFASDXCVZXVSDGHFDGHVBCX28}   \end{align} \par Proof of \eqref{DFGRTHVBSDFRGDFGNCVBSDFGDHDFHDFNCVBDSFGSDFGDSFBDVNCXVBSDFGSDFHDFGHDFTSADFASDFSADFASDXCVZXVSDGHFDGHVBCX88}: Reversing the roles of $j$ and $m-j$ and proceeding as in the above argument, we may write   \begin{align}   \begin{split}   \mathcal{D}_3   &   \leq   C\tau^b \sum_{m=5}^\infty \sum_{j=[m/2]+1}^{m} \left(\frac{\tau^{(j-3)_+}}{(j-3)!}\Vert (\epsilon \partial_t)^{j-1} \partial_t E \Vert_{L^2} \right)    \left(\frac{\tau^{(m-j-2)_+}}{(m-j-2)!} \Vert (\epsilon \partial_t)^{m-j+1} u \Vert_{L^2}\right)^{1/4}\\   &\indeq\indeq   \times   \left(\frac{\tau^{(m-j)_+}}{(m-j)!} \Vert D^2 (\epsilon \partial_t)^{m-j+1} u \Vert_{L^2}\right)^{3/4} \mathcal{B}_{j,m},   \end{split}   \label{DFGRTHVBSDFRGDFGNCVBSDFGDHDFHDFNCVBDSFGSDFGDSFBDVNCXVBSDFGSDFHDFGHDFTSADFASDFSADFASDXCVZXVSDGHFDGHVBCX62}   \end{align}  where   \begin{align}   \begin{split}   \mathcal{B}_{j,m}   &   =   \frac{m!}{j!(m-j)!} \frac{(j-3)!(m-j-2)!^{1/4} (m-j)!^{3/4}}{(m-3)!}   \leq   \frac{Cm^3}{j^3}    \leq   C,   \end{split}   \llabel{l Jc9p kcXe skylzA gk s iU3 wRy nWy bw hr8t nMQ8 DogR2A cG i GED CKI cu4 cn L4yE KBoZ J48AO6 0g N EZX IcU 9He Ol wvkh 5leQ 81jZFS 0R u UpQ zmK RRo 2W otHV E5FU Y1IXkZ B8 u TyZ qXD dTL BV Y1eu 664w wehRFE zB p fCV DSL 6rn Pg xh42 KIuc 4VmyhP 2V I ReJ HNz Yft e7 woAH XFtc TFjC4R Hv 4 Jj2 t0g FMq Yo lk0g Jfpr 9lLAyb ZW Q oq0 AAx lfh ma qV1P Kpnc Q1wG12 VC J 2rf Ndq IFn Wl TR0x xwPB 3ZlDew pf K EaP AaA Cff cq 3u8v 8Hi9 qNCbpk Ih C WjH iyB dLX fX VGxc H35r N6nqwT iD A QGR 1oq RkO oO LscX ULOt XDRWY1 Sx r TJN ib4 oc4 et 2mk9 k60M cri5Pm vS B zlq vgx ZIr zm GuDz o1rr VrNbPJ Te a SZp 9L5 Q1A xm xjDw 9uKI bQayUx yP H jr0 AaV x06 wY 5jcn jK7Q E2Fx0U 61 u zTY OZY kvX m2 h6wG NqbM yW2Iwu 8z 8 ccz cIh iVc GD dsum 80Wb emmOHY n6 a zLR 9Ir Kqg r1 NIVv n8jZ Cn1m1D zX x PP0 ZAe ZFp q2 RgGO qnlN 7sCAbU Ud m 44v vjm 73z uX 3cyj xImc HRY6TW 0o g OSs 8wc gzw FH 3k3c QypC 0iaaFV vK h kHi 25U 1zc Tk RsQB aL11 aLH1gu n5 e 9nK Gds zOK af Q4HN fHHL 5lVHkq pI 9 zjo ddg LMw 7F DsF5 iOv8 HjsVvI nr F DDS NBM cOk M3 ZK2Y eHKO jE2Fq7 cq K gjj JQX 9ym Lk NJZA OREJ MBJvRE 5X c Hac UZl THm ij JVL5 HvGV L14sfZ 4Z q 4y4 8S4 ABC h3 H6hY 1rDT lRp44K Gv 0 Jnz tTM Wya 2l iQgS iKeu 3mjHy9 Ow Z Nn9 Mgh 99i lf wZQl 4Bnb AH5roC nX N mIk o2r Kwb rL 7nJL nMXC bXHFDr W7 i jjC PMY KwA Ja kVe1 Q4s3 kLDFGRTHVBSDFRGDFGNCVBSDFGDHDFHDFNCVBDSFGSDFGDSFBDVNCXVBSDFGSDFHDFGHDFTSADFASDFSADFASDXCVZXVSDGHFDGHVBCX63}   \end{align} and         \begin{align}        b= m-3 -(j-3)_+ - \left(\frac{m-j-2}{4}\right)_+ - \left(\frac{3m-3j}{4}\right)_+ \geq 0,        \label{DFGRTHVBSDFRGDFGNCVBSDFGDHDFHDFNCVBDSFGSDFGDSFBDVNCXVBSDFGSDFHDFGHDFTSADFASDFSADFASDXCVZXVSDGHFDGHVBCX341}        \end{align}   since $m \geq j\geq [m/2]+1$. From \eqref{DFGRTHVBSDFRGDFGNCVBSDFGDHDFHDFNCVBDSFGSDFGDSFBDVNCXVBSDFGSDFHDFGHDFTSADFASDFSADFASDXCVZXVSDGHFDGHVBCX62}--\eqref{DFGRTHVBSDFRGDFGNCVBSDFGDHDFHDFNCVBDSFGSDFGDSFBDVNCXVBSDFGSDFHDFGHDFTSADFASDFSADFASDXCVZXVSDGHFDGHVBCX341}, we obtain   \begin{align}   \mathcal{D}_3   &   \leq   C \Vert \partial_t E \Vert_{B(\tau)} \Vert u \Vert_{A(\tau)}.   \label{DFGRTHVBSDFRGDFGNCVBSDFGDHDFHDFNCVBDSFGSDFGDSFBDVNCXVBSDFGSDFHDFGHDFTSADFASDFSADFASDXCVZXVSDGHFDGHVBCX22}   \end{align} \par Analogously, the second sum in \eqref{DFGRTHVBSDFRGDFGNCVBSDFGDHDFHDFNCVBDSFGSDFGDSFBDVNCXVBSDFGSDFHDFGHDFTSADFASDFSADFASDXCVZXVSDGHFDGHVBCX40} may be separated according to low and high values of $j$, obtaining   \begin{align}   \begin{split}   \sum_{m=1}^\infty \frac{\tau^{(m-3)_+}}{(m-3)!} F_{2,m}   &   =   \sum_{m=1}^4 \sum_{j=1}^{[m/2]} \frac{\tau^{(m-3)_+}}{(m-3)!} \binom{m}{j} \Vert (\epsilon \partial_t)^j (Ev) (\epsilon \partial_t)^{m-j} \nabla u \Vert_{L^2}\\   &\indeq   +   \sum_{m=5}^\infty \sum_{j=1}^{[m/2]} \frac{\tau^{(m-3)_+}}{(m-3)!} \binom{m}{j} \Vert (\epsilon \partial_t)^j (Ev) (\epsilon \partial_t)^{m-j} \nabla u \Vert_{L^2}   \\&\indeq   +   \sum_{m=5}^\infty \sum_{j=[m/2]+1}^m \frac{\tau^{(m-3)_+}}{(m-3)!} \binom{m}{j} \Vert (\epsilon \partial_t)^j (Ev) (\epsilon \partial_t)^{m-j} \nabla u \Vert_{L^2}\\   &   =   \mathcal{D}_4   +   \mathcal{D}_5   +   \mathcal{D}_6    .   \end{split}    \label{DFGRTHVBSDFRGDFGNCVBSDFGDHDFHDFNCVBDSFGSDFGDSFBDVNCXVBSDFGSDFHDFGHDFTSADFASDFSADFASDXCVZXVSDGHFDGHVBCX345}   \end{align} We claim   \begin{align}   &   \mathcal{D}_4   \leq   C,   \label{DFGRTHVBSDFRGDFGNCVBSDFGDHDFHDFNCVBDSFGSDFGDSFBDVNCXVBSDFGSDFHDFGHDFTSADFASDFSADFASDXCVZXVSDGHFDGHVBCX343}   \\   &   \mathcal{D}_5   \leq   C \Vert Ev \Vert_{A(\tau)} \Vert \nabla u \Vert_{A_1(\tau)},   \label{DFGRTHVBSDFRGDFGNCVBSDFGDHDFHDFNCVBDSFGSDFGDSFBDVNCXVBSDFGSDFHDFGHDFTSADFASDFSADFASDXCVZXVSDGHFDGHVBCX98}   \\   &   \mathcal{D}_6   \leq   C \Vert Ev \Vert_{A(\tau)} \Vert u \Vert_{A(\tau)}   .   \label{DFGRTHVBSDFRGDFGNCVBSDFGDHDFHDFNCVBDSFGSDFGDSFBDVNCXVBSDFGSDFHDFGHDFTSADFASDFSADFASDXCVZXVSDGHFDGHVBCX99}   \end{align}    Proof of \eqref{DFGRTHVBSDFRGDFGNCVBSDFGDHDFHDFNCVBDSFGSDFGDSFBDVNCXVBSDFGSDFHDFGHDFTSADFASDFSADFASDXCVZXVSDGHFDGHVBCX343}: Proceeding as in the proof of \eqref{DFGRTHVBSDFRGDFGNCVBSDFGDHDFHDFNCVBDSFGSDFGDSFBDVNCXVBSDFGSDFHDFGHDFTSADFASDFSADFASDXCVZXVSDGHFDGHVBCX343}, we obtain that the low-order mixed space-time derivatives may be estimated by $C$.      Proof of \eqref{DFGRTHVBSDFRGDFGNCVBSDFGDHDFHDFNCVBDSFGSDFGDSFBDVNCXVBSDFGSDFHDFGHDFTSADFASDFSADFASDXCVZXVSDGHFDGHVBCX98}: As in \eqref{DFGRTHVBSDFRGDFGNCVBSDFGDHDFHDFNCVBDSFGSDFGDSFBDVNCXVBSDFGSDFHDFGHDFTSADFASDFSADFASDXCVZXVSDGHFDGHVBCX11}, we have   \begin{align}   \begin{split}   \mathcal{D}_5   &   \leq   C\sum_{m=5}^\infty \sum_{j=1}^{[m/2]} \frac{\tau^{(m-3)_+}}{(m-3)!} \binom{m}{j} \Vert (\epsilon \partial_t)^j (Ev) \Vert_{L^2}^{1/4}\Vert D^2 (\epsilon \partial_t)^j (Ev) \Vert_{L^2}^{3/4} \Vert (\epsilon \partial_t)^{m-j} \nabla u \Vert_{L^2}\\   &   \leq   C \sum_{m=5}^\infty \sum_{j=1}^{[m/2]} \left(\Vert (\epsilon \partial_t)^j (Ev) \Vert_{L^2} \frac{\tau^{(j-3)_+}}{(j-3)!}\right)^{1/4}   \left(\Vert D^2 (\epsilon \partial_t)^j (Ev) \Vert_{L^2} \frac{ \tau^{(j-1)_+}}{(j-1)!}\right)^{3/4} \\   &\indeqtimes   \left(\Vert (\epsilon \partial_t)^{m-j} \nabla u \Vert_{L^2} \frac{\tau^{(m-j-3)_+}}{(m-j-3)!}\right)\\   &   \leq   C \Vert Ev \Vert_{A(\tau)} \Vert \nabla u \Vert_{A_1(\tau)}   .   \end{split}    \llabel{tx6K r6 G FSZ 3p5 0EV qT 1lFb C8D0 SfvSSE 7x s zb5 Opu 1sK vI BTBz xYS9 DC4d2C Ha X E32 k2P hWJ pI KRUG R6jE k3Ub97 lT l geP Nh6 WL1 Hw MBoN SsGO ApcRCI VD J hqs 4Ls 5jH Z4 4q5o 2rEb US2df4 tG X VQl qsn VnU BP H7Ex opAt w4btxL Oq y pBl Sie Jik pN Zcw6 lCFv XFT1J7 ct H RLL HZi Oe7 Vh JxSR v9io nQmjmE I7 O LLf GoI DF0 91 xVSQ bR5t oemTef aX W BP6 gFx 6kY zD 3a8N 6lqD nI0ajp Hv s B0r nu3 DB0 DH kEOq jF9B 4yoXQY 8N i TJi ZEf LDh BZ sayA MjR9 eRHFEo Cj E iYm jLW KYe uZ 33pn UNCa xYxRvR Au z 2O4 FWY O6u ed uQUT r6UQ AlTZB9 RV 1 XCb 4bp ZML ii D0kW K3v5 5XG22G pg a mWl Cyc ZpJ iI 9cUn 8Sph EDdsSO PC 7 giw ndT oMn VV zP5u 0zvu bEhjtN sE k M3k sk5 ofO ef Q4kP qYCo syN7VE Qq n AEW ktk lvV XC JsNY XQ48 CECbJ5 rR B E6i 4d1 eQA is bhtZ 76mM QqYWp6 ax 4 VTH IcU yKJ 0H 3zws NhQq LgZcIo tI V agL bOo BG5 nb 24XO Iy1v 7TEc75 ez Y gtP qJx L65 MP nnWH PSlF 50NmjH ge t Uns kLG w5l Kj zZic IUlN 9c8krT Np W CO0 HOe WI1 JH rEIJ 5F8z vjo4Zg uS c sJW JS2 1tE ze L5sS mDGM p24eme AG 5 41r pdP N9M ug hrLV PcBJ EoGXKf HO t dG0 Xk7 fHh 5X FgbQ Piix KeS2Kq aS D fAL 5qC O9Q Gf NZTP 1oBJ S449nE KR r 9UP SA6 5Ou n2 xK9l yFSN JjXhjh md m HYD C0u NFl o4 klkU tHt8 GWCT6v dJ N E73 e9f lxS Ud r3a4 Be4Q 3s8RXn JY L 4Dk KUq dbC A8 o55T kB7p AgM6Ga OL X A4r SJX 7DP Ac U2ep ONh9 8LiMYK s2 x aGC DFGRTHVBSDFRGDFGNCVBSDFGDHDFHDFNCVBDSFGSDFGDSFBDVNCXVBSDFGSDFHDFGHDFTSADFASDFSADFASDXCVZXVSDGHFDGHVBCX85}   \end{align}    Proof of \eqref{DFGRTHVBSDFRGDFGNCVBSDFGDHDFHDFNCVBDSFGSDFGDSFBDVNCXVBSDFGSDFHDFGHDFTSADFASDFSADFASDXCVZXVSDGHFDGHVBCX99}: As in \eqref{DFGRTHVBSDFRGDFGNCVBSDFGDHDFHDFNCVBDSFGSDFGDSFBDVNCXVBSDFGSDFHDFGHDFTSADFASDFSADFASDXCVZXVSDGHFDGHVBCX62}, we arrive at   \begin{align}   \begin{split}   \mathcal{D}_6   &   \leq   C \sum_{m=5}^\infty \sum_{j=[m/2]+1}^m \frac{\tau^{(m-3)_+}}{(m-3)!}\binom{m}{j} \Vert (\epsilon \partial_t)^j (Ev) \Vert_{L^2} \Vert (\epsilon \partial_t)^{m-j} \nabla u \Vert_{L^2}^{1/4} \Vert_{L^2} \Vert D^2 (\epsilon \partial_t)^{m-j} \nabla u \Vert_{L^2}^{3/4} \\   &   \leq   C \sum_{m=5}^\infty \sum_{j=[m/2]+1}^m \left(\Vert (\epsilon \partial_t)^j (Ev) \Vert_{L^2} \frac{\tau^{(j-3)_+}}{(j-3)!}\right)    \left(\Vert (\epsilon \partial_t)^{m-j} \nabla u \Vert_{L^2} \frac{\tau^{(m-j-2)_+}}{(m-j-2)!}\right)^{1/4} \\   &\indeqtimes   \left(\Vert D^2(\epsilon \partial_t)^{m-j} \nabla u \Vert_{L^2} \frac{ \tau^{(m-j)_+}}{(m-j)!}\right)^{3/4}\\   &   \leq   C \Vert Ev \Vert_{A(\tau)} \Vert u \Vert_{A(\tau)}.   \end{split}    \llabel{0Zr trm O9 TOwv YgQ8 X2stz6 i8 Z Gdj EIY pMI vk 4mHo HzNO kBhY3G pQ y kPw BJf Iri Ip 8qUh CHMW tmvKzw eL r SHV vBb ezG Lw B1Eb guGo aFD8Iu 8V S vCc ywH dsb a9 BaFI CoXl dpGoL0 ma v 80D AEd u5a JJ 90XO chPj j9FDDs pu C tYh 9BB N4M qe 9PXb t2b6 Xi9zba 2z G ZqG V0f KdN x6 WCWC wA3d e5L5Rg lv y 7dQ KJ5 PXI N9 eIx5 7QP6 SOm4Kq TY S TQb mJv NJe 74 Yrf9 jLBM p6f9vG Np e oAh 9qK BWP KQ rOxz 6bD2 apVIc0 lY 4 0O5 uBT pyF 4j 1eqI 3qhD RludWD 3v S Mn3 hr8 Sb5 49 exAJ WxCH XhRrcm HQ S q8h dBr W8i sk 71oI u8Tj eFQSC5 Bz d pwB URk R0R Dz FjgE 1YNE xpbhdA iK i hOG HN0 Cjp pH bZXY sQa7 wBYKds fs O bkr ZWX R4Q IC 3ctX TDSH E9Cavw RM t otj ttW l2O 6h 2kdi XQ62 1kwNoK Nc B zvH vTH vAq u1 Giax ggeJ C4PdN7 Fc 1 5Qm ADY H8A V4 Z46d mQZk q0jwjg vm F 6x7 h2e Kz1 0W JTZ4 tRkN EjZ762 p3 M mIR oBN wcT JT TDZ5 PU9h yNjWxj rD C 7Mg sU0 eGu G4 Lwpo VDXN KuedYf 13 x qpP AcH MYK 8Z uGPT PUA6 e1LTA5 dq h XNC vTD SjN XM 7jcS DLOB GdHGo3 EO J i3g Jss P4h Q8 FC9X IJ8L cOQ2u5 Fa B jTg iWM 8KA L8 fkiD Uk49 k6SxXe Ns V k2A StX 3DO UJ 4el2 IoQS v2Sz15 aN 7 MkI N7U Uqm AT ZcTr uWSq N8FF35 vO E 0sv tmG n4B gA HFxQ NcpE J7ozAJ Vi m il5 4dj z0k Ab qaw9 3AJe FqKAv5 gS v cJG IXW Xnd Qw V4Pc dBMs xAR3K2 KQ z Pe6 fwD 2kE rT z8tZ SjGm vpoU1I 2S 9 UTr xuN ANY eJ Mbeh BWR2 YrVhd7 Fg b 5OR kJC PaI Zq T0rDFGRTHVBSDFRGDFGNCVBSDFGDHDFHDFNCVBDSFGSDFGDSFBDVNCXVBSDFGSDFHDFGHDFTSADFASDFSADFASDXCVZXVSDGHFDGHVBCX89}   \end{align}    Collecting the above estimates \eqref{DFGRTHVBSDFRGDFGNCVBSDFGDHDFHDFNCVBDSFGSDFGDSFBDVNCXVBSDFGSDFHDFGHDFTSADFASDFSADFASDXCVZXVSDGHFDGHVBCX82}--\eqref{DFGRTHVBSDFRGDFGNCVBSDFGDHDFHDFNCVBDSFGSDFGDSFBDVNCXVBSDFGSDFHDFGHDFTSADFASDFSADFASDXCVZXVSDGHFDGHVBCX88} and \eqref{DFGRTHVBSDFRGDFGNCVBSDFGDHDFHDFNCVBDSFGSDFGDSFBDVNCXVBSDFGSDFHDFGHDFTSADFASDFSADFASDXCVZXVSDGHFDGHVBCX345}--\eqref{DFGRTHVBSDFRGDFGNCVBSDFGDHDFHDFNCVBDSFGSDFGDSFBDVNCXVBSDFGSDFHDFGHDFTSADFASDFSADFASDXCVZXVSDGHFDGHVBCX99}, we obtain from \eqref{DFGRTHVBSDFRGDFGNCVBSDFGDHDFHDFNCVBDSFGSDFGDSFBDVNCXVBSDFGSDFHDFGHDFTSADFASDFSADFASDXCVZXVSDGHFDGHVBCX40},    \begin{align}   \begin{split}   &   \sum_{m=1}^\infty \frac{\tau^{(m-3)_+}}{(m-3)!} F_{1,m}   +   \sum_{m=1}^\infty \frac{\tau^{(m-3)_+}}{(m-3)!} F_{2,m}     \leq   \mathcal{D}_1   +   \mathcal{D}_2   +   \mathcal{D}_3   +   \mathcal{D}_4   +   \mathcal{D}_5   +   \mathcal{D}_6   \\&\indeq   \leq        C        +        C\Vert \partial_t E \Vert_{B(\tau)}  \Vert u \Vert_{\tilde{A}_1(\tau)}        +        C \Vert u \Vert_{\tilde{A}_1}        +        C \Vert \partial_t E \Vert_{B(\tau)} \Vert u \Vert_{A(\tau)}   \\&\indeq\indeq         +         C \Vert Ev \Vert_{A(\tau)} \Vert \nabla u \Vert_{A_1(\tau)}         +         C \Vert Ev \Vert_{A(\tau)} \Vert u \Vert_{A(\tau)},   \end{split}   \label{DFGRTHVBSDFRGDFGNCVBSDFGDHDFHDFNCVBDSFGSDFGDSFBDVNCXVBSDFGSDFHDFGHDFTSADFASDFSADFASDXCVZXVSDGHFDGHVBCX101}   \end{align} where we estimate $\Vert \partial_t E \Vert_{B(\tau)}$ using Lemma~\ref{L06}, and $\Vert Ev \Vert_{A(\tau)}$  with \eqref{DFGRTHVBSDFRGDFGNCVBSDFGDHDFHDFNCVBDSFGSDFGDSFBDVNCXVBSDFGSDFHDFGHDFTSADFASDFSADFASDXCVZXVSDGHFDGHVBCX328} and Lemma~\ref{L07}. \par In order to estimate the dissipative term $\Vert \nabla u \Vert_{A_1(\tau)}$, we recall the elliptic regularity for the div-curl system        \begin{align}        \Vert \nabla v \Vert_{L^2}        \leq        C \Vert \dive v \Vert_{L^2}        +        C \Vert \curl v \Vert_{L^2}        ,        \label{DFGRTHVBSDFRGDFGNCVBSDFGDHDFHDFNCVBDSFGSDFGDSFBDVNCXVBSDFGSDFHDFGHDFTSADFASDFSADFASDXCVZXVSDGHFDGHVBCX350}        \end{align} which, together with the definition of the $A_1(\tau)$ norm, leads to        \begin{align}        \Vert \nabla u \Vert_{A_1(\tau)}        \leq        C \Vert L(\partial_x) u \Vert_{A_1(\tau)}        +        C \Vert \curl v \Vert_{A_1(\tau)}.        \label{DFGRTHVBSDFRGDFGNCVBSDFGDHDFHDFNCVBDSFGSDFGDSFBDVNCXVBSDFGSDFHDFGHDFTSADFASDFSADFASDXCVZXVSDGHFDGHVBCX351}        \end{align} To treat the divergence part of the dissipative term, we rewrite the equation \eqref{DFGRTHVBSDFRGDFGNCVBSDFGDHDFHDFNCVBDSFGSDFGDSFBDVNCXVBSDFGSDFHDFGHDFTSADFASDFSADFASDXCVZXVSDGHFDGHVBCX01} as   \begin{align}   L(\partial_x) u = -E(S, \epsilon u)( \epsilon \partial_t u + \epsilon v\cdot \nabla u).   \label{DFGRTHVBSDFRGDFGNCVBSDFGDHDFHDFNCVBDSFGSDFGDSFBDVNCXVBSDFGSDFHDFGHDFTSADFASDFSADFASDXCVZXVSDGHFDGHVBCX41}   \end{align} For $m\in \mathbb{N}$, we apply $(\epsilon \partial_t)^m$ to the equation~\eqref{DFGRTHVBSDFRGDFGNCVBSDFGDHDFHDFNCVBDSFGSDFGDSFBDVNCXVBSDFGSDFHDFGHDFTSADFASDFSADFASDXCVZXVSDGHFDGHVBCX41}, obtaining   \begin{align}   \begin{split}   \Vert (\epsilon \partial_t)^m L(\partial_x) u \Vert_{L^2}   &   \leq   C\sum_{j=0}^m \binom{m}{j} \Vert (\epsilon \partial_t)^j E (\epsilon\partial_t)^{m-j+1} u \Vert_{L^2}    +   C\epsilon \Vert Ev \Vert_{L^\infty} \Vert (\epsilon \partial_t)^m \nabla u \Vert_{L^2} \\   &\indeq   +   C\epsilon \sum_{j=1}^m \Vert (\epsilon \partial_t)^j (Ev) (\epsilon \partial_t)^{m-j} \nabla u \Vert_{L^2}   .   \end{split}   \llabel{C 7uHz BBJmte 0N u Ccx SRK Oa3 tH x0qd FbV4 fvCE0e ox Q PxY HrY keI 7J z2tx 2O10 v7Jyht 7Q S guu qOS 2nJ 8V kDNE av2p j6rfex zl W RDZ Q9M 6TS 1p 8sDD Jug0 zLrm00 Hw S PIG fIU OG2 hg LO6p bxkw ShuRYO eq h Bgt 9iu aEP YJ 2V5S LDZS yqYjlF cX N frG fHj IAx Gr VA5C iWhV EIq50c B3 l PmN Cf0 W4U nR dqhe Lc2o 7UbYJK A7 t CYh HQL 0Kc BR JZw9 ApfT QpnCYO GR S H4T Veu 5Rp 9o p5Y6 2O65 VgNakY MH o zQ9 8YE xil Qs cn7X aXvi EALeFt yb g iei OCU Ae8 F5 n11j 8Wzq RoHaKp G8 t NWA b7G tIE U5 LguO Npy3 H36FNH Ys 3 yuB h2H YAn L1 eGrW wSyY Hm0XM9 Om y Vrb r21 m8W p4 gVk4 Btcf vzKkcV 3t s lwA pzJ oZV 6P zOE7 NT3W TICa8G FG H UC4 zj9 wCW OT jG9b bFsA JDkJPz zU S puK 9Uf dfC NY aIRy EwPf il1wrH bl Z Xo8 0DT N5K Sf XTbJ 7Slj LvcdWQ sJ j rTf o2g NBb SG OAqV JGm2 PWlot3 5f 5 jcI 5QB c4z 9c W8xT EzvO ItH3oZ xN W LOG fN0 iZq 78 D5f9 XbFL 9M0icL Ud d 02D t2l M0D Bj 7aKk IUtv 9zVYv3 qk J d0l cUk 3I7 n5 Bc4v YnmK 4gIDpY di e cxX g2Q wRR B7 xdkN JWjM XR5oJ5 kc J Okb iBM oo0 UP UEh5 5var uPvVin uL j 3Uc aoQ 1Cr Kf 2Xve 4wOk 644Ix3 29 I eZ7 SMl twO MO YFJF aDYC rotcFl wx M gKO aJt Lki RH CrKk 799v aRmHUX e3 Y ByJ 1B1 nvP TX FGnk Ue4B HzBYbH 7P 8 xYM qXw rwj 8n QmR3 myQM l9qzDT 5m z ZeI I6b Q8S pa 7nFQ LcCt e5TC9A yi 2 qlq x2n eSJ u9 LlzD UiHH CMo5uf zy k CVn uv6 3Ik vB Xhw8 f0Qe 6vJ4ki DFGRTHVBSDFRGDFGNCVBSDFGDHDFHDFNCVBDSFGSDFGDSFBDVNCXVBSDFGSDFHDFGHDFTSADFASDFSADFASDXCVZXVSDGHFDGHVBCX42}   \end{align} From here we arrive at   \begin{align}   \begin{split}   \Vert L(\partial_x) u \Vert_{A_1(\tau)}   &   =   \sum_{m=1}^\infty \Vert (\epsilon \partial_t)^m L(\partial_x) u \Vert_{L^2} \frac{\tau^{(m-3)_+}}{(m-3)!}\\   &   \leq   \sum_{m=1}^2 \Vert E (\epsilon \partial_t)^{m+1} u \Vert_{L^2} \frac{\tau^{(m-3)_+}}{(m-3)!}   +   \sum_{m=3}^\infty \Vert E \Vert_{L^\infty} \Vert (\epsilon \partial_t)^{m+1} u \Vert_{L^2} \frac{\tau^{m-3}}{(m-3)!}\\   &\indeq   +   C\epsilon\sum_{m=1}^\infty \sum_{j=1}^m \binom{m}{j} \Vert (\epsilon \partial_t)^{j-1} \partial_t E (\epsilon \partial_t)^{m-j+1} u \Vert_{L^2}\frac{\tau^{(m-3)_+}}{(m-3)!}    \\&\indeq   +   C\epsilon \sum_{m=1}^\infty \Vert (\epsilon \partial_t)^m \nabla v \Vert_{L^2} \frac{\tau^{(m-3)_+}}{(m-3)!}   \\   &\indeq   +   C \epsilon \sum_{m=1}^\infty \sum_{j=1}^m \binom{m}{j} \Vert (\epsilon \partial_t)^j (Ev) (\epsilon\partial_t)^{m-j} \nabla u \Vert_{L^2} \frac{\tau^{(m-3)_+}}{(m-3)!}   \\&   \leq   C+   C\Vert E \Vert_{L^\infty} \Vert u \Vert_{\tilde A_1}   +   C\epsilon \Vert \nabla v \Vert_{A_1(\tau)}   +   C\epsilon\sum_{m=1}^\infty \frac{\tau^{(m-3)_+}}{(m-3)!} F_{1,m}   +   C\epsilon \sum_{m=1}^\infty \frac{\tau^{(m-3)_+}}{(m-3)!}  F_{2,m}   ,   \end{split}   \label{DFGRTHVBSDFRGDFGNCVBSDFGDHDFHDFNCVBDSFGSDFGDSFBDVNCXVBSDFGSDFHDFGHDFTSADFASDFSADFASDXCVZXVSDGHFDGHVBCX46}   \end{align} where we used the notation from \eqref{DFGRTHVBSDFRGDFGNCVBSDFGDHDFHDFNCVBDSFGSDFGDSFBDVNCXVBSDFGSDFHDFGHDFTSADFASDFSADFASDXCVZXVSDGHFDGHVBCX40} in the last inequality.  The third term on the far right side of \eqref{DFGRTHVBSDFRGDFGNCVBSDFGDHDFHDFNCVBDSFGSDFGDSFBDVNCXVBSDFGSDFHDFGHDFTSADFASDFSADFASDXCVZXVSDGHFDGHVBCX46} may be absorbed  in the left side of \eqref{DFGRTHVBSDFRGDFGNCVBSDFGDHDFHDFNCVBDSFGSDFGDSFBDVNCXVBSDFGSDFHDFGHDFTSADFASDFSADFASDXCVZXVSDGHFDGHVBCX351} when $\epsilon$ is sufficiently small,  and the fourth and fifth terms may be absorbed into the left side of \eqref{DFGRTHVBSDFRGDFGNCVBSDFGDHDFHDFNCVBDSFGSDFGDSFBDVNCXVBSDFGSDFHDFGHDFTSADFASDFSADFASDXCVZXVSDGHFDGHVBCX101}  when $\epsilon$ is sufficiently small depending on $M_{\epsilon, \kappa}(T)$. \par To treat the curl part of the dissipative term, using the similar technique for the curl estimate above, we have        \begin{align}        \begin{split}        &        \Vert (\epsilon \partial_t)^m \curl v \Vert_{L^2}        =        \Vert (\epsilon \partial_t)^m \curl (R_0 r_0 v )\Vert_{L^2}       \\&\indeq        \leq        \sum_{j=0}^m \binom{m}{j} \Vert (\epsilon \partial_t)^j R_0 (\epsilon \partial_t)^{m-j} \curl (r_0 v) \Vert_{L^2}        +     \sum_{j=0}^m \binom{m}{j} \Vert (\epsilon \partial_t)^j \nabla R_0 (\epsilon \partial_t)^{m-j} r_0 v \Vert_{L^2}.        \end{split}    \llabel{Q2 T ufK LrB aXD Zj c3HK C0JR d8rDc4 uJ F 5Sv owJ sqR 14 T2ug Lt1j w3NNGZ ay h kJJ KtU lrN N1 FpTM tKn2 6Rwqof YU G RvJ Opu p6W 1d A2f5 BDmo 8RGRjv zD G VBH A0F yT2 FO wqAE nyY5 IaT7um RQ n JJ5 YmU 9O5 Xa uxdj 5yBQ LNcL4l qo u 9wA v9K l1R Gl AtyI 1yZu ubgDSx QC P kIE ern ynD jU fRWv V2Ze IXtbPW bv t BHs 5a8 aHR p7 1Yyu N5LJ tAFz0g Mw 5 z9q tTa pmC O1 IRGW wIzD JCz8M9 8v M 1fG lfo fyN p0 FBpU gMMg hQL9A6 sq c Zuj N1G 6ZF dN vx3h NMR1 TDOVdF cK 6 jV4 TNL SPh XB Dsgg sSEO ycDfQ5 1Y 1 kJB lVC W2T Ah v3Ar EPMs Ecs5Ep o9 Z zzn Ua8 fQB Ut Welx DaXC zrusIP SY h j9E QgR Fxe Zd nNIz l806 znWOwE WF t RVZ u67 3kZ eR FOUT vTHG 5rMZLS 97 J fWL kDn b2H NA HNYv V8uY vhS5ji ta a K8v UTR 9RS UO jOdz lsss gYfrUL VH G xNk Ilm HCS Xh NgFX mZo6 6QMgrd tF t Px5 GW0 skn rP 4n0Z fDWy A0I29F R9 c iVJ 3zQ DIB Sf BHPd l6hF yCZ5la w0 H P37 D1k Cta 1U iCBD sQZo NBTQom 9I y 2d0 mAN jA8 jQ 2lxm pAKy aPbD5P as P fWX mrX gg5 jj BntZ mZUq ZK7gFv wS j 973 eTw yng Xl DscS 2R1j HLmeEo kU r xmq sE7 Wll VP S01J x8rj 36WcqW rP V CdK avk rzz z1 gqmv aPPD GKCgWf 47 s uKy GH6 6pl pt 8aN4 8xRw ouRcsr KZ x xVC uH4 lN8 gJ tYm8 QPVQ Jgcyff Ge n 3b4 oOb Jl4 le oWkA L6zH XWxbsd Vw O WP4 Fjv xhN Lx tOWh xNdg 34S5Ib iq n v3I Uw0 gTU hX jEwb 3ulU Sb9Omy 6k y 45l OpD qIy Qh nzJP 7UVd 8FmcxG FI 3 M6R owG ODFGRTHVBSDFRGDFGNCVBSDFGDHDFHDFNCVBDSFGSDFGDSFBDVNCXVBSDFGSDFHDFGHDFTSADFASDFSADFASDXCVZXVSDGHFDGHVBCX182}        \end{align} We use a similar technique as in the proofs of \eqref{DFGRTHVBSDFRGDFGNCVBSDFGDHDFHDFNCVBDSFGSDFGDSFBDVNCXVBSDFGSDFHDFGHDFTSADFASDFSADFASDXCVZXVSDGHFDGHVBCX79}--\eqref{DFGRTHVBSDFRGDFGNCVBSDFGDHDFHDFNCVBDSFGSDFGDSFBDVNCXVBSDFGSDFHDFGHDFTSADFASDFSADFASDXCVZXVSDGHFDGHVBCX80}, obtaining        \begin{align}        \begin{split}        \Vert \curl v \Vert_{A_1(\tau)}        &        \leq        \sum_{m=1}^\infty \sum_{j=0}^m \binom{m}{j} \Vert (\epsilon \partial_t)^j R_0 (\epsilon \partial_t)^{m-j} \curl (r_0 v) \Vert_{L^2} \frac{\tau^{(m-3)_+}}{(m-3)!}\\        &\indeq        +        \sum_{m=1}^\infty \sum_{j=0}^m \binom{m}{j} \Vert (\epsilon \partial_t)^j \nabla R_0 (\epsilon \partial_t)^{m-j} r_0 v \Vert_{L^2} \frac{\tau^{(m-3)_+}}{(m-3)!}\\        &        \leq        C \Vert R_0 \Vert_{A(\tau)} \Vert \curl (r_0 v) \Vert_{\tilde{B}(\tau)}        +        C \Vert R_0 \Vert_{A(\tau)} \Vert \curl (r_0 v) \Vert_{B(\tau)}\\        &\indeq        +        C \Vert R_0 \Vert_{A(\tau)} \Vert r_0 v \Vert_{A(\tau)}        +        C \Vert R_0 \Vert_{\tilde{A}(\tau)} \Vert r_0 v \Vert_{A(\tau)}.        \label{DFGRTHVBSDFRGDFGNCVBSDFGDHDFHDFNCVBDSFGSDFGDSFBDVNCXVBSDFGSDFHDFGHDFTSADFASDFSADFASDXCVZXVSDGHFDGHVBCX354}          \end{split}        \end{align} Since $R_0$ is a function of $S$, it satisfies the inhomogeneous transport equation        \begin{align}        \partial_t R_0 + v \cdot \nabla R_0 = 0        .    \llabel{LF wI dQr6 oBks CxS02w id K MyF Yvc B12 zR yPha 2sY9 5yuhSK hU q Ar8 mJe Oj4 CS cV9N VVmA Sf0mzA xn j Qkk wOe 00B Fe ADuw xsKH FBQuJX Uv u VF3 bc3 7DQ 7G HBwv eDwB QVvxIM by O gQC WoC snb AL h68c MgNK PbJnWL Qj r VFq xfp wMY R4 YeQH 6l0N 2Cc9yO tR U D13 Kcg 2tL 90 rSZG K2lg O2P4hG jT D JHx Yci Aby Or CYF0 5v1Y sExhDV sR l h4V TTo z6w HO 6zXA ONUq 6zPmGW LS o Wg9 Rmv dxl Mp apKX ocmZ dnSQDj pt b FKF iWY rhX dy ckgZ tYp9 WyaAfl Bb V FIg 0gp OJA O2 7Fkg 6q8s bN2nnr OL f Fld 3ih qde Aw TMbI k4SJ gojG5n Qo d OKA ekl Hpy vc kztq QAqZ oAFTrz AF g yCs NVi GdR A4 Yip9 H3bl uAWAGG 4M 4 MqD GjE rwy 8V dtt2 RIsK YB6Ldd Zh Z acD mWL myr fW 3YZr kxTw 8ppbzr PT 0 XOd URe Q2e gL 5UaR Uh1s hxMHP4 2A 8 UaC 55w 08Z me we8Q NIR0 NDDCx1 jr R K09 wy6 oO6 wj EsP0 etAF TujB5i OO l vlS Eiw obL zV sWtS bh4k dw1hY5 CK o yqd BVT Pds l9 XzFo fecG gXefsf kB R 1Ln jyf 4H4 nA bSnq 9CjK PFkNut cV M UpB Y8y s1D nQ FUnD e6iW gr6NB7 5o C P8f wSC iw5 eA pwb5 fuop mxOAe4 UL G 28u rDR gY6 5o XknO b5xi RTqWQn tO n L23 3h0 mES rG 2kJP oTma FYVGSt Ss d EAZ MLw OpI k4 0UTZ jl3W sROwpD 20 t cdM Yy1 ov6 ve jFE0 N5Dz rPCZpq N5 q R8p bN4 PQF ro ucgs qfpP pFJ2p1 Mx f ofm U61 vZp Gl gO75 CWnC kDHWf4 gV i RS5 L5X bwQ jp Ezb9 a1wQ 0kLnOP 6M L 0R9 pXj iO0 lc O598 BL8B fJVuAx xt 7 Prq 1Qn Yru on IEyo ozXDFGRTHVBSDFRGDFGNCVBSDFGDHDFHDFNCVBDSFGSDFGDSFBDVNCXVBSDFGSDFHDFGHDFTSADFASDFSADFASDXCVZXVSDGHFDGHVBCX183}        \end{align} Then Lemma~\ref{L02} and \eqref{DFGRTHVBSDFRGDFGNCVBSDFGDHDFHDFNCVBDSFGSDFGDSFBDVNCXVBSDFGSDFHDFGHDFTSADFASDFSADFASDXCVZXVSDGHFDGHVBCX26} imply that        \begin{align}        \frac{d}{dt} \Vert R_0 \Vert_{A(\tau)}        \leq        \Vert R_0 \Vert_{\tilde{A}(\tau)} (\dot{\tau} + C \Vert v \Vert_{A(\tau)})        +         C \Vert R_0 \Vert_{A(\tau)}        +        C \Vert v \Vert_{A(\tau)}        +        C.    \label{DFGRTHVBSDFRGDFGNCVBSDFGDHDFHDFNCVBDSFGSDFGDSFBDVNCXVBSDFGSDFHDFGHDFTSADFASDFSADFASDXCVZXVSDGHFDGHVBCX184}        \end{align}   Coupling \eqref{DFGRTHVBSDFRGDFGNCVBSDFGDHDFHDFNCVBDSFGSDFGDSFBDVNCXVBSDFGSDFHDFGHDFTSADFASDFSADFASDXCVZXVSDGHFDGHVBCX309}, \eqref{DFGRTHVBSDFRGDFGNCVBSDFGDHDFHDFNCVBDSFGSDFGDSFBDVNCXVBSDFGSDFHDFGHDFTSADFASDFSADFASDXCVZXVSDGHFDGHVBCX29}, and \eqref{DFGRTHVBSDFRGDFGNCVBSDFGDHDFHDFNCVBDSFGSDFGDSFBDVNCXVBSDFGSDFHDFGHDFTSADFASDFSADFASDXCVZXVSDGHFDGHVBCX184}, we arrive at  \begin{align}   \begin{split}   &   \frac{d}{dt}    \left( \Vert u \Vert_{A_E} + \Vert R_0 \Vert_{A(\tau)} + \Vert \curl (r_0 v) \Vert_{B(\tau)}   \right)   \\   &\indeq   \leq    \dot{\tau} \Vert u \Vert_{\tilde{A}_E}    +   C\Vert u \Vert_{A_1(\tau)}   +   C \sum_{m=1}^\infty \frac{\tau^{(m-3)_+}}{(m-3)!} \Vert F \Vert_{L^2}   +    \Vert R_0 \Vert_{\tilde{A}(\tau)} (\dot{\tau} + C \Vert v \Vert_{A(\tau)})   +    C \Vert R_0 \Vert_{A(\tau)}   \\   &\indeq\indeq   +   C \Vert v \Vert_{A(\tau)}   +   \dot{\tau} \Vert \curl (r_0 v) \Vert_{\tilde{B}(\tau)}   +   C \Vert G \Vert_{B(\tau)}   +   C \Vert v \Vert_{B(\tau)}   +   C   .   \label{DFGRTHVBSDFRGDFGNCVBSDFGDHDFHDFNCVBDSFGSDFGDSFBDVNCXVBSDFGSDFHDFGHDFTSADFASDFSADFASDXCVZXVSDGHFDGHVBCX600}   \end{split}   \end{align}    Collecting the estimates \eqref{DFGRTHVBSDFRGDFGNCVBSDFGDHDFHDFNCVBDSFGSDFGDSFBDVNCXVBSDFGSDFHDFGHDFTSADFASDFSADFASDXCVZXVSDGHFDGHVBCX29}, \eqref{DFGRTHVBSDFRGDFGNCVBSDFGDHDFHDFNCVBDSFGSDFGDSFBDVNCXVBSDFGSDFHDFGHDFTSADFASDFSADFASDXCVZXVSDGHFDGHVBCX101}, \eqref{DFGRTHVBSDFRGDFGNCVBSDFGDHDFHDFNCVBDSFGSDFGDSFBDVNCXVBSDFGSDFHDFGHDFTSADFASDFSADFASDXCVZXVSDGHFDGHVBCX351}, \eqref{DFGRTHVBSDFRGDFGNCVBSDFGDHDFHDFNCVBDSFGSDFGDSFBDVNCXVBSDFGSDFHDFGHDFTSADFASDFSADFASDXCVZXVSDGHFDGHVBCX46}, \eqref{DFGRTHVBSDFRGDFGNCVBSDFGDHDFHDFNCVBDSFGSDFGDSFBDVNCXVBSDFGSDFHDFGHDFTSADFASDFSADFASDXCVZXVSDGHFDGHVBCX354}, and \eqref{DFGRTHVBSDFRGDFGNCVBSDFGDHDFHDFNCVBDSFGSDFGDSFBDVNCXVBSDFGSDFHDFGHDFTSADFASDFSADFASDXCVZXVSDGHFDGHVBCX600}, we arrive at   \begin{align}   \begin{split}   &   \frac{d}{dt}    \left( \Vert u \Vert_{A_E} + \Vert R_0 \Vert_{A(\tau)} + \Vert \curl (r_0 v) \Vert_{B(\tau)}   \right)   \\   &\indeq   \leq   \Vert u \Vert_{\tilde{A}_E} \left( \dot{\tau} + C \Vert \partial_t E \Vert_{B(\tau)} + C   +   C\Vert Ev \Vert_{A(\tau)}   \right)   +   C \Vert \partial_t E \Vert_{B(\tau)} \Vert u \Vert_{A(\tau)}   +    C\Vert Ev \Vert_{A(\tau)} \Vert u \Vert_{A(\tau)}   \\   &\indeq\indeq   +   C \Vert Ev \Vert_{A(\tau)}   +    C \Vert Ev \Vert_{A(\tau)} \Vert R_0 \Vert_{A(\tau)} \Vert \curl (r_0 v) \Vert_{B(\tau)}   +   C \Vert Ev \Vert_{A(\tau)} \Vert R_0 \Vert_{A(\tau)} \Vert r_0 v \Vert_{A(\tau)}   \\   &\indeq\indeq   +   \Vert R_0 \Vert_{\tilde{A}(\tau)} \left(\dot{\tau} + C \Vert v \Vert_{A(\tau)} +    C \Vert Ev \Vert_{A(\tau)} \Vert r_0 v \Vert_{A(\tau)}   \right)   +   C \Vert R_0 \Vert_{A(\tau)}   +   C \Vert v \Vert_{A(\tau)}   \\   &\indeq\indeq   +   \Vert \curl (r_0 v) \Vert_{\tilde{B}(\tau)} \left( \dot{\tau} + C \Vert Ev \Vert_{A(\tau)} \Vert R_0 \Vert_{A(\tau)}   \right)   +   C \Vert G \Vert_{B(\tau)}   +   C \Vert v \Vert_{B(\tau)}   +   C   ,   \end{split}    \llabel{y p1Whlu el I mMI hVo 5Q3 Vp iXFU UekX CCc5ao vm N tGU KQ7 OQr MJ IWKt z2Qi khWFX1 R4 E rAy YBu 0Y7 iK QuZC hw1U NMXlmg ii E KnL BFy oPX YF hsa9 k70g MnaJNt s5 D UWb hve LO8 So eTEL F0Ru v1MbVA gf I 2yJ Fuh CYX zy tWBt sjaq 6tMa55 21 k Xqz pMi Y2w Wv 544g KPBH B8EXVv 0W r mHO Wxp GiV dn VtO5 sZCU t2voXg hO E HVd Kqy APi jH 8BoS vmtK jKUBM8 vV P n1p lG5 7kl zz Bl9z eznz aySOHG E6 x 0Xb RQf jjl hu Do8X i53V 9dpOwc eh i D9G Snt AtF Fd ADt9 wbyg pJW32j Ff v JvZ 1bX xUW Oi W3O7 5z0w 8rrE2P dm f R9f TMQ wvf 6L Czkc TS9t eGWHWN 63 8 PNg R2n rTp 2z lTmI tjLU 8QIpTD WQ t 9jB oTY xqX mJ PlTv 39zt qqPzkf 55 M Y2q rqB mXU dV QeB7 rOlz Nub5Fq aV K Ebl Ht7 0GP pY wD6T QgVI SCLy8G Zo l mqE bfo 7hu kI 0uVh cFwU MmkDiY 9v c AYV igH MFU Vr zo4o TZAs 2ct4Qz 2b c sSe fVm Lqa Jh rxp6 JWDe IJAajx a4 t Y2z hHm z1x KC BeYP UNJU ZG5jto T5 f E6H u9B AZe ZX lgN9 M7nw 86AmBa u8 l lhP CuE Mux Lipschitz u30y Zb5D TLYAoC PR y uTL uan NMk t2 CGZS zUv2 fonuRr Ol E oJC b4a 4ET uS jsfu iPuG 0qp6T5 8K K bI2 Yod h1X nH 8sQY FdTi 0E76Qa IM q 4Di L0T s0T PB z81w EIIn 8HZWhJ uM 2 gMO q3G p4f fO M83q HOZr BolMiy T4 E g3b XoK skT am OIOH AII3 TVggqG pH A GfN Z60 i6d h2 JD7M yQ06 GK5Uca Nd z ZxU Tri 969 wE j90P vJwW 59DaTz iT Y tlk jJB EEI 9g PeNJ 03gd bPhJtD yW v 9p2 jI6 qJh Dx hX0a AhBA 7gQI2DFGRTHVBSDFRGDFGNCVBSDFGDHDFHDFNCVBDSFGSDFGDSFBDVNCXVBSDFGSDFHDFGHDFTSADFASDFSADFASDXCVZXVSDGHFDGHVBCX112}   \end{align} where we appealed to   \begin{align}    \Vert u \Vert_{\tilde{A}_1(\tau)}    \leq    C \Vert u \Vert_{\tilde{A}_E}    \llabel{S nD s yNx mNh APB A8 18hB ZdNn RcMJ2f qq a rzT NCV ecw BA sflk F4S4 LA9t3Z W9 G InZ A82 E9A l6 GAhg AZM7 CtESKS 3F b i9t tMt DiE gN p2r3 kjEd Qtygsn 2L Y Gjd KnE Y05 bY ymad bCAV qhL7IJ 62 D Fm2 92D SCA 4i RHDM OdZ3 e30PWU oW 4 oK4 aRU RjS Hu JGmM Ib1f qnUpm0 IO i M3u wLF E2a 6C CZOL svs3 wH9a6d VF h CS6 tSQ T1o LO QPer VRNY bCGGNj bt F BSD 2IK xbq MA AEfn BsfY 6DJP01 7e b bsb atx 9xM xn aFkk Zjq6 pecNfG AA q HKi gAw 9Gy 5E St6Q hqQa myZzPv 8F y cS3 hAe hzW gD os21 TOCn 85kSCw 0q u seV qmX wvN 0t NkRD VGS1 3xgehb yY G Lkk VIf MOr mJ ijJa IpG9 g4c2yt xY 0 Yei HhK k9c On JIVw WR51 Uy1uTS 1z 4 arU JZX p78 zj TemS pw5W SbP3wd s5 D N2A gcf SkC Ee NZnP 8Wvp p6GAoM bU y E2J yOy Zxy 8H c1NL 30g4 o8dJA3 F8 O mDh 8mQ 49x ze OkX5 Nvu0 OsENCI pY O opU kBa LoB MK gh43 i1zW uaFegI dS D Cko k9m lVT T4 VQNH QHlU Xs7h9S H5 N gfg MBm 6Ju Mp NpEp yDUX g8CKNF ht 1 9fr dD7 jRI gc 4kaE UKMR 4kTC0x 06 l Px5 Nhj 3lb GQ YP8e RgEB n8r5Rw YU f z3A SEA h1e FY Qk9P ZHtL 4mBxwm v8 T 39d XXW tgR gB Jn5F RhyE bUGCRE M2 s ls8 sNC cqP y7 keU9 9axt UX6cZ2 OS c ppi ibb aL6 co I6Nw mOCg QNCDIH Wu O tXu hQG JLt Rs D7o6 MVLT Y2HTvd eJ M swx sNa qjm ee 7dvC Mq6n pXN1fp 83 K jm8 8YG lR2 Kg akrc wcdn HCtbWl la 3 DS6 qvh RkL VH H7yC Z8r9 E3qSFM JZ P Pzn Z7f BLk Wm fnK8 JxrS x0VcXd Sb F j5V amkDFGRTHVBSDFRGDFGNCVBSDFGDHDFHDFNCVBDSFGSDFGDSFBDVNCXVBSDFGSDFHDFGHDFTSADFASDFSADFASDXCVZXVSDGHFDGHVBCX164}   \end{align}   by the boundedness of $\Vert E^{-1/2} \Vert_{L^\infty}$.          Now, assume that the radius $\tau(t)$ decreases sufficiently fast so that  the factors next to $\Vert u \Vert_{\tilde{A}_E}$, $\Vert R_0 \Vert_{\tilde{A}(\tau)}$, and $\Vert \curl (r_0 v ) \Vert_{\tilde{B}(\tau)}$ are  less than or equal to $0$. Integrating the resulting inequality on $[0,t]$, we get   \begin{align}   \begin{split}   \Vert u \Vert_{A_E}   \leq   \Vert u(0) \Vert_{A_E}    +   \Vert R_0 (0) \Vert_{A(\tau)}    +   \Vert \curl (r_0 v)(0) \Vert_{B(\tau)}    +   t Q(M_{\epsilon, \kappa}(t))   \leq   C   +   t Q(M_{\epsilon, \kappa}(t)),   \end{split}    \llabel{ ev3 oM UIhO QvVH lL9gn3 zA R uNZ kz5 jqP O5 JinK 2Qby IMjIX4 Zd b KD1 KkW How KT ugzJ Y5NT nNoYtP Jv B qRs Xhe P1a uU bK1z lcsz OHp7Yp OR W DgF aiN 4mU SY 4H91 gjEN M5MgdY dn r ifW 7GE 1m2 hA HsnP Ccul IJwwco gM B xI7 Xyq Gnd 0a xMBC 2blR dzBmTY x2 E 1qN 72w nMQ px wJEU WzM0 Ff1m6x Gq L KTf YVH v0A BS JK71 8ntf 8YSoL2 sh x CO3 ipS qwl ta 2Q4u SRiU pqj6JI P6 R Kri coX sKH dr 8OaT vEas TdFXsx DG 3 J55 rbd rU8 hY 37tX DroH lMNixz mn a wPP xHc 1ga DM rRPD cNiI BP8OX9 zr c jpp pPX 1T4 Wf QIdL 62XY HDKiby Qd x wQK 915 71P KP DwLC ocWm 5zM3y9 wQ l MqK nur nJx DQ j378 1To8 imMeS3 XF m Ooe wUQ oqN 7V Kyex DoPf NR4aj8 9I F LCf p6j NUM X8 SSTS eEOR E0pPsr Vp h VFR Koi 9qK P5 schl LkLk kLrWZA wV T rf4 JBr CsR pT sUxE LByu amQ6ew DG Y oVx VSo aqP Cr Icmg O1eC wxSO0B BG S CwX f9V eFs zg 39Gu 0Moo onhnAq sE b 8oN RdW c8B u6 2ON2 7dB2 nFy9Vv P9 E yq0 nxy xZ7 GR jUsN fDXY cyV5nH sP K lrR WaM L3x 9O fNS9 6Cgd BLKXts p2 x tUH zSq TXm 2m GRcW htpG VFVlDc nH F Asp aPV w6J 1v VMqP 7C6Y YCUeYn L1 0 2Mo N0r kry Ng zmmY 6cI6 omeWFb YV B 7EQ BKj 0hM IF QYmS Cqyw 2FQGnz 62 H 4VD wg1 1OL Kn hUec KiyV j8RuLC md d sIi qZM X8i 0j hohB jguo Tqum14 lL L O84 LDU 29Z lV H0gX NWun OudtPq fH 9 Iwz 23H NIn K6 uY54 PVjv PJb0Rr bx Z bqA PYW hgr TT K4Gq LB1b 6yqH5B 24 z vhk VVy azf ad 79h3 pDFGRTHVBSDFRGDFGNCVBSDFGDHDFHDFNCVBDSFGSDFGDSFBDVNCXVBSDFGSDFHDFGHDFTSADFASDFSADFASDXCVZXVSDGHFDGHVBCX93}   \end{align} and since $   \Vert u \Vert_{A_1}   \leq   C \Vert u \Vert_{A_E} $, the proof is concluded. \end{proof} \par \subsection{Energy equation for the mixed derivatives ($A_2$ norm)} \label{sec6.3} Here we estimate the mixed space-time analytic norm. For this purpose, denote   \begin{align}   \begin{split}    \Vert u \Vert_{A_2(\tau)}    &    =    \sum_{m=1}^\infty \sum_{j=1}^m \sum_{\vert \alpha \vert =j}\Vert \partial^\alpha (\epsilon \partial_t)^{m-j} u \Vert_{L^2} \frac{\kappa^{(j-3)_+} \tau(t)^{(m-3)_+}}{(m-3)!}   ,   \end{split}    \llabel{YKR x7Np7p 2x B v9Q nJc tEG D6 wgzk 8ZHM mXjXMm rp 6 yBj q1g MLN rH ocwg 6EZN ZrrwMm wn 3 Hbz B1w H8g xd qgzU VK1f kfNrGS a7 6 qr0 fq3 f8N 0f xb4v BawG NH0ZoC iw Y 1ZY eKu a9k xE rKIE 01CC jUymat 9X c Ty8 cbG Wm1 vM Vanu CIZs 8JyLcv wa 5 WfY Dh5 lh4 rI 37En uL5A x1MuIN vP i swL D6J oJk nG 3l2t R8TT AlrMhB RX v RPZ HVA A0Y 5R fHUR lZV8 lC7oXK Lj B d9f 9mF X7f ik DuaY JUvl jbZQJQ wo z ksi ZxN c5n Ea L4gX ivyd MFchxJ 1l j kcE Mii bBU 12 Zumc 0S39 blV9Xd 6f d qpH ngB x1l fP gD57 Ya7S KouGJZ k0 9 p8g QGw Wtv Bd 1vyR DvmP yeScYF xe b EyP 89Q BBH En dJE9 1IR6 oA0OrF y7 p A3y 8AN VMM jL 5Sct kuRb CyP7Au ZI a mZc li7 THG Px zSQ6 6gtO g4b85K Tq b ogV EOF kXH T7 KJIc hMCb 2wen4R mq 6 pq5 yqp 80h Ln lKi3 J1iE OPuQS7 dS V 69Z mSQ dVe ey iiqq Ogc9 xIdkPk 5T A rJM 6tj VR5 qK 0KIQ mN9m P0FPBQ 8f o ckL JRw Z4q H0 rQWF FeI3 t9XnSo wG j oEC wQF 6fL pW fYXg ry12 Iyad2e m1 E E2i OZS bp1 Fn Rxmq OS5p oJTSfO ui 5 ePX KVG ldk z0 7xrm 5m6e vbQ7n2 BL I 1NM sxF 87b Jk 04DZ 6DVk QF5BX9 ab y OCb l33 nrb 0Y zrzP KSkS wjRDb0 er c FgV 17x QAo Et ekQQ 3V5S 4N7gVI NW h zJr gDy jBM ve Lfrf C4mw 0BtDdR 5r o p2k HfB ed8 6e t1QQ xXlJ SQfsNL Dd b aCP koF Ae8 LL 0XdO BApe 63zGAO AQ T wmF 4Yb q6C pV XUvO 26Oz ii2JWK SQ j VDr fxP bf3 1P MG1I a5Yn nI35sz nZ u Wec BMI ay5 iu zVNN fAZA DEauRS 1m DFGRTHVBSDFRGDFGNCVBSDFGDHDFHDFNCVBDSFGSDFGDSFBDVNCXVBSDFGSDFHDFGHDFTSADFASDFSADFASDXCVZXVSDGHFDGHVBCX90}   \end{align} and let    \begin{align}   \begin{split}    \Vert u \Vert_{\tilde A_2(\tau)}     &    =    \sum_{m=4}^\infty     \sum_{j=1}^m     \sum_{\vert \alpha \vert = j} \Vert \partial^\alpha (\epsilon \partial_t)^{m-j} u \Vert_{L^2} \frac{\kappa^{(j-3)_+} (m-3)\tau(t)^{m-4}}{(m-3)!}    \end{split}    \llabel{O 6jg DZG 3tF iL itQg H0aE pH2NEA xz V uSG M14 Qt3 aG QSIR mm1r 4qpDQK MF w SOr 3oX 4G3 cb 4Efc Z7kb q8n7zH ed g 0sn 4NK AoE mO 024C ARDq sDAezB 7G h ePE QzU PjG lP f3Zx uoKU SWHU0U Oh 1 sYO gjy tRl Dg 3DkY hOm7 QYHyPx rR N 6lV SN3 RGq 1g 0wFl tIrs yraBxg uZ 6 jgC 7Am r7v bm x6FN PWZg fJ0qwR k5 J b0N VWE jqA Aj 6R8d fhnu Vzkdsx F3 G u0l Iki 612 Ee m0Cj 21Qz vJZ7vt Fz i 6cz K0F XsU Dp 6QD9 aFdL IPmJrR rV o 33L cSH rlb J1 KyIh 4ngI mbgClV j0 8 nB5 pOc cHf dc USjt s7pR qLkPji kS r E02 Ncv G4g TX Tr4N eEA9 lIVnf6 9C w 6jI KcV UEX 31 EQ8R Qt1B PS3Ob6 PX 0 8iA s3g pLB rQ gEm0 DCTm COHhK8 MP b 5Gj uDY b7Y Qe ir17 X7zB m3M4Yn jN 6 e1u hw5 OdE Gq X4cK Ua5T Fi20Ri RL A FoK Txc aEz My P4FL JjiW paqZI3 lb N dF9 1g5 8GO cb skA7 Vdg9 3VtNlD N8 k WiB bGQ Jes 7X 66b7 JNu0 0Nci4m pc B UoN csS Hjb 7u yM3x YrvX fPaUZq pC N agr nVg Qmp Gx 8pYw 88Y4 0YvJSa 8Z v Q4Q d4v kys sd Ualw oSGf Fwt3ZV Qm g 8db ZA4 uSz r0 tJyq n8aQ Tg3VGz 9c M P2l jhA zaU kc I6K3 UaNw WHHeJD if G xvh Lwt sJN Bg zzU7 p5e7 T1jlNQ Qu G fC4 gSP MhY 0n W6W0 pdS5 gYmYeR Sd 2 sV8 4Bc 7rG vI OX71 T1bP AOvN7c l1 k yNA Rwp 6NI e6 evMg p1QB v51nbQ Ul S Lvg YLw dR2 T6 9ASA 1M9O gtyi3P 4j g uia Cve z60 PK J5YR mF6Y 2vTG4r c6 t pHL MfQ R61 Po pib2 EvVU 4XjpTs 9j 0 kfh 0vn iGR yc bVsu CmyQ HgJE6s RS b jvf FQN hYv DFGRTHVBSDFRGDFGNCVBSDFGDHDFHDFNCVBDSFGSDFGDSFBDVNCXVBSDFGSDFHDFGHDFTSADFASDFSADFASDXCVZXVSDGHFDGHVBCX91}   \end{align}    be the corresponding dissipative analytic norm. Note that   \begin{equation}    \Vert u(t) \Vert_{A}    =        \Vert u(t) \Vert_{A_1}    +   \Vert u(t) \Vert_{A_2}     .    
\llabel{qx 36XL 5Ewe 5pZBhV OH s X0y XXw UY0 Rl b2lw A8GQ 8z7g11 io 4 og6 p99 V71 7w qK50 O5RE sBEMQ0 91 I M2y fkd 7nN hQ Jlfc 1Q0i dDkUbE Ep j Ye7 qfG sYO QJ aNxv i5ET eKq11c Rp c fHP 185 AHy Mv 9X2Z s698 3J3xgs l8 N Tn7 u4P Uiz b7 U4JR x1xE y8AmNM KL Y ddX hhd 2Mr XR sUmt ntyW gCtmRS Pf V kEJ 2Jy 2zp Nz X80H XN4V RloB8Q Ir m jRh atE FrL V3 Y85v bn37 KIt0ZF Gr W 3qO UWf Uxl RJ Fp5m RzW2 AWcYw8 1Y 5 RLP xxR tRq CP yWDE YlHN HSslDe OK i g6E bHh I65 69 5htf ukyT bqGKt4 nH K HoH HUq PG4 tT lfry bDyV qMfLgS zo 6 Lvp mVd SIK Fi upOA sn9Q laOl2A Wn D 2iw 2NI e6k Hk WN7r mqux Im9Rez Ng Q VsD Yid kVk RF TU2w nTVx vOotbJ mk U gZA 53U BOo 0n ZzBn OHpb MCkRfQ t6 D KD3 vgs 4bq Hc brny nvVn 2kqS3h Pg Q asD TlM ghY Dd oZMc cOpL T5WVju iU t MgW bp7 vfM nK uxcJ dNC7 OlLPiy Os i sIx HO5 AgP 1m hZEP ASHM 7Y4djN EE J jCu BLJ GVD q9 nOOQ a5o7 boiQVe 7M u EkK ft8 UFZ f1 PJMg V918 ITlqdS v8 9 uis 0dK lCW C0 3xER AnIw CoOOjw gY a 47Z 8ui cW9 RW HQD5 6kbp 8LhPWP Km p rvo Ozr VxO id MatL qHqx 2GV1Ck H0 l Cso Yis kQW DL uQyi Sect TUilUZ Xe E Ipi y1D aCz Bc xIGz Ohc3 LstcT6 MW z zsg kqW WGo bW dzTF ETHf 4mxIGm Fl 3 KrK lZ4 aK2 O3 260H U03p 0Q2T6W j7 T oSL Klf ybp tN DLS7 JFTZ 3hXBCu sL z ZuP LCi fqK oe XbPY t0R9 D6eYLh nu c TVw tkO mvE Q1 GXX5 4pGl UyOQBy 4G Y ket 7X6 5pc qJ zGlc XXAa nDFGRTHVBSDFRGDFGNCVBSDFGDHDFHDFNCVBDSFGSDFGDSFBDVNCXVBSDFGSDFHDFGHDFTSADFASDFSADFASDXCVZXVSDGHFDGHVBCX128}   \end{equation} \par \cole \begin{Lemma}  \label{L04} Given $M_0>0$, there exist a function $Q$ and $\epsilon_0$ such that for $\epsilon \in (0, \epsilon_0)$, $\kappa<1$,  the solution of~\eqref{DFGRTHVBSDFRGDFGNCVBSDFGDHDFHDFNCVBDSFGSDFGDSFBDVNCXVBSDFGSDFHDFGHDFTSADFASDFSADFASDXCVZXVSDGHFDGHVBCX01} satisfies   \begin{align}   \Vert u(t) \Vert_{A_2}   \leq   C   +   \left( t+\epsilon + \tau +\kappa\right)Q(M_{\epsilon, \kappa}(t)).    \label{DFGRTHVBSDFRGDFGNCVBSDFGDHDFHDFNCVBDSFGSDFGDSFBDVNCXVBSDFGSDFHDFGHDFTSADFASDFSADFASDXCVZXVSDGHFDGHVBCX105}   \end{align} \end{Lemma} \colb \par \begin{proof}[Proof of Lemma~\ref{L04}] By the definition of the $A_2(\tau)$ norm, we have        \begin{align}        \begin{split}        \Vert u \Vert_{A_2(\tau)}        &        =        \sum_{m=1}^\infty \sum_{j=1}^m \sum_{\vert \alpha \vert =j} \Vert \partial^\alpha (\epsilon \partial_t)^{m-j} u \Vert_{L^2} \frac{\kappa^{(j-3)_+} \tau^{(m-3)_+}}{(m-3)!}        \\        &        =        \sum_{m=1}^3 \sum_{j=1}^m \sum_{\vert \alpha \vert =j} \Vert \partial^\alpha (\epsilon \partial_t)^{m-j} u \Vert_{L^2} \frac{1}{(m-3)!}        +        \sum_{m=4}^\infty \sum_{j=4}^m \sum_{\vert \alpha \vert =j} \Vert \partial^\alpha (\epsilon \partial_t)^{m-j} u \Vert_{L^2} \frac{\kappa^{j-3} \tau^{m-3}}{(m-3)!}        \\&\indeq        +        \sum_{m=4}^\infty \sum_{j=1}^3 \sum_{\vert \alpha \vert =j} \Vert \partial^\alpha (\epsilon \partial_t)^{m-j} u \Vert_{L^2} \frac{\tau^{m-3}}{(m-3)!}        =        \mathcal{P}_1        +        \mathcal{P}_2        +        \mathcal{P}_3,        \end{split}    \llabel{2YHHc Bi k FpU mF2 z1B qs 9Pq0 uhHE kzTQd0 WC K CEx BJK Wit DX hICD ebnU gz4Udq 9G k yBa O8N 0yW L0 O0iK NMSy Agarms Fn e 4xD tOz gVR K3 2um1 MqAS tAux9D qC B v8p bhl I4C F2 NKxe 6zo6 vGs8QE OM Q 8yD dJC y7T EI 6PRh J7Y7 7oXjLE E8 T Zdk zbX fNu jt Fx0C aZiD 4N8K5n Wm X gWz pAS k4h Ei a5VO rx0O 7INJ2z 0B g hNF AaZ 8cn rc ypS3 XIOa V1YeDb yR V wgO dVZ 4Jk Ph NJXJ XOvl QuWoXK hw A qlT RRy KyI it 8IBk hodH FDInjN JT U uBi 3sk cDC 2a BDwP Qtwu T3KfbB 7I 9 olG BoM tDQ mI moQW nnpU oCq7gX qW F YkI LZL 9h5 Kz I9i2 LOfe O58k80 ga Z jvp Y3z hQr 45 8vOd txn0 R6a5W7 OC J fRH CGo rjC Pb sAms Yt5D rP9XU3 N8 b Fsd vAZ Wlk OX Gwem rty1 pX6cxB 3c X iO1 90P uRu ff DfNf 5L3D E6DMEa bd R urg WCv FJY g6 uPjB 8h7Z NIL7sE CY N vkg vsb 7K3 VV b5JL Kq68 thYpvT Wq Y riK nGj ZNf Iy 5Qf3 tgdY OPVUzS m3 E SGU Z5R 4kC FX c1vP d10o o09mXi Ji B e4w MFF y92 Nz VpAv R63h 4TshsK fZ A Rlw Kox A0H zn hlrZ KNgc BbPqON y0 i RDH SGL 1Ll PE CGnC llnP LoSBKr 46 d UE6 XvX 1PR xI d6BA 7Ayo 6DA4hn VL M kyr byi WzU Wc 1dc0 EKa6 VytPgO F6 w 1mC MZV yCH 3f qSYA aCRU gI9VMy nm 0 7aT FqE Efk EQ jtsR 4yMD nYUOnR 99 3 H4F N1t ZvR jM rNrS m41P xerpK2 VZ L 9zJ R3i PU9 rY gnIt algn QQFjOB Ee V jS2 F4z P0v OX 1sWS zTLA riSMKL Co A sfM yZu PzZ PX 06p0 ZD3B gMXZbD 2h K Hs2 8JF 90o Ka vi8u 4iHZ 6caRRC ND 7 4fQDFGRTHVBSDFRGDFGNCVBSDFGDHDFHDFNCVBDSFGSDFGDSFBDVNCXVBSDFGSDFHDFGHDFTSADFASDFSADFASDXCVZXVSDGHFDGHVBCX185}        \end{align} where we split the sum according to the high and low values of $j$ and $m$. We claim        \begin{align}        &        \mathcal{P}_1        \leq        C,        \label{DFGRTHVBSDFRGDFGNCVBSDFGDHDFHDFNCVBDSFGSDFGDSFBDVNCXVBSDFGSDFHDFGHDFTSADFASDFSADFASDXCVZXVSDGHFDGHVBCX360}        \\        &        \mathcal{P}_2        \leq        C+ \left( t + \epsilon + \kappa + \tau\right)Q(M_{\epsilon, \kappa}(t))        ,        \label{DFGRTHVBSDFRGDFGNCVBSDFGDHDFHDFNCVBDSFGSDFGDSFBDVNCXVBSDFGSDFHDFGHDFTSADFASDFSADFASDXCVZXVSDGHFDGHVBCX361}        \\        &        \mathcal{P}_3        \leq        C + \left( t + \epsilon + \tau\right)Q(M_{\epsilon, \kappa}(t))        .         \label{DFGRTHVBSDFRGDFGNCVBSDFGDHDFHDFNCVBDSFGSDFGDSFBDVNCXVBSDFGSDFHDFGHDFTSADFASDFSADFASDXCVZXVSDGHFDGHVBCX362}        \end{align} Firstly, \eqref{DFGRTHVBSDFRGDFGNCVBSDFGDHDFHDFNCVBDSFGSDFGDSFBDVNCXVBSDFGSDFHDFGHDFTSADFASDFSADFASDXCVZXVSDGHFDGHVBCX360} follows by using  Sobolev inequalities and Remark~\ref{R04}. \par Proof of \eqref{DFGRTHVBSDFRGDFGNCVBSDFGDHDFHDFNCVBDSFGSDFGDSFBDVNCXVBSDFGSDFHDFGHDFTSADFASDFSADFASDXCVZXVSDGHFDGHVBCX361}: We rewrite equation \eqref{DFGRTHVBSDFRGDFGNCVBSDFGDHDFHDFNCVBDSFGSDFGDSFBDVNCXVBSDFGSDFHDFGHDFTSADFASDFSADFASDXCVZXVSDGHFDGHVBCX01} as   \begin{align}   L(\partial_x) u = - E(S,\epsilon u) (\epsilon \partial_t u + \epsilon v\cdot \nabla u)   .   \label{DFGRTHVBSDFRGDFGNCVBSDFGDHDFHDFNCVBDSFGSDFGDSFBDVNCXVBSDFGSDFHDFGHDFTSADFASDFSADFASDXCVZXVSDGHFDGHVBCX20}   \end{align} For $m \geq 3$, and $\vert \alpha \vert = j$ where $3\leq j \leq m$, we commute $\partial^\alpha (\epsilon \partial_t)^{m-j}$ with \eqref{DFGRTHVBSDFRGDFGNCVBSDFGDHDFHDFNCVBDSFGSDFGDSFBDVNCXVBSDFGSDFHDFGHDFTSADFASDFSADFASDXCVZXVSDGHFDGHVBCX20}, and using div-curl regularity \eqref{DFGRTHVBSDFRGDFGNCVBSDFGDHDFHDFNCVBDSFGSDFGDSFBDVNCXVBSDFGSDFHDFGHDFTSADFASDFSADFASDXCVZXVSDGHFDGHVBCX350}, we obtain   \begin{align}   \begin{split}   \Vert \nabla \partial^\alpha (\epsilon \partial_t)^{m-j} u \Vert_{L^2}   &   \leq   C\Vert L(\partial_x) \partial^\alpha (\epsilon \partial_t)^{m-j} u \Vert_{L^2}   +   C\Vert \curl (\partial^\alpha (\epsilon \partial_t)^{m-j} v )\Vert_{L^2}\\   &   \leq    C\sum_{k=0}^{m-j} \sum_{l=0}^{j} \sum_{\vert \beta \vert = l, \beta \leq \alpha} \binom{\alpha}{\beta} \binom{m-j}{k} \Vert \partial^\beta (\epsilon \partial_t)^k E  \partial^{\alpha - \beta} (\epsilon \partial_t)^{m-j-k+1} u \Vert_{L^2}\\   &\indeq\indeq   +   C\epsilon \Vert Ev \Vert_{L^\infty} \Vert \partial^\alpha (\epsilon \partial_t)^{m-j} \nabla u \Vert_{L^2}\\   &\indeq\indeq   +      C \epsilon \sum_{l=0}^{j} \sum_{\substack{k = 0\\l+k \geq 1}}^{m-j} \sum_{\vert \beta \vert = l, \beta \leq \alpha} \binom{\alpha}{\beta} \binom{m-j}{k} \Vert \partial^\beta (\epsilon \partial_t)^k (Ev) \partial^{\alpha-\beta} (\epsilon \partial_t)^{m-j-k} \nabla u \Vert_{L^2}\\   &\indeq\indeq   +   C\Vert \partial^\alpha (\epsilon \partial_t)^{m-j} \curl v \Vert_{L^2}   .   \end{split}   \label{DFGRTHVBSDFRGDFGNCVBSDFGDHDFHDFNCVBDSFGSDFGDSFBDVNCXVBSDFGSDFHDFGHDFTSADFASDFSADFASDXCVZXVSDGHFDGHVBCX24}   \end{align} The second term on the far right side of \eqref{DFGRTHVBSDFRGDFGNCVBSDFGDHDFHDFNCVBDSFGSDFGDSFBDVNCXVBSDFGSDFHDFGHDFTSADFASDFSADFASDXCVZXVSDGHFDGHVBCX24} can be absorbed into the left side when $\epsilon$ is sufficiently small. Multiply the above estimate with appropriate weights and then sum, with change of variables we obtain    \begin{align}   \begin{split}   \mathcal{P}_2   =   &   \sum_{m=4}^\infty \sum_{j=4}^m \sum_{\vert \alpha \vert =j} \Vert \partial^\alpha (\epsilon \partial_t)^{m-j} u \Vert_{L^2} \frac{\kappa^{j-3}\tau^{m-3}}{(m-3)!}   \\&   \leq   C\sum_{m=3}^\infty \sum_{j=3}^m \sum_{\vert \alpha \vert =j} \Vert \nabla \partial^\alpha (\epsilon \partial_t)^{m-j} u \Vert_{L^2} \frac{\kappa^{j-2}\tau^{m-2}}{(m-2)!}   \\   &   \leq   C\kappa \sum_{m=4}^\infty \sum_{j=3}^{m-1} \sum_{\vert \alpha \vert = j} \sum_{k=0}^{m-j-1} \sum_{l=0}^{j} \sum_{\vert \beta \vert =l, \beta \leq \alpha}\frac{\kappa^{j-3}\tau^{m-3}}{(m-3)!} \binom{\alpha}{\beta} \binom{m-j-1}{k}    \Vert \partial^\beta (\epsilon \partial_t)^k E  \partial^{\alpha - \beta} (\epsilon \partial_t)^{m-j-k} u \Vert_{L^2}\\   &\indeq   +   C\epsilon \sum_{m=3}^\infty \sum_{j=3}^{m} \sum_{\vert \alpha \vert = j }   \sum_{l=0}^j \sum_{\substack{k = 0\\l+k \geq 1}}^{m-j} \sum_{\vert \beta \vert = l, \beta \leq \alpha} \frac{\kappa^{j-2}\tau^{m-2}}{(m-2)!} \binom{\alpha}{\beta} \binom{m-j}{k} \\   &\indeqtimes   \Vert \partial^\beta (\epsilon \partial_t)^k (Ev) \partial^{\alpha-\beta} (\epsilon \partial_t)^{m-j-k} \nabla u \Vert_{L^2}   \\   &\indeq   +   C\sum_{m=3}^\infty \sum_{j=3}^m \sum_{\vert \alpha \vert = j} \frac{\kappa^{j-2} \tau^{m-2}}{(m-2)!} \Vert \partial^\alpha (\epsilon \partial_t)^{m-j} \curl v \Vert_{L^2}   \\   &   \leq   C\kappa \Vert E u \Vert_{A(\tau)}    +   C\epsilon ( \Vert Ev\Vert_{A(\tau)} + \Vert Ev \Vert_{L^2} ) \Vert u \Vert_{A(\tau)}   +   C \Vert \curl v \Vert_{B(\tau)}\\   &   \leq   C + \left( t + \kappa + \tau + \epsilon\right) Q(M_{\epsilon, \kappa}(t)),    \end{split}   \label{DFGRTHVBSDFRGDFGNCVBSDFGDHDFHDFNCVBDSFGSDFGDSFBDVNCXVBSDFGSDFHDFGHDFTSADFASDFSADFASDXCVZXVSDGHFDGHVBCX25}   \end{align} where the last inequality follows from the estimates \eqref{DFGRTHVBSDFRGDFGNCVBSDFGDHDFHDFNCVBDSFGSDFGDSFBDVNCXVBSDFGSDFHDFGHDFTSADFASDFSADFASDXCVZXVSDGHFDGHVBCX328} and \eqref{DFGRTHVBSDFRGDFGNCVBSDFGDHDFHDFNCVBDSFGSDFGDSFBDVNCXVBSDFGSDFHDFGHDFTSADFASDFSADFASDXCVZXVSDGHFDGHVBCX356}. \par Proof of \eqref{DFGRTHVBSDFRGDFGNCVBSDFGDHDFHDFNCVBDSFGSDFGDSFBDVNCXVBSDFGSDFHDFGHDFTSADFASDFSADFASDXCVZXVSDGHFDGHVBCX362}: For $m \geq 4$ and $\vert \alpha \vert = j$ where $1\leq j \leq 3$, we proceed as in \eqref{DFGRTHVBSDFRGDFGNCVBSDFGDHDFHDFNCVBDSFGSDFGDSFBDVNCXVBSDFGSDFHDFGHDFTSADFASDFSADFASDXCVZXVSDGHFDGHVBCX24}--\eqref{DFGRTHVBSDFRGDFGNCVBSDFGDHDFHDFNCVBDSFGSDFGDSFBDVNCXVBSDFGSDFHDFGHDFTSADFASDFSADFASDXCVZXVSDGHFDGHVBCX25} and obtain   \begin{align}        \begin{split}        \mathcal{P}_3        &        =        \sum_{m=4}^\infty \sum_{j=1}^3 \sum_{\vert \alpha \vert = j} \Vert \partial^\alpha (\epsilon \partial_t)^{m-j} u \Vert_{L^2} \frac{\tau^{m-3}}{(m-3)!}        \\        &        \leq        C\sum_{m=3}^\infty \sum_{j=0}^2 \sum_{\vert \alpha \vert = j} \Vert \nabla \partial^\alpha (\epsilon \partial_t)^{m-j} u \Vert_{L^2} \frac{\tau^{m-2}}{(m-2)!}        \\        &        \leq        C\sum_{m=3}^\infty \sum_{j=0}^2 \sum_{\vert \alpha \vert = j} \Vert L(\partial_x) \partial^\alpha (\epsilon \partial_t)^{m-j} u \Vert_{L^2} \frac{\tau^{m-2}}{(m-2)!}        +        C\sum_{m=3}^\infty \sum_{j=0}^2 \sum_{\vert \alpha \vert = j} \Vert \curl (\partial^\alpha (\epsilon \partial_t)^{m-j} u )\Vert_{L^2}\frac{\tau^{m-2}}{(m-2)!}      .        \end{split}       \label{DFGRTHVBSDFRGDFGNCVBSDFGDHDFHDFNCVBDSFGSDFGDSFBDVNCXVBSDFGSDFHDFGHDFTSADFASDFSADFASDXCVZXVSDGHFDGHVBCX84}        \end{align} Therefore,        \begin{align}        \begin{split}        \mathcal{P}_3        &        \leq        C\sum_{m=3}^\infty \sum_{j=0}^2 \sum_{\vert \alpha \vert = j} \sum_{l=0}^j \sum_{\vert \beta \vert =l , \beta\leq \alpha} \frac{\tau^{m-2}}{(m-2)!} \Vert \partial^\beta E  \partial^{\alpha - \beta} (\epsilon \partial_t)^{m-j+1} u \Vert_{L^2}\\   &\indeq   +        C \epsilon\sum_{m=3}^\infty \sum_{j=0}^2 \sum_{\vert \alpha \vert = j} \sum_{k=1}^{m-j} \sum_{l=0}^{j} \sum_{\vert \beta \vert =l, \beta \leq \alpha} \frac{\tau^{m-2}}{(m-2)!}  \binom{m-j}{k}    \Vert \partial^\beta (\epsilon \partial_t)^{k-1} \partial_t E  \partial^{\alpha - \beta} (\epsilon \partial_t)^{m-j-k+1} u \Vert_{L^2}\\   &\indeq   +   C\epsilon \sum_{m=3}^\infty \sum_{j=0}^2 \sum_{\vert \alpha \vert = j }   \sum_{l=0}^j \sum_{\substack{k = 0\\l+k \geq 1}}^{m-j} \sum_{\vert \beta \vert = l, \beta \leq \alpha} \frac{\tau^{m-2}}{(m-2)!} \binom{m-j}{k}   \Vert \partial^\beta (\epsilon \partial_t)^k (Ev) \partial^{\alpha-\beta} (\epsilon \partial_t)^{m-j-k} \nabla u \Vert_{L^2}   \\   &\indeq   +   C\sum_{m=3}^\infty \sum_{j=0}^2 \sum_{\vert \alpha \vert = j} \frac{\tau^{m-2}}{(m-2)!} \Vert \partial^\alpha (\epsilon \partial_t)^{m-j} \curl v \Vert_{L^2}    .   \end{split}   \label{DFGRTHVBSDFRGDFGNCVBSDFGDHDFHDFNCVBDSFGSDFGDSFBDVNCXVBSDFGSDFHDFGHDFTSADFASDFSADFASDXCVZXVSDGHFDGHVBCX363}    \end{align} \colb The second term on the far right side is estimated by using Lemmas~\ref{L03}  and~\ref{L06},  while the third and fourth term can be estimated analogously as in \eqref{DFGRTHVBSDFRGDFGNCVBSDFGDHDFHDFNCVBDSFGSDFGDSFBDVNCXVBSDFGSDFHDFGHDFTSADFASDFSADFASDXCVZXVSDGHFDGHVBCX25}. For the first term, denoted by $\mathcal{P}_{31}$,  we use the div-curl regularity to reduce the spatial derivative. We split it according to the values of $j$, obtaining \begin{align} 	\begin{split} 		\mathcal{P}_{31} 		& 		= 		\sum_{m=3}^\infty \frac{\tau^{m-2}}{(m-2)!}  \Vert E (\epsilon \partial_t)^{m+1} u \Vert_{L^2} 		+ 		\sum_{m=3}^\infty \sum_{\vert \alpha \vert = 1} \sum_{l=0}^1 \sum_{\vert \beta \vert = l, \beta \leq \alpha} \frac{\tau^{m-2}}{(m-2)!}  \Vert \partial^\beta E \partial^{\alpha -\beta} (\epsilon \partial_t)^{m} u \Vert_{L^2}\\ 		&\indeq 		+ 		\sum_{m=3}^\infty \sum_{\vert \alpha \vert = 2} \sum_{l=0}^2 \sum_{\vert \beta \vert = l, \beta \leq \alpha} \frac{\tau^{m-2}}{(m-2)!}  \Vert \partial^\beta E \partial^{\alpha -\beta} (\epsilon \partial_t)^{m-1} u \Vert_{L^2}\\ 		& 		\leq 		C\Vert u \Vert_{A_1(\tau)}  		+ 		C \sum_{m=2}^\infty \Vert \nabla (\epsilon \partial_t)^m u \Vert_{L^2} \frac{\tau^{m-2}}{(m-2)!} 		+ 		C \sum_{m=3}^\infty \Vert \nabla^2 (\epsilon \partial_t)^{m-1} u \Vert_{L^2} \frac{\tau^{m-2}}{(m-2)!} 		. 		\label{DFGRTHVBSDFRGDFGNCVBSDFGDHDFHDFNCVBDSFGSDFGDSFBDVNCXVBSDFGSDFHDFGHDFTSADFASDFSADFASDXCVZXVSDGHFDGHVBCX365} 	\end{split} \end{align} The first term on the far side is bounded by $C + t Q(M_{\epsilon, \kappa}(t))$ by Lemma~\ref{L09}, while the  second and third terms can be estimated analogously to \eqref{DFGRTHVBSDFRGDFGNCVBSDFGDHDFHDFNCVBDSFGSDFGDSFBDVNCXVBSDFGSDFHDFGHDFTSADFASDFSADFASDXCVZXVSDGHFDGHVBCX84}--\eqref{DFGRTHVBSDFRGDFGNCVBSDFGDHDFHDFNCVBDSFGSDFGDSFBDVNCXVBSDFGSDFHDFGHDFTSADFASDFSADFASDXCVZXVSDGHFDGHVBCX365}. Combining the resulting inequalities, we obtain        \begin{align}        \mathcal{P}_3        \leq        C + \left( t+ \epsilon + \tau \right) Q(M_{\epsilon, \kappa}(t)),    \llabel{ Q8f vSO 0I vBix TXUH ooXi2R fl O tql 7Yd ECN Nd YgMs TVmZ EyfsTm 84 Z Xte 4k3 ON2 Ze cVTC OtKg mifZlr 6g Y i6z yNU yHF KS r4sE 8yjR tqJIWz ef E DZs Nbe VAc 04 Pcuw CEJo N38qw7 s2 e EmP 3aa 0j8 ka lF2z zeY5 g2LYeA q9 o 0du D6R iqN gl sdtL KNDI nhmvdz 7w w TLf GFF J7N 8y pFxa Fq72 IPRRDa eo N jaT sfA ATQ W6 nAv8 y4Pg McdnQs 5j y vpU Tpr 8wA PQ ulnk 1d7r n6uVUb cC 8 Rwk LPE Gpf aV tl4x nncU gDqsd8 v8 p sr7 8rF BXy Ig rPnu OwmJ XttPrG R0 v V5k mkp 6Gg a5 kVwN sVZ3 63PdPU xP 5 o5u 2i5 CGt df CRWi 6LLr EzIdBN xJ 0 vGI ZID dPH HX 3e9z pNFv nEAS15 cq k d2C eYj Mx5 Zr E2dN qaaG 4TgrZI rS w wau myM ZhR Ug RYxR 3JiA lnVJpc ZF j XEx qjM TBf GG JQjE 2DNG w1itGo Hn l goa XYX lNd cu Bi31 RR40 4GhUbH 7P T qg8 O0M Q2O Kb XROs ixwl ydyJDA dD 6 tSM yve yI2 St J5r3 fRAN Ywyvbr iX k c4G zlX xMr yx FQeN uHS5 tD9JjC DP l V9K GI0 5nT ri Ugbl nkLT gMCCyL zK E Di6 zAw Mbd k2 5LnH 3Hs6 TdaXLy wa N qAk 98v gGK FG HAZr OJpQ N4xYzm rP 3 Vev 7wl efc B8 AzGJ D3iP GqAhKP Al G RSC aOl F7L RC mOLG Fe5q DFlcE4 lS 8 ODD LvU FET Dv 30DN UYrw jF1zBs 7e T jSn L3n FsG vz cLRC L5Mb BPpQ49 wM J DBk c80 xEt kY mVP0 NelV uFj7MF DC b vcN e0V 1Rj UB XwUR lZWe aVhbmo Ne g Dfh GIM IVk j3 UYNI QCXX GNkf5T Cw J kkI MdV GQi Ed VDmP fOjG flRHvE wE M 5T3 J9C lJM Gk 00ee rENN NAnKmF zv d g6e Jzp e1P wo wiDFGRTHVBSDFRGDFGNCVBSDFGDHDFHDFNCVBDSFGSDFGDSFBDVNCXVBSDFGSDFHDFGHDFTSADFASDFSADFASDXCVZXVSDGHFDGHVBCX55}        \end{align} and the lemma then follows by \eqref{DFGRTHVBSDFRGDFGNCVBSDFGDHDFHDFNCVBDSFGSDFGDSFBDVNCXVBSDFGSDFHDFGHDFTSADFASDFSADFASDXCVZXVSDGHFDGHVBCX360}--\eqref{DFGRTHVBSDFRGDFGNCVBSDFGDHDFHDFNCVBDSFGSDFGDSFBDVNCXVBSDFGSDFHDFGHDFTSADFASDFSADFASDXCVZXVSDGHFDGHVBCX362}.        \end{proof} \par \begin{proof}[Proof of Lemma~\ref{L01}] The inequality \eqref{DFGRTHVBSDFRGDFGNCVBSDFGDHDFHDFNCVBDSFGSDFGDSFBDVNCXVBSDFGSDFHDFGHDFTSADFASDFSADFASDXCVZXVSDGHFDGHVBCX56} follows by using \eqref{DFGRTHVBSDFRGDFGNCVBSDFGDHDFHDFNCVBDSFGSDFGDSFBDVNCXVBSDFGSDFHDFGHDFTSADFASDFSADFASDXCVZXVSDGHFDGHVBCX19}, \eqref{DFGRTHVBSDFRGDFGNCVBSDFGDHDFHDFNCVBDSFGSDFGDSFBDVNCXVBSDFGSDFHDFGHDFTSADFASDFSADFASDXCVZXVSDGHFDGHVBCX76}, and \eqref{DFGRTHVBSDFRGDFGNCVBSDFGDHDFHDFNCVBDSFGSDFGDSFBDVNCXVBSDFGSDFHDFGHDFTSADFASDFSADFASDXCVZXVSDGHFDGHVBCX105}. \end{proof} \par \startnewsection{The Mach limit}{sec06} \par In this section, we prove the second main theorem on the Mach limit in the space~$X$. \par \begin{proof}[Proof of Theorem~\ref{T02}] Let $\delta>0$ be a small constant, which is to be determined below. For the sake of contradiction,  we assume that $(v^\epsilon, p^\epsilon, S^\epsilon)$  does not converge to $(v^{(\inc)}, 0, S^{(\inc)})$ in $L^2 ([0,T_0], X_\delta)$.  Then there exists a sequence $(v^{\epsilon_n}, p^{\epsilon_n}, S^{\epsilon_n})$ which does not converge to $(v^{(\inc)}, 0, S^{(\inc)})$ in $L^2 ([0,T_0], X_\delta)$ as $\epsilon_n \to 0$.  Recall from \cite[Theorem~1.4]{MS01} that $(v^{\epsilon_n}, p^{\epsilon_n}, S^{\epsilon_n})$ converges to $(v^{(\inc)}, 0, S^{(\inc)})$ in $L^2 ([0,T_0], L^2 (\mathbb{R}^3))$ as $\epsilon_n \to 0$. For $ k,n \in \mathbb{N}$, we define $v_{kn}(t) = v^{\epsilon_k}(t) - v^{\epsilon_n}(t)$. For $m\in \mathbb{N}$ and $ \alpha \in \mathbb{N}^3_0$, using integration by parts and the Cauchy-Schwarz inequality leads to \begin{align} \begin{split} 	\Vert \partial^\alpha v_{kn} \Vert_{L^2_x}^2  	& 	=  	\bigl\langle \partial^\alpha v_{kn}, \partial^\alpha v_{kn} \bigr\rangle 	= 	(-1)^{\vert \alpha \vert} \bigl\langle v_{kn}, \partial^{2\alpha} v_{kn} \bigr\rangle 	\leq 	\Vert v_{kn} \Vert_{L^2_x} \Vert \partial^{2\alpha} v_{kn} \Vert_{L^2_x}. \end{split} \llabel{te aQEE MCD7RZ pA m bWp H6d 6Nu 7F 03lK PPmO ZWiqHT YH 1 cUi bZW 7ar jG RgbE V6cj GXfM98 nJ e wFC 3Y1 tqU fT QoNo VLm1 xWc5IQ Ps 9 zqI HxP giW 0v h1LW ZJwQ APcMtZ ME e Y9c 0oA aCk mY Y9Tl PLTb 10ZkLG IX w A4l v3s 5rH KA TMeX OXNN PsTqYf 3c e HUH fny Klk y5 Ji5k 1Rt5 z0JpRu Az J NQC UAn pvY Ax 0q4T WLIK qiCLxO 61 f 8fq UoO smd t5 SEWS gkH6 rP7J3Q 1A 0 HzC Cb4 Cqo bC SkVk KtzF 1G2bd1 zz f RiY Tdi fkj S3 IPrb 8Nxj SGVYvv 1r M ZNY APJ aNG wX 8cY1 3qh7 5fT2iv 2i E MXV 5Pj Y6m vW vSsm ArR9 9Q4HvJ Lk f Dd6 fF9 c3S 0U ti7E NUGe VbaEOG Ra h 4HG 3fT BXq YM rgNY 8rCq KsFJce FX 6 W8Z gSo XKB Py FpiY WwUf PGUz7R gb J qgr exS HdF W9 054n 5Bc3 d43rrf ew P oxZ GQd JJX Po XK9c GKyH oU1jWA Dy Y I5H aKU xXo dt 3sFG QHNJ qXz2DF xO z 57t ACk Wg2 z5 Hf9y LNpV r9cCYL cr q W5p 5Hs vxu sa IFGU wzjd 8bcw06 nm c YPf RnO XXY Dq MRf3 Echd Ran40L 9D N WiE kki bzL 0L 7WBC y7WM OTRxLZ fN A QEl GN8 fYM sw Ausc wuu8 yQednd XO k ddM dlE 5OB LS WPpI cNqQ AKldWl BK w umR ZaL COa 2i WZZw NQ6w ef7TP3 KN A M2Z 49G 4qT Xj 0Q38 AsqU CvEbIv qF M GZy JDi FXh IV 9q7T 4Kvq mmwoiw 4S U 1IH L9L UGs ka 6BIZ pyiQ YASKZk gu 7 zep fhB hz6 uv w0hq k4Yh 7ISUoP zA o icD eof JnO 6w 6Zta 72Aj eNjbZV p1 a vY4 Xm0 0Fi xg 2yrH 1nB5 ae8FhY 81 3 uYj ber wjX yI loLP 9dL5 tlM4hR W6 b Irg iU6 DnZ Na OhQQ ULtK e2XgSMDFGRTHVBSDFRGDFGNCVBSDFGDHDFHDFNCVBDSFGSDFGDSFBDVNCXVBSDFGSDFHDFGHDFTSADFASDFSADFASDXCVZXVSDGHFDGHVBCX201} \end{align} Summing over $\vert \alpha \vert = m$ with $m\in \mathbb{N}$ such that $m\geq 4$ and using the Minkowski and H\"older inequality, we obtain \begin{align} \begin{split} 	& 	\sum_{m=4}^\infty \sum_{\vert \alpha \vert = m} \Vert \partial^\alpha v_{kn} \Vert_{L^2_{x,t}} \frac{\delta^{(m-3)_+}}{(m-3)!} 	\leq 	\Vert v_{kn} \Vert_{L^2_{x,t}}^{1/2} \sum_{m=4}^\infty \sum_{\vert \alpha \vert = m} \Vert \partial^{2\alpha} v_{kn} \Vert_{L^2_{x,t}}^{1/2} \frac{\delta^{(m-3)_+}}{(m-3)!}\\ 	&\indeq 	= 	\Vert v_{kn} \Vert_{L^2_{x,t}}^{1/2} \sum_{m=4}^\infty \sum_{\vert \alpha \vert = m} \left( \Vert \partial^{2\alpha} v_{kn} \Vert_{L^2_{x,t}} \frac{\kappa^{(2m-3)_+} \tau^{(2m-3)_+}}{(2m-3)!} \right)^{1/2} \frac{(2m-3)!^{1/2}\delta^{(m-3)_+}}{(m-3)!\kappa^{(2m-3)_+/2}\tau^{(2m-3)_+/2}} 	\\&\indeq 	\leq 	C M^{1/2}\Vert v_{kn} \Vert_{L^2_{x,t}}^{1/2}  	\sum_{m=4}^\infty \sum_{\vert \alpha \vert = m}  \frac{(2m-3)!^{1/2}\delta^{(m-3)_+}}{(m-3)!\kappa^{(2m-3)_+/2}\tau^{(2m-3)_+/2}} 	\\&\indeq 	\leq 	M^{1/2}\Vert v_{kn} \Vert_{L^2_{x,t}}^{1/2}  	\sum_{m=4}^\infty  \frac{C^{m}(2m-3)!^{1/2}\delta^{(m-3)_+}}{(m-3)!\kappa^{(2m-3)_+/2}\tau^{(2m-3)_+/2}} 	, \end{split} \label{DFGRTHVBSDFRGDFGNCVBSDFGDHDFHDFNCVBDSFGSDFGDSFBDVNCXVBSDFGSDFHDFGHDFTSADFASDFSADFASDXCVZXVSDGHFDGHVBCX206} \end{align} where $C>0$ is a fixed universal constant and $M$ is as in \eqref{DFGRTHVBSDFRGDFGNCVBSDFGDHDFHDFNCVBDSFGSDFGDSFBDVNCXVBSDFGSDFHDFGHDFTSADFASDFSADFASDXCVZXVSDGHFDGHVBCX17}. Now choose $\delta>0$ sufficiently small so that $\delta/ \kappa \tau \leq 1/C C_0$ on the whole time interval~$[0,T_0]$, where $C_0$ is sufficiently  large, and obtain \begin{align} \begin{split} 	\sum_{m=4}^\infty \sum_{\vert \alpha \vert = m}  	\Vert \partial^\alpha v_{kn} \Vert_{L^2_{x,t}} \frac{\delta^{(m-3)_+}}{(m-3)!} 	& 	\leq 	M^{1/2} 	\delta^{-3/2} 	\Vert v_{kn}  \Vert_{L^2_{x,t}}^{1/2}  	\sum_{m=4}^\infty  \frac{(2m-3)!^{1/2}}{C_0^{m}(m-3)!} 	\leq   	M^{1/2}        	\delta^{-3/2} 	\Vert v_{kn} \Vert_{L^2_{x,t}}^{1/2}  	, \end{split} \label{DFGRTHVBSDFRGDFGNCVBSDFGDHDFHDFNCVBDSFGSDFGDSFBDVNCXVBSDFGSDFHDFGHDFTSADFASDFSADFASDXCVZXVSDGHFDGHVBCX94} \end{align} where we used Stirling's formula and assumed $C_0$ to be sufficiently large so the sum converges. Analogously, we set $p_{kn}(t) = p^{\epsilon_k}(t) - p^{\epsilon_n}(t)$ and $S_{kn}(t) = S^{\epsilon_k}(t) - S^{\epsilon_n}(t)$ and proceed as in above obtaining \begin{align} \sum_{m=4}^\infty \sum_{\vert \alpha \vert = m} \Vert \partial^\alpha p_{kn} \Vert_{L^2_{x,t}} \frac{\delta^{m-3}}{(m-3)!} \leq M^{1/2} \delta^{-3/2} \Vert v_{kn} \Vert_{L^2_{x,t}}^{1/2}  , \label{DFGRTHVBSDFRGDFGNCVBSDFGDHDFHDFNCVBDSFGSDFGDSFBDVNCXVBSDFGSDFHDFGHDFTSADFASDFSADFASDXCVZXVSDGHFDGHVBCX207} \end{align} and  \begin{align} \sum_{m=4}^\infty \sum_{\vert \alpha \vert = m} \Vert \partial^\alpha S_{kn} \Vert_{L^2} \frac{\delta^{m-3}}{(m-3)!} \leq M^{1/2} \delta^{-3/2} \Vert S_{kn} \Vert_{L^2}^{1/2}  . \label{DFGRTHVBSDFRGDFGNCVBSDFGDHDFHDFNCVBDSFGSDFGDSFBDVNCXVBSDFGSDFHDFGHDFTSADFASDFSADFASDXCVZXVSDGHFDGHVBCX208} \end{align} Note also that $\Vert v_{kn}\Vert_{L^2_t H^{3}_x}^2 \leq C  \Vert v_{kn}\Vert_{L^2_t H^{4}_x}^{3/2}\Vert v_{kn}\Vert_{L^2}^{1/2}  \leq C\Vert v_{kn}\Vert_{L^2_{x,t}}^{1/2}$, by Remark~\ref{R04}, with analogous inequalities for $p_{kn}$ and $S_{kn}$. Since $M$ and $\delta$ are fixed constants, we infer  from \eqref{DFGRTHVBSDFRGDFGNCVBSDFGDHDFHDFNCVBDSFGSDFGDSFBDVNCXVBSDFGSDFHDFGHDFTSADFASDFSADFASDXCVZXVSDGHFDGHVBCX94}--\eqref{DFGRTHVBSDFRGDFGNCVBSDFGDHDFHDFNCVBDSFGSDFGDSFBDVNCXVBSDFGSDFHDFGHDFTSADFASDFSADFASDXCVZXVSDGHFDGHVBCX208} that the sequence $\{(v^{\epsilon_n}, p^{\epsilon_n}, S^{\epsilon_n})\}$ is Cauchy in $L^2([0,T_0], X_\delta)$ which implies that it converges in $L^2([0,T_0], X_\delta)$, which is a contradiction.  Therefore, $\{(v^\epsilon, p^\epsilon, S^\epsilon)\}$ is convergent  and converges to $(v^{(\inc)}, 0, S^{(\inc)})$ in $L^2([0,T_0], X_\delta)$ as $\epsilon \to 0$. \end{proof} \par \startnewsection{Analyticity assumptions on the initial data}{secinitial} In this section, we  assume that  the initial data satisfies \eqref{DFGRTHVBSDFRGDFGNCVBSDFGDHDFHDFNCVBDSFGSDFGDSFBDVNCXVBSDFGSDFHDFGHDFTSADFASDFSADFASDXCVZXVSDGHFDGHVBCX36}, and intend to prove that for smaller $n$ we have  	\begin{align} 	& 	\sum_{n=0}^3	\sum_{j=0}^\infty \sum_{\vert \alpha \vert = j} 	\Vert \partial^\alpha (\epsilon \partial_t)^n u(0) \Vert_{L^2}  	\frac{ \tau_0^{(j+n-3)_+}}{(j+n-3)!} 	\leq 	\Gamma 	, 	\label{DFGRTHVBSDFRGDFGNCVBSDFGDHDFHDFNCVBDSFGSDFGDSFBDVNCXVBSDFGSDFHDFGHDFTSADFASDFSADFASDXCVZXVSDGHFDGHVBCX560} 	\end{align} and 	\begin{align} 	& 	\sum_{n=0}^3	\sum_{j=0}^\infty \sum_{\vert \alpha \vert = j} 	\Vert \partial^\alpha (\epsilon \partial_t)^n S(0) \Vert_{L^2}  	\frac{\tau_0^{(j+n-3)_+}}{(j+n-3)!} 	\leq 	\Gamma 	, 	\label{DFGRTHVBSDFRGDFGNCVBSDFGDHDFHDFNCVBDSFGSDFGDSFBDVNCXVBSDFGSDFHDFGHDFTSADFASDFSADFASDXCVZXVSDGHFDGHVBCX561} 	\end{align} where $\Gamma>0$ is a sufficiently large constant depending on $M_0$; for larger values of $n$, we claim that there exists a sufficiently small parameter $\lambda>0$  depending on $M_0$, such that for all $k\geq 4$ we have 	\begin{align} 	& 	\sum_{n=4}^k	\sum_{j=0}^\infty \sum_{\vert \alpha \vert = j} 	\Vert \partial^\alpha (\epsilon \partial_t)^n u(0) \Vert_{L^2}  	\frac{\lambda^{n-3} \tau_0^{(j+n-3)_+}}{(j+n-3)!} 	\leq 	1 	\label{DFGRTHVBSDFRGDFGNCVBSDFGDHDFHDFNCVBDSFGSDFGDSFBDVNCXVBSDFGSDFHDFGHDFTSADFASDFSADFASDXCVZXVSDGHFDGHVBCX510} 	\end{align} and        \begin{align} 	& 	\sum_{n=4}^k \sum_{j=0}^\infty \sum_{\vert \alpha \vert = j} 	\Vert \partial^\alpha (\epsilon \partial_t)^n S(0) \Vert_{L^2}  	\frac{\lambda^{n-3} \tau_0^{(j+n-3)_+}}{(j+n-3)!} 	\leq 	1	 	. 	\label{DFGRTHVBSDFRGDFGNCVBSDFGDHDFHDFNCVBDSFGSDFGDSFBDVNCXVBSDFGSDFHDFGHDFTSADFASDFSADFASDXCVZXVSDGHFDGHVBCX511} 	\end{align} In \eqref{DFGRTHVBSDFRGDFGNCVBSDFGDHDFHDFNCVBDSFGSDFGDSFBDVNCXVBSDFGSDFHDFGHDFTSADFASDFSADFASDXCVZXVSDGHFDGHVBCX510} and \eqref{DFGRTHVBSDFRGDFGNCVBSDFGDHDFHDFNCVBDSFGSDFGDSFBDVNCXVBSDFGSDFHDFGHDFTSADFASDFSADFASDXCVZXVSDGHFDGHVBCX511} we then choose  $\tilde{\tau}_0 = \lambda \tau_0 /2$ and using \eqref{DFGRTHVBSDFRGDFGNCVBSDFGDHDFHDFNCVBDSFGSDFGDSFBDVNCXVBSDFGSDFHDFGHDFTSADFASDFSADFASDXCVZXVSDGHFDGHVBCX560}--\eqref{DFGRTHVBSDFRGDFGNCVBSDFGDHDFHDFNCVBDSFGSDFGDSFBDVNCXVBSDFGSDFHDFGHDFTSADFASDFSADFASDXCVZXVSDGHFDGHVBCX511}, we get 	\begin{align} 	\begin{split} 	\Vert (p_0, v_0, S_0) \Vert_{A(\tilde{\tau}_0)} 	& 	= 	\sum_{n=0}^\infty \sum_{j=0, n+j\geq 1}^\infty \sum_{\vert \alpha \vert = j} 	\Vert \partial^\alpha (\epsilon \partial_t)^n  (u, S)(0)\Vert_{L^2} 	\frac{ \tilde{\tau}_0^{(j+n-3)_+}}{(j+n-3)!} 	\\ 	& 	\leq 	\sum_{n=0}^3 	\sum_{j=0}^\infty \sum_{\vert \alpha \vert = j} 	\Vert \partial^\alpha (\epsilon \partial_t)^n (u, S)(0) \Vert_{L^2} 	\frac{\tau_0^{(j+n-3)_+}}{(j+n-3)!} 	\\ 	&\indeq 	+ 	\sum_{n=4}^\infty  	\frac{1}{2^{n-3}} 	\sum_{j=0}^\infty \sum_{\vert \alpha \vert = j} 	\Vert \partial^\alpha (\epsilon \partial_t)^n (u, S)(0) \Vert_{L^2} 	\frac{\lambda^{n-3} \tau_0^{(j+n-3)_+}}{(j+n-3)!} 	\leq 	\Gamma 	+ 	\sum_{n=4}^\infty \frac{1}{2^{n-3}} 	\leq 	C 	, 	\end{split}    \llabel{ z1 1 h7b RiT 8Ly 9w U9Hu w7D2 x80TDp Sr B NsU dBW i1Y P1 xBcZ lA4I tnhR1L Od c Rt7 42s 9Wf r5 RM92 xZyF AG3lDX yL 7 M2O A8a yjH 6V 8RLj cv7x NNRblQ 79 5 qD9 WBe Dpb uV xBWf PncT 3LceZG wJ A WKt CWS zth Gp HNHo lqre bIsUCK Ty B ykp S5s TTD jj qake bcrK eVLZg5 HP V 1jS 9o7 2iy cM 8hGY 23oH AWyfBU Cs i nGk 1nc WZT ig 5shm TsC0 lDPDqc Pp t OuS qct YcY hK 6uQl 3WMP xsaBR5 Qm f huH y67 wL2 Rx wBHL wXuO Odf78Z fh k 0Fu zb1 OKJ y8 ipeW 8knK mTon5A nX 7 GqH exO YHp VH GZeO vyW3 VBgfYN RX b uTQ vfR AzH Nm luXb BxMs pgn4uh qx D 4Nn 5Zd yTZ lD 9jii gTcM vBUM1t ZH I 0Dd Fzj reB yV AI6D me58 SY4cWB Sa S rhg z7h 4w7 O7 9ZYy MKxl sIbGM8 AI Y 0wH M5f rjB yf 311Q kYIC GteOS7 dD c j7k P92 Pfd Rs FD75 KJaO 7vYnmC lO N hJo ZuK qKO f0 MKPT iTaI uRqhU5 EJ t u5M Z9N UJ4 uw JGZG 4hF1 7EdkpE 5l 6 uZi t9Y GCi 1P wKiE Fesj wo3QoJ HZ M Ygw G2B Ojp He J3rs 5Pi8 Fzkx8Y Dg O 3ov Gua Nhp XG EneF L8Et ljhTJO dQ x XFp dOL Sv9 fp BaUs vwR2 L2lVEz FN u CNk M8x Fah Hv ZhEn QpOP zVGYCk 36 A YhS xLk WSg T5 wSwk Vjjm VCuMlR CV 1 LJ1 Lxy bqt 6H vHpo 6ExR 3SVQsW JJ S Uw7 aZS 4Oz 58 7MuX sKwt My4h1g vq W OBX JqY 58o Gk F4fI RwHE pAZ60s dp Z QuA e02 vXZ 0u 0YsV ZMRE YGpzwT 6W s 6kb RM1 nu3 dw VlxI svl7 snLBTT N5 E Qaf c4i Xr0 IQ cqTe vf7O b20B6t Pt a SlH uRE 3xb ZC RGKX xyaz u8XEKR zC M SVx W5D DFGRTHVBSDFRGDFGNCVBSDFGDHDFHDFNCVBDSFGSDFGDSFBDVNCXVBSDFGSDFHDFGHDFTSADFASDFSADFASDXCVZXVSDGHFDGHVBCX186} 	\end{align} obtaining \eqref{DFGRTHVBSDFRGDFGNCVBSDFGDHDFHDFNCVBDSFGSDFGDSFBDVNCXVBSDFGSDFHDFGHDFTSADFASDFSADFASDXCVZXVSDGHFDGHVBCX166}.  In the remainder of this section, we prove \eqref{DFGRTHVBSDFRGDFGNCVBSDFGDHDFHDFNCVBDSFGSDFGDSFBDVNCXVBSDFGSDFHDFGHDFTSADFASDFSADFASDXCVZXVSDGHFDGHVBCX560}--\eqref{DFGRTHVBSDFRGDFGNCVBSDFGDHDFHDFNCVBDSFGSDFGDSFBDVNCXVBSDFGSDFHDFGHDFTSADFASDFSADFASDXCVZXVSDGHFDGHVBCX511}. \par For $n=0$, we use the assumption  \eqref{DFGRTHVBSDFRGDFGNCVBSDFGDHDFHDFNCVBDSFGSDFGDSFBDVNCXVBSDFGSDFHDFGHDFTSADFASDFSADFASDXCVZXVSDGHFDGHVBCX36} on the initial data to obtain 	\begin{align} 	\sum_{j=0}^\infty \sum_{\vert \alpha \vert = j} 	\Vert \partial^\alpha (u, S)(0) \Vert_{L^2}  	\frac{\tau_0^{(j-3)_+}}{(j-3)!} 	\leq 	\Gamma_0 	, 	\label{DFGRTHVBSDFRGDFGNCVBSDFGDHDFHDFNCVBDSFGSDFGDSFBDVNCXVBSDFGSDFHDFGHDFTSADFASDFSADFASDXCVZXVSDGHFDGHVBCX550} 	\end{align} for some constant $\Gamma_0 >0$. Next, for $n=1$, we apply $\partial^\alpha$ to \eqref{DFGRTHVBSDFRGDFGNCVBSDFGDHDFHDFNCVBDSFGSDFGDSFBDVNCXVBSDFGSDFHDFGHDFTSADFASDFSADFASDXCVZXVSDGHFDGHVBCX02} where $\vert \alpha \vert = j\in{\mathbb N}_0$, which leads to 	\begin{align} 	\partial^\alpha \partial_t S = -\sum_{\beta \leq \alpha} \binom{\alpha}{\beta} \partial^\beta v \cdot \partial^{\alpha - \beta} \nabla S         .    \llabel{D8h 5O 2TOE oLOa 9Xnni0 P9 M PVn vRl MHM z9 ktYH YuJw Wqj4m7 h1 i 364 30c gmt 1p 1Shm 0Uaz sD8tmm 6p 2 bXM 7Di hSr nd 9sqd N1OU b0VLbf YD z DXU 930 cyC 6Y qfWq Ee4H eZMNg9 wd 2 yon leH Nqm ml lxgz nC28 iQ2zbF 68 m uRk 0Tv BSk 9d LE46 kTqc 93PZP2 Jn E T1M pgt Rie 0q z9UV dqm6 D5PaqR 6w j Sx1 zD1 BrV xE c1RL 7FUd 69kjGE me f LVA Igh g11 8y Tyxf Jw9x RcephR Tw p ucP KuZ kzl S3 dEco 32xr gE8lnc C6 d djb VAe U9j yF zdge C9Tu UDjxKt v0 Y 4cZ BT2 03i Cw G1vU BujR rosbyZ Yu 1 f0b OAu WiV zE fOsb RCCk kU4hYM pc L 47j DoR qc4 4Z Ze8L OtAP 7gfaEe DP A RT5 7Me UfV 0P fmtT nxgf SaBPn8 cC Z XbY bsC WQg he heTc uFhK GK1IMQ tX W CIq xyw ylt ID 7aCO v8rL 0u8Y07 PR k PD2 HQX RgH j2 84iH MosI hTXwEo um 3 eYf R4U KKT cg ZKDh czP2 09UHEZ 34 o TdH mfR KW3 SI tIc1 X38y zbYIhw Q1 R cgg seU UTZ Tv 1Ra0 m2RJ uQLeCx wh a KEa TSV Uio gD Uxgn zqsW jV7PV5 Eu B 2L0 9cz ZuE YH urdi PITA A3ngcZ m0 a yab Zvj 0MO dU 4O8g L0up zcRWG1 0G E 3CC WV8 6zA 6z MJLf zf8x zXPsRD Oy C Mqh wua 6xz Lo 1BgM UhYs KHfuwo tK e O6s eTV 9tM Fi Qai8 ZtJg wIY56g Pu u 1yT LHG o9R 1f qtcY zybk TePkid 8S A hJ6 8KS mFE uu CuLU c3Au nELyIA YR e DzS qH8 cqI vl MwYV 3oN9 qwEhdX 2g Q dz3 vhJ 0Si 5B GV6b SPYw 7nLBPi IT t RoJ 1R9 dde bw czqq qWtF popMhN Df O kF3 4Yb cGd 6C Sb02 H5Cx 2JcMxJ Sk n V8c A9r Sdr Da KEJU vTDFGRTHVBSDFRGDFGNCVBSDFGDHDFHDFNCVBDSFGSDFGDSFBDVNCXVBSDFGSDFHDFGHDFTSADFASDFSADFASDXCVZXVSDGHFDGHVBCX187} 	\end{align} Therefore, 	\begin{align} 	\begin{split} 	& 	\sum_{j=0}^\infty \sum_{\vert \alpha \vert = j}  	\Vert \partial^\alpha \epsilon \partial_t S \Vert_{L^2} 	\frac{\tau_0^{(j-2)_+}}{(j-2)!} 	\\ 	&\indeq\indeq 	\leq 	C\epsilon \sum_{j=0}^\infty \sum_{\vert \alpha \vert = j} \sum_{l=0}^j \sum_{\beta \leq \alpha, \vert \beta \vert = l}  	\binom{\alpha}{\beta}  	\Vert \partial^\beta v \cdot \partial^{\alpha -\beta} \nabla S \Vert_{L^2} 	\frac{\tau_0^{(j-2)_+}}{(j-2)!}         . 	\label{DFGRTHVBSDFRGDFGNCVBSDFGDHDFHDFNCVBDSFGSDFGDSFBDVNCXVBSDFGSDFHDFGHDFTSADFASDFSADFASDXCVZXVSDGHFDGHVBCX524} 	\end{split} 	\end{align} We split the right side of \eqref{DFGRTHVBSDFRGDFGNCVBSDFGDHDFHDFNCVBDSFGSDFGDSFBDVNCXVBSDFGSDFHDFGHDFTSADFASDFSADFASDXCVZXVSDGHFDGHVBCX524} according to low and high values $l$.  By H\"older and Sobolev inequalities, we have 	\begin{align} 	\begin{split} 	& 	\sum_{j=0}^\infty \sum_{\vert \alpha \vert = j}  	\Vert \partial^\alpha \epsilon \partial_t S \Vert_{L^2} 	\frac{\tau_0^{(j-2)_+}}{(j-2)!} 	\\ 	&\indeq 	\leq 	C\epsilon \sum_{j=0}^\infty \sum_{\vert \alpha \vert =j} \sum_{0\leq l\leq [j/2]} \sum_{\beta \leq \alpha, \vert \beta \vert = l}  	\left( \Vert D^2 \partial^\beta v \Vert_{L^2} \frac{\tau_0^{(l-1)_+}}{(l-1)!} 	\right)^{3/4} 	\left( \Vert \partial^\beta v \Vert_{L^2} \frac{\tau_0^{(l-3)_+}}{(l-3)!} 	\right)^{1/4} 	\\ 	& 	\indeqtimes 	\left( \Vert \partial^{\alpha-\beta} \nabla S \Vert_{L^2} \frac{\tau_0^{(j-l-2)_+}}{(j-l-2)!}  	\right) 	\frac{(l-1)!^{3/4} (l-3)!^{1/4} (j-l-2)! j!}{(j-2)!(j-l)!l!} 	\\ 	&\indeq\indeq 	+ 	C\epsilon \sum_{j=0}^\infty \sum_{\vert \alpha \vert =j} \sum_{[j/2]+1\leq l\leq j} \sum_{\beta \leq \alpha, \vert \beta \vert = l}  	\left( \Vert \partial^{\beta} v \Vert_{L^2} \frac{\tau_0^{(l-3)_+}}{(l-3)!} 	\right) 	\left( \Vert D^2 \partial^{\alpha-\beta} \nabla S \Vert_{L^2} \frac{\tau_0^{(j-l)_+}}{(j-l)!}  	\right)^{3/4} 	\\ 	& 	\indeqtimes 	\left( \Vert \partial^{\alpha-\beta} \nabla S \Vert_{L^2} \frac{\tau_0^{(j-l-2)_+}}{(j-l-2)!} 	\right)^{1/4} 	\frac{(j-l)!^{3/4} (j-l-2)!^{1/4} (l-3)! j!}{(j-2)!(j-l)!l!} 	. 	\label{DFGRTHVBSDFRGDFGNCVBSDFGDHDFHDFNCVBDSFGSDFGDSFBDVNCXVBSDFGSDFHDFGHDFTSADFASDFSADFASDXCVZXVSDGHFDGHVBCX521} 	\end{split} 	\end{align} One may check that 	\begin{align} 	\frac{(l-1)!^{3/4} (l-3)!^{1/4} (j-l-2)! j!}{(j-2)!(j-l)!l!} 	\leq 	 C 	 , 	 \llabel{Gf qZ2u96 5w b LTQ rc8 rAT C6 el8Q ybVb 1hPbs4 TZ G 5R2 06Z GIE 1l ltPs CGAW w5Zg04 Wl L gZg BXf EVq pY zy9H y2jh DJCUoM CV k 25F joe bTR na xVMD yg9H 1F8JW0 tN D q2u SHl Y1o Dn HJ0c ca6l 3ntftL sL i Dbz qtO ODm ef ASJj k7YX hKoYyf YV c XXD ZI2 eRa hN f5Pp KnNW pgg9tg 7B W U4y uAe UeC pM P11c 9XBz cGCb4W Kh 4 KOs tO4 lHa Gz jWLv p2k1 dL83ma Jq k QP7 Z6K bIb zx DQ4N Fadi 5W3dUI HP M Mfr REU J0I j4 f7wF QTq0 isbfI1 aj r 3jM wtn Nvb RP TcUh Cfvh 05YdEH 6r q MHv IvD VdD s8 iJnV vZSC PcsAGj T1 q djy 4jw PCG qf JAE1 HvJG qfyYJC Pq n Yib ZPD q0a 53 MaT8 K3lW s3pZaB 1g 3 Gtv dgq 3hm Vz S3i2 6Kd5 j5uJmT N7 q 7bd lcN XWe xN SGSF 5MfO yt58V7 ha h iOJ coB llX kN md8F njAm 4yadvW dP z qyu olc Vog xL RPuM Zd08 vXAPZR 1v T SqM 2lX VJ4 KL THL3 b8co 2L69zG l8 s IPf ssK wwF 68 Ad4f ejxz iYuMMX TN 4 AgG G8T fu4 iH wNgw yJLM JUzogN kx o gvN yPs bAl 1O Cnqg wErP pSYyCH gc 7 xY8 dPS kSf i0 vEOz yOER SfIDt5 2c w q3c wbt Cbw ZV YXFY LUZV yp9IRV Af w VIC KPS zyl hl QEhU F5jt pdFyJT zI Q xZD TB5 Jy3 zv dWNT Q21d GtbOB8 9y I d00 4sn knQ Ep PhdT 6rE0 3fWcak J5 a Psv s4X CSN gl h7tG ZE7h Mi6upy ph d k4o bPE tIn oX esiL 3WQi EepSQf Sc K XUD US3 UKV 0U vwpt B8Yq HR56TF Bw T DJ6 UfK sLM yR bKdA IhP1 1CGBMD At G bX0 QxT N5O yM ETyE 9sp5 BpB6vJ wy N 0dM 7uj mSh LG o7S7 GIXt WAmAyU yt uDFGRTHVBSDFRGDFGNCVBSDFGDHDFHDFNCVBDSFGSDFGDSFBDVNCXVBSDFGSDFHDFGHDFTSADFASDFSADFASDXCVZXVSDGHFDGHVBCX522} 	\end{align} for $l \leq [j/2]$, while 	\begin{align} 	\frac{(j-l)!^{3/4} (j-l-2)!^{1/4} (l-3)! j!}{(j-2)!(j-l)!l!} 	\leq  	C         , 	\label{DFGRTHVBSDFRGDFGNCVBSDFGDHDFHDFNCVBDSFGSDFGDSFBDVNCXVBSDFGSDFHDFGHDFTSADFASDFSADFASDXCVZXVSDGHFDGHVBCX523} 	\end{align} for $l \geq [j/2]+1$. Collecting the estimates \eqref{DFGRTHVBSDFRGDFGNCVBSDFGDHDFHDFNCVBDSFGSDFGDSFBDVNCXVBSDFGSDFHDFGHDFTSADFASDFSADFASDXCVZXVSDGHFDGHVBCX550} and \eqref{DFGRTHVBSDFRGDFGNCVBSDFGDHDFHDFNCVBDSFGSDFGDSFBDVNCXVBSDFGSDFHDFGHDFTSADFASDFSADFASDXCVZXVSDGHFDGHVBCX521}--\eqref{DFGRTHVBSDFRGDFGNCVBSDFGDHDFHDFNCVBDSFGSDFGDSFBDVNCXVBSDFGSDFHDFGHDFTSADFASDFSADFASDXCVZXVSDGHFDGHVBCX523}, we obtain 	\begin{align} 	\sum_{j=0}^\infty \sum_{\vert \alpha \vert = j}  	\Vert \partial^\alpha \epsilon \partial_t S \Vert_{L^2} 	\frac{\tau_0^{(j-2)_+}}{(j-2)!} 	\leq     C( \Vert v \Vert_{A_0(\tau_0)} + \Vert v \Vert_{L^2}) \Vert S \Vert_{A_0(\tau_0)}      \leq     \Gamma_1 	, 	\llabel{ tBg CdP t9a gc yrAO gPZR ecCrvr 6I k udt pAw adX 1k SoF1 0f3C Ra71Yj 29 n oBH PYu lF6 Mr Z7LI ean2 2rtcaa hK S b1X QyI oBv 0H PVj4 VNx7 u7E0xI CS M M2t Xfv N0x Ku SBEY 77pe NCyPLA UU j sBD iEt LL4 8q QHRi 8DHl MkSfpq cs S Fhx 5FY elv z1 aMqT 4qwK 4vviuy TK x 19I KQP PsU hG Vlpb XHED ogEgqo Gp T Qdz r8X KIS KM Wk5Y iGhw 1B30YE JU C 4P1 h04 7k0 pL lJ2Z I1Nb latvBG Y8 0 S91 kpm 114 jV wqWA gzxo LyzcVF sb 8 5uy mbN PEd Fm 2sxy LRQS TJEsFn 8d R D5w 9Qe 5iw jT 5qJD Z4Lg 9jakVb Bs q EVS x6r FLb Zx i8kC KGfU 760EMy 4f 2 4uC Ja8 stq fY z76a qSiB duMEcp 4b B NKB W85 Z0f Zr iQnb K2FJ 6LEuYu Mp j 4Fm Uyr uZH sb dMix V9Hd DDkFK4 8m G gDK EW0 k9u eb tJeh NB7g mthgi6 QW m kpq 0BR X33 XI S020 8tUE POOQsp HN R CEX Exk 4S4 lE MJ3p FnaS Vsy2Me 44 7 jjt PcF gI9 d2 sA1y CyD0 wd3Mai SB N ARE GXK QOf Ti 48cC 1OXL zTRyhm Qd L x6o 3bM LqA X6 OGny A0Pi VIemDe z6 u 40R BMz rrL Yp mAIP VhkH Y3NuQO RG 3 LYU 7Ke BnP vq iySV QZvl MsaCVI BK V gJT ePP pRf xB Lq9S jLXD z2AaEO 7f w Ytg T9N yTg Vc 3AJI brfn OaVVia Bj U ue8 0Ki twM Q0 Du2i CyPS 20aoKD ih 1 syd 1pM bxd 5h 1kt7 fDil 8v9TDO e4 r SPg ZQa dBG b0 dJwM ungo aGxgxU tb 5 ucS 9wZ uLV a1 GLsq thtf ZD55HW 1Z f 9cj V0J 7p0 Cp uwlj NF54 LbaR8E nv 6 n6X 1yB l4p DM mRdR E2nM f8KHCo 9E 3 bP7 8w3 PnD N6 Qza8 d0hY i5jA6D e9 E F5x sEP dIh YDFGRTHVBSDFRGDFGNCVBSDFGDHDFHDFNCVBDSFGSDFGDSFBDVNCXVBSDFGSDFHDFGHDFTSADFASDFSADFASDXCVZXVSDGHFDGHVBCX525} 	\end{align} where $\Gamma_1 = Q(\Gamma_0)$ and 	\begin{align} 	\Vert u \Vert_{A_0(\tau_0)} 	= 	\sum_{j=1}^\infty \sum_{\vert \alpha \vert = j} 	\Vert \partial^\alpha u \Vert_{L^2} 	\frac{\tau_0^{(j-3)_+}}{(j-3)!} 	.    \llabel{u cwXP NhOD MFgbmm AT S 3hN wWm SmC 5I mVbd tD8r JfS82v ii F oRE Moe VwI a6 9qtN WM9H doLwLt 9z J N3b 8m6 JB3 Ex 3l3v tqm4 z5vTZb Qq B VOJ IXC EaU SB zRHZ zACM TCLFxY 8H g cSs hB9 plt cW RCTo YmM7 dKQCLA kL h m18 i9l QpY xs gdi8 EsYf VzeAXB vZ a 4Ty I70 gB5 wz hiIj 8VQB oFsRsp Za g uMP BkA dsg uJ STqT 7e5p nChyCC hW t 931 rpu ZKj fN JlTP jS1j qjEHNi sT a eyK N9v dXZ 8L NnEa OuJW p9NJmO gY b xJU zPC lD3 aj BmaV 0bM8 sdhs8M j7 f 2zx zms atF dh cvgw uHju J4BL9u K8 Q sTf AHl 00K ss Xk7p 0H8u 5ngTTm Lb X wxP kNe PKf Un YY0K zWaG B0VejT KL f Si0 gKs q5t rP Hvf3 Qiil f8EDzv mF E Io8 yB9 fmE jM txIX CtsC wQeIbW ir r FIz mZg 5EM xM PvtE vtd1 JiHy8Y mI X qhg aew Puf Hr EtPG elIt Jy2zoP Jj p KUF NBr JaA MP aIdR JZe7 1IiTwV IM Z ILy cP4 pOp c1 0lTY ZdtQ 9mnhXQ Ww f iJY uqj G7c LU QNxU gmTI sS9kX8 fO H Ayw kEV zb5 EA DFKp ONnq Oo2naG rU P F86 tRs nn1 Fs zQ3q rhmi UiIcyC Fw W xo3 MKs jXk 0a uWMd Lcsn kelyKA oK L FfR d5K ne7 AM pq01 Gv88 TB5JRB GU t 0X9 aW0 Kkr eZ 23k0 ZVgn xnP87T fD i e68 sBr mPj a1 cBoO FPed yIQl5t Vu 4 RWe tq3 TG7 Fl jxr0 4t2R iMpaaU ng i 4N8 u0T 5Ub v8 Ivge mxQk PAXjWn EK l aDE C7W Dgi UM jZjO o3X6 mBVKJz vZ C aQF b6z ozo ym ECCl ijWz jx5IXU TZ o wb3 0yZ AuR Tp q4f0 xVJp OAVV57 QE 4 9IM gyU tmz l7 8mGk i3Z6 CzRECY oX P iVi GGM ANp Eh 9vw2 35qm aqDFGRTHVBSDFRGDFGNCVBSDFGDHDFHDFNCVBDSFGSDFGDSFBDVNCXVBSDFGSDFHDFGHDFTSADFASDFSADFASDXCVZXVSDGHFDGHVBCX188} 	\end{align} As for \eqref{DFGRTHVBSDFRGDFGNCVBSDFGDHDFHDFNCVBDSFGSDFGDSFBDVNCXVBSDFGSDFHDFGHDFTSADFASDFSADFASDXCVZXVSDGHFDGHVBCX560}, we rewrite the equation \eqref{DFGRTHVBSDFRGDFGNCVBSDFGDHDFHDFNCVBDSFGSDFGDSFBDVNCXVBSDFGSDFHDFGHDFTSADFASDFSADFASDXCVZXVSDGHFDGHVBCX01} as 	\begin{align} 	\epsilon \partial_t u  	=  	- 	\epsilon v \cdot \nabla  u 	- 	\tilde E L(\partial_x) u  	, 	\label{DFGRTHVBSDFRGDFGNCVBSDFGDHDFHDFNCVBDSFGSDFGDSFBDVNCXVBSDFGSDFHDFGHDFTSADFASDFSADFASDXCVZXVSDGHFDGHVBCX501} 	\end{align} where we denoted $\tilde E(S, \epsilon u) = E^{-1}(S, \epsilon u)$. Applying $\partial^\alpha$ to \eqref{DFGRTHVBSDFRGDFGNCVBSDFGDHDFHDFNCVBDSFGSDFGDSFBDVNCXVBSDFGSDFHDFGHDFTSADFASDFSADFASDXCVZXVSDGHFDGHVBCX501},	where $\vert \alpha \vert = j\geq0$, we get 	\begin{align} 	\begin{split} 	\Vert \partial^\alpha \epsilon \partial_t u \Vert_{L^2} 	& 	\leq 	C\epsilon \sum_{l=0}^j \sum_{\beta \leq \alpha, \vert \beta \vert =l}  	\binom{\alpha}{\beta}	 	\Vert \partial^\beta v \cdot \partial^{\alpha -\beta} \nabla u \Vert_{L^2} 	\\ 	&\indeq 	+ 	C\sum_{l=0}^j \sum_{\beta \leq \alpha, \vert \beta \vert =l}  	\binom{\alpha}{\beta}	 	\Vert \partial^\beta \tilde E \partial^{\alpha - \beta} \nabla u \Vert_{L^2} 	, 	\end{split}    \llabel{jNIH 2S e UmE NcG Rss Kg wXn4 zAeI vlB1LC ic T 5hm sKr dIO 3b H67A GV7U MxCGGZ 97 c d7V Pno tBH Y9 TAiv 2scr G1PlLX MQ 6 e7V Je0 jWv Gt luSL Vyxc CSyTNR pb L iKZ YOY C9z 0W OcOz pEpf U4sSOz xn w J6a lIq htK Da 8iJr vXWb PfXSCC uu d Jn8 NAP WT6 e2 TDng GBlT 4HcaWS ye M dNe FTB J5s IP KIaa X2Eu j3lLEu b4 X 0oy tUZ nv3 FT 5R6S 4qQ0 SagCcu jM m tDf aGh fxe 5t aXdI VJQJ egqeCe 6H l 6XI 0wv GRP Qc 9exX bdAV TCWPQL FM J gUi 3Vq oO3 OX mI8h rRxI ccz2h7 Yf 8 8Bj WH4 09L GD dvmW Koqs gj5KdT 26 a rTQ zdv SRp sl 6iSP RyAb VoV8ZY Wg e rDD ray CD6 Oy PIk3 Fyt0 HW0Eyn tB l 6MO Gbj IjJ wx rKUQ 0lIJ xVziJH 6N G Tij m6k Wwi rX 5bPo 2eRS 41GTsK 8B 6 Mev nXf z0D Ar 68Ot 8dMR ACSCIV Lt l pep oGt BzQ Xj OJQb k3HM lT5Rqe sy 1 YXv jFW bUE Ko lySs cD21 ajmqab eV P Se4 nak TAt TE Bewe 7He5 2jWgEa 9q Q IeP us2 DJ5 bu P0oH SVGS 08xqGh QP i UXC j4Y WkG AT nWTY uChV Wjsg8J A2 J 0aV YJX 3Gp XP YMla uJA6 vPXhbm gd l yJr cmP KyG 3n 1TEV z1nR 65hHg9 dt D Smk 8Sj 3gk 1s 0Sbu QVdZ yhgnOo Xw d VQT umh WRY 8f Rhid i53B g2fBcG PY C PKG UbR Yvo 54 rf8j UmIK qxbxc3 Jh T XCq MTl H0D 3W 7QRa Drcl 0xOqql 4a z FqS E0l idv nk gUtJ Uhtf iUXCps sl Y SVH I2j xjy O9 rrT4 k1p8 iOkJcl oc b PDz Q5y 4b2 4u PlLU lm67 y5nFhA 0q P Dgp Vna Uw6 uM zRH7 q50c 819Ucp hP s od4 rFM cQH 8Q 3eDc DNb3 azTWIF GV x dxn DFGRTHVBSDFRGDFGNCVBSDFGDHDFHDFNCVBDSFGSDFGDSFBDVNCXVBSDFGSDFHDFGHDFTSADFASDFSADFASDXCVZXVSDGHFDGHVBCX189} 	\end{align} from where 	\begin{align} 	\begin{split} 	& 	\sum_{j=0}^\infty \sum_{\vert \alpha \vert = j}  	\Vert \partial^\alpha \epsilon \partial_t u \Vert_{L^2} 	\frac{ \tau_0^{(j-2)_+}}{(j-2)!} 	\\ 	&\indeq 	\leq 	C\epsilon \sum_{j=0}^\infty \sum_{\vert \alpha \vert = j} \sum_{l=0}^j \sum_{\beta \leq \alpha, \vert \beta \vert = l}  	\binom{\alpha}{\beta}  	\Vert \partial^\beta v \cdot \partial^{\alpha -\beta} \nabla u \Vert_{L^2} 	\frac{ \tau_0^{(j-2)_+}}{(j-2)!} 	\\ 	&\indeq\indeq 	+ 	C \sum_{j=0}^\infty \sum_{\vert \alpha \vert = j} \sum_{l=0}^j \sum_{\beta \leq \alpha, \vert \beta \vert = l}  	\binom{\alpha}{\beta}  	\Vert \partial^\beta \tilde E \partial^{\alpha -\beta}  \nabla u \Vert_{L^2} 	\frac{ \tau_0^{(j-2)_+}}{(j-2)!} 	= 	I_1 	+ 	I_2      . 	\label{DFGRTHVBSDFRGDFGNCVBSDFGDHDFHDFNCVBDSFGSDFGDSFBDVNCXVBSDFGSDFHDFGHDFTSADFASDFSADFASDXCVZXVSDGHFDGHVBCX527} 	\end{split} 	\end{align}	 The term $I_1$ can be estimated analogously as in \eqref{DFGRTHVBSDFRGDFGNCVBSDFGDHDFHDFNCVBDSFGSDFGDSFBDVNCXVBSDFGSDFHDFGHDFTSADFASDFSADFASDXCVZXVSDGHFDGHVBCX524}--\eqref{DFGRTHVBSDFRGDFGNCVBSDFGDHDFHDFNCVBDSFGSDFGDSFBDVNCXVBSDFGSDFHDFGHDFTSADFASDFSADFASDXCVZXVSDGHFDGHVBCX523}, obtaining $ 	\mu_1 	\leq 	Q(\Gamma_0) $. For the term $I_2$, we proceed as in \eqref{DFGRTHVBSDFRGDFGNCVBSDFGDHDFHDFNCVBDSFGSDFGDSFBDVNCXVBSDFGSDFHDFGHDFTSADFASDFSADFASDXCVZXVSDGHFDGHVBCX524}--\eqref{DFGRTHVBSDFRGDFGNCVBSDFGDHDFHDFNCVBDSFGSDFGDSFBDVNCXVBSDFGSDFHDFGHDFTSADFASDFSADFASDXCVZXVSDGHFDGHVBCX523}, obtaining $ 	\mu_2 	\leq 	C \Vert \tilde E \Vert_{A_0(\tau_0)} \Vert u \Vert_{A_0(\tau_0)}  	+ 	C\Vert \tilde E \Vert_{L^\infty} \Vert u \Vert_{A_0(\tau_0)} $. One may easily check that the product rules in Lemmas~\ref{L03} and~\ref{L05} hold for the norm $A_0(\tau_0)$. Thus we have 	\begin{align} 	\Vert \tilde E \Vert_{A_0(\tau_0)} 	\leq 	Q(\Vert u \Vert_{A_0(\tau_0)} + \Vert u \Vert_{L^2}, 	\Vert S \Vert_{A_0(\tau_0)} + \Vert S \Vert_{L^2}) 	\leq 	Q(\Gamma_0) 	. 	\label{DFGRTHVBSDFRGDFGNCVBSDFGDHDFHDFNCVBDSFGSDFGDSFBDVNCXVBSDFGSDFHDFGHDFTSADFASDFSADFASDXCVZXVSDGHFDGHVBCX526} 	\end{align} Combining \eqref{DFGRTHVBSDFRGDFGNCVBSDFGDHDFHDFNCVBDSFGSDFGDSFBDVNCXVBSDFGSDFHDFGHDFTSADFASDFSADFASDXCVZXVSDGHFDGHVBCX527}--\eqref{DFGRTHVBSDFRGDFGNCVBSDFGDHDFHDFNCVBDSFGSDFGDSFBDVNCXVBSDFGSDFHDFGHDFTSADFASDFSADFASDXCVZXVSDGHFDGHVBCX526}, we may write 	\begin{align} 	\sum_{j=0}^\infty \sum_{\vert \alpha \vert = j}  	\Vert \partial^\alpha \epsilon \partial_t u \Vert_{L^2} 	\frac{\tau_0^{(j-2)_+}}{(j-2)!} 	\leq 	\Gamma_1 	,    \llabel{sWB jVS 94 82Fg 8WWi GHFWTL bG h XN8 UVj l6Z sQ Wm95 1UvK 2g79jX I5 I HJU yCS bxU ZZ yTyV FULI fyt1Nk a3 k C7d OS9 9Hs jB z0yF mxdq m2REG3 IT F JPz X5B pz8 TB uRjV 3yZz WZDWFo Ry n jo8 6ci HsJ 8h PSh0 RjWB EhS4Hg X6 b QRd XOc g6U fp wqkg oC7H vukudC vq X ROv KYU WxP Ux 3LZA 0c7I LhNLhR 3Q d 2wQ 0d8 ugl VD Xx7Y oA00 xYf6P8 qh 6 DPg cBj Omx by O78O yVff YWS9NS vq j vK6 tZt lkD xg JBx2 yOS8 ooEUnh P6 4 V7A Ia6 BHA 5f 6PO4 vp3x rltZrm gH 0 XaY 3bD e46 Hk 0xhS mde9 Q0b67u JT 8 3dY D8G yB2 wq aRwJ IdFV iC71YK P6 k CZX eqm MCy Ww PciG sKJz LFBb3v 7x B 4N0 wex BT9 WA YuEx PhpP Q70uzU qR V kZJ rEn rxM Oj 4vyp Da09 huVk5k Ql i t9j aha qZV eR 93QO szbs 96VknY IW C aJu cdE H6X U1 RE5D MrJO x7uCYZ kg B RKV 3hy uDk OM LA8P quqM 7phON2 z8 S jG3 TH2 A9J ky znFw jMn1 fis7pe WS W 3Yg Ze5 d8J pe WYiL xQCc Hdqwo3 vk t cz9 6b2 XbO TV BwGa Aemi jgDvnc k5 R SaC HMc tSL QQ NhcC QSDH i8gxVX wk y R2r ciL oa3 6M qjRL dyy8 IyxQco 3F 3 MYK g4s sJY rp Kd89 2LpD cXGkqY mX K EyU Wj2 xra X8 GTXr W0Wf Ek0kD1 IA B IZR maJ eM5 W2 TwH7 PYGQ eLtkYR hn w VNn SLZ B3y Uk x1Ub AuTu UqmveV GG G qki uC0 Npb hW 5vxa yTtZ uS8msU se d Hou mBL yfS to GjPv Gtvw KVPD7D p9 7 ITM L9x oPt fY vWgM 981j YqzMvN 5o l Oek csD tBy gD 4tM9 RN8r oD0aMF nw B 0du jUS XWN S6 dIU1 oWZo wpOGKS rZ Y K8E Wqw 0a9 DF twlDFGRTHVBSDFRGDFGNCVBSDFGDHDFHDFNCVBDSFGSDFGDSFBDVNCXVBSDFGSDFHDFGHDFTSADFASDFSADFASDXCVZXVSDGHFDGHVBCX192} 	\end{align} where $\Gamma_1 = Q(\Gamma_0)$. 	 For $n=2$ and $n=3$, the proof is completely analogous and we obtain 	\begin{align} 	\sum_{j=0}^\infty \sum_{\vert \alpha \vert = j}  	\Vert \partial^\alpha (\epsilon \partial_t)^n S \Vert_{L^2} 	\frac{\tau_0^{(j+n-3)_+}}{(j+n-3)!} 	\leq 	\Gamma_n 	\llabel{X 793N 5RPvfw nT Y 2qW z9q 3Ij Un wryN xy7B qK4WCu WU U gli xaK W9e mC Dyk1 N9oe EucQSw 0J x Jzo xTg WVF Pu VKHv 6Qab FOy6Py it a 2lz NMO Slw l3 tUX4 Cns3 rPWtqX OJ 7 d6V Iiq KJh 09 dYGa 34H4 TBn6tC 71 Q shE Fyp bYA Gx FTsX EyQR muq5Fd 3o I EHe DQq ZzV 5l d0co S1CY EOwTpX oM 5 qS4 lGP Exm G0 6eA6 AbDo HqzKoB 3w T SrM ZJK T0l Js zsfI GbkR s9ThUZ rd F opw 2aQ SQF 67 Dhxb kfht TpBV6b jg H QUm YDS QB4 Jn j6Og 3EP0 vcfd4i P3 3 Ij9 cHC Vjk gA CMO8 wbch QD4ww1 r7 8 zSK Fpb QHM sI XmQp jl2s mcw1yb 0m U n5w FPQ Vpi WX ccx4 WajN D6ufj5 vR L sbU o3j EqA Tm Nlrm Lrea Mk3SHx tb 3 Qda lIU oM8 S2 Aa1R sOvV Aua3Ga vW w DXC BRg 4oC cP VXj6 cs0G iZUmxr An B hV8 AEI fRA nJ 2f6f qVxV aQQeNy f5 W cGG WUq UY6 sE 2fBh O6dT d4Ls1o SD 7 z4j Y8A nuX 6u JJtf vNMW BvRdMy DQ z ZEV OCI KFg dt NCLC LEQw 9iYE7C Or 5 UZ5 4IX n6e K0 IYi3 gAfV gU4s9u b2 A 9Nd tji tDS zt A0Fz Xsh4 F9zGGa 8N S KsL J80 0Rz wv vvjI Gg7I 8qPCFt CZ T 9Bo eUA QN0 zB txIP 0mXv 9ObxPh p8 r 8PV u0w jDG Zv b17p hD3P mk8wYp V1 A Sq6 LDa bJN 7N dFch OYV6 zbVAMF mS u FNs 4sO HZp LM aOMn nSta rUnlGf eY D drz KIe 2Je nL d3Xy ha9V hyA5p4 g8 4 sPo dIy sJ8 gW 2Vhg Gbuo 6bLWEb xX E raP YP2 DAV ie Z9Jj z2Jq qJxIi2 DQ b 0o1 zhP ZrR rT JoAG RZDN 2T6pPG T4 9 jgN g1q ne0 Y8 Pbvz Eh32 g9Hfzp Vq c dE7 UNs U79 jC iQ8N Yfe2 ijZZEB DFGRTHVBSDFRGDFGNCVBSDFGDHDFHDFNCVBDSFGSDFGDSFBDVNCXVBSDFGSDFHDFGHDFTSADFASDFSADFASDXCVZXVSDGHFDGHVBCX600} 	\end{align} and 	\begin{align} 	\sum_{j=0}^\infty \sum_{\vert \alpha \vert = j}  	\Vert \partial^\alpha (\epsilon \partial_t)^n u \Vert_{L^2} 	\frac{\tau_0^{(j+n-3)_+}}{(j+n-3)!} 	\leq 	\Gamma_n 	, 	\llabel{fD D Nre L8R Ary bj jcNC DZRx nZ13eY 3g 5 P0F twu Iks VX Mw68 84Sl HGnOGk 6P E qgG ati Cpb jH shyX aAV0 calJi2 OE B nxM qxd DRa vI shIB vaHY NBNjeW Xc j jjT Lbj edE RD Htfi xFZJ Tk7QT9 2e W XSa 81y 8CN ev s5Vu LJsJ faXS42 2U c 4pu 3dR LRl LB DLOf Nosp tTxsTj vh h com DWZ ttc dF oJ8q a9NT ew8Ed8 Kz 5 PAO xlG gH5 Mc s33F PeN0 76t8tR qR Q lXJ hR8 5ll eB 2k6c 44hg VgCr8k lG q HTA upL H1z 8u ZvGw VQFa pfkA9v Zi 1 Czn Tyi iZk 4r sSH8 qaY6 v1htCi gc J 7m3 ODk LlT 0M k8NB wHc2 7CFRuY R1 Y BZL g5D tid q2 kCjJ R2r4 W29CAN LH n ATB KbC zta mB tSOg Gi3m Z63sPE 6d M aJ7 la9 oeP Et vhUw 6tcG iXbdDl CP h Cq6 jr0 QtR Vr 5pwv Ggxr 6WRPdP ma s hqa C28 qSD eL zAyO PSKi wb2RzZ Ih C asb 1CQ 9hs 2c WbFI fK61 uq4mQa qv Z UjO TN4 n6S kw MHvT tH14 GQ3BmL RR 7 UW8 Yoy y4V Nm kdIL ZQn3 c2iNcu 2D f ydI bUg eYO cH fDba 0nNH tf3tZU Jd N Bgq AEj CsS PQ 34iz S14u u5WMAF pD I iqI IUE Hrq mI bQt0 e120 yYKfb0 WD i AqL dUv zPb KR GIyJ nHLK pZP1jT DQ 3 ogR VJU 2lg 47 xN2t 0aI8 PWGxdQ MD Y sGR 6Hj gNk AZ cdC8 nHET 2AQFAi n5 p XSe Dp2 XrD SX qOAT Plxj Q5UomW c1 u qBz 3Bh 7sv 7l FLRZ QFHB 6EXPbW 4U F aEp SeJ ZOH iL yQ4c Qstc fjJhI4 nR 3 WB6 qE9 sC0 8m 0vpL fjbD LwIEwm ef a 244 FcA Rm3 mc I8VU 4dkh iUBOPQ Cr b Cq7 8hX 2uu gU IYRA 9ZFK apioyQ Bo B yJq 8ti YUL 8L 4AKo Kzll hFIdcF HI T sRb 6FW cDFGRTHVBSDFRGDFGNCVBSDFGDHDFHDFNCVBDSFGSDFGDSFBDVNCXVBSDFGSDFHDFGHDFTSADFASDFSADFASDXCVZXVSDGHFDGHVBCX601} 	\end{align} for sufficiently large $\Gamma_n$ depending on $\Gamma_0$.  Summing over $n$ from $0$ to $3$, we obtain \eqref{DFGRTHVBSDFRGDFGNCVBSDFGDHDFHDFNCVBDSFGSDFGDSFBDVNCXVBSDFGSDFHDFGHDFTSADFASDFSADFASDXCVZXVSDGHFDGHVBCX560} and \eqref{DFGRTHVBSDFRGDFGNCVBSDFGDHDFHDFNCVBDSFGSDFGDSFBDVNCXVBSDFGSDFHDFGHDFTSADFASDFSADFASDXCVZXVSDGHFDGHVBCX561} for sufficiently large $\Gamma = Q(\Gamma_0)$. We fix $\Gamma$ for the rest of the proof. \par Next, we prove \eqref{DFGRTHVBSDFRGDFGNCVBSDFGDHDFHDFNCVBDSFGSDFGDSFBDVNCXVBSDFGSDFHDFGHDFTSADFASDFSADFASDXCVZXVSDGHFDGHVBCX510} and \eqref{DFGRTHVBSDFRGDFGNCVBSDFGDHDFHDFNCVBDSFGSDFGDSFBDVNCXVBSDFGSDFHDFGHDFTSADFASDFSADFASDXCVZXVSDGHFDGHVBCX511} for all $k \geq 4$ using induction and starting with the case $k=4$. First, we apply $\partial^\alpha (\epsilon \partial_t)^3$ to \eqref{DFGRTHVBSDFRGDFGNCVBSDFGDHDFHDFNCVBDSFGSDFGDSFBDVNCXVBSDFGSDFHDFGHDFTSADFASDFSADFASDXCVZXVSDGHFDGHVBCX02},  where $\vert \alpha \vert = j \geq 0$, obtaining 	\begin{align} 	\partial^\alpha (\epsilon\partial_t)^3 \partial_t S 	= 	- \sum_{\beta \leq \alpha} \sum_{n=0}^3  	\binom{\alpha}{\beta} \binom{3}{n}	 	\partial^\beta (\epsilon \partial_t)^n v \cdot \partial^{\alpha - \beta} (\epsilon \partial_t)^{3-n} \nabla S 	.    \llabel{1T 5h znft 3Qcn RUUAwu Nt E Yr6 EwL F9Q 5J 4phn rQNK prsLVl nE Z tYm 96M YY5 xj 6VMj Meod 8Hnpo7 m9 V nhj JIG vqD 4a qeAm Vj60 2tHUM3 uK z BMg mZ6 0dj Ru yVrF LJmn swSr5Y 7Y f UVP 7Gq M7w CN 7bjE riFY GBahv1 pX 8 YgA KOl wT4 F1 Sua8 RdcA DdU8QL ju c vEg WAp CTa tg hMs2 oKCy g02TZ4 gq D dhK KkB xlZ c3 Hw0l pEdw qWHH3S He 1 yi3 9Mc lMy Tr nGf3 0Mau iLGigv w1 r 8uH wuy H7o L0 gfYC DTvR CWgkyK bd k OFz U71 3vD s4 k85q i2gh xe0CXn C3 U H8s ntV Ne4 Vh 1xcR 3sEp 1xIXel eR I qqH OUr GrU vH jaCd W9So cvRooy CT 3 Asn jTi KKv yT aii6 4j2o WzuxMU zX N ghs iIo k8e qc AIWd aR6N POAorD Cr n zwk 7bA bNo 3x tbKR BibG DIdYoh RI s tDO ZXJ m5g II Vyms btlV DXTm1K i4 h Jpx EKO WEb 12 CCEQ iE4G aq88tM Tf Y kI9 enx HMM 9x f8Xp 8T8w HiIPvn zQ L KuD Hen jHc z3 pkx8 k0hv At2f9I Gi N 9La bqq xod wn tuQU vi9V 7B77on 09 t 9hM z5J p4M uF oiZ2 bd52 v7bUML BB F BFC jO6 ydc Mw GOgO cz8p zfiNr9 Oy 7 zVm AVY aDq cE hyPL vDuw qxIlgS dh 4 8y9 9MZ I4e 05 LsxU x1qj 0Y2i7N 4h q 3gv k9h Hzl 7q TmLy jZg3 xlNsOQ 6p z Jz7 YAh mTc JX JQMJ FpVj naFqk2 BF B NW7 aTB XcQ BR uJ9p QJ2R kypDsq cg i Tp2 gt8 zCS VJ bomQ D4CW AD6fdU 8D 6 3En sLV AzN 1F vTeQ QMlH Y9Pi6N dQ n lTK EA8 2oz rv VdP7 bFZe NJohJA Nh R F2A 1xs lEl I0 cOSw Zv3n Bttl5M 6x s xc9 DJB 0QR 3J X5X1 zhwf b7x5Ft 2b 0 Ibz GN3 WDu Ve 26MG se2DFGRTHVBSDFRGDFGNCVBSDFGDHDFHDFNCVBDSFGSDFGDSFBDVNCXVBSDFGSDFHDFGHDFTSADFASDFSADFASDXCVZXVSDGHFDGHVBCX193} 	\end{align} Using the splitting argument as in \eqref{DFGRTHVBSDFRGDFGNCVBSDFGDHDFHDFNCVBDSFGSDFGDSFBDVNCXVBSDFGSDFHDFGHDFTSADFASDFSADFASDXCVZXVSDGHFDGHVBCX524}--\eqref{DFGRTHVBSDFRGDFGNCVBSDFGDHDFHDFNCVBDSFGSDFGDSFBDVNCXVBSDFGSDFHDFGHDFTSADFASDFSADFASDXCVZXVSDGHFDGHVBCX521},  	\begin{align} 	\begin{split} 	& 	\sum_{j=0}^\infty \sum_{\vert \alpha \vert =j}  	\Vert \partial^\alpha (\epsilon \partial_t)^{4} S \Vert_{L^2} 	\frac{\lambda \tau_0^{j+1}}{(j+1)!} 	\\ 	&\indeq 	\leq 	C\epsilon \lambda 	\sum_{j=0}^\infty \sum_{\vert \alpha \vert =j} \sum_{0\leq l\leq [j/2]} \sum_{\beta \leq \alpha, \vert \beta \vert = l} \sum_{n=0}^3 	\left( \Vert \partial^{\alpha-\beta} (\epsilon \partial_t)^{3-n} \nabla S \Vert_{L^2} \frac{\tau_0^{(j-n-l+1)_+}}{(j-n-l+1)!}  	\right) 	\\ 	&\indeqtimes 	\left( \Vert \partial^\beta (\epsilon \partial_t)^n v \Vert_{L^2} \frac{\tau_0^{(l+n-3)_+}}{(l+n-3)!} 	\right)^{1/4} 	\left( \Vert D^2 \partial^\beta (\epsilon \partial_t)^n v \Vert_{L^2} \frac{\tau_0^{(l+n-1)_+}}{(l+n-1)!} 	\right)^{3/4} 	\\ 	&\indeq\indeq 	+ 	C\epsilon \lambda 	 \sum_{j=0}^\infty \sum_{\vert \alpha \vert =j} \sum_{[j/2]+1\leq l\leq j} \sum_{\beta \leq \alpha, \vert \beta \vert = l} \sum_{n=0}^3 	\left( \Vert \partial^{\beta} (\epsilon \partial_t)^n v \Vert_{L^2} \frac{ \tau_0^{(l+n-3)_+}}{(l+n-3)!} 	\right) 	\\ 	&\indeqtimes 	\left( \Vert D^2 \partial^{\alpha-\beta} (\epsilon \partial_t)^{3-n} \nabla S \Vert_{L^2} \frac{\tau_0^{(j-n-l+3)_+}}{(j-n-l+3)!}  	\right)^{3/4} 	\\ 	& 	\indeqtimes 	\left( \Vert \partial^{\alpha-\beta} (\epsilon \partial_t)^{3-n} \nabla S \Vert_{L^2} \frac{\tau_0^{(j-n-l+1)_+}}{(j-n-l+1)!}  	\right)^{1/4} 	. 	\label{DFGRTHVBSDFRGDFGNCVBSDFGDHDFHDFNCVBDSFGSDFGDSFBDVNCXVBSDFGSDFHDFGHDFTSADFASDFSADFASDXCVZXVSDGHFDGHVBCX531} 	\end{split} 	\end{align} Appealing to \eqref{DFGRTHVBSDFRGDFGNCVBSDFGDHDFHDFNCVBDSFGSDFGDSFBDVNCXVBSDFGSDFHDFGHDFTSADFASDFSADFASDXCVZXVSDGHFDGHVBCX560} and \eqref{DFGRTHVBSDFRGDFGNCVBSDFGDHDFHDFNCVBDSFGSDFGDSFBDVNCXVBSDFGSDFHDFGHDFTSADFASDFSADFASDXCVZXVSDGHFDGHVBCX561}, we arrive at 	\begin{align} 	\begin{split} 	& 	\sum_{j=0}^\infty \sum_{\vert \alpha \vert =j} 	\Vert \partial^\alpha (\epsilon \partial_t)^{4} S \Vert_{L^2} 	\frac{\lambda \tau_0^{j+1}}{(j+1)!} 	\\ 	&\indeq 	\leq 	C \lambda 	\sum_{n=0}^3 \left( \sum_{j=0}^\infty \sum_{|\alpha| =j} \Vert \partial^\alpha (\epsilon \partial_t)^n u \Vert_{L^2} \frac{\tau_0^{(j+n-3)_+}}{(j+n-3)!} 	\right) 	\\ 	&\indeqtimes 	 \left( \sum_{j=0}^\infty \sum_{|\alpha| =j} \Vert \partial^\alpha (\epsilon \partial_t)^{3-n} S \Vert_{L^2} \frac{ \tau_0^{(j-n)_+}}{(j-n)!} 	 \right) 	 \leq          \frac12          , 	\label{DFGRTHVBSDFRGDFGNCVBSDFGDHDFHDFNCVBDSFGSDFGDSFBDVNCXVBSDFGSDFHDFGHDFTSADFASDFSADFASDXCVZXVSDGHFDGHVBCX571} 	\end{split} 	\end{align} where we set $\lambda= 1 /Q(\Gamma)$, concluding the proof of \eqref{DFGRTHVBSDFRGDFGNCVBSDFGDHDFHDFNCVBDSFGSDFGDSFBDVNCXVBSDFGSDFHDFGHDFTSADFASDFSADFASDXCVZXVSDGHFDGHVBCX511} for $k=4$. As for \eqref{DFGRTHVBSDFRGDFGNCVBSDFGDHDFHDFNCVBDSFGSDFGDSFBDVNCXVBSDFGSDFHDFGHDFTSADFASDFSADFASDXCVZXVSDGHFDGHVBCX510}, we apply $\partial^\alpha (\epsilon \partial_t)^3$ to \eqref{DFGRTHVBSDFRGDFGNCVBSDFGDHDFHDFNCVBDSFGSDFGDSFBDVNCXVBSDFGSDFHDFGHDFTSADFASDFSADFASDXCVZXVSDGHFDGHVBCX501}, where $|\alpha| = j \geq 0$, obtaining 	\begin{align} 	\begin{split} 	\Vert \partial^\alpha (\epsilon \partial_t)^4 u \Vert_{L^2} 	& 	\leq 	C\epsilon 	 \sum_{l=0}^j \sum_{\beta \leq \alpha, \vert \beta \vert =l}  	\binom{\alpha}{\beta}	 	\binom{3}{n} 	\Vert \partial^\beta (\epsilon \partial_t)^n v \cdot \partial^{\alpha -\beta} (\epsilon \partial_t)^{3-n} \nabla u \Vert_{L^2} 	\\ 	&\indeq 	+ 	C\sum_{l=0}^j \sum_{\beta \leq \alpha, \vert \beta \vert =l}  	\binom{\alpha}{\beta}	 	\binom{3}{n} 	\Vert \partial^\beta (\epsilon \partial_t)^n \tilde E \partial^{\alpha - \beta} (\epsilon \partial_t)^{3-n} \nabla u \Vert_{L^2} 	. 	\end{split}    \llabel{z Aoo0HZ 1m 0 xD8 gVL tjX oN Lw5r JF1K 3qAPzY Xc V 0fF PE7 CkF iD wO4R LPlw BZdHRl NH q tkK HhJ nOg F6 sj7m j0a6 BBnK5e Kg l 4br 6YO VQq IG CXsc lFut Y0KJQg LB s 2Hw 9hw HGn cw T3Fg xDo7 sYsdP3 LZ h P2l 6xk OH4 To YuaO ZGlW St22kt 3t 5 BLk QW6 bXl 99 eQY1 sV4P z1Ntr3 yx h DxZ XM6 ZQp zN gtLt UhYA Efs5WL Is w 38w rM6 OVi Ai oqnO 185k dDdCXy YL 1 xwO 3n0 fzq 1o jsIG XT11 OAkgMX BX W diP 2pW aBm 66 PaLR TxfB HlEm2l Xv P ONB DY5 c0W YY XtuP QJOw HQcJR6 B0 b l3z Weo 9qq 0R tEO2 Y69j 2nEf7U B6 z 0fK cYY 9fZ aI 39Vd emxz 8XYo8N lh O UAA yjT 6GW sm OOIP SFgK HxIoJD yZ B mxp OZ0 1G9 7K vFT4 7hA8 JWuU0n cZ 5 NiK cOF GPI S8 m8f9 rDiW dVxlqg fr J lfl jVM cgf ef Yzxa 36Mb 0ao6aV tA a BZV XDi qnK u4 Df2d ACqI smiDkg qv n MaF n5R 4aw u7 IvwB 2IoN WZvaKK Kz 3 99A Pe0 tbq Ja 3s2N C1io jSyqlB E6 9 XPB vjL 0kC oQ TMSw aXji ZJO3qJ 7o i Qdx 5KM 05i uS 8kbn 3gZ0 i4WFs8 T8 I xl9 OXj fhr JB Gzj9 nGHX k4n5o3 Xt K 4D2 BRQ ciB iy Apiv 6UIr VkdRSx oT e n5c GCk xAf b0 3VDR 2x5P KE1Cu2 K6 f pVw WPS ubu cm 4kmf PNrT NqB9sg t6 D s8h 7lK lbL u3 Iwr3 g7MH 4CGz3l hH i fPV hX2 PG8 vL rBcj RkqG A5gVzv 96 c d0C CcM Pcn yT 7vo8 qGhQ DlxKoQ 9g S pG2 vH8 Wtw VJ nyFS KhZj tZJLmy 83 m VFG zXO MQw Wb 9AGV ak43 kSZzmW qJ L pjs 9xX Xec UK qCWG Aa3T 4iq6nB 6J M nkj 1vR HYM Be Mddm 5Lkb k5J6NM d9 J DFGRTHVBSDFRGDFGNCVBSDFGDHDFHDFNCVBDSFGSDFGDSFBDVNCXVBSDFGSDFHDFGHDFTSADFASDFSADFASDXCVZXVSDGHFDGHVBCX194} 	\end{align} Therefore, we get 	\begin{align} 	\begin{split} 	& 	\sum_{j=0}^\infty \sum_{\vert \alpha \vert = j}  	\Vert \partial^\alpha (\epsilon \partial_t)^4 u \Vert_{L^2} 	\frac{\lambda \tau_0^{j+1}}{(j+1)!} 	\\ 	&\indeq 	\leq 	C\epsilon \lambda 	 \sum_{j=0}^\infty \sum_{\vert \alpha \vert = j} \sum_{l=0}^j \sum_{\beta \leq \alpha, \vert \beta \vert = l} \sum_{n=0}^3 	\binom{\alpha}{\beta}  	\binom{3}{n} 	\Vert \partial^\beta (\epsilon \partial_t)^n v \cdot \partial^{\alpha -\beta} (\epsilon \partial_t)^{3-n}  \nabla u \Vert_{L^2} 	\frac{ \tau_0^{j+1}}{(j+1)!} 	\\ 	&\indeq\indeq 	+ 	C\lambda 	 \sum_{j=0}^\infty \sum_{\vert \alpha \vert = j} \sum_{l=0}^j \sum_{\beta \leq \alpha, \vert \beta \vert = l} \sum_{n=0}^3 	\binom{\alpha}{\beta}  	\binom{3}{n} 	\Vert \partial^\beta (\epsilon \partial_t)^n \tilde E \partial^{\alpha -\beta} (\epsilon \partial_t)^{3-n}  \nabla u \Vert_{L^2} 	\frac{\tau_0^{j+1}}{(j+1)!} 	\\ 	&\indeq 	= 	I_{41} 	+ 	I_{42}     .     \label{DFGRTHVBSDFRGDFGNCVBSDFGDHDFHDFNCVBDSFGSDFGDSFBDVNCXVBSDFGSDFHDFGHDFTSADFASDFSADFASDXCVZXVSDGHFDGHVBCX575} 	\end{split} 	\end{align}	 The term $I_{41}$ can be estimated as in \eqref{DFGRTHVBSDFRGDFGNCVBSDFGDHDFHDFNCVBDSFGSDFGDSFBDVNCXVBSDFGSDFHDFGHDFTSADFASDFSADFASDXCVZXVSDGHFDGHVBCX531}--\eqref{DFGRTHVBSDFRGDFGNCVBSDFGDHDFHDFNCVBDSFGSDFGDSFBDVNCXVBSDFGSDFHDFGHDFTSADFASDFSADFASDXCVZXVSDGHFDGHVBCX571}, obtaining $ 	I_{41}  	\leq 	1/2 $, while $I_{42}$ can be treated in a similar fashion as in \eqref{DFGRTHVBSDFRGDFGNCVBSDFGDHDFHDFNCVBDSFGSDFGDSFBDVNCXVBSDFGSDFHDFGHDFTSADFASDFSADFASDXCVZXVSDGHFDGHVBCX526}, arriving at $ 	I_{42} 	\leq 	C\lambda 	\Vert \tilde{E} \Vert_{A_3(\tau_0)} \Vert u \Vert_{A_3(\tau_0)} 	+ 	C\lambda 	 \Vert \tilde{E} \Vert_{L^\infty} \Vert u \Vert_{A_3(\tau_0)} $, where for each $k\geq 3$, we denote 	\begin{align} 	\Vert u \Vert_{A_k(\tau_0)} 	= 	\sum_{n=0}^k \sum_{j=0, j+n \geq 1}^\infty \sum_{|\alpha| = j} 	\Vert \partial^\alpha (\epsilon \partial_t)^n u\Vert_{L^2} 	\frac{\lambda^{(n-3)_+} \tau_0^{(j+n-3)_+}}{(j+n-3)!} 	. 	\label{DFGRTHVBSDFRGDFGNCVBSDFGDHDFHDFNCVBDSFGSDFGDSFBDVNCXVBSDFGSDFHDFGHDFTSADFASDFSADFASDXCVZXVSDGHFDGHVBCX574} 	\end{align} One can easily check that Lemma~\ref{L03} and~\ref{L05} hold for the $A_k(\tau_0)$-norm for each $k\geq 3$. Therefore,  $ 	I_{42}  	\leq 	1/2 $ by choosing $\lambda = 1/Q(\Gamma)$.	 There, we obtain \eqref{DFGRTHVBSDFRGDFGNCVBSDFGDHDFHDFNCVBDSFGSDFGDSFBDVNCXVBSDFGSDFHDFGHDFTSADFASDFSADFASDXCVZXVSDGHFDGHVBCX510} for $k=4$.  \par Now we assume that we have \eqref{DFGRTHVBSDFRGDFGNCVBSDFGDHDFHDFNCVBDSFGSDFGDSFBDVNCXVBSDFGSDFHDFGHDFTSADFASDFSADFASDXCVZXVSDGHFDGHVBCX510} and \eqref{DFGRTHVBSDFRGDFGNCVBSDFGDHDFHDFNCVBDSFGSDFGDSFBDVNCXVBSDFGSDFHDFGHDFTSADFASDFSADFASDXCVZXVSDGHFDGHVBCX511} for some $k \geq 4$, and prove them for  $k+1$. For $n \geq 3$, we apply $\partial^\alpha (\epsilon \partial_t)^n$ to \eqref{DFGRTHVBSDFRGDFGNCVBSDFGDHDFHDFNCVBDSFGSDFGDSFBDVNCXVBSDFGSDFHDFGHDFTSADFASDFSADFASDXCVZXVSDGHFDGHVBCX02},  where $\vert \alpha \vert = j\geq 0$, obtaining 	\begin{align} 	\partial^\alpha (\epsilon\partial_t)^n \partial_t S 	= 	- \sum_{\beta \leq \alpha} \sum_{m=0}^n \partial^\beta (\epsilon \partial_t)^m v \cdot \partial^{\alpha - \beta} (\epsilon \partial_t)^{n-m} \nabla S 	,    \llabel{3t7 sBd 9xF WZ wOPY Mzii wrrIeY XB j 9Wa MAD EjR XM jFxQ XEkm K3arQY Me Y 7YT wMN wmU cR W3lk JNKN YvZB6G Dk L O5i QlR JSS sb yySn CjSG C2hae0 Hs q Umb y1L JdF Xw dJ3m 4m6e 8btOXL vZ 9 K0A u3k cEv HU k84D rvxe s9BIt5 ZD z 2TL ESs qrz Mm 4B11 oJpP Atuzgh 2Q f pWt HZ8 6sn 5w RYZ7 3gX5 P8ZE0j 0a 5 eki NP5 nYz Mh nj9D 0ZU3 7xHuJw vO m c7B 13W iwu 22 4aSx 9UsU 9Z5ZpC EC E wXD B7j d8O m8 aGxE DihJ fcdsAt Mx r hLL k1P z04 9D 7aWr Kzqb F1UK2f gP h lms mZH 4RA v6 uYph ELLS h5GRHC M1 E ARj JLs Xq7 wi AmC9 9rAs BKHuEN tC 5 tOE z7C umv JL m0Xh Pnnu R3Udac 0W T DwV jT7 Eo7 zK C3Xl pSbo EUDayL GU H j9H Amk kEZ 23 69Lk kd0G u9GxEu yK r OFE 1AF BY8 tV MAgl W7Xs g8LQBT eV A SKm cbT q4X XC kr7Y V6xf XLt6Kr ot D TFQ fIf QbE Gz E9Qu zWQ2 uA0YFS j1 b Sxs LGD v4D Cl rbdi Mocr 3JO4Hs nF g FPX x95 PBh Ug vRu7 q4lq T4t8k3 ph f dR3 DAJ Mfk 22 K0MC 8dQG xNj7rD eu 0 35s zHb 6Kq bw 4Kdz RD8n asx3Qj ii J QOq iqz t1I 9x 7uCT xoor 4MK2YX h8 f bgN g2y HP6 BX Z3fW UHlb wyC53o 8V h ceC QaQ nhL Iu h3Ul xs4M UHQ3Ym ca p FGU aNz lCi Yh Bls9 CoL2 Xn9Ftz n4 S TOR Hfs 4iJ yF gxkd qFJx PcbfcW JR V VmT l33 QXX In G1JE u5ff wYgkEH f1 2 Ylu dXy z4r EX zzi0 Ov1K R1cIiL Ha i aF8 yx9 uKU JW pppq dWIn LcEzQc 10 a 3gh EDw CiJ 88 hnTj AF36 MyM7bR RN a ghX Nk7 zPE iV zgjE VD4i mE2VtL wm O k7v giq pIl oWDFGRTHVBSDFRGDFGNCVBSDFGDHDFHDFNCVBDSFGSDFGDSFBDVNCXVBSDFGSDFHDFGHDFTSADFASDFSADFASDXCVZXVSDGHFDGHVBCX195} 	\end{align} from where 	\begin{align} 	\begin{split} 	& 	\sum_{j=0}^\infty \sum_{\vert \alpha \vert =j} 	\Vert \partial^\alpha (\epsilon \partial_t)^{n+1} S \Vert_{L^2} 	\frac{\lambda^{n-2} \tau_0^{(j+n-2)_+}}{(j+n-2)!} 	\\ 	&\indeq 	\leq 	C\epsilon \sum_{j=0}^\infty \sum_{\vert \alpha \vert = j} \sum_{l=0}^j \sum_{\beta \leq \alpha, \vert \beta \vert = l} \sum_{m=0}^n 	\binom{\alpha}{\beta} \binom{n}{m} 	\\ 	&\indeqtimes	 	\Vert \partial^\beta (\epsilon \partial_t)^m v \cdot \partial^{\alpha -\beta}  (\epsilon \partial_t)^{n-m} \nabla S \Vert_{L^2} 	\frac{\lambda^{n-2} \tau_0^{(j+n-2)_+}}{(j+n-2)!} 	. 	\end{split}    \llabel{ Tu9E 2G7e jssvDI 96 t n6B 2Lx Sl5 b4 oHHW 1Lmz gSlAiT Y1 4 dN4 OoJ gn7 gG hJCu L116 kQ1wZu 6a I MAL PPG Zrk AB KEM7 8uPj PGxGkR X9 y zOx 60R wY2 9a 9fdG CqtM w1AYAD Vl v fIG FCH 1qb e4 3u9G FL0i yld9ne ef 7 xEU Nlg X88 XK PDHE Cfq4 9YBJIM QJ d YD2 aps cK6 Lr BcnK b3Sm 5KdmtN xz a z07 fgr 7Tr tM gaX2 H6xw drXzEp II 4 6Wk I5o kRz BV IYtc kJeY TgwZl0 Ak v voQ Cs6 kdx sv 3YyB BTbm rus3FO PE R ou8 aFs ryb Eb Drwg bZf2 KOyFo6 et R SQe 10R edR 9e lN9K 43tj a5v4B2 hG w fWV CzU RqI 8D aPhr Zykm GTNdmk hN j K2j eP9 ELH WD yQMK ydco NHqIdC VD j mu5 lBn a71 cc l5JI SZjq EYwQud Ax s k25 3sR tNH 1i wRKe uGYh lfy7lo Kj s vcP 3TW m1A ML SWRs uFCi Opnt4m gj C a7N nvR LfQ KO mL4P nnq8 boPSV1 YC Y H7E 9zu 79Q za aqNU nWOk Ol35DI ez r PWU 5Qb 20q 4a AkXX T2Ru 4FWl9M jU u 5b1 bVL Q4j bK lwgI OnTZ NFLdb5 B6 5 v1V ccB 1gS Pm 0FhP n6h8 YTC3jy i7 v OL0 QqA vXj Wo h7ke oeXx JcLb7a iZ R qzN q4L Cen KO KVns S6Ee RyJYks gx y bzz BhF m0N Zu x26b C39t uzA0PB qW D Ucy rLF dpf ht CK7Q 0oOg Uthn3l aR o GyK hbp fcY sc wPtQ mvyo oxDPvH lT O 1a2 46b 7Hb 1h OzH4 AOt3 g4WtmP IP x rEL t3L 7eb 4r 48o6 h9bh 0MfExn KB L bnf iud OED iQ 25Mb CZFX sj6T5u Re g idF 4ad mCv a6 0pdl x5fU uaMst5 uB D MxA ZRG JON 3g XRE7 ha6d idccEw Hb S bBD iTm dRl kN udTM sFCQ YmRFVY EX F ehB meY pqX Qk QIdc ErYS AvgDFGRTHVBSDFRGDFGNCVBSDFGDHDFHDFNCVBDSFGSDFGDSFBDVNCXVBSDFGSDFHDFGHDFTSADFASDFSADFASDXCVZXVSDGHFDGHVBCX196} 	\end{align} We split the above sum according to the low and high values of $l+n$. Using a similar argument as in \eqref{DFGRTHVBSDFRGDFGNCVBSDFGDHDFHDFNCVBDSFGSDFGDSFBDVNCXVBSDFGSDFHDFGHDFTSADFASDFSADFASDXCVZXVSDGHFDGHVBCX531}, we get 	\begin{align} 	\begin{split} 	& 	\sum_{j=0}^\infty \sum_{\vert \alpha \vert =j} 	\Vert \partial^\alpha (\epsilon \partial_t)^{n+1} S \Vert_{L^2} 	\frac{\lambda^{n-2} \tau_0^{(j+n-2)_+}}{(j+n-2)!} 	\\ 	&\indeq 	\leq	 	C\epsilon \lambda 	 \sum_{j=0}^\infty \sum_{|\alpha| = j} \sum_{l=0}^j \sum_{\beta \leq \alpha, \vert \beta \vert = l} \sum_{m=0}^n 	\left( \Vert D^2 \partial^\beta (\epsilon \partial_t)^m v \Vert_{L^2} 	\frac{\lambda^{(m-3)_+} \tau_0^{(l+m-1)_+}}{(l+m-1)!} 	\right)^{3/4} 	\\ 	& 	\indeqtimes 	\left( \Vert \partial^\beta (\epsilon \partial_t)^m v \Vert_{L^2} 	\frac{\lambda^{(m-3)_+} \tau_0^{(l+m-3)_+}}{(l+m-3)!} 	\right)^{1/4} 	\\ 	&\indeqtimes 	\left( \Vert \partial^{\alpha -\beta} (\epsilon \partial_t)^{n-m} \nabla S \Vert_{L^2}  	\frac{\lambda^{(n-m-3)_+} \tau_0^{(j+n-l-m-2)_+}}{(j+n-l-m-2)!} 	\right) 	\mathbbm{1}_{\{0 \leq l+m \leq [(j+n)/2]\}} 	\\ 	&\indeq\indeq 	+ 	C\epsilon \lambda 	 \sum_{j=0}^\infty \sum_{|\alpha| = j} \sum_{l=0}^j \sum_{\beta \leq \alpha, \vert \beta \vert = l} \sum_{m=0}^n 	\left( \Vert D^2 \partial^{\alpha-\beta} (\epsilon \partial_t)^{n-m} \nabla S \Vert_{L^2} 	\frac{\lambda^{(n-m-3)_+} \tau_0^{(j+n-l-m)_+}}{(j+n-l-m)!} 	\right)^{3/4} 	\\ 	& 	\indeqtimes 	\left( \Vert \partial^{\alpha-\beta} (\epsilon \partial_t)^{n-m} \nabla S \Vert_{L^2} 	\frac{\lambda^{(n-m-3)_+} \tau_0^{(j+n-l-m-2)_+}}{(j+n-l-m-2)!} 	\right)^{1/4} 	\\ 	&\indeqtimes 	\left( \Vert \partial^{\beta} (\epsilon \partial_t)^m v \Vert_{L^2}  	\frac{\lambda^{(m-3)_+} \tau_0^{(l+m-3)_+}}{(l+m-3)!} 	\right) 	\mathbbm{1}_{\{[(j+n)/2]+1 \leq l+m \leq j+n\}}         , 	\end{split}        \llabel{DoY 1D b f50 Oc9 x3k Bj a3Uf CLcp x9gaTk xp a m22 x8U dcR GR ir8g 7LwO cEnelX tp Q 0Pi p92 bCb fC jEvW OMoM pgYxDv eS B 9a7 fn9 YyF 6C NdWc jqZG 7NaXSu oz t jPY K87 zso q2 HxjZ 9sJI t0fwFk 7Q 2 xja 4wo tnv w2 qjwn QYMG RRfX6k 1v Y NEW zf0 SBP t5 xbfI GNSj mx9tRN BQ N KV4 z3x 1Ez fl Umag fgCt tQ2PGX Hb a btc lxc YoV xj GEyX LTpO jVOAfG 9K 7 KMK KMM OC7 Hn Hvpd B0SS ZlylU7 rn L xdM tfv DSN VG Y91v 5HNL tSpFOT N1 v x0K na9 P3O zP pENq pWAp cnLHGz Wp q Tra 5GZ z64 Rl CSGi n2em lUiy9b Ym g cpW 0F0 4d0 mB eKKT VdQ5 0kadoy jl 3 LZk 1PW TnU f5 FKNt qbLe Tm2NWS jY y O1R P5J kja aV avTF y2Iu IXoH2O re s gmD kN3 AKf c1 nePW TXa9 7oDl4Y Kl E 2eX cbe 0OM AU 2t6R 9GFX eSqblS yw n TB0 w9W BLk jO eVRv u4Jc EUrll0 rt y K9C kqc u81 31 u6Q2 bll8 NO6Ysp Vz B S9j FnS cx6 uh GbkM jKX1 S9Qg3H td o VPg tGL chi ml NffJ cEGb PaAQZd TF I ZGo GYB 09Y TW cudB 1pU1 WCb5kf Up R 7yL NAN fwO XK B8Iv VdBP HoeqJZ OP k bND yzr AqR DM 1qJZ gH7q 9uBXVy CT d ND2 TTt Mo4 fC 8lfu LO3Q Zk4cTR wq 2 GKH GFZ Fvj Ll ZvA4 GMEi VC9Xkt 4R k v6g gDV ciL Ks T4AV QCOo VkESMP Kt A bKC mJO xxD Cm hqZ3 n8st qaTtEw Jx Y X2N dqn bDZ 5B iln4 iPvR XBuio3 ZO b ssX duE nI3 XJ a79Y klF8 RPVXRa 7y o abp PS1 6c1 sz qG9x ZfWN D0uPLB do K hCT 6E4 sRe wB 4rf9 FW94 D9xB4y Yh 7 fAM L3c M6o t4 uJoL CrER iBVqaj qm n Uwg FDFGRTHVBSDFRGDFGNCVBSDFGDHDFHDFNCVBDSFGSDFGDSFBDVNCXVBSDFGSDFHDFGHDFTSADFASDFSADFASDXCVZXVSDGHFDGHVBCX129} 	\end{align} from where 	\begin{align} 	\begin{split} 	& 	\sum_{j=0}^\infty \sum_{\vert \alpha \vert =j} 	\Vert \partial^\alpha (\epsilon \partial_t)^{n+1} S \Vert_{L^2} 	\frac{\lambda^{n-2} \tau_0^{(j+n-2)_+}}{(j+n-2)!} 	\\ 	&\indeq 	\leq	 	C\epsilon \lambda 	 \sum_{m=0}^n  	\left(	 \sum_{j=0}^\infty \sum_{|\alpha| = j} \Vert \partial^\alpha (\epsilon \partial_t)^m v \Vert_{L^2} \frac{\lambda^{(m-3)_+} \tau_0^{(j+m-3)_+}}{(j+m-3)!} 	\right) 	\\ 	&\indeqtimes 	\left(	  \sum_{j=0}^\infty \sum_{|\alpha| = j} \Vert \partial^\alpha (\epsilon \partial_t)^{n-m} S \Vert_{L^2} \frac{\lambda^{(n-m-3)_+} \tau_0^{(j+n-m-3)_+}}{(j+n-m-3)!} 	\right) 	. 	\label{DFGRTHVBSDFRGDFGNCVBSDFGDHDFHDFNCVBDSFGSDFGDSFBDVNCXVBSDFGSDFHDFGHDFTSADFASDFSADFASDXCVZXVSDGHFDGHVBCX529} 	\end{split} 	\end{align} \colb Summing the above estimate in $n$ from $3$ to $k$, we get 	\begin{align} 	\begin{split} 	& 	\sum_{n=4}^{k+1} \sum_{j=0}^\infty \sum_{|\alpha| = j}  	\Vert \partial^\alpha (\epsilon \partial_t)^{n} S \Vert_{L^2} 	\frac{\lambda^{n-3} \tau_0^{(j+n-3)_+}}{(j+n-3)!} 	= 	\sum_{n=3}^{k} \sum_{j=0}^\infty \sum_{|\alpha| = j}  	\Vert \partial^\alpha (\epsilon \partial_t)^{n+1} S \Vert_{L^2} 	\frac{\lambda^{n-2} \tau_0^{(j+n-2)_+}}{(j+n-2)!}        ,       \llabel{yL 3cf NT 9pVC SQgF DkkfGn y3 u 2sx 8rZ 1jW k0 1Mfn HZoD 5LVclb 9i h o0L lJz wbf XR FiQs 03w6 Ptdg9y fy R VsD YOu vZG X0 0V3r W3hj hkw7c1 1j S uzF Kl4 2qx Jg anNI ym3n BFnYQz dt B 21p q9L 71v gm 70nT gLwn f6bTHS BJ i gZr V3I bWo 6l XMEc sUi9 cVc4iq 4O z J6I meG lOL 8t 512R GgCk iewMfb b8 f 5H0 5CB dpR Pe Q37L JdRe E30Fpd Zr R Pkx Lfv fTx 48 Gzwu vHS8 SmZugd 9C 1 DsD NyD G9Q TZ dkgb QZst FczLX3 as 8 zoL Ave f63 FM aeeY NpZn by5poe PI B xcV UHX qpr dY 5f5A ryLz JQtDKf xQ 0 xU2 e9T m80 f2 rdNr QKl6 eWN4t3 q6 b 7lA IC5 leK kE XKNF TRMv TvSn3M Si O azA j5M DLw f1 pcFM FF5I lnU3EF ow e 4my vMW 1gS cP GKfB QKpE mOO62n Wd U jw4 ZvU p3p eG DjOW X96Y cSRZFf 5a F fNV QPY 0Bb n8 AC4j cGUI 7zVEOU C7 U 3n5 SYE Tg1 Bb r7s2 hvMU vkmzZA YV X j2j CzQ cjK gS Njkz ntZw RkiC4k 7A x L5Y WX6 9Kf 5H MCgs nE1K x4sWPz W3 M HTM Tip qSv iZ CkzD rWFj QugJY5 ur B p8c km2 EJU St HplY AcTc CldxPC WM t 98S 6vE pzo nl QQ8h G1wh KJ43e1 GA F 5bw wkq ym9 IR ugvp js55 onAjat SD s PBt caO Z5f ky u1bn QMPX tIiDUQ oy 5 rUD sMP Rm8 HH vq0D xs1p G0dQBd zP z EUV rs8 klo XL Ckeb 12Dx oaOy9D lM u 55p ZQW yX0 ED CSYp klsT ywQMiP CF D GlD Ql0 JFM gU KJKG kXqg UU3Yd5 ub 0 XQe 5dj 4CX L6 cBQ2 ec1R Vz5U1d mH q T36 jab 5J8 BL XV5S 99mt OgmFpM xX m k5Z jhm Din XC nA0O RIei OTRGeC Mh n cqZ q2b D9l aN T17NDFGRTHVBSDFRGDFGNCVBSDFGDHDFHDFNCVBDSFGSDFGDSFBDVNCXVBSDFGSDFHDFGHDFTSADFASDFSADFASDXCVZXVSDGHFDGHVBCX86} 	\end{split} 	\end{align} which is bounded from above by 	\begin{align} 	\begin{split} 	& 	C \lambda 	\sum_{n=3}^{k} \sum_{m=0}^n  	\left(	 \sum_{j=0}^\infty \sum_{|\alpha| = j} \Vert \partial^\alpha (\epsilon \partial_t)^m v \Vert_{L^2} \frac{\lambda^{(m-3)_+} \tau_0^{(j+m-3)_+}}{(j+m-3)!} 	\right) 	\\ 	&\indeqtimes 	\left(	  \sum_{j=0}^\infty \sum_{|\alpha| = j} \Vert \partial^\alpha (\epsilon \partial_t)^{n-m} S \Vert_{L^2} \frac{\lambda^{(n-m-3)_+} \tau_0^{(j+n-m-3)_+}}{(j+n-m-3)!} 	\right) 	\\ 	& 	\indeq 	\leq 	C\lambda 	\left(	\sum_{m=0}^{k} \sum_{j=0}^\infty \sum_{|\alpha| = j} \Vert \partial^\alpha (\epsilon \partial_t)^m v \Vert_{L^2} \frac{\lambda^{(m-3)_+} \tau_0^{(j+m-3)_+}}{(j+m-3)!} 	\right) 	\\ 	&\indeqtimes 	\left(	\sum_{m=0}^{k} \sum_{j=0}^\infty \sum_{|\alpha| = j} \Vert \partial^\alpha (\epsilon \partial_t)^m S \Vert_{L^2} \frac{\lambda^{(m-3)_+} \tau_0^{(j+m-3)_+}}{(j+m-3)!} 	\right) 	. 	\label{DFGRTHVBSDFRGDFGNCVBSDFGDHDFHDFNCVBDSFGSDFGDSFBDVNCXVBSDFGSDFHDFGHDFTSADFASDFSADFASDXCVZXVSDGHFDGHVBCX578} 	\end{split} 	\end{align} By \eqref{DFGRTHVBSDFRGDFGNCVBSDFGDHDFHDFNCVBDSFGSDFGDSFBDVNCXVBSDFGSDFHDFGHDFTSADFASDFSADFASDXCVZXVSDGHFDGHVBCX560} and \eqref{DFGRTHVBSDFRGDFGNCVBSDFGDHDFHDFNCVBDSFGSDFGDSFBDVNCXVBSDFGSDFHDFGHDFTSADFASDFSADFASDXCVZXVSDGHFDGHVBCX561}, and the inductive hypothesis \eqref{DFGRTHVBSDFRGDFGNCVBSDFGDHDFHDFNCVBDSFGSDFGDSFBDVNCXVBSDFGSDFHDFGHDFTSADFASDFSADFASDXCVZXVSDGHFDGHVBCX510}--\eqref{DFGRTHVBSDFRGDFGNCVBSDFGDHDFHDFNCVBDSFGSDFGDSFBDVNCXVBSDFGSDFHDFGHDFTSADFASDFSADFASDXCVZXVSDGHFDGHVBCX511} for $k$, we arrive at 	\begin{align} 	\sum_{n=4}^{k+1} \sum_{j=0}^\infty \sum_{|\alpha| = j}  	\Vert \partial^\alpha (\epsilon \partial_t)^{n} S \Vert_{L^2} 	\frac{\lambda^{n-3} \tau_0^{(j+n-3)_+}}{(j+n-3)!} 	\leq         \frac12 	,    \llabel{ Z4aX cfJ5VP AN H DhO lBZ Y0N IZ TDGG 8laT M53v1c wW 3 O3n 2q6 L2u HF 8Moe JaUq ns9tpC zL m Q7A 8EU BXu 3a hDO9 i6oH 0Donq3 UN q HFv PIf qRD gY 7NWh 37SS PZgDvl Nd r krH j26 MLb rn 9Nds nI9k fgZQRP PH z Mu8 Cd7 qru Ce zbvv 2O1k OJMlHj 2w y vNR x46 aS3 Xd 66oO mTFX abbYyP Z2 m wKH 0ID ZVk iL 3GMM kSdR 3CQJT8 Kb 7 hTL PuR cyk SC YPbO 5h63 Q4dsKP iY e FUO CkO 91A DB NqMb wohi dmRvrF JF v 0Bk EhM dYA Ho LWOh xqEP A2vPMU e4 a IPW MJ5 hU1 jW 0Suf Sn47 42HypF Dz z Jqq c54 uhx 8G hadF eSeQ aSNqqT C5 a MZK jS6 c2b Bn fVIk dzKy JOBNrj s6 I 2Cc sGH 1Be Pz apx3 WrTj Whmecb Jp w bXe r87 khQ jZ bgUB uYB9 t8D9Or PD X 7Od rGd Qyn Io XgV9 jFzu p9fYh8 hC W vSG dY0 Jva 9P fv3Y vOae 8Qnpr7 Oe j zYS kyU BNg NN DZ0A fXMu 6S8wfW 9o J gEG Yy6 dhS oH ZCc8 ZN6z Ncl1C4 jh k ttp IHc s3j xH 2Gt3 4kx3 GsPrIz xu c yn4 fmM d3W Y9 V3sj NHia US7VDi Ni K 1Pr tvn vyO JO FJxi ym43 ejIUIu Mg k jkU w61 DxM 2Y NZGi 5EyP Yi7g0f OO c pB0 m48 di7 Lj AJ1u 11Ny d4z1Nj iK V 39h aoy 7Qs WZ ZWlX wByr hRilV7 Zu e zCV Wor wjG 1w KEJj Dd3Q 7KVJa5 MA y 1RT iEy cU2 7u v9Ot G2VA QMioy9 xB Q nzN chE R3h vo Ig1O mOoZ ZPQN3z PV T Dp8 eoC ea9 EF HzFi GOTC zvCU7r Ii 9 Cq8 kUh RiK lb AlUg HhHL KV80Ls D8 5 uAW oeo VkL dc zibm u4Aw 2dZIpl n6 H RMK A6Q 1aw Bw XhPg y61R Iu1OMb xV y aBj z8l dtz rr Z4XF P8oV l6YN6u ODFGRTHVBSDFRGDFGNCVBSDFGDHDFHDFNCVBDSFGSDFGDSFBDVNCXVBSDFGSDFHDFGHDFTSADFASDFSADFASDXCVZXVSDGHFDGHVBCX197} 	\end{align} where we choose $\lambda = 1/Q(\Gamma)$, which leads to \eqref{DFGRTHVBSDFRGDFGNCVBSDFGDHDFHDFNCVBDSFGSDFGDSFBDVNCXVBSDFGSDFHDFGHDFTSADFASDFSADFASDXCVZXVSDGHFDGHVBCX511} for $k+1$. \par As for \eqref{DFGRTHVBSDFRGDFGNCVBSDFGDHDFHDFNCVBDSFGSDFGDSFBDVNCXVBSDFGSDFHDFGHDFTSADFASDFSADFASDXCVZXVSDGHFDGHVBCX510}, we apply $\partial^\alpha (\epsilon \partial_t)^n$ to \eqref{DFGRTHVBSDFRGDFGNCVBSDFGDHDFHDFNCVBDSFGSDFGDSFBDVNCXVBSDFGSDFHDFGHDFTSADFASDFSADFASDXCVZXVSDGHFDGHVBCX501} where  $|\alpha| = j \geq 0$ and $n \geq 3$.  Similarly to \eqref{DFGRTHVBSDFRGDFGNCVBSDFGDHDFHDFNCVBDSFGSDFGDSFBDVNCXVBSDFGSDFHDFGHDFTSADFASDFSADFASDXCVZXVSDGHFDGHVBCX527}, we obtain 	\begin{align} 	\begin{split} 	& 	\sum_{j=0}^\infty \sum_{\vert \alpha \vert = j}  	\Vert \partial^\alpha (\epsilon \partial_t)^{n+1} u \Vert_{L^2} 	\frac{\lambda^{n-2} \tau_0^{j+n-2}}{(j+n-2)!} 	\\ 	&\indeq 	\leq 	C\epsilon  \lambda 	 \sum_{j=0}^\infty \sum_{\vert \alpha \vert = j} \sum_{l=0}^j \sum_{\beta \leq \alpha, \vert \beta \vert = l} \sum_{m=0}^{n} 	\binom{\alpha}{\beta} \binom{n}{m} 	\Vert \partial^\beta (\epsilon \partial_t)^m v \cdot \partial^{\alpha -\beta}  (\epsilon \partial_t)^{n-m} \nabla u \Vert_{L^2} 	\frac{\lambda^{n-3} \tau_0^{j+n-2}}{(j+n-2)!} 	\\ 	&\indeq\indeq 	+ 	C \lambda 	 \sum_{j=0}^\infty \sum_{\vert \alpha \vert = j} \sum_{l=0}^j \sum_{\beta \leq \alpha, \vert \beta \vert = l} \sum_{m=0}^n 	\binom{\alpha}{\beta} \binom{n}{m} 	\Vert \partial^\beta (\epsilon \partial_t)^m \tilde E \partial^{\alpha -\beta}  (\epsilon \partial_t)^{n-m} \nabla u \Vert_{L^2} 	\frac{\lambda^{n-3}\tau_0^{j+n-2}}{(j+n-2)!} 	\\ 	&\indeq 	= 	J_{1n} 	+ 	J_{2n} 	. 	\label{DFGRTHVBSDFRGDFGNCVBSDFGDHDFHDFNCVBDSFGSDFGDSFBDVNCXVBSDFGSDFHDFGHDFTSADFASDFSADFASDXCVZXVSDGHFDGHVBCX538} 	\end{split} 	\end{align} For the term $J_{1n}$, we proceed as in \eqref{DFGRTHVBSDFRGDFGNCVBSDFGDHDFHDFNCVBDSFGSDFGDSFBDVNCXVBSDFGSDFHDFGHDFTSADFASDFSADFASDXCVZXVSDGHFDGHVBCX529}--\eqref{DFGRTHVBSDFRGDFGNCVBSDFGDHDFHDFNCVBDSFGSDFGDSFBDVNCXVBSDFGSDFHDFGHDFTSADFASDFSADFASDXCVZXVSDGHFDGHVBCX578}, obtaining 	\begin{align} 	\begin{split} 	\sum_{n=3}^k J_{1n}  	& 	\leq 	\frac{1}{2}         . 	\label{DFGRTHVBSDFRGDFGNCVBSDFGDHDFHDFNCVBDSFGSDFGDSFBDVNCXVBSDFGSDFHDFGHDFTSADFASDFSADFASDXCVZXVSDGHFDGHVBCX580} 	\end{split} 	\end{align} For the term $J_{2n}$, we split the sum according to the low and high values of $l+n$. Proceeding as in \eqref{DFGRTHVBSDFRGDFGNCVBSDFGDHDFHDFNCVBDSFGSDFGDSFBDVNCXVBSDFGSDFHDFGHDFTSADFASDFSADFASDXCVZXVSDGHFDGHVBCX529}, we arrive at 	\begin{align} 	\begin{split} 	J_{2n} 	& 	\leq 	C \lambda 	\sum_{j=0}^\infty \sum_{|\alpha| =j}  	\Vert \tilde{E} \Vert_{L^\infty}	 \Vert \partial^\alpha (\epsilon \partial_t)^n \nabla u \Vert_{L^2} \frac{\lambda^{n-3} \tau_0^{j+n-2}}{(j+n-2)!} 	\\ 	&\indeq 	+ 	C\lambda 	\sum_{j=0}^\infty \sum_{|\alpha| = j} \sum_{l=0}^j \sum_{\beta \leq \alpha, \vert \beta \vert = l} \sum_{m=0}^n 	\left( \Vert D^2 \partial^\beta (\epsilon \partial_t)^m \tilde{E} \Vert_{L^2} 	\frac{\lambda^{(m-3)_+} \tau_0^{(l+m-1)_+}}{(l+m-1)!} 	\right)^{3/4} 	\\ 	& 	\indeqtimes 	\left( \Vert \partial^\beta (\epsilon \partial_t)^m \tilde{E} \Vert_{L^2} 	\frac{\lambda^{(m-3)_+} \tau_0^{(l+m-3)_+}}{(l+m-3)!} 	\right)^{1/4} 	\\ 	&\indeqtimes 	\left( \Vert \partial^{\alpha -\beta} (\epsilon \partial_t)^{n-m} \nabla u \Vert_{L^2}  	\frac{\lambda^{(n-m-3)_+} \tau_0^{(j+n-l-m-2)_+}}{(j+n-l-m-2)!} 	\right) 	\mathbbm{1}_{\{1 \leq l+m \leq [(j+n)/2]\}} 	\\ 	&\indeq 	+ 	C \lambda 	 \sum_{j=0}^\infty \sum_{|\alpha| = j} \sum_{l=0}^j \sum_{\beta \leq \alpha, \vert \beta \vert = l} \sum_{m=0}^n 	\left( \Vert D^2 \partial^{\alpha-\beta} (\epsilon \partial_t)^{n-m} \nabla u \Vert_{L^2} 	\frac{\lambda^{(n-m-3)_+} \tau_0^{(j+n-l-m)_+}}{(j+n-l-m)!} 	\right)^{3/4} 	\\ 	& 	\indeqtimes 	\left( \Vert \partial^{\alpha-\beta} (\epsilon \partial_t)^{n-m} \nabla u\Vert_{L^2} 	\frac{\lambda^{(n-m-3)_+} \tau_0^{(j+n-l-m-2)_+}}{(j+n-l-m-2)!} 	\right)^{1/4} 	\\ 	&\indeqtimes 	\left( \Vert \partial^{\beta} (\epsilon \partial_t)^m \tilde{E} \Vert_{L^2}  	\frac{\lambda^{(m-3)_+} \tau_0^{(l+m-3)_+}}{(l+m-3)!} 	\right) 	\mathbbm{1}_{\{[(j+n)/2]+1 \leq l+m \leq j+n\}} 	, 	\end{split}    \llabel{g E yPH bI2 1A8 rK 3FJt 1q8B b7SGo2 Cs e eiI tYT jq2 V5 PpCq 6IYG dMyhSG qC X wgg d0c Hh5 zb b5gK LLHb 3F117F NZ T qA1 9qD YXS I1 p5FH 2awj azGAub 3B j KfS AJa EkE XT Vb0u e2bo DjfYUQ 59 u PnU Ldu eeM nv m4PS sOOc NNweXS vL 1 C0y 4Td KRN cF oiKo WeMK gvTFeh Jp W hvJ 89B L2e 6u DvHB dPJV TE5eAr Pw D NCh 3FZ GmH tf V1mA 41UZ Jz6VFh 0s O BNu INC Veq qq JlTJ HswC JasIgg wZ k sq5 AYo 8Uy GP YIRs g18Q rvzsET to S ID2 g1u pY1 3q zGSS MrfG Y9rGHG jq x 72J afO xsh UC 2fqL h9TL PkUepr nO W Uqf I92 pTl t1 M34f Q3Ks uJYaFZ J5 i lYe tbv 1Ly Xu lBm6 zWjO XdRhj6 V0 u LKS GbV bLl hS yisD 6tDU z9fVH2 uB l Vwt x8H sKZ I6 v1uG w6gZ 21pthw 5g E M7H eSO nq2 lY IAdH 9Lrl NXx5QC 4Q V 6F0 hrQ Gui js vbkB snXQ xhE9BX Xq Q GtL RUs p9s Vl 6ZG7 9ClP bdomO2 nV 2 23e Anc xQz 7Y C8cD UFRo 5oAcNO si X KZ9 dIF Rb9 ZB SeG8 TeyD m4gqCd vP d 8JQ SGD R6d wT ZsCy g3sm 7EzJVR z6 c Xa5 LTg CEK Ig 2x9K xctS pS3OiS gw L wii 1bI OuY zq aSUe ywbj ArIRQA Ce r kSv wxc do1 24 XlBp lnjF wl4oi4 k2 y IFk v5T Zn8 tl Zxx2 EKbW WvBJDe 6L A pGj GeF gvI A1 DKwz 3CqL QKg9Ix Nw 2 Ddo Khv RgE Ok g9Ll 9diM Fc34WW mR C 4k5 e92 dnk fA roO2 FVrR BEH7qh Q0 W QsW ybm M0n 13 aoAI NKon WDZPkb HS D TjN 4Xz OmC VH bbUv BDb5 BLjvnM YS n qf1 p9e UFj hq ZwFf wlew j6jtTs C6 g Soz OZ0 0JA 4Q 5wgM 5mHx pyNazR RS Y 87k Gqk u5DFGRTHVBSDFRGDFGNCVBSDFGDHDFHDFNCVBDSFGSDFGDSFBDVNCXVBSDFGSDFHDFGHDFTSADFASDFSADFASDXCVZXVSDGHFDGHVBCX158} 	\end{align} and thus 	\begin{align} 	\begin{split} 	J_{2n} 	& 	\leq 	C\lambda    \sum_{j=0}^\infty \sum_{|\alpha| = j} 	 \Vert \partial^\alpha (\epsilon \partial_t)^n u \Vert_{L^2} \frac{\lambda^{(n-3)_+} \tau_0^{(j+n-3)_+}}{(j+n-3)!} 	\\ 	&\indeq 	+ 	C \lambda 	\sum_{m=0}^n 	\left( \sum_{j=0}^\infty \sum_{|\alpha| = j, m+j \geq 1}  	\Vert \partial^\alpha (\epsilon \partial_t)^m \tilde{E} \Vert_{L^2} \frac{\lambda^{(m-3)_+} \tau_0^{(j+m-3)_+}}{(j+m-3)!}  	\right) 	\\ 	&\indeqtimes 	\left(	\sum_{j=0}^\infty \sum_{|\alpha| = j}  	\Vert \partial^\alpha (\epsilon \partial_t)^{n-m} u \Vert_{L^2} \frac{\lambda^{(n-m-3)_+} \tau_0^{(j+n-m-3)_+}}{(j+n-m-3)!} 	\right) 	. 	\end{split} 	\llabel{M 2K Ux25 s1ZP f51Bev Fa L eOp dDi OLX AR UzWz FaRw m7v0mE vN U UdT ZrE sJL R6 MNKq S4NJ n2MzGi kZ M 244 Tib Xqm Hq xmt9 Ahci v45lz8 UG z ZNl U9L 7LT dw 1H7D BeGe jqRBAh sT n 0r0 Hxt l0D zx VuP7 cyPS NqHEjY 3T g mqN tEd N3U xK I8XS nGSQ PbpsLj dR H TWs Dzx LCm 7O XWip v9dw VVQpDS pT Z SEM 40J yUr Df vUci SvZj OYQZip Jm b m70 63m hmd Xg Ye7E ILjZ PsrLR8 YZ t Uvd CVy uXE Gf IXqm BXeq 2WpVnD Bd G xPP pZ1 4Fp fI zorT LNz0 Pf7imA 7N i aL6 qeU 2j1 33 ajyr YVL2 lEQD1V Zt m I5d tKs 4bl wi aSti 2X4G 98PGKy lU k kaw GJx PBF 8U pNBZ v5TF O1pvAP CN O CHh yJy 06R au LDUX PEOJ 6JEbGE A0 4 BKF kjX WHf Rs cKNM WcH2 XjNNQz X3 W s1p oad uEO fH kSu2 H5Dg 5AYb9m 4r V hPE fGM J3Z ah yqOf qDOs YVVXom w4 i KBo enG Lls X9 pQ8E hCI4 CltiaQ xs H Yn1 AYF SMP mi 7aeD U3Lz 0OXgkd 6g T 0kY 1Ld fBK ky BdN2 rlu1 g5YPjA Kk o W7y g2O d3M zS QOdN 6Yq1 Us6ZCT 55 7 bEs xKO cup w8 Ennf vqfy fyOw9i eo s blJ YtK zkC Vy mr45 8FJg EmgXve Fa t iuX 2Yz GxL Rl fCzY ifPa tjzVpw OB W yUH UTX GWo Un 0B4g 3DFJ vsOwde g6 I CyR Khy CDd P7 mrPc 2c8B O8zrC0 Ud o X1G LRP lIY Si 3zmc dHeF 4Hg1Gg aQ S y46 W5c 6Ei hb SJsX XYX5 dYVd4P 2G 1 Zrx LpM Hpp b2 eZ4c gprW NpngPZ 6N u zpA XFs Xip io laSo FFQV NBu8Kv uC z fq1 qrW 2cl Cu CSAj 9OcI rx08zY Fr G Mvs Nl5 jYd Mk VYqT SGx7 onMKDw JU R Ji0 ByP sV9 NC FTpw CQq9DFGRTHVBSDFRGDFGNCVBSDFGDHDFHDFNCVBDSFGSDFGDSFBDVNCXVBSDFGSDFHDFGHDFTSADFASDFSADFASDXCVZXVSDGHFDGHVBCX533} 	\end{align} \colb Summing the above estimate in $n$ from $3$ to $k$	, we obtain 	\begin{align} 	\begin{split} 	\sum_{n=3}^k J_{2n} 	& 	\leq 	C\lambda 	\sum_{n=3}^k \sum_{j=0}^\infty \sum_{|\alpha| = j} 	 \Vert \partial^\alpha (\epsilon \partial_t)^n u \Vert_{L^2} \frac{\lambda^{(n-3)_+} \tau_0^{(j+n-3)_+}}{(j+n-3)!} 	 \\ 	 &\indeq 	 + 	 C\lambda 	  \sum_{n=3}^k \sum_{m=0}^n 	\left( \sum_{j=0}^\infty \sum_{|\alpha| = j, m+j \geq 1}  	\Vert \partial^\alpha (\epsilon \partial_t)^m \tilde{E} \Vert_{L^2} \frac{\lambda^{(m-3)_+} \tau_0^{(j+m-3)_+}}{(j+m-3)!}  	\right) 	\\ 	&\indeqtimes 	\left(	\sum_{j=0}^\infty \sum_{|\alpha| = j}  	\Vert \partial^\alpha (\epsilon \partial_t)^{n-m} u \Vert_{L^2} \frac{\lambda^{(n-m-3)_+} \tau_0^{(j+n-m-3)_+}}{(j+n-m-3)!} 	\right) 	\\ 	& 	\leq 	C\lambda \Vert u \Vert_{A_k(\tau_0)} 	+  	C\lambda \Vert \tilde{E} \Vert_{A_k(\tau_0)} ( \Vert u \Vert_{A_k(\tau_0)} + \Vert u \Vert_{L^2} ) 	, 	\end{split}         \label{DFGRTHVBSDFRGDFGNCVBSDFGDHDFHDFNCVBDSFGSDFGDSFBDVNCXVBSDFGSDFHDFGHDFTSADFASDFSADFASDXCVZXVSDGHFDGHVBCX162} 	\end{align} where we used the $A_k(\tau_0)$ norm in \eqref{DFGRTHVBSDFRGDFGNCVBSDFGDHDFHDFNCVBDSFGSDFGDSFBDVNCXVBSDFGSDFHDFGHDFTSADFASDFSADFASDXCVZXVSDGHFDGHVBCX574}. The first term on the right side of above can be estimated by $1/4$, for sufficiently small $\lambda = 1/Q(\Gamma)$. For the second term of the right-hand side of \eqref{DFGRTHVBSDFRGDFGNCVBSDFGDHDFHDFNCVBDSFGSDFGDSFBDVNCXVBSDFGSDFHDFGHDFTSADFASDFSADFASDXCVZXVSDGHFDGHVBCX162},  it is easy to check that the product rules in Lemmas~\ref{L03} and~\ref{L05} hold for the norm $A_k (\tau_0)$,  and the function $Q$ in Lemma~\ref{L05} is independent of~$k$. Therefore, from \eqref{DFGRTHVBSDFRGDFGNCVBSDFGDHDFHDFNCVBDSFGSDFGDSFBDVNCXVBSDFGSDFHDFGHDFTSADFASDFSADFASDXCVZXVSDGHFDGHVBCX162} and the inductive hypothesis \eqref{DFGRTHVBSDFRGDFGNCVBSDFGDHDFHDFNCVBDSFGSDFGDSFBDVNCXVBSDFGSDFHDFGHDFTSADFASDFSADFASDXCVZXVSDGHFDGHVBCX510}--\eqref{DFGRTHVBSDFRGDFGNCVBSDFGDHDFHDFNCVBDSFGSDFGDSFBDVNCXVBSDFGSDFHDFGHDFTSADFASDFSADFASDXCVZXVSDGHFDGHVBCX511} for $k$, we obtain 	\begin{align} 	\begin{split} 	\sum_{n=3}^k J_{2n} 	& 	\leq 	\frac{1}{4} 	+ 	C\lambda 	\Vert \tilde{E} \Vert_{A_k (\tau_0)} (\Vert u \Vert_{A_k (\tau_0)} + \Vert u \Vert_{L^2} )      \\      &      \leq      \frac{1}{4}      +       \lambda       Q(\Vert u \Vert_{A_k(\tau_0)} + \Vert u \Vert_{L^2}, \Vert S \Vert_{A_k(\tau_0)} + \Vert S \Vert_{L^2} )      \leq      \frac{1}{2}.      \end{split} 	\label{DFGRTHVBSDFRGDFGNCVBSDFGDHDFHDFNCVBDSFGSDFGDSFBDVNCXVBSDFGSDFHDFGHDFTSADFASDFSADFASDXCVZXVSDGHFDGHVBCX537} 	\end{align} Finally, combining \eqref{DFGRTHVBSDFRGDFGNCVBSDFGDHDFHDFNCVBDSFGSDFGDSFBDVNCXVBSDFGSDFHDFGHDFTSADFASDFSADFASDXCVZXVSDGHFDGHVBCX538}, \eqref{DFGRTHVBSDFRGDFGNCVBSDFGDHDFHDFNCVBDSFGSDFGDSFBDVNCXVBSDFGSDFHDFGHDFTSADFASDFSADFASDXCVZXVSDGHFDGHVBCX580}, and \eqref{DFGRTHVBSDFRGDFGNCVBSDFGDHDFHDFNCVBDSFGSDFGDSFBDVNCXVBSDFGSDFHDFGHDFTSADFASDFSADFASDXCVZXVSDGHFDGHVBCX537}, 	\begin{align} 	\begin{split}         & 	\sum_{n=4}^{k+1} \sum_{j=0}^\infty \sum_{\vert \alpha \vert = j}  	\Vert \partial^\alpha (\epsilon \partial_t)^{n} u \Vert_{L^2} 	\frac{\lambda^{n-3} \tau_0^{j+n-3}}{(j+n-3)!} 	\\&\indeq 	= 	\sum_{n=3}^{k} \sum_{j=0}^\infty \sum_{\vert \alpha \vert = j}  	\Vert \partial^\alpha (\epsilon \partial_t)^{n+1} u \Vert_{L^2} 	\frac{\lambda^{n-2} \tau_0^{j+n-2}}{(j+n-2)!} 	\leq 	\sum_{n=3}^k J_{1n} + \sum_{n=3}^k J_{2n} 	\leq 	1, 	\end{split}    \llabel{ RaPJdH cq E yr6 bTj jpP Uj qm4g FE6f GDVPeu lT R vG3 Mrk yR2 89 7brm pAPL W38JJY Qv H den g1p ECE ZE kqEx 8sHg LBfpXE v7 M Ewd XUN nml B6 gLKV ZN8b Wjemgx P9 n 02L 9YE N4B iU WyHj rZoF 7kCwhc O5 j Ghk HC7 Ti4 Lq 9Myq F852 Av2TYG Ii b 4em zWf yPL 3F KDWq hj0Z 4lGsn6 Xr l zLn tuX afe Ps W4Rt 7fMU JP7lfJ Dl W sce jNc YXt Ba Jijv nzQ8 IDkhlZ vd v 2vM Rv8 1aq 9N 4kfF SAnP bPJqqc Gq 0 uEu 1aR TU9 Jl aKjs YTDM 9WOW43 CV h bNq 4ML eQi Xu e5W1 5r5u l8dDuw OR 9 gEP WpV fLk p1 nZ5l XeLT wLMnHj BZ q Bzc Tsn HTv 2M jMoc Pc8B NlonaA B0 3 hmK sij dW9 NL bFXt zlfq mUDugM rD E QCJ FHX Upa kY oL9b gV77 89dp3e 4I H qx6 MCX npW sn Rdxi OSwv s39gDd Tn R IzE IUj yOZ QL Sibt 4qQC p9OmeI nv 4 bnY B8e KEe FN zMMn ZbtP wFYB1Q FJ Y OEB QKv HPT LF fGCY LUu2 3t2XEV Rh U Qht rIl 26u jl 3k35 TpgK TSo181 3h t Ruw CRE TWO DQ eLZ5 9JHA LacWpY PR Z lzi iXE wiZ EC 6bpp Kui7 c423I4 b6 j Fbp fbA ydI X9 ttpk 0MlP RlxSv4 D3 p WgQ dUP cj9 Js xKmW 21bd gs7MKU uW v 8AD MQ9 Ftj zz JNlp 9UdX 7WfgMD oG Z 7F8 bWO GgZ Kf 9pgG kcZh pj0bKC rP r B7D cLE nMq pe XtC2 MBE2 DyrKmV gi 5 Xqv oWY snb sp 4JIl 3uGK a5ug2V G3 G zRj ixg NE4 bE rl7P Y2xt knCJC3 SY Y 0CQ V9Y 14e cQ NFq4 u3cy qHxH7r hO D jdk 6Ta 6Si YK bD6Y uWDU dwlWWR pJ x rDp UNv GMl 8X 8Zj8 7aGT uM8h6u eZ 2 hO6 eYz mYT 5Q 9OJT jR8z JEY943 T6 Y kDFGRTHVBSDFRGDFGNCVBSDFGDHDFHDFNCVBDSFGSDFGDSFBDVNCXVBSDFGSDFHDFGHDFTSADFASDFSADFASDXCVZXVSDGHFDGHVBCX198} 	\end{align} concluding the proof of \eqref{DFGRTHVBSDFRGDFGNCVBSDFGDHDFHDFNCVBDSFGSDFGDSFBDVNCXVBSDFGSDFHDFGHDFTSADFASDFSADFASDXCVZXVSDGHFDGHVBCX510} for $k+1$. \par \startnewsection{The Mach limit in a Gevrey norm}{secgevrey}  Theorem~\ref{T01} shows that if the initial data is analytic, then the Mach limit holds in an analytic norm. In this section, we show that if, more generally, the initial data is Gevrey, then the Mach limit holds in the Gevrey norm. \par Thus, assume the initial data is Gevrey regular that satisfies 	\begin{align} 	\sum_{m=0}^\infty \sum_{|\alpha| = m}  	\Vert \partial^\alpha (p_0^\epsilon, v_0^\epsilon, S_0^\epsilon) \Vert_{L^2}  	\frac{\tau_0^{(m-3)_+}}{(m-3)!^s}	 	\leq M_0, 	\label{DFGRTHVBSDFRGDFGNCVBSDFGDHDFHDFNCVBDSFGSDFGDSFBDVNCXVBSDFGSDFHDFGHDFTSADFASDFSADFASDXCVZXVSDGHFDGHVBCX700} 	\end{align} where $s\geq 1$ is the Gevrey index and , $\tau_0,M_0>0$ are fixed constants.  Note that when $s=1$ we recover the class of real-analytic functions.  Also, for the Sobolev regularity, we assume that we have~\eqref{DFGRTHVBSDFRGDFGNCVBSDFGDHDFHDFNCVBDSFGSDFGDSFBDVNCXVBSDFGSDFHDFGHDFTSADFASDFSADFASDXCVZXVSDGHFDGHVBCX540}. \par Similarly to \eqref{DFGRTHVBSDFRGDFGNCVBSDFGDHDFHDFNCVBDSFGSDFGDSFBDVNCXVBSDFGSDFHDFGHDFTSADFASDFSADFASDXCVZXVSDGHFDGHVBCX366}, we define the mixed weighted Gevrey norm 	\begin{align} 	\Vert u \Vert_{G(\tau)}  	= 	\sum_{m=1}^\infty \sum_{j=0}^m \sum_{|\alpha| = j}  	\Vert \partial^\alpha (\epsilon \partial_t)^{m-j} u \Vert_{L^2}  	\frac{\kappa^{(j-3)_+} \tau(t)^{(m-3)_+}}{(m-3)!^s} 	, 	\llabel{8h dmg WX4 GC cAcQ kn06 SLj9Ls HL G mqH BaQ GaM fK F6Fi Zzzy UtEzRl uY T PjO Qba cMo Mo vPNI cZRL JkCH6J aM d cad qto pHK RL vHJG UsK6 ayTC8W oK i 7fg THv dgf 6d mqJd Cq3h eZusxu eJ L t0X J50 HUJ qG 6nsC lpmo uI1wSv ra i BiS ewQ tdv 5Q 0hPz hV71 OVjIB9 2O t pZh HMI hLk xh 6uQT 6DnK eTkF7D LM W zjS a31 dSZ ma roX6 E5we eiej4p tq N mPE Ap8 dMU v8 30uE vUFF CVsurQ G7 G IA5 KGC lzG Gw OYcg iniQ HMad3O vU Q mUP 2No sjI cx qZbb cRp6 ewtgwJ MN 9 6dC erk 9DP RZ 6vce 09MJ ksrHQF Lm d bVj 4At uWY lV EsQ9 i2FV RcOZ6g SI m t2r nzC 6Th wB FIYc cJM8 q7RZTt OM d Vf2 oIA kFK Dw UUUP BtZw 4OPtIY YT J WMA BoA 8Nd FQ M59U ED1f 0FX6RP X1 l 2sl Zb1 RhT Iy iLRq 3Dx4 rSOTpT nB s lhn hYy QQJ rT qcJO IwXm Xrs7Je ic V X9C CKs pSn 0z 8nP2 TyFA 4JRatx O7 E 9Gp 5JB VgC aI 39Nc djZY P05KVz 6f i hIG RV2 ORg Cl V8He KhnD ZPCqp4 uY P Qqb LIU g6M D6 o0xw bokp zeb6fw BV u yDh TB3 0Il 6f Q7rk wGQU yyDV6A bO a ShW PT2 gJd b5 4ssc JuD1 zHnLlD Z1 o CpA boa Chl gN yFOD NlVg pVGXo9 bZ Q x3b L5Z hFy xX PaQ6 aYJK BSIKxj Bz i j4t CDv Hxy A9 rvLx lQq2 0UaIWA wg L N3v qPb QOp pe LxWM yLaH NPpUWz qi 6 Rqg gUi Qx1 3C Syqz V4yh MTxnf6 UL b 3ue Y85 djo GL KJ4K DCdt 1uYd3E rE V FQJ lPo exK iE Xdrc VLgi Xlg0YW r9 2 KjQ MQi hzX 3W o37y aBzK MebZh0 F6 o 8vS Dbt 8Rg PJ QmnP wDy8 IXn6wL bn 1 Hyb 0qW 42F vH DFGRTHVBSDFRGDFGNCVBSDFGDHDFHDFNCVBDSFGSDFGDSFBDVNCXVBSDFGSDFHDFGHDFTSADFASDFSADFASDXCVZXVSDGHFDGHVBCX701} 	\end{align} where $\tau \in (0,1]$ represents the mixed space-time Gevrey radius and $\kappa \in (0,1]$ is a fixed parameter depending on $M_0$. Proceeding as in Section~\ref{secinitial}, we can prove that with $\kappa =1 $  we have 	\begin{align} 	\Vert (p_0^\epsilon, v_0^\epsilon, S_0^\epsilon) \Vert_{G(\tilde{\tau}_0)} \leq Q(M_0) 	\label{DFGRTHVBSDFRGDFGNCVBSDFGDHDFHDFNCVBDSFGSDFGDSFBDVNCXVBSDFGSDFHDFGHDFTSADFASDFSADFASDXCVZXVSDGHFDGHVBCX702} 	\end{align}	  for some $\tilde{\tau}_0>0$ depending on $\tau_0$ and $M_0$. Thus \eqref{DFGRTHVBSDFRGDFGNCVBSDFGDHDFHDFNCVBDSFGSDFGDSFBDVNCXVBSDFGSDFHDFGHDFTSADFASDFSADFASDXCVZXVSDGHFDGHVBCX702} holds for any $\kappa \in (0,1]$ as it is an increasing function of $\kappa$. We also define the analyticity radius function as 	\begin{align} 	\tau(t) = \tau(0) - Kt 	, 	\label{DFGRTHVBSDFRGDFGNCVBSDFGDHDFHDFNCVBDSFGSDFGDSFBDVNCXVBSDFGSDFHDFGHDFTSADFASDFSADFASDXCVZXVSDGHFDGHVBCX703} 	\end{align} where $\tau(0) \leq \min \{\tilde{\tau}_0, 1\}$ is a sufficiently small parameter, and $K\geq 1$ is a sufficiently large parameter depending on $M_0$. We shall work on the time interval $[0, T_0]$ where $T_0>0$ respects \eqref{DFGRTHVBSDFRGDFGNCVBSDFGDHDFHDFNCVBDSFGSDFGDSFBDVNCXVBSDFGSDFHDFGHDFTSADFASDFSADFASDXCVZXVSDGHFDGHVBCX148} and Remark~\ref{R04}. \par The first theorem generalizes Theorem~\ref{T01} by showing uniform boundedness in the Gevrey norms. \par \cole \begin{Theorem} \label{T03} Assume that the initial data $(p_0^\epsilon, v_0^\epsilon, S_0^\epsilon) $ satisfies \eqref{DFGRTHVBSDFRGDFGNCVBSDFGDHDFHDFNCVBDSFGSDFGDSFBDVNCXVBSDFGSDFHDFGHDFTSADFASDFSADFASDXCVZXVSDGHFDGHVBCX540} and \eqref{DFGRTHVBSDFRGDFGNCVBSDFGDHDFHDFNCVBDSFGSDFGDSFBDVNCXVBSDFGSDFHDFGHDFTSADFASDFSADFASDXCVZXVSDGHFDGHVBCX700}, where $s\geq 1$ and  $\tau_0,M_0>0$. There exist sufficiently small constants $\kappa, \tau(0), \epsilon_0,T_0>0$, depending on $\tau_0$, $s$, and $M_0$, such that   \begin{align}    \Vert (p^\epsilon, v^\epsilon, S^\epsilon)(t) \Vert_{G(\tau)}    \leq    M    \comma 0<\epsilon\leq \epsilon_0    \commaone t\in[0,T_0]    ,    \label{DFGRTHVBSDFRGDFGNCVBSDFGDHDFHDFNCVBDSFGSDFGDSFBDVNCXVBSDFGSDFHDFGHDFTSADFASDFSADFASDXCVZXVSDGHFDGHVBCX704}   \end{align} where $\tau$ is as in \eqref{DFGRTHVBSDFRGDFGNCVBSDFGDHDFHDFNCVBDSFGSDFGDSFBDVNCXVBSDFGSDFHDFGHDFTSADFASDFSADFASDXCVZXVSDGHFDGHVBCX703} and  $\KK$ and $M$ are sufficiently large constants depending on $s$ and $M_0$. \end{Theorem} \colb \par \begin{proof}[Proof of Theorem~\ref{T03}] We proceed exactly as in Sections~\ref{sec03}--\ref{sec05}, obtaining the a~priori estimates analogous to \eqref{DFGRTHVBSDFRGDFGNCVBSDFGDHDFHDFNCVBDSFGSDFGDSFBDVNCXVBSDFGSDFHDFGHDFTSADFASDFSADFASDXCVZXVSDGHFDGHVBCX56}. Then we use a similar argument as in Section~\ref{sec02} to prove \eqref{DFGRTHVBSDFRGDFGNCVBSDFGDHDFHDFNCVBDSFGSDFGDSFBDVNCXVBSDFGSDFHDFGHDFTSADFASDFSADFASDXCVZXVSDGHFDGHVBCX704} We omit further details. \end{proof} \par Similarly to \eqref{DFGRTHVBSDFRGDFGNCVBSDFGDHDFHDFNCVBDSFGSDFGDSFBDVNCXVBSDFGSDFHDFGHDFTSADFASDFSADFASDXCVZXVSDGHFDGHVBCX169}, we introduce the spatial Gevrey norm 	\begin{align} 	\Vert u \Vert_{Y_\delta} 	= 	\sum_{m=1}^\infty \sum_{|\alpha|= m}  	\Vert \partial^\alpha u \Vert_{L^2} 	\frac{\delta^{(m-3)_+}}{(m-3)!^s} 	,    \llabel{v8cy r3So tvuHC1 IF S u6S RBs bwV UK 9df2 55x6 qoaLts 6d w 6vD rAY I2e 4J cdZK 4fOe 7rDpZe E0 M fgm lIu wpO rq bksV W9r7 Qz8vn4 I1 l BBC uCs dhP SW 9XBz ikz1 VvD70I Mw O M7v FxI eL1 PT ZJHP 2fPg 0x7ZJ2 VA B ZVB MRb ici ZW MKAo ZQoL Yqd787 31 j 0i4 ahP juX u6 TGRv F3HN WgEh3B 1P 6 J30 nfG QHr of Yz7y EiDW fFeFcH Pn E xw6 MaR zfH br VLS4 aPxS IIDNri FQ S GCS QJZ KWy xQ z1DC p64M iUWcjr Y1 K BxW kBx L8y vM 269m GGng bAXMSe Gd U 1JD 5AR Kwj rc T3AM yhEz BV4k5f 6O b gO3 AHW VTJ Kb QBJq h0L7 faNewo ei 6 bFk LHc EDM lb Bxo7 14Xq w8HpUh d2 w lC2 yTQ v45 3Z 68DD j3MA nlw3OC UD Q X4K yxv RUm vo K3QD PYzg C0ebBJ 75 C Etk Vka VyE 7g JEzX Etx5 MwGhh0 6w H lNx ZYY Dod Zb x0WW 5Dnk R5yn9s Vo n N0i ADK GK2 19 KiIU Ox6L f8gbG6 Fq R HKd ZoE BNA yI W3Sl c10z yYUnsl XP 6 C6m Siy rWZ 1B X1rN wTNh xvqrSj Di r 7he lkG pwP hy uS0J D7iS XkfdY8 Kr 8 CjQ sQZ ZJA qm CPis hy9l awd7eA pO 9 eKa jRe Pat qJ CEMn gj96 deozm8 jf 4 nvc yWW mHu Jr ixxC 14wJ N8AYmK um o qT5 8ve Azq Dd K9ic n4ta 1bryDP fu Q KK5 aVJ hJS Pk 3gsO Gir9 9vTdAt FK G SNr JqG x1X Hr tGVe QCsG tfmIOU jS X YLz VfM Fwq fa cKOB 8vXS AcuObi By P RxX V92 ZNp Uq Eh6y lwET exGxsC If W ZtB qhk HzK K4 vY6q C8Hv 8nINvQ v7 z EHF 3bp DpD Ly dbWi JpZN HwFTjG sk J oAk gpx mqP fK bK0F S6vl gc5tgh N9 y Nlz Wbj P6j nz X5zK zOdP XtpDDFGRTHVBSDFRGDFGNCVBSDFGDHDFHDFNCVBDSFGSDFGDSFBDVNCXVBSDFGSDFHDFGHDFTSADFASDFSADFASDXCVZXVSDGHFDGHVBCX81} 	\end{align} where $\delta>0$ is as in \eqref{DFGRTHVBSDFRGDFGNCVBSDFGDHDFHDFNCVBDSFGSDFGDSFBDVNCXVBSDFGSDFHDFGHDFTSADFASDFSADFASDXCVZXVSDGHFDGHVBCX205}.  \par The next theorem provides convergence of the solution in \eqref{T03} to the corresponding incompressible Euler equation in the Gevrey space. \par \cole \begin{Theorem} \label{T04} Let $\delta>0$ be as in \eqref{DFGRTHVBSDFRGDFGNCVBSDFGDHDFHDFNCVBDSFGSDFGDSFBDVNCXVBSDFGSDFHDFGHDFTSADFASDFSADFASDXCVZXVSDGHFDGHVBCX205},  and assume that the initial data $(v_0^\epsilon, S_0^\epsilon)$ converges to  $(v_0, S_0)$ in $Y_\delta$ and in $L^2$ as $\epsilon \to 0$, and $S_0^\epsilon$ decays sufficiently rapidly at infinity in the sense    \begin{align}    \vert S_0^\epsilon (x) \vert           \leq          C \vert x \vert^{-1-\zeta},        \indeq\indeq         |\nabla S_0^\epsilon (x) \vert                 \leq          C |x \vert^{-2-\zeta},    \llabel{5H sO e KGH D4K rRU OM r34S 7i2A 2nBldq ZP O w4C fBf Pxk hT hKQH OISH 83HhCN E7 p blM zbE ONf R5 JyOS IxvE z7hArh Wc O Dlk cav wVC LJ TkaX kGyM UH0k1v yt L kTm IdB khF wX f03Z eiNU gtLyyF Uc d wGO dEI ftN GK 1hPn C2gs 9xOWgT uI K Aqz ewx s4p q4 GCmf G9JA WdIFWf e7 F A7Q KUO w8h 5U ofdS nNNb pLwqXo I5 K S54 QEF keP 2o 54uU qFq7 nsqSUV 38 V R5z wrX ZrZ kC RRE8 zQZf vVfSv1 Se I 8qW i1r g7H J2 GUIK kAdI rcaiGk Yq n W0B 7LO 6hr n0 dEoc 03Rw kRwldl lf a 5y9 dCR KS7 QA BR45 kHX2 qIT1Jd Cw H rvs 7Eg GkV 1S 0faL O9yU AAU30z sZ Z zyW rAF FNE dp VlRr cCJQ XM4V1k Qi w aRv 7A2 xyK tP hNiX 1GkQ anMzA0 t5 R CyQ fIf CkU Wo hCNk rA1v h6NW3U 9X L 1Zh EXW sps Qx zprN MOY1 scb6wK kK z 5sC Yl3 v0e aO dpwH wlv0 dr0iW3 QI j LJr iIt tC0 lc IlKF 1imc C4QL0Q uw 3 AYE 31f WB4 lf 1wuW gD2D TnlA7v fb F 7xS 2vu l5l zj o8Nx 5u0p pqCCaD oZ 9 J99 QPU q5a cS daIZ zzSY bYsgpO gz u Bf9 HSe VZa 9v A3X6 rF1B 8KrQ9U vr 6 27t gvW fHA 3U 4tXs W1ud AXL9VC SE K m7R g5Y Hkz P6 zg57 fkyy 0XZstL ib T lVi r3H wxE bd yrN9 o5Z5 q8Ircg 90 O ky2 AuP Ojw aq PdvJ wjdq icEcPM ZW g CQv 1dR yjt uQ qb9A obNe hd2QON hV z 4CJ lmM 5GG Me oJ80 ddN0 FcvfYq yf 7 ydu aU0 NS3 tQ x9OC 1rGO QcinIc ts Q wxg tNg cO0 eU tYB0 v0FA hWTtXK Q3 p p4a 2f7 rlJ Th PEGT KVWG xvWKKr 06 O VFR Y0j Y9T W9 rhCs fSSj FTVCs8 Og C Rqm 7vDFGRTHVBSDFRGDFGNCVBSDFGDHDFHDFNCVBDSFGSDFGDSFBDVNCXVBSDFGSDFHDFGHDFTSADFASDFSADFASDXCVZXVSDGHFDGHVBCX49}    \end{align} for $0<\epsilon \leq \epsilon_0$ and some constants $C$ and $\zeta>0$. Then $(v^\epsilon, p^\epsilon, S^\epsilon)$ converges to  $(v^{({\rm inc})}, 0, S^{({\rm inc})})$ in $C ([0,T_0], Y_\delta)$, where $(v^{({\rm inc})}, S^{({\rm inc})})$ is the solution to \eqref{DFGRTHVBSDFRGDFGNCVBSDFGDHDFHDFNCVBDSFGSDFGDSFBDVNCXVBSDFGSDFHDFGHDFTSADFASDFSADFASDXCVZXVSDGHFDGHVBCX202}--\eqref{DFGRTHVBSDFRGDFGNCVBSDFGDHDFHDFNCVBDSFGSDFGDSFBDVNCXVBSDFGSDFHDFGHDFTSADFASDFSADFASDXCVZXVSDGHFDGHVBCX204} with the initial data $(w_0, S_0)$, and $w_0$ is the unique solution of    \begin{align}    &\dive w_0 = 0,     \llabel{E nZP Qc DlIv NiPm cJoxnj Hb R 2gp iUh FQo dy 24tB 6aED PmZIKa gR p rnN CTb yp3 7i 924c S6Ku DUsCRM k1 2 QMa twj JLz Vm hQ87 Czcv IeiXxX D6 c I6D Q51 omp 9u 0LJT 6I0V tanJup ws I vde oIo hAy 3k HP4C 8KrP IEbcDS b0 I eZ3 fJd EN7 AL PpKU zqug HWekJ6 FU h efB 0x6 Hw0 Fw 7uLE ZtNF T9yMaQ hd E d18 m78 PSm AK QATN NlX1 zHQ5sc Oj e 5RW NYz 4Js vE tiZI BM0L 3lnAof XZ W kQe FgX K1E DH xdnO 6TOh xVgp3y BS J Jtd Y3c l3V MB 6JRR k94S 1iIAyh rx K WyX 2Ap rnJ iT 4cVJ zTju o8WspZ OI c 9t1 5lv hzY 08 Dddw fJND eJ45cJ Au 2 37x NoR CZ4 K2 i6n6 bhkj fupin0 nu a 2mm H4w 9RN tg PRJ3 v2ou IjSSQB Vz b 4ZQ WAz DJr Fz QNm4 6muQ mOCZby jV O Hdj p3S DZ6 Xe TvjO A6Ap djP0F3 Tw U FSn aZL HdC sv byjZ 9M7o xqLiN6 AH k RLQ j1t 13a 7y 8j1I h5fP 9nRjeN cY 0 hRs 5nL U2Y bW MvN7 m8AZ SD8AKe cK n Evl 0Dj 5Rd nr lhUk Xocp zs9MNa uX s EVv rkh ZBm kA yb7v yffi P7UpTl 5t z dMR pZG vog rT ibOe C6Fp LJilJ4 Tc o 6hI IUY jNt w5 0fWM zyj3 6MXb0g I4 6 1MK nc5 ma6 4n TEl2 Toow H395FF uQ O oqy yGT WTH aJ m2dz BKx5 CbGrFn 4o l yYC fOI k2u pm hYkV Eczi qLI1pP AC 1 tYS TSk 05S GU 5IfR PMG7 C747Vz wC E XBI VAN Upd na MiuI Ykp5 qr5V5q gU 0 APh wwU crW qw IkxZ Zxbp wZbt4u dA 9 rdE l7I pYu g7 sIjz 0fJX j97vO6 Ty 5 QE1 BFo WIS dH BIjg XsOW QN62qC UF S UAK YGh CpE Vo 6Tmi 592Z e8BQwk vZ V ejh yEy ICY 0e 0xKJ DFGRTHVBSDFRGDFGNCVBSDFGDHDFHDFNCVBDSFGSDFGDSFBDVNCXVBSDFGSDFHDFGHDFTSADFASDFSADFASDXCVZXVSDGHFDGHVBCX155}       \\&    \curl (r_0 w_0) = \curl (r_0 v_0)    \llabel{iFio qkK5u5 Y1 x Xrk GF7 H9O gm f2he oOiw r0Vg7K mx O Zqw ZT2 VDC Mh 2fAI JUEO ZVsVRt yt 0 04p tuq 0jU y8 hAc5 z4nv HpufkN KA x Gsh 0qG n3L oD y0Dq 7eh0 dk4Ktm ic S vy9 V4D R8E mj lT20 a9tp 4YIHU7 dj T 7Zn NRu i4z n6 7tlA 8XX7 czYuyj Cy O YHD oEU 8u2 qW AP9d Aqqp YUrNsG PV y sAs EDR 43o EG HgLM y4Dw 0TSDxa nb 1 5D4 sJ2 1uC T3 ES4e v6MZ bhhU3P AD Z IqZ COB oKX br GSoJ BqVX 2kcZnV 2l 4 CCM NEh i73 QK KjKL WYYv u4gjeT YO Y 4R4 iIh BEp Gn 2x1v IV3f 1IxAmc f4 g Bxy TkY gas rI 8v34 uN7u HsujPj Ja X S1Y fkJ 8OU 6G C3P9 TGdJ tEz8K1 MW x 5bF xvN plo y6 Qs0N 4iio oKScnQ 09 Q SEu Lgs B5w wv LjzG Tb5r dwHt2o i8 s ok1 6Vl RXw 42 GqOr OClS 27j40x qK M ggZ vlg K2C 45 Xazd aiC3 pfY4lP qo B gu9 Sqn 83A YL ZXgh r0Dr QrJt7g cy y j13 ALq UwA 31 mx8C hSxj kwHvu2 EZ Z NH6 nJP T5U ae C9WT kREA ycER9t nd o q8z Jrz Py9 T7 A5go Bxuc Vb8QDJ zM 7 ITX nGU 7sb 14 oa19 2ytl 22MdHy RA A ueG qnm H6K QF n3yz Xb6J QI5NjG hJ 0 50R 5L5 oZ9 hG x0Da PsuP LDLkAA vB L TJP 6Oc iFH DX EQQC zUtd vPG19b b2 x pk8 wwI cC3 Og 37mm OqH1 fnpAOe Wv P s34 uRe rQ5 PP yQ3C UgCK 3cTczZ ye F VMX sWm V6n qa VDuS luWo WSowFj JR 4 Jxd bgS zt3 2K oXpu ANll 87vIcu IE 0 2Ep sAt 4XQ 5z 3aCM dIpH fazZcJ pm G jdn SKI TIM B9 H5KN 0BR9 JaeRoP wr K wdu 7Ag eLS jx Czrt Ujyh Xlbhzp 1w 3 Cr4 Z3s Lsl vJ 0AsR a5jM qHz6IL yvDFGRTHVBSDFRGDFGNCVBSDFGDHDFHDFNCVBDSFGSDFGDSFBDVNCXVBSDFGSDFHDFGHDFTSADFASDFSADFASDXCVZXVSDGHFDGHVBCX156}    ,   \end{align} with $r_0 = r(S_0, 0)$. \end{Theorem} \colb \par \begin{proof}[Proof of Theorem~\ref{T04}] Theorem~\ref{T04} follows by using arguments analogous to those in  Section~\ref{sec06}. \end{proof} \par \section*{Acknowledgments} JJ was supported in part by the NSF Grant DMS-2009458 and by the Simons foundation, IK was supported in part the NSF grant DMS-1907992, while LL was supported in part by the NSF grants DMS-2009458 and DMS-1907992. \par  \end{document}